%% file: main.tex
\documentclass[10pt]{extarticle}
\usepackage[letterpaper,margin=1.0in]{geometry}
\usepackage{bm} 
\usepackage{graphicx}
\usepackage{subfig}
\usepackage{xfrac} 
\usepackage{mathtools} 
\usepackage{fancyhdr} 
\usepackage{tabulary} 
\usepackage{enumitem} 
\usepackage[colorlinks]{hyperref} 
\usepackage{float} 
\usepackage{color} 
\usepackage[titletoc,title]{appendix} 
\usepackage{amsthm}
\usepackage{amsmath,amssymb,amsfonts}
\newtheorem{thm}{Theorem}[section]

\newtheorem{cor}[thm]{Corollary}
\newtheorem{lem}[thm]{Lemma}

\usepackage{algorithm}
\usepackage{algorithmic}

\DeclareGraphicsExtensions{.eps,.pdf,.png,.jpg,.JPG,.gif}
\usepackage{etoolbox}
\patchcmd{\appendices}{\quad}{: }{}{}
\makeatletter
\newcommand{\customlabel}[2]{%
   \protected@write \@auxout {}{\string \newlabel {#1}{{#2}{\thepage}{#2}{#1}{}} }%
   \hypertarget{#1}{#2}
}
\makeatother
\newcommand{\norm}[1]{\left\lVert#1\right\rVert}
\newcommand\scalemath[2]{\scalebox{#1}{\mbox{\ensuremath{\displaystyle #2}}}}
\begin{document}

\LARGE
\begin{center}
	Observability Analysis of Aided INS with Heterogeneous Features of Points, Lines and Planes
\end{center}

\normalsize
{
	\begin{center}
		Yulin Yang - \href{yuyang@udel.edu}{yuyang@udel.edu} \\
		Guoquan Huang - \href{mailto:ghuang@udel.edu}{ghuang@udel.edu}
	\end{center}
}
\begin{center}
	Department of Mechanical Engineering \\
	University of Delaware, Delaware, USA
\end{center}

\input{sections/abstract}
\input{sections/introduction}

\input{sections/problem_formulation}

\input{sections/point}
\input{sections/line}

\input{sections/plane}

\input{sections/pt_line_plane}

\input{sections/global}
\input{sections/range_bearing}

\input{sections/simulation1}
\input{sections/simulation2}

\input{sections/conclusion}




%
%
\begin{appendices}
\input{appendix/state_transition}
\input{appendix/measurements_point}
\input{appendix/proof_projection}

\input{appendix/line_jacobian}

\input{appendix/single_line}
\input{appendix/direct_line}

\input{appendix/multiple_lines}

\input{appendix/plane_obs}
\input{appendix/multiple_planes}
\input{appendix/plane_feature}
\input{appendix/lem1}

\input{appendix/lem2}

\input{sections/lie_derivative}

\end{appendices}
\input{sections/ack}


%
%
\newpage
\addcontentsline{toc}{section}{References}

\bibliographystyle{IEEEtran}  
\bibliography{libraries/library}
\newpage

\end{document}

%% file: sections/abstract.tex


\begin{abstract}

In this paper, we perform a thorough observability analysis for linearized inertial navigation systems (INS) aided by exteroceptive range and/or bearing sensors (such as cameras, LiDAR and sonars) 
with different geometric features (points, lines and planes). 
While the observability of vision-aided INS (VINS) with point features has been extensively studied in the literature,
we analytically show that the general aided INS with point features preserves the same observability property -- that is, 4 unobservable directions, corresponding to the global yaw  and the global position of the sensor platform. 
We further prove that there are at least 5 (and 7) unobservable directions for the linearized aided INS with a single line (and plane) feature;
and, for the first time, analytically derive the unobservable subspace for the case of multiple lines/planes. 
%
Building upon this,  
we examine the system observability of the linearized aided INS with different combinations of points, lines and planes, and show that, in general, the system preserves at least 4 unobservable directions,
while if global measurements are available, as expected, some unobservable directions diminish. 
In particular, when using plane features, we propose to use a minimal, closest point (CP) representation;
and we also study in-depth the effects of 5 degenerate motions identified on  observability.
To numerically validate our analysis, we develop and evaluate both EKF-based visual-inertial SLAM  and visual-inertial odometry (VIO) using heterogeneous geometric features 
in Monte Carlo simulations.

\end{abstract}


%% file: sections/introduction.tex

\section{Introduction}
%
%
%
%

Inertial navigation systems (INS) have been widely used for providing 6 degrees-of-freedom (DOF) pose estimation \textsl{•}when navigating in 3D~\cite{Chatfield1997}. 
However, due to the noises and biases that corrupt the IMU readings, simple integration of the local angular velocity and linear acceleration measurements can cause large drifts in a short period of time, in particular, when using cheap MEMS IMUs. 
To mitigate  this issue, additional sensors (e.g., optical camera\cite{Mourikis2007ICRA,Mourikis2009TRO}, LiDAR\cite{Zhang2017ICRA,Geneva2018IROS}, and imaging sonar\cite{Yang2017icra}) are often used, i.e., aided INS.
%
Among  possible exteroceptive sensors, optical cameras -- which are low-cost and energy-efficient while providing rich environmental information -- are ideal aiding sources for INS and thus, vision-aided INS (i.e., VINS) have recently prevailed, in particular, when navigating in GPS-denied environments (e.g., indoors)~\cite{Hesch2013TRO,Hesch2014IJRR,Li2013IJRR,Leutenegger2014IJRR,Huang2014ICRA,Qin2017TRO,Mur2017ICRA}.
While many different VINS algorithms were developed in last decade, the extended Kalman filter (EKF)-based methods are still among the most popular ones, 
such as multi-state constraint Kalman filter (MSCKF) \cite{Mourikis2007ICRA}, observability-constrained (OC)-EKF~\cite{Hesch2013TRO,Huang2010IJRR}, optimal-state constraint (OSC)-EKF~\cite{Huang2015ISRR}, and right invariant error (RI)-EKF~\cite{Zhang2017ICRAa}. 

System observability plays an important role in consistent state estimation \cite{Huang2012thesis} and thus, 
significant research efforts have been devoted to the observability analysis of VINS.
For example, it has been proved in \cite{Martinelli2012TRO} that biases, velocity, and roll and pitch angles in VINS are observable;
in \cite{Hesch2013TRO,Kottas2012ISER} the null space of observability matrix (unobservable subspace) of linearized VINS was analytically derived;
and in \cite{Hesch2014IJRR,Guo2013ICRA} the Lie-derivative based nonlinear observability analysis was presented. 
However, since most of the current VINS algorithms (e.g., \cite{Hesch2013TRO,Hesch2014IJRR,Li2013IJRR,Leutenegger2014IJRR,Huang2014ICRA,Qin2017TRO})
are developed based on point features, the observability analysis is performed primarily using point measurements.
Very few have yet studied the observability properties of the aided INS with heterogeneous geometric features of points, lines and planes 
that are extracted from range and/or bearing sensor measurements. 

In this paper, 
building upon our recent conference publications~\cite{Yang2018ICRA,Yang2018IROS},
we perform a thorough  observability analysis for the linearized aided INS using points, lines, planes features and their different combinations.
In particular, we propose to use the closest point (CP) from the plane to the origin to represent plane features, 
because this parameterization is not only minimal but also can directly be used for the plane error state update in vector space.
Moreover, we perform in-depth study to identify 5 degenerate motions, which is of practical significance,
as these motions may negatively impact the system observability properties by causing more unobservable directions and thus exacerbate the VINS estimators (e.g., see \cite{Martinelli2013IROS,Wu2017ICRA}). 
The insights obtained from the observability analysis are leveraged when developing our EKF-based VINS algorithms (including visual-inertial odometry (VIO) and visual-inertial SLAM (VI-SLAM)) 
using heterogeneous geometric features, which are evaluated in simulations to validate our analysis.

In particular, the main contributions of this work include:
\begin{itemize}

	\item In the case of point features, we generalize the VINS observability analysis to encompass any type of aiding sensors (such as 3D LiDAR, 2D imaging sonar and stereo cameras) and analytically show that the same observability properties remain (i.e., 4 unobservable directions). 
	
	\item In the case of line (or plane) features, we perform observability analysis for the linearized aided INS with line (or plane) features and show that there exist at least 5 (or 7) unobservable directions for a single line (or plane) feature. Moreover, we analytically derive the unobservable subspaces for multiple lines (or planes) in the state vector, without  any assumption about features.  
	In particular, we advocate to use the closest point (CP) representation, which is simple and compact, for both plane feature parameterization and error state propagation. 
	
	\item In the case of different combinations of point, line and plane features, we show that in general there are at least 4 unobservable directions that are  analytically derived, for the linearized aided INS with heterogeneous features.
	
	\item By employing the spherical coordinates for the point feature, 
	we identify 5 degenerate motions that cause the aided INS to have more unobservable directions. 
	On the other hand, we study in-depth the effects of global measurements on the system observability, 
	and show that they, as expected, will greatly improve the observability.  
		
	\item To validate our observability analysis of linearized aided INS,
	we develop the EKF-based VI-SLAM and MSCKF-based VIO  using {\em heterogeneous} geometric features (i.e., points, lines, planes, and their different combinations) 
	and perform extensive Monte-Carlo simulations by comparing the standard and the benchmark (ideal) filters.

\end{itemize}

%

\section{Related Work}

Aided INS is a classical research topic with significant body of literature~\cite{Farrell2008}
and has recently been re-emerging in part due to the advancement of sensing and computing technologies.
In this section, we briefly review the related literature closest to this work by focusing on the vision-aided INS.

\subsection{Aided INS with Points, Lines, and Planes}

As mentioned earlier, vision-aided INS (VINS) arguably is among the most popular localization methods in particular for resource-constrained sensor platforms such as mobile devices and micro aerial vehicles (MAVs) navigating  in GPS-denied environments (e.g., see~\cite{Li2013ICRAb,Guo2014RSS,Leutenegger2014IJRR,Shen2013ICRA}).
While most current VINS algorithms focus on using point features (e.g., \cite{Hesch2013TRO,Hesch2014IJRR,Li2013IJRR,Leutenegger2014IJRR}), 
line and plane features may not be blindly discarded in structured environments~\cite{Kottas2013icra,Yu2015IROS,Zheng2018ICRA,He2018Sensors,Guo2016ICRA,Guo2013iros,Hesch2010ICRA,Kaess2015icra,Wu2017ICRA}, 
in part because: 
(i) they are ubiquitous and compact  in many urban or indoor environments (e.g., doors, walls, and stairs),
(ii) they can be detected and tracked  over a relatively long time period,
and (iii) they are more robust in texture-less environments compared to point features.

In the case of utilizing line features, 
Kottas et al. \cite{Kottas2013icra} represented the line with a quaternion and a distance scalar and studied the observability properties for linearized VINS with this line parameterization. 
Yu et al. \cite{Yu2015IROS} proposed a minimal four-parameter representation of line features for VIO using rolling-shutter cameras,
while Zheng et al. \cite{Zheng2018ICRA} used two end points to represent a line and designed point/line VIO based on MSCKF. 
Recently, He et al. \cite{He2018Sensors} employed the Pl\"ucker representation for line parameterization and orthogonal representation~\cite{Bartoli2005CVIU} for line error states, and developed a tightly-couple keyframe-based inertial-aided mono SLAM system. 
%


In the case of exploiting plane features,
Guo et al. \cite{Guo2013iros} analyzed the observability of VINS using both point and plane features, 
while assuming the plane orientation is \textit{a priori} known. 
The authors have shown that VINS with only plane bearing measurements have 12 unobservable directions but 4 if both point and plane measurements are present. 
Hesch et al.~\cite{Hesch2010ICRA} developed a  2D LiDAR-aided INS algorithm that jointly estimates the perpendicular structural planes associated with buildings, along with the IMU states.  
However, one particular  challenge of using plane features is the plane parameterization~\cite{Kaess2015icra}. 
A conventional method is to use the plane normal direction and a distance scalar, which, however, is over-parameterized,
resulting in singular information matrix in least-squares optimization if not treated carefully. 
Alternatively, one may use a spherical parametrization (two angles and one distance scalar)~\cite{Servant2008ICPR}, 
which is minimal but might suffer from singularities similar to gimbal lock for Euler angles. 
Kaess~\cite{Kaess2015icra} used a unit quaternion   for plane representation by leveraging the quaternion error states for propagation,
while  Wu et al.~\cite{Wu2017ICRA} employed both a quaternion and a distance scalar  for planes but assuming 2D quaternion error states.

\subsection{VINS Observability Analysis}

As system observability is important for consistent estimation~\cite{Huang2012thesis}, 
in our prior work~\cite{Huang2008ICRA,Huang2008ISER,Huang2010IJRR,Huang2009ICRA,Huang2013TRO,Huang2011AURO,Huang2011IROS,Huang2014ICRA}, 
we have been the first to design observability-constrained consistent estimators for robot localization and mapping problems. 
Since then, significant research efforts have been devoted to the observability analysis of VINS. 
In particular, 
in~\cite{Jones2011IJRR,Hernandez2015ICRA}  the system's indistinguishable trajectories were examined from the observability perspective. 
By employing the concept of continuous symmetries as in~\cite{Martinelli2011TRO}, 
Martinelli~\cite{Martinelli2012TRO} analytically derived the closed-form solution of VINS 
and identified that IMU biases, 3D velocity, global roll and pitch angles  are observable. 
He has also examined the effects of degenerate motion~\cite{Martinelli2013IROS}, minimum available sensors~\cite{Martinelli2014ICRA}, cooperative VIO~\cite{Martinelli2018CoRR} and unknown inputs~\cite{Martinelli2017BOOK,Martinelli2018TAC} on the system observability. 
Based on the Lie derivatives and observability matrix rank test~\cite{Hermann1977TAC}, 
Hesch et al. \cite{Hesch2014IJRR} analytically showed that the monocular VINS has 4 unobservable directions, i.e.,  the global yaw and the global position of the exteroceptive sensor. 
Guo et al. \cite{Guo2013iros} extended this method to the RGBD-camera aided INS that preserves the same unobservable directions if both point and plane measurements are available.
With the similar idea, 
in~\cite{Mirzaei2007IROS,Kelly2011IJRR,Guo2013ICRA}, the observability of IMU-camera (monocular, RGBD) calibration was analytically studied, which shows  that the extrinsic transformation between the IMU and  camera is observable given generic motions.
Additionally, 
in~\cite{Panahandeh2013IROS,Panahandeh2016JIRS}, the system with a downward-looking camera measuring point features from horizontal planes was shown to have the observable global $z$ position of the sensor.

More importantly, as in practice VINS estimators are built upon the linearized system,
it necessitates to perform  observability analysis for the linearized VINS whose observability properties can be exploited when designing an estimator.
For instance, 
Li et al. \cite{Li2012ICRA,Li2013IJRR} performed observability analysis for the linearized VINS (without considering biases) and adopted the idea of first-estimates Jacobian~\cite{Huang2008ISER} to improve filter consistency. 
Analogously, 
in \cite{Hesch2013WAFR,Hesch2013TRO,Kottas2012ISER}, the authors conducted observability analysis for the linearized VINS with full states (including IMU biases) and analytically showed the system unobservable directions by finding the right null space of the observability matrix~\cite{Hesch2013TRO}. 
Based on this analysis, the observability-constrained (OC)-VINS algorithm was developed. 
%
%
In this work, we thus primarily focus on the observability analysis of the linearized aided INS with heterogeneous features and developing estimation algorithms to validate it.

%% file: sections/problem_formulation.tex

\section{Problem Statement} \label{sec:pb}

In this section, we  describe the system and measurement models for the aided INS with different geometric features,
providing the basis for our ensuing observability  analysis.
The state vector of the aided INS 
contains the current IMU state $\mathbf{x}_{IMU}$ and the feature state $^G\mathbf{x}_{\mathbf{f}}$:
%
	\begin{align} \label{eq:state000}
	\mathbf{x}
		&=
		\begin{bmatrix}
		\mathbf{x}^{\top}_{IMU} & ^G\mathbf{x}_{\mathbf{f}}^{\top}
		\end{bmatrix}^{\top} \notag\\
		&=
		\begin{bmatrix}
		{^I_G\bar{q}^{\top}} & 
		\mathbf{b}^{\top}_{g} & 
		{^G\mathbf{V}^{\top}_{I}} & 
		\mathbf{b}^{\top}_{a}  & 
		{^G\mathbf{P}^{\top}_{I}}  & 
		{^G\mathbf{x}^{\top}_{\mathbf{f}}}
		\end{bmatrix}^{\top}
	\end{align}

In the above expressions, $^I_G\bar{q}$ is a unit quaternion represents the rotation from  the global frame $\{G\}$ to the current IMU frame $\{I\}$, whose  corresponding rotation matrix is $^I_G\mathbf{R}(\bar{q})$. $\mathbf{b}_{g}$ and $\mathbf{b}_{a}$ represent the gyroscope and accelerometer biases, respectively, while $^G\mathbf{V}_{I}$ and $^G\mathbf{P}_{I}$ denote the current IMU velocity and position in the global frame. $^G\mathbf{x}_{\mathbf{f}}$  denotes the generic features, which can be points, lines, planes or their different combinations. 
%

\subsection{IMU Propagation Model}

The IMU kinematic model is given by~\cite{Trawny2005_Q_TR}:
	\begin{align}
	{^I_G\dot{\bar{q}}(t)} &= \frac{1}{2}\boldsymbol{\Omega}\left({^I\boldsymbol{\omega}}(t)\right){^I_G\bar{q}(t)} 	\notag \\
	{^G\dot{\mathbf{P}}_I(t)} &= {^G\mathbf{V}_{I}(t)}, \quad 	{^G\dot{\mathbf{V}}_I(t)} = {^G\mathbf{a}(t)}  \notag\\
	\dot{\mathbf{b}}_g(t) &= \mathbf{n}_{wg}(t), \quad ~~~	\dot{\mathbf{b}}_a(t) = \mathbf{n}_{wa}(t)  \notag\\
	^G\dot{\mathbf{x}}_{\mathbf{f}}(t) &= \mathbf{0}_{m_{\mathbf{f}}\times 1}  
	\label{eq:imu-system}
	\end{align} 
where $\boldsymbol{\omega}$ and $\mathbf{a}$ are the angular velocity and linear acceleration, respectively. $\mathbf{n}_{wg}$ and $\mathbf{n}_{wa}$ are the zero-mean Gaussian noises driving the gyroscope and accelerometer biases. 
$m_{\mathbf{f}}$ is the dimension of $^G\mathbf{x}_{\mathbf{f}}$, and $
\scalemath{0.9}{
	\boldsymbol{\Omega}(\boldsymbol{\omega}) \triangleq
	\begin{bmatrix}
	-\lfloor \boldsymbol{\omega}\times \rfloor   &   \boldsymbol{\omega} \\
	-\boldsymbol{\omega}^{\top}  &  0
	\end{bmatrix}}$, 
$
\scalemath{0.9}{
	\lfloor\boldsymbol{\omega}\times\rfloor
	\triangleq 
	\begin{bmatrix}
	0     &        -\omega_3      &  \omega_2 \\
	\omega_3     &       0        &   -\omega_1 \\
	-\omega_2   &  \omega_1      &     0    
	\end{bmatrix}}$. 
The continuous-time linearized error-state equation is given by:\footnote{Throughout this paper
$\hat x$ is used to denote the estimate of a random variable $x$,
while $\tilde x = x-\hat x$ is  the error in this estimate.
$\mathbf 0_{m\times n}$  and $\mathbf 0_n$ denote $m\times n$ and $n \times n$
matrices of zeros, respectively, and  $\mathbf I_n$ is the identity matrix.
}
%
\begin{align}
\scalemath{1}{
	\dot{\tilde{\mathbf{x}}}(t)}
& \simeq
	\begin{bmatrix}
	\mathbf{F}_c(t)   &  \mathbf{0}_{15\times m_{\mathbf{f}}}  \\
	\mathbf{0}_{m_{\mathbf{f}}\times 15}  & \mathbf{0}_{m_{\mathbf{f}}}
	\end{bmatrix}
	\tilde{\mathbf{x}}(t) + 
	\begin{bmatrix}
	\mathbf{G}_c(t) \\
	\mathbf{0}_{m_{\mathbf{f}}\times 12}
	\end{bmatrix}
	\mathbf{n}(t)
	=:
	\mathbf{F}(t)\tilde{\mathbf{x}}(t) + \mathbf{G}(t)\mathbf{n}(t)
\end{align}
%
where $\mathbf{F}_c(t)$ and $\mathbf{G}_c(t)$ are the continuous-time error-state transition matrix and noise Jacobian matrix, respectively. $\scalemath{1}{\mathbf{n}(t)=\begin{bmatrix}\mathbf{n}^{\top}_{g}  &  \mathbf{n}^{\top}_{wg}  & \mathbf{n}^{\top}_a & \mathbf{n}^{\top}_{wa}\end{bmatrix}^{\top}}$
are modeled as zero-mean Gaussian noise with autocorrelation 
$\scalemath{1}{\mathbb{E} \left[\mathbf{n}(t)\mathbf{n}^{\top}(\tau) \right] =\mathbf{Q}_c\delta(t-\tau)}$. 
Note that $\mathbf{n}_g(t)$ and $\mathbf{n}_a(t)$ are the Gaussian noises contaminating the angular velocity and linear acceleration measurements.
The discrete-time state transition matrix $\boldsymbol{\Phi}_{(k+1,k)}$ from time $t_k$ to $t_{k+1}$, 
can be derived from $\dot{\boldsymbol{\Phi}}_{(k+1,k)} = \mathbf{F}(t_k) \boldsymbol{\Phi}_{(k+1,k)}$  
with the identity as the initial condition, which is given by~\cite{Hesch2013TRO} (see Appendix \ref{apd_state_transition}):
%
\begin{align} \label{eq:PHI-matrix}
\scalemath{0.9}{
\boldsymbol{\Phi}_{(k+1,k)} = 
\begin{bmatrix}
\boldsymbol{\Phi}_{11} & 
\boldsymbol{\Phi}_{12} & 
\mathbf{0}_3 & 
\mathbf{0}_3 & 
\mathbf{0}_3 & 
\mathbf{0}_{m_{\mathbf{f}}\times 3}
\\
\mathbf{0}_3 & 
\mathbf{I}_3 &  
\mathbf{0}_3 & 
\mathbf{0}_3 & 
\mathbf{0}_3 & 
\mathbf{0}_{m_{\mathbf{f}}\times 3}
\\
\boldsymbol{\Phi}_{31} & 
\boldsymbol{\Phi}_{32} & 
\mathbf{I}_3 & 
\boldsymbol{\Phi}_{34} & 
\mathbf{0}_3 & 
\mathbf{0}_{m_{\mathbf{f}}\times 3}
\\
\mathbf{0}_3 & 
\mathbf{0}_3 &  
\mathbf{0}_3 & 
\mathbf{I}_3 & 
\mathbf{0}_3 & 
\mathbf{0}_{m_{\mathbf{f}}\times 3}
\\
\boldsymbol{\Phi}_{51} & 
\boldsymbol{\Phi}_{52} & 
\boldsymbol{\Phi}_{53} &
\boldsymbol{\Phi}_{54} & 
\mathbf{I}_3 & 
\mathbf{0}_{m_\mathbf{f}\times 3}
\\
\mathbf{0}_{3\times m_{\mathbf{f}}} & 
\mathbf{0}_{3\times m_{\mathbf{f}}} &  
\mathbf{0}_{3\times m_{\mathbf{f}}} & 
\mathbf{0}_{3\times m_{\mathbf{f}}} & 
\mathbf{0}_{3\times m_{\mathbf{f}}} & 
\mathbf{I}_{m_{\mathbf{f}}}					
\end{bmatrix}}	
\end{align}
where
$\bm\Phi_{ij}$ is the $(i,j)$ block of this matrix, and in particular, $\bm\Phi_{54}$ will be useful for the ensuring analysis for pure translation and its expression can be analytically given by (see \cite{Hesch2013TRO,Wu2017ICRA}) with pure translation:
\begin{align} \label{eq:PHI54}
\scalemath{.9}{
    \boldsymbol{\Phi}_{54}
    =
    -\int_{t_1}^{t_k} \int_{t_1}^{t_s}
    {^G_{I_{\tau}}\mathbf{R}}\mathbf{d}{\tau}\mathbf{d}_{t_s}
    =
    -{^G_{I_1}\mathbf{R}}
    \int_{t_1}^{t_k} \int_{t_1}^{t_s}
    \mathbf{d}{\tau}\mathbf{d}_{t_s}
    =
    -\frac{1}{2}{^G_{I_1}\mathbf{R}}\delta t^2_k
    }
\end{align}
where $\delta t_k=t_{k} -t_1$ is the time elapse from the beginning.
With the state transition matrix~\eqref{eq:PHI-matrix},
we can also analytically or numerically compute the discrete-time noise covariance: 
%
\begin{align}
	\mathbf{Q}_k =
	\int_{t_k}^{t_{k+1}}\boldsymbol{\Phi}_{(k,\tau)}
	\mathbf{G}_{c}(\tau)
	\mathbf{Q}_c
	\mathbf{G}^{\top}_{c}(\tau)
	\boldsymbol{\Phi}^{\top}_{(k,\tau)}\mathbf{d}\tau
\end{align}

\subsection{Point Measurements}

Note that point measurements from different exteroceptive sensors (such as monocular/stereo camera, acoustic sonar, and LiDAR)  in the aided INS 
 generally can be modeled as range and/or bearing observations which are functions of the relative  position of the point feature expressed in the sensor frame ${^I\mathbf{P}_{\mathbf{f}}} =\begin{bmatrix} ^Ix_{\mathbf{f}} & ^Iy_{\mathbf{f}} & ^Iz_{\mathbf{f}}\end{bmatrix}^{\top}$(see~\cite{Yang2018ICRA}):\footnote{Note that  in this work we assume the frame of the aiding exteroceptive sensor coincides with the IMU frame in order to keep our analysis concise.}
%
\begin{align}\label{eq_x_meas}
	\mathbf{z}_{p} &=
	\begin{bmatrix}
	z^{(r)} \\ \mathbf{z}^{(b)}
	\end{bmatrix}
	=
	\begin{bmatrix}
	\sqrt{{^{I}\mathbf{P}_{\mathbf{f}}}^{\top}{^{I}\mathbf{P}_{\mathbf{f}}}} + n^{(r)} \\
	\mathbf{h}_{\mathbf{b}}\left({^I\mathbf{P}_{\mathbf{f}}},\mathbf{n}^{(b)}\right)
	\end{bmatrix} \\
	{\rm with}~~
	{^I\mathbf{P}_{\mathbf{f}}} &= {^I_G\mathbf{R}}\left(^G\mathbf{P}_{\mathbf{f}}-{^G\mathbf{P}_I}\right)
	\label{eq_feat}
\end{align}
where $\mathbf{h}_{\mathbf{b}}(\cdot)$ is a generic bearing measurement function whose concrete form depends on the particular sensor used (more comprehensive cases can be found in Appendix \ref{apd_meas}).
%
In~\eqref{eq_x_meas}, ${n}^{(r)}$ and $\mathbf{n}^{(b)}$ 
are zero-mean Gaussian noises (inferred from sensor raw data) for the range and bearing measurements. 
We then linearize these measurements about the current state estimate:
\begin{align} \label{eq:linearized-point-meas}
	\tilde{\mathbf z}_p 
	&\simeq 
	\mathbf H_I  \tilde{\mathbf x} + \mathbf H_n \mathbf n
	= 	
	\begin{bmatrix}
	\frac{\partial \tilde z^{(r)}}{\partial {^I\tilde{\mathbf{P}}_{\mathbf{f}}}} \frac{\partial {^I\tilde{\mathbf{P}}_{\mathbf{f}}}}{\partial  \tilde{\mathbf x}} \tilde{\mathbf x} + n^{(r)}  \\
	\frac{\partial \tilde {\mathbf z}^{(b)}}{\partial {^I\tilde{\mathbf{P}}_{\mathbf{f}}}} \frac{\partial {^I\tilde{\mathbf{P}}_{\mathbf{f}}}}{\partial  \tilde{\mathbf x}} \tilde{\mathbf x} + \frac{\partial \tilde {\mathbf z}^{(b)}}{\partial \mathbf n^{(b)}} \mathbf n^{(b)}
	\end{bmatrix}   
	=: \begin{bmatrix}
	\mathbf{H}_{r} \mathbf H_{\mathbf f} \tilde{\mathbf x}+ n^{(r)}  \\
	\mathbf{H}_{b}\mathbf H_{\mathbf f} \tilde{\mathbf x}+\mathbf{H}_{n}\mathbf{n}^{(b)}
	\end{bmatrix}
	=: \underbrace{\begin{bmatrix} \mathbf{H}_{r} \\ \mathbf{H}_{b}  \end{bmatrix}}_{\mathbf{H}_{proj}} \mathbf H_{\mathbf f} \tilde{\mathbf x}  + 
	\begin{bmatrix}  
	n^{(r)} \\ \mathbf{H}_{n}\mathbf{n}^{(b)}
	\end{bmatrix} 
\end{align}
%
%
The interested readers are referred to Appendix \ref{apd_meas} 
 for detailed derivations of these Jacobians when using different sensors.

\subsection{Line Measurements}

We propose to use the Pl{\"u}cker representation for the line feature in the state vector but the minimal \textit{orthonormal} representation for the error state,
which was introduced  in~\cite{Bartoli2005CVIU}. 
Specifically, the Pl{\"u}cker representation can be initialized by the two end points $\mathbf{P}_{1}$ and $\mathbf{P}_{2}$ of a line segment $\mathbf{L}$, as:
	\begin{equation}\label{eq_line_rep}
	\mathbf{L}
	=
	\begin{bmatrix}
	\lfloor\mathbf{P}_{1}\times\rfloor\mathbf{P}_{2} \\
	\mathbf{P}_{2}-\mathbf{P}_{1}
	\end{bmatrix}
	=
	\begin{bmatrix}
	\mathbf{n}_{L} \\ \mathbf{v}_{L}
	\end{bmatrix}
	\end{equation}
where $\mathbf{n}_{L}$ and $\mathbf{v}_{L}$ are the normal and directional vectors (which are {\em not} normalized to be unit)  for the line $\mathbf{L}$, 
which clearly is over parameterized for 4 DOF lines.
A minimal parameterization of the error state is desirable for covariance propagation and update. 
To this end,  we have:
	\begin{equation} \label{eq:def-L}
	\scalemath{0.9}{
		\mathbf{L}^\top = 
		\begin{bmatrix}
		\mathbf{n}_{L}^\top \ \  \mathbf{v}_{L}^\top
		\end{bmatrix}
		=
		\begin{bmatrix}
		\frac{\mathbf{n}_{L}}{\norm{\mathbf{n}_{L}}} &
		\frac{\mathbf{v}_{L}}{\norm{\mathbf{v}_{L}}} &
		\frac{\mathbf{n}_{L}\times \mathbf{v}_{L}}{\norm{\mathbf{n}_{L}\times\mathbf{v}_{L}}} 
		\end{bmatrix}
		\begin{bmatrix}
		\norm{\mathbf{n}_{L}} & 0 \\
		0 & \norm{\mathbf{v}_{L}} \\
		0 & 0 
		\end{bmatrix}	}
	\end{equation}
Based on this, we  define:
\begin{align}\label{eq:def-RL}
\scalemath{1}{
	\mathbf{R}_{L}(\bm\theta_L)}
&
\scalemath{1}{
	=
	\exp(-\lfloor\boldsymbol{\theta}_{L}\times \rfloor)
	=
	\begin{bmatrix}
	\frac{\mathbf{n}_{L}}{\norm{\mathbf{n}_{L}}} &
	\frac{\mathbf{v}_{L}}{\norm{\mathbf{v}_{L}}} &
	\frac{\mathbf{n}_{L}\times \mathbf{v}_{L}}{\norm{\mathbf{n}_{L}\times\mathbf{v}_{L}}}
	\end{bmatrix}
}
\\
\scalemath{1}{
	\mathbf{W}_{L}(\phi_L)}
&
\scalemath{.9}{=
	\eta
	\begin{bmatrix}
	w_1 & -w_2 \\
	w_2 &  w_1
	\end{bmatrix}
	=
	\frac{1}{\sqrt{\norm{\mathbf{n}_{L}}^2+\norm{\mathbf{v}_{L}}^2}}
	\begin{bmatrix}
	\norm{\mathbf{n}_{L}} & -\norm{\mathbf{v}_{L}} \\
	\norm{\mathbf{v}_{L}} & \norm{\mathbf{n}_{L}} 
	\end{bmatrix}
}\label{eq:def-WL}
\end{align}
where $w_1 = \norm{\mathbf{n}_{L}}$, $w_2 = \norm{\mathbf{v}_{L}}$ and $\eta = \frac{1}{\sqrt{w^2_1+ w^2_2}}$.
Since $\scalemath{0.9}{\mathbf{R}_{L}\in \mathbf{SO}(3)}$ and $\scalemath{0.9}{\mathbf{W}_{L}\in \mathbf{SO}(2)}$, we  define the error states for these parameters as $\delta \boldsymbol{\theta}_{L}$ and $\delta \phi_{L}$ corresponding to $\mathbf{R}_L$ and $\mathbf{W}_L$, respectively. 
With that, the state vector with the line feature can be written as [see \eqref{eq:state000}]:
\begin{align}
\mathbf{x}
&
\scalemath{1}{
	=
	\begin{bmatrix}
	{^I_G\bar{q}}^{\top} & \mathbf{b}^{\top}_{\mathbf{g}} & {^G\mathbf{V}}^{\top}_{I} & \mathbf{b}^{\top}_{\mathbf{a}} &  {^G{\mathbf{P}}}^{\top}_{I}  &  {^G{\mathbf{L}}}^{\top}
	\end{bmatrix}^{\top}
}
\end{align}
\noindent
where $
\scalemath{1}{
	{^G\mathbf{L}}
	=
	\begin{bmatrix}
	^G\mathbf{n}^{\top}_{L} & ^G\mathbf{v}^{\top}_{L}
	\end{bmatrix}^{\top}
}
$
and
$
\scalemath{1}{
	{^G\tilde{\mathbf{L}}}
	=
	\begin{bmatrix}
	\delta \boldsymbol{\theta}^{\top}_{L} &
	\delta \phi_{L}
	\end{bmatrix}^{\top}
}$.

In particular, the visual line measurement is given by the distance from two ending points $\mathbf{x}_s$ and $\mathbf{x}_e$ of line segment 
to the line  in the image (also see our prior work \cite{Zuo2017IROS}):
\begin{align} \label{line-meas-im}
\mathbf{z}_{l} &= 
\begin{bmatrix}
\frac{\mathbf{x}^{\top}_{s}\mathbf{l}'}{\sqrt{l^2_1+ l^2_2}} 
&
\frac{\mathbf{x}^{\top}_{e}\mathbf{l}'}{\sqrt{l^2_1+ l^2_2}} 
\end{bmatrix}^{\top}  
\\
\mathbf{l}'
&=
\scalemath{1}{
	\mathbf{K} {^I\mathbf{n}_{L}}
	=
	\begin{bmatrix}
	l_1 & l_2 &l_3
	\end{bmatrix}^{\top}
} \label{eq:def-l-prime}
\\
\mathbf{K}&=\begin{bmatrix}
f_2  &  0   &  0 \\
0    &  f_1 &  0 \\
-f_2c_1  & -f_1c_2 & f_1f_2
\end{bmatrix}
\\
\scalemath{1}{
	^{I}\mathbf{L}}  &= \begin{bmatrix} {^I\mathbf{n}_{L}} \\ {^I\mathbf{v}_{L}} \end{bmatrix} 
=
\scalemath{1}
{
	\begin{bmatrix}
	^I_G\mathbf{R}  & -{^I_G\mathbf{R}}\lfloor {^G\mathbf{P}_I}\times \rfloor \\
	\mathbf{0}_3 & ^I_G\mathbf{R} 
	\end{bmatrix}
	{^G\mathbf{L}} 
}\label{eq:line-L-in-I}
\end{align}
where $\mathbf{K}$ is the projection matrix for line  (not point) features (see Appendix \ref{apd_line_proj}), with $f_1$, $f_2$, $c_1$ and $c_2$ as the camera intrinsic parameters. 
The relationship~\eqref{eq:line-L-in-I} is derived based on the geometry ${^G\mathbf{P}_i} = {^G\mathbf{P}_I} + {^G_I\mathbf{R}} {^I\mathbf{P}_i}$ ($i=1,2$) [see \eqref{eq_line_rep}].
Moreover, the measurement Jacobian can be computed as follows (Appendix~\ref{apd_line_jacob}):
\begin{align} \label{eq:H_I-line}
\mathbf H_I &= \frac{\partial \tilde{\mathbf{z}}_l}{\partial \tilde{\mathbf{l}}'}   \frac{\partial \tilde{\mathbf{l}}'}{\partial  \tilde{\mathbf x}} =: \mathbf{H}_{l}  \mathbf{H}_{\mathbf f} \\
\mathbf{H}_{l} &= 
\frac{1}{l_n}
\begin{bmatrix}
u_1 - \frac{l_1e_1}{l^2_n}   &  v_1-\frac{l_2e_1}{l^2_n}   &  1 \\
u_2 - \frac{l_1e_2}{l^2_n}   &  v_2-\frac{l_2e_2}{l^2_n}   &  1 
\end{bmatrix} \label{eq:def-H_l}
\end{align}
where we have used the following identities:
\begin{align}
	e_1 &= \mathbf{x}^{\top}_s\mathbf{l}', \ e_2 = \mathbf{x}^{\top}_e\mathbf{l}', \ l_n = \sqrt{\left(l^2_1 + l^2_2\right)}
	\nonumber \\
	\mathbf{x}_s &=\begin{bmatrix}u_1, v_1, 1\end{bmatrix}^{\top}, \ \mathbf{x}_e=\begin{bmatrix}u_2, v_2, 1\end{bmatrix}^{\top}
	\nonumber
\end{align}
%


\subsection{Plane Measurements}

\subsubsection{Overview of plane representations}

\begin{table} 
	\centering
	\caption{Different Plane Representations}
	\label{tab_plane}
	\renewcommand{\arraystretch} {1.25}
	\begin{tabular}{|l|l|l|}
		\hline
		\# & Models                                                                     & Parameters                                                           \\ \hline
		1  & $\pi_1 p_x+ \pi_2 p_y + \pi_3 p_z + \pi_4 = 0$                                                             & $\pi_1$, $\pi_2$, $\pi_3$, $\pi_4$                                                         \\ \hline
		2  & $\mathbf{n}^{\top}_{\pi}\mathbf{P}_{\mathbf{f}}-d = 0$                                          & $\mathbf{n}_{\pi}$, $\norm{\mathbf{n}_{\pi}}=1$, $d$ \\ \hline
		3  & $ \left[\cos\phi\cos \theta \   \cos \phi \sin \theta \ \sin \phi  \right] \mathbf{P}_{\mathbf f}  - d = 0 $      & $\theta, \phi, d$                                                          \\ \hline
		4  & $p_z = a p_x + b p_y + c$                                                                & $a,b,c$ with $\mathbf{e}^{\top}_3\mathbf{n}_{\pi} \neq 0$                                                                    \\ \hline
		5  & $ap_x+bp_y+cp_z+1 =0$                                                              & $a,b,c$ with $d\neq 0$                                                                    \\ \hline
		6  & $\bar{q}=\frac{1}{\sqrt{1+d^2}}\begin{bmatrix}\mathbf{n}_{\pi} \\ d \end{bmatrix}$ & $\bar{q}$, $ \norm{\bar{q}}=1$                                                         \\ \hline
		7  & $\boldsymbol{\Pi} = d \mathbf{n}$                                                         & $\boldsymbol{\Pi} =[x_{\pi}, y_{\pi}, z_{\pi}]^{\top}$                                  \\ \hline
	\end{tabular}
\end{table}

Different representations have been developed for plane estimation~\cite{Weingarten2006PHD}
and we summarize in Table~\ref{tab_plane} the most commonly used plane parameterizations along with the new representations introduced in this paper.
Note that $\mathbf{p}_{\mathbf{f}}=[p_x, p_y,p_z]^{\top}$ represents the point in the plane $\boldsymbol{\Pi}$. 
Model 1 \cite{Hartley2004} is the most general representation using the homogeneous coordinates $(\pi_i,i\in{1\ldots 4})$. 
Model 2 (Hesse form) uses the unit normal direction $\mathbf{n}_{\pi}=\left[n_x, n_y, n_z\right]^{\top}$  and the shortest distance from the origin to the plane $d$. 
Both models are {\em not} minimal, so the information matrix will become singular if directly using them for least-squares optimization.  
Model~3 is similar to the spherical coordinate, which parameterizes a plane with 2 angles (horizontal angle $\theta$ and elevation angle $\phi$) and distance~$d$. 
Model~3 is appealing since it is a minimal parameterization while suffering from singularities when $\phi = \pm\frac{\pi}{2}$ which is similar to the gimbal lock issue for Euler angles. 
Model 4 \cite{Weingarten2004ICRA} and Model 5 \cite{Proenca2017CORR} are used under certain conditions. 
Model 6~\cite{Kanatani1996,Kaess2015icra} uses a unit 4-dimensional vector for plane representation.
In~\cite{Kaess2015icra}   it was treated as a unit quaternion and thus quaternion error states to represent the plane error propagation.
In~\cite{Wu2017ICRA}  both quaternion and distance $d$ were used for plane parameterization while the  error states contain 2D quaternion errors and the distance error. 
The error states for both quaternion related plane representations lack physical interpretation. 
Since different representations have their own advantages, some work has combined these models. 
For example, both Models 2 and 4 are used in \cite{Weingarten2004ICRA}, and 
in \cite{Proenca2017CORR}  Model 5 was used for plane fitting and Model 7 for formulating the cost function of plane matching. 
Our recent work \cite{Yang2018ICRA} employed Models 2 and 3 for plane and its error state, respectively.

\subsubsection{Closest point (CP) parameterization}

A plane $\boldsymbol{\Pi}$ is  often represented by the normal direction $\mathbf{n}_{\pi}$ and the shortest distance $d$ from the origin.
However, we propose to use Model 7, the closest point (CP) of plane to the origin, for our plane parameterization. 
%
%
This is due to the facts: (i) it is a minimal representation; and (ii) its error states are in vector space can be interpreted  geometrically.
Note that there is one degenerate case associated with this representation, i.e., when $d=0$, which, however, can be easily avoided in practice.  
%
%
As in practice, we can extract plane features, e.g., from 3D point clouds,  direct measurements of plane features are given by:
%
\begin{equation} \label{plane-meas-model}
	\mathbf{z}_{\pi} ={d{\mathbf{n}_{\pi}}} + \mathbf {n}^{(\pi)} =: {^I\boldsymbol{\Pi}} + \mathbf {n}^{(\pi)}
\end{equation}
where $^I\boldsymbol{\Pi}$ represents the plane in the sensor's local frame and $\mathbf {n}^{(\pi)}$ represents the plane measurement noise.
%


To compute the measurement Jacobian of this plane measurement~\eqref{plane-meas-model},
note first that the plane parameters (Model 2) in the global frame can be transformed to the local frame as:
\begin{eqnarray}
\begin{bmatrix}
{^I\mathbf{n}_{\pi}} \\
{^Id}
\end{bmatrix}
=
\begin{bmatrix}
{^I_G\mathbf{R}}  & \mathbf{0}_{3\times 1} \\
-{^G\mathbf{P}^{\top}_I}  &   1
\end{bmatrix}
\begin{bmatrix}
{^G\mathbf{n}_{\pi}} \\
{^Gd}
\end{bmatrix}
\end{eqnarray}
The corresponding closest points  (Model 7) have:
\begin{align} 
{^I\boldsymbol{\Pi}} &= {^Id}{^I\mathbf{n}_{\pi}}
=
\left( -{^G\mathbf{P}^{\top}_I} {^G\mathbf{n}_{\pi}} + {^Gd} \right) {^I_G\mathbf{R}} {^G\mathbf{n}_{\pi}}
=
-  {^G\mathbf{n}^{\top}_{\pi}} {^G\mathbf{P}_I}{^I_G\mathbf{R}} {^G\mathbf{n}_{\pi}} + {^Gd}{^I_G\mathbf{R}} {^G\mathbf{n}_{\pi}}
\label{eq:local-global-plane}
\end{align}
We now can compute the local-plane measurement Jacobian w.r.t. the plane feature using the chain rule:
\begin{align} \label{eq:plane-jac-feat}
\frac{\partial {^I\tilde{\boldsymbol{\Pi}}}}{\partial {^G\tilde{\boldsymbol{\Pi}}}} 
=
\frac{\partial {^I\tilde{\boldsymbol{\Pi}}}}{\partial {^G\tilde{\mathbf{n}}_{\pi}}} 
\frac{\partial {^G\tilde{\mathbf{n}}_{\pi}}} {\partial {^G\tilde{\boldsymbol{\Pi}}}}
+
\frac{\partial {^I\tilde{\boldsymbol{\Pi}}}}{\partial {^G\tilde{d}}} 
\frac{\partial {^G\tilde{d}}} {\partial {^G\tilde{\boldsymbol{\Pi}}}} 
\end{align}
where the pertinent intermedian Jacobians are computed as:
\begin{align}
\frac{\partial {^I\tilde{\boldsymbol{\Pi}}}}{\partial {^G\tilde{\mathbf{n}}_{\pi}}} 
&=
{^I_G\hat{\mathbf{R}}}
\left(
\left(
{^G\hat{d}} - {^G\hat{\mathbf{n}}^{\top}_{\pi}}{^G\hat{\mathbf{P}}_{I}}
\right)\mathbf{I}_{3} 
- {^G\hat{\mathbf{n}}_{\pi}}{^G\hat{\mathbf{P}}^{\top}_{I}}
\right)
\\
\frac{\partial {^L\tilde{\boldsymbol{\Pi}}}}{\partial {^G\tilde{d}}} 
&=
{^I_G\hat{\mathbf{R}}} {^G\hat{\mathbf{n}}_{\pi}}\\
\frac{\partial {^G\tilde{\mathbf{n}}_{\pi}}} {\partial {^G\tilde{\boldsymbol{\Pi}}}}
&=
\frac{
	\left(
	{^G\hat{\mathbf{n}}^{\top}_{\pi}}{^G\hat{\mathbf{n}}_{\pi}}\mathbf{I}_3 - {^G\hat{\mathbf{n}}_{\pi}}{^G\hat{\mathbf{n}}^{\top}_{\pi}}
	\right)
	}
{^G\hat{d}} 
\\
\frac{\partial {^G\tilde{d}}} {\partial {^G\tilde{\boldsymbol{\Pi}}}} 
&=
\frac{
	\begin{bmatrix}
	{x_{\pi}}  &  {y_{\pi}} & {z_{\pi}}
	\end{bmatrix}
	}{\sqrt{x^2_{\pi} + y^2_{\pi}+ z^2_{\pi}}}
=
{^G\hat{\mathbf{n}}^{\top}_{\pi}}
\end{align}
Substitution of the above expressions into \eqref{eq:plane-jac-feat} yields:
\begin{align} \label{eq:plane-jac0}
\frac{\partial {^I\tilde{\boldsymbol{\Pi}}}}{\partial {^G\tilde{\boldsymbol{\Pi}}}} 
= &
\frac{
\scalemath{1}	
{
{^I_G\hat{\mathbf{R}}} 
\bigg(
\left(
{^G\hat{d}} - {^G\hat{\mathbf{n}}^{\top}_{\pi}}{^G\hat{\mathbf{P}}_{I}}
\right)\mathbf{I}_{3}
- {^G\hat{\mathbf{n}}_{\pi}}{^G\hat{\mathbf{P}}^{\top}_{I}}
+
2 {^G\hat{\mathbf{n}}^{\top}_{\pi}}{^G\hat{\mathbf{P}}_{I}}
{^G\hat{\mathbf{n}}_{\pi}}{^G\hat{\mathbf{n}}^{\top}_{\pi}}
\bigg)
}	
}{^G\hat{d}}	
\end{align}
%
%
%
%
%
%
%
We now compute the measurement Jacobian w.r.t. to the IMU states by applying perturbation of $\delta \boldsymbol{\theta}$ and $^G\tilde{\mathbf{P}}_I$ on~\eqref{eq:local-global-plane}:
\begin{align}
{^I\boldsymbol{\Pi}} =& 
\scalemath{1}{
-  {^G\hat{\mathbf{n}}^{\top}_{\pi}} 
{^G\hat{\mathbf{P}}_I}
{^I_G\hat{\mathbf{R}}} 
{^G\hat{\mathbf{n}}_{\pi}} 
+ {^G\hat{d}}
{^I_G\hat{\mathbf{R}}} 
{^G\hat{\mathbf{n}}_{\pi}}
} 
+
\bigg(
{^G\hat{d}}
- {^G\hat{\mathbf{n}}^{\top}_{\pi}} 
{^G\hat{\mathbf{P}}_I}
\bigg)
\lfloor {^I_G\hat{\mathbf{R}}} 
{^G\hat{\mathbf{n}}_{\pi}}  \times  \rfloor 
\delta \boldsymbol{\theta} 
\\
=& 
-  {^G\hat{\mathbf{n}}^{\top}_{\pi}} 
{^G\hat{\mathbf{P}}_I}
{^I_G\hat{\mathbf{R}}} 
{^G\hat{\mathbf{n}}_{\pi}} 
+ {^G\hat{d}}
{^I_G\hat{\mathbf{R}}} 
{^G\hat{\mathbf{n}}_{\pi}}
-   
{^L_G\hat{\mathbf{R}}} 
{^G\hat{\mathbf{n}}_{\pi}} {^G\hat{\mathbf{n}}^{\top}_{\pi}}{^G\tilde{\mathbf{P}}_I}
\end{align}
which immediately provides the desired Jacobians:
\begin{align} \label{eq:plane-jac1}
\frac{\partial {^I\tilde{\boldsymbol{\Pi}}}}{\partial \delta\boldsymbol{\theta}} 
&=
\left(
{^G\hat{d}}
- {^G\hat{\mathbf{n}}^{\top}_{\pi}} 
{^G\hat{\mathbf{P}}_I}
\right)
\lfloor {^I_G\hat{\mathbf{R}}} 
{^G\hat{\mathbf{n}}_{\pi}}  \times  \rfloor 
\\
\frac{\partial {^I\tilde{\boldsymbol{\Pi}}}}{\partial {^G\tilde{\mathbf{P}}_{I}}}
&=
-   
{^I_G\hat{\mathbf{R}}} 
{^G\hat{\mathbf{n}}_{\pi}} {^G\hat{\mathbf{n}}^{\top}_{\pi}} 
\label{eq:plane-jac2}
\end{align}
Stacking \eqref{eq:plane-jac0}, \eqref{eq:plane-jac1} and \eqref{eq:plane-jac2} yields the complete the measurement Jacobian of the plane measurement w.r.t. the state~\eqref{eq:state000}:
\begin{align} \label{eq:H_I-plane}
\mathbf H_I = 
\begin{bmatrix}
\frac{\partial {^I\tilde{\boldsymbol{\Pi}}}}{\partial \delta\boldsymbol{\theta}} & \mathbf 0_{3\times 9} & \frac{\partial {^I\tilde{\boldsymbol{\Pi}}}}{\partial {^G\tilde{\mathbf{P}}_{I}}} & \frac{\partial {^I\tilde{\boldsymbol{\Pi}}}}{\partial {^G\tilde{\boldsymbol{\Pi}}}} 
\end{bmatrix}
\end{align}
%

\subsection{Observability Analysis}

Observability analysis for the linearized aided INS can be performed in a similar way as in~\cite{Huang2010IJRR,Hesch2013TRO}. 
In particular,  the observability matrix $\mathbf{M}(\mathbf{x})$ is given by:
\begin{equation}\label{eq_obs_equation}
\scalemath{1}{
	\mathbf{M}{(\mathbf{x})} = 
	\begin{bmatrix}
	\mathbf{H}_{I_1} \boldsymbol{\Phi}_{(1,1)}\\
	\mathbf{H}_{I_2} \boldsymbol{\Phi}_{(2,1)} \\
	\vdots \\
	\mathbf{H}_{I_k} \boldsymbol{\Phi}_{(k,1)}
	\end{bmatrix}
}
\end{equation}
where $\mathbf{H}_{I_k}$ is the measurement Jacobian at time step~$k$. 
The unobservable directions  span the right null space of this matrix.
%

%% file: sections/point.tex

\section{Observability Analysis of Aided INS with Homogeneous Features}
\label{sec:homogeneous-features}

In this section, we perform observability analysis for the linearized systems of aided INS using {\em homogeneous} geometric features including only points, lines and planes;
and the observability analysis for aided INS with {\em heterogeneous} geometric features will be presented in next section.

\subsection{Aided INS with Point Features} \label{sec:points}

We first consider the aided INS with point features and conduct the observability analysis in a similar way as in~\cite{Huang2010IJRR,Hesch2013TRO}.
In particular,
as the unobservable directions of this aided INS span the right null space of $\mathbf{M}(\mathbf{x})$~\eqref{eq_obs_equation},
we compute the measurement Jacobians $\mathbf{H}^{(p)}_{I_k}$  based on \eqref{eq_x_meas} as follows [see~\eqref{eq:linearized-point-meas}]: 
\begin{align}
	\scalemath{1}{
	\mathbf{H}^{(p)}_{I_k}}
	&=
	\scalemath{1}{
	\begin{bmatrix}
	\mathbf{H}_{r,k}\\
	\mathbf{H}_{b,k}
	\end{bmatrix}
\underbrace{
\begin{bmatrix}
\lfloor {^{I_k}\hat{\mathbf{P}}_{\mathbf{f}}} \times\rfloor
& 
\mathbf{0}_3 
& 
\mathbf{0}_3
&
\mathbf{0}_3
&
-^{I_k}_G\hat{\mathbf{R}}
&
^{I_k}_G\hat{\mathbf{R}}
\end{bmatrix}}_{\mathbf H_{\mathbf f, k}}
}
\nonumber
\\
&=
\scalemath{1}{
\mathbf{H}_{proj,k}~{^{I_k}_G\hat{\mathbf{R}}}
\begin{bmatrix}
\mathbf{H}_{p1}& 
\mathbf{0}_3 & 
\mathbf{0}_3 & 
\mathbf{0}_3 & 
-\mathbf{I}_3 & 
\mathbf{I}_3
\end{bmatrix}}	\label{eq:H_I-point}
\end{align}
where 
we have used \eqref{eq_feat} and \eqref{eq:linearized-point-meas} as well as the following matrix:
%
\begin{align}
\mathbf{H}_{p1}  
	&= 
	\scalemath{1}{
		\lfloor \left({^G\hat{\mathbf{P}}_{\mathbf{f}}}-{^{G}\hat{\mathbf{P}}_{I_k}}\right)\times \rfloor 
		{^{I_k}_G\hat{\mathbf{R}}}^{\top} 
	}
\end{align}
Specifically, 
for each block row of $\mathbf{M}(\mathbf{x})$ [see \eqref{eq_obs_equation}], we  have:
\begin{equation}\label{eq:block-row-single-point}
\scalemath{1}{
	\mathbf{H}^{(p)}_{I_k}\boldsymbol{\Phi}_{(k,1)}
	=
	\mathbf{H}_{proj,k}~
	{^{I_k}_G\hat{\mathbf{R}}} 
	\begin{bmatrix}
	\boldsymbol{\Gamma}_1 & \boldsymbol{\Gamma}_2 & \boldsymbol{\Gamma}_3 & \boldsymbol{\Gamma}_4
	& -\mathbf{I}_3  & \mathbf{I}_3
	\end{bmatrix}
}
\end{equation}
where
\begin{align}
	\boldsymbol{\Gamma}_1
	&=
\scalemath{1}{
	\left\lfloor \left( {^G\hat{\mathbf{P}}_{\mathbf{f}}} -
	{^G\hat{\mathbf{P}}_{I_1}} 
	-
	{^G\hat{\mathbf{V}}_{I_1}}\delta t_k
	+
	\frac{1}{2}{^G\mathbf{g}}(\delta t_k)^2 \right) \times \right\rfloor
	{^G_{I_1}}\hat{\mathbf{R}}
}
	\nonumber
	\\
	\boldsymbol{\Gamma}_2
	&=
\scalemath{1}{	
	\left\lfloor \left({^G\hat{\mathbf{P}}_{\mathbf{f}}}-{^G\hat{\mathbf{P}}_{I_k}}\right)\times \right\rfloor
	{^{I_k}_G\hat{\mathbf{R}}^{\top}}
	\boldsymbol{\Phi}_{12}
	-\boldsymbol{\Phi}_{52}
}
\nonumber
	\\
	\boldsymbol{\Gamma}_3
	&=-
\scalemath{1}{	
	\mathbf{I}_3\delta t_k
}
	,\ 
	\boldsymbol{\Gamma}_4
	=
	-\boldsymbol{\Phi}_{54}
	\label{eq:GAMMA4}
\end{align}
%
where $\scalemath{1}{g=\norm{^G\mathbf{g}}}$ and ${^G\mathbf{g}}=\left[0,0,-g\right]^{\top}$, 
Note that for the analysis purpose, we assume that in computing different Jacobians the linearization points for the same state variables remain the same. 
By inspection, it is not difficult to see that the null space of $\mathbf{M}(\mathbf{x})$ in this case is given by:
\begin{align} \label{eq_point_obs}
&\scalemath{1}{
	\mathbf{N}
}
 = 
\scalemath{1}{
	\begin{bmatrix}
	\mathbf{N}_g  & \mathbf{0}_{12\times 3}   \\
	-\lfloor ^G\hat{\mathbf{P}}_{I_1}\times\rfloor {^G\mathbf{g}}  &  \mathbf{I}_3   \\
	-\lfloor ^G\hat{\mathbf{P}}_{\mathbf{f}}\times\rfloor {^G\mathbf{g}}  & \mathbf{I}_3   \\
	\end{bmatrix}
}
=:
\begin{bmatrix}
\mathbf{N}_{r}   &  \mathbf{N}_{p}  
\end{bmatrix}
\end{align}
where $\mathbf{N}_{g}$ is defined by: 
\begin{align}
	&
	\scalemath{1}{
	\mathbf{N}_{g} = 
	\begin{bmatrix}
	  \left(^{I_1}_G\hat{\mathbf{R}}{^G\mathbf{g}}\right)^{\top} &
	 \mathbf{0}_{1\times 3} &
	 -\left(\lfloor ^G\hat{\mathbf{V}}_{I_1}\times\rfloor {^G\mathbf{g}}\right)^{\top} &
	 \mathbf{0}_{1\times 3} &
	\end{bmatrix}^{\top}
	}
\end{align} 

It is interesting to notice that in \eqref{eq_point_obs}, $\mathbf{N}_{p}$ corresponds to the sensor's global translation, while $\mathbf{N}_{r}$ relates to the global rotation around the gravity direction.
We thus see that the system has at least 4 unobservable directions ($\mathbf{N}_{p}$ and $\mathbf{N}_r$).
%
Moreover, in analogy to \cite{Huang2010IJRR,Hesch2014IJRR,Guo2013ICRA}, 
we have further performed the nonlinear observability analysis based on Lie derivatives~\cite{Hermann1977TAC} for the continuous-time nonlinear aided INS, which is summarized  as follows:
\begin{lem}
	The continuous-time nonlinear aided INS  with point features (detected from generic range and/or bearing measurements), 
	has 4 unobservable directions. 
\end{lem}
\begin{proof}
	See Appendix \ref{apd_proof}.
\end{proof}

%% file: sections/line.tex
\subsection{Aided INS with Line Features}
\label{sec:lines}

When navigating in structured environments, line features might be ubiquitous and should be exploited in the aided INS to improve performance.
In the following, we perform observability analysis for the aided INS with line features to provide insights for building consistent estimators.

\subsubsection{Single Line}

With the line measurements \eqref{line-meas-im}, the  measurement Jacobian is computed by (see~\eqref{eq:H_I-line} and Appendix~\ref{apd_line_jacob}):
	\begin{align} \label{eq:H_I-line}
	&
	\scalemath{1}{
		\mathbf{H}^{(l)}_{I_k}
		=
		\mathbf{H}_{l,k}
		\underbrace{
		\mathbf{K}{^{I_k}_G\hat{\mathbf{R}}}
		\begin{bmatrix}
		\mathbf{H}_{l1} &
		\mathbf{0}_{3\times 9} &
		\lfloor {^G\hat{\mathbf{v}}_L}\times\rfloor &
		\mathbf{H}_{l2} &
		\mathbf{H}_{l3}
		\end{bmatrix}
		}_{\mathbf H_{\mathbf f, k}}
	}
	\end{align}
where we have employed the following identities:
\begin{align}
\mathbf{H}_{l1}	 &=
\scalemath{1}{
	\left(
	\lfloor {^G\hat{\mathbf{n}}_{L}}\times \rfloor
	- \lfloor \lfloor {^G\hat{\mathbf{P}}_{I_k}}\times \rfloor {^G\hat{\mathbf{v}}_{L}}\times \rfloor
	\right) {^{I_k}_G\hat{\mathbf{R}}}^{\top} 
}
\\ 
\mathbf{H}_{l2}	 &= 
\scalemath{1}{
	{\lfloor {^G\hat{\mathbf{n}}_{L}}\times \rfloor
		-	\lfloor {^G\hat{\mathbf{P}}_{I_k}}\times \rfloor\lfloor {^G\hat{\mathbf{v}}_{L}}\times \rfloor}
}
\\ 
\mathbf{H}_{l3} &= 
\scalemath{1}{
	{-\left(
		\frac{w_2}{w_1}{^G\hat{\mathbf{n}}_{L}}+\frac{w_1}{w_2}\lfloor {^G\hat{\mathbf{P}}_{I_k}}\times \rfloor{^G\hat{\mathbf{v}}_{L}}
		\right)} 
} 
\end{align}
With this, the block row of the observability matrix $\mathbf{M}(\mathbf{x})$~\eqref{eq_obs_equation}  at time step $k$ can be written as:
\begin{align}
&\scalemath{1}{
	\mathbf{H}^{(l)}_{I_k}\boldsymbol{\Phi}_{(k,1)}}
	=
	\mathbf{H}_{l,k}
	\mathbf{K}{^{{I_k}}_G\hat{\mathbf{R}}} 
	\begin{bmatrix}
	\boldsymbol{\Gamma}_{l1} &
	\boldsymbol{\Gamma}_{l2} &
	\lfloor {^G\hat{\mathbf{v}}_L\times}\rfloor\delta t_k &
	\boldsymbol{\Gamma}_{l3} &
	\lfloor {^G\hat{\mathbf{v}}_L\times}\rfloor &
	\boldsymbol{\Gamma}_{l4} &
	\boldsymbol{\Gamma}_{l5}
	\end{bmatrix}
\end{align}
where 
\begin{align}
\scalemath{1}{
	\boldsymbol{\Gamma}_{l1}
}
&
=
\scalemath{1}{
	(\lfloor {^G\hat{\mathbf{n}}_{L}\times}\rfloor
	-
	\lfloor \lfloor{^G\hat{\mathbf{P}}_{I_k}}\times \rfloor
	{^G\hat{\mathbf{v}}_{L}}\times \rfloor
	+
	\lfloor {^G\hat{\mathbf{v}}_{L}\times}\rfloor
	\lfloor{^G\hat{\mathbf{P}}_{I_1}}\times \rfloor +
}
	\lfloor {^G\hat{\mathbf{v}}_{L}\times}\rfloor
	\lfloor{^G\hat{\mathbf{V}}_{I_1}}\times \rfloor\delta t_k	
	- 
	\frac{1}{2}\lfloor {^G\hat{\mathbf{v}}_{L}\times}\rfloor
	\lfloor {^G\hat{\mathbf{g}}\times}\rfloor\delta t^2_k
\nonumber
\\
&\hspace{-15pt}
\scalemath{1}{	
	-
	\lfloor {^G\hat{\mathbf{v}}_{L}\times}\rfloor
	\lfloor{^G\hat{\mathbf{P}}_{I_k}}\times \rfloor)
	{^{G}_{I_1}\hat{\mathbf{R}}}
}	\nonumber
\\
\scalemath{1}{
	\boldsymbol{\Gamma}}_{l2}
&
=
\scalemath{1}{
	\left(\lfloor {^G\hat{\mathbf{n}}_{L}\times}\rfloor
	-
	\lfloor \lfloor{^G\hat{\mathbf{P}}_{I_k}}\times \rfloor
	{^G\hat{\mathbf{v}}_{L}}\times \rfloor\right)
	{^{I_k}_G\hat{\mathbf{R}}}^{\top}
	\boldsymbol{\Phi}_{12} + \lfloor {^G\hat{\mathbf{v}}_{L}\times}\rfloor\boldsymbol{\Phi}_{52}
}
\nonumber
\\
\boldsymbol{\Gamma}_{l3}
&=
\lfloor {^G\hat{\mathbf{v}}_{L}\times}\rfloor\boldsymbol{\Phi}_{54}	
, \
\boldsymbol{\Gamma}_{l4} = \mathbf{H}_{l2}
, \
\boldsymbol{\Gamma}_{l5} = \mathbf{H}_{l3}
\nonumber
\end{align}
Therefore, we  have the following result:
\begin{lem}
\label{lem_single_line}
	The aided INS with a single line feature has at least 5 unobservable directions denoted by $\mathbf{N}_l$:
	\begin{align}\label{eq_line_obs}
	\scalemath{1}{
		\mathbf{N}_{l}}
	&
	=\!\!
	\scalemath{1}{
		\begin{bmatrix}
		\mathbf{N}_g & \mathbf{0}_{12\times 3}  &	\mathbf{N}_{\mathbf{v}}  \\
		-{\lfloor{^G\hat{\mathbf{P}}_{I_1}}\times \rfloor{^G\mathbf{g}}} &	{^{G}_{L}\hat{\mathbf{R}}} &  \mathbf{0}_{3\times 1}\\
		-{^{G}\mathbf{g}} & {\frac{w_2}{w_1}}{^{G}\hat{\mathbf{v}}_{\mathbf{e}}} {\mathbf{e}^{\top}_1}  & \mathbf{0}_{3\times 1} \\
		0  & \eta^2w^2_2{\mathbf{e}^{\top}_3} & 0 
		\end{bmatrix}}
	\scalemath{1}{=:
		\begin{bmatrix}
		\mathbf{N}_{l1} & \mathbf{N}_{l2:5} 
		\end{bmatrix}
	}	
	\end{align}\noindent
	where $^G\hat{\mathbf{n}}_{\mathbf{e}}$ and $^G\hat{\mathbf{v}}_{\mathbf{e}}$ are the normalized unit vectors of $^G\hat{\mathbf{n}}_L$ and $^G\hat{\mathbf{v}}_L$, respectively, and $\mathbf{N}_{\mathbf{v}}$ and  ${^G_{L}\hat{\mathbf{R}}}$ are defined by:
	\begin{align}
	\mathbf{N}_{\mathbf{v}} 
	&= 
	\begin{bmatrix}
	\mathbf{0}_{1\times 3} &\mathbf{0}_{1\times 3} & {^G\hat{\mathbf{v}}^{\top}_{\mathbf{e}}} & \mathbf{0}_{1\times 3} 
	\end{bmatrix}^{\top}
	\\
	{^G_{L}\hat{\mathbf{R}}} 
	&=
	\begin{bmatrix}
	{^G\hat{\mathbf{n}}_{\mathbf{e}}} & {^G\hat{\mathbf{v}}_{\mathbf{e}}} & \lfloor{^G\hat{\mathbf{n}}_{\mathbf{e}}}\times \rfloor {^G\hat{\mathbf{v}}_{\mathbf{e}}}
	\end{bmatrix}
	\end{align}
\end{lem}
\begin{proof}
	See Appendix \ref{apd_single_line}. 
\end{proof}
It is not difficult to see that $\mathbf{N}_{l1}$ relates to the sensor rotation around the gravitational direction, $\mathbf{N}_{l2:4}$ associates with the sensor's global translation, and $\mathbf{N}_{l5}$ corresponds to the sensor motion along the line direction.  
Note also that the above analysis is based on the projective line measurement model \eqref{line-meas-im}. 
Additionally, in Appendix \ref{apd_direct_line}, we have also considered the direct line measurement model,
for example, by extracting lines from point clouds, and show that the same unobservable subspace $\mathbf{N}_{l}$~\eqref{eq_line_obs} holds.

\subsubsection{Multiple Lines}

We extend the analysis to the case where  $l>1$ general, unparallel lines are included in the state vector.
To this end, by noting that the rotation between line $i$ and line $j$ ($i,j\in\{1,\ldots l\}$) is given by:
${^{L_i}_{L_j}\hat{\mathbf{R}}}={^G_{L_i}\hat{\mathbf{R}}}^{\top}{^G_{L_j}\hat{\mathbf{R}}}$,
 we can prove that in this case there are 4 unobservable directions.
\begin{lem} \label{lem_multiple_lines}
	The aided INS with $l>1$ general, unparallel line features has at least 4 unobservable  directions:
\begin{align}
\scalemath{1}{
	\mathbf{N}_{L}
}
&
\scalemath{1}{
	=
	\begin{bmatrix}
	\mathbf{N}_g &  \mathbf{0}_{12\times 3} \\
	-\lfloor{^G\hat{\mathbf{P}}_{I_1}}\times \rfloor{^G\mathbf{g}} & {^G_{L_i}\hat{\mathbf{R}}} \\
	-{^{G}\mathbf{g}} &  \frac{w_{1,2}}{w_{1,1}}{^G\hat{\mathbf{v}}_{\mathbf{e} 1}}{\mathbf{e}^{\top}_1}{^{L_1}_{L_i}\hat{\mathbf{R}}}\\
	0 & \eta^2_1w^2_{1,2}\mathbf{e}^{\top}_3{^{L_1}_{L_i}\hat{\mathbf{R}}} \\
	\vdots &\vdots \\
	-{^{G}\mathbf{g}} & \frac{w_{l,2}}{w_{l,1}}{^G\hat{\mathbf{v}}_{\mathbf{e} {l}}}{\mathbf{e}^{\top}_1}{^{L_l}_{L_i}\hat{\mathbf{R}}}\\
	0 & \eta^2_{l}w^2_{l,2}\mathbf{e}^{\top}_3{^{L_l}_{L_i}\hat{\mathbf{R}}} \\
	\end{bmatrix}
}
\scalemath{1}{
	=:
	\begin{bmatrix}
	\mathbf{N}_{L1} & \mathbf{N}_{L2:4}
	\end{bmatrix}}
\end{align}	
where $\eta_i$, $w_{1,i}$, $w_{2,i}$ and $^G\hat{\mathbf{v}}_{\mathbf{e}i}$ 
are the parameters related to line $i$ ($i\in \{1\ldots l\}$) [see~\eqref{eq:def-WL}].
\end{lem}
\begin{proof}
	See Appendix \ref{apd_line}.
\end{proof}

We want to point out again that 
the above Lemma \ref{lem_multiple_lines} holds under the assumption that not all the lines are parallel;
if  all  parallel lines, the system will have one more unobservable direction, 
which coincides with the line direction $^G\hat{\mathbf{v}}_{\mathbf{e}}$.

%% file: sections/plane.tex
\subsection{Aided INS with Plane Features} \label{sec:planes} 

Now we perform observability analysis of the aided INS with plane features 
that are important geometric features commonly seen in structured environments.
In particular, our analysis is based on the CP parameterization of plane features.

\subsubsection{Single Plane}

We first consider a single plane feature $^G\boldsymbol{\Pi}$ included in the state vector:
\begin{equation}
\scalemath{1}{
\mathbf{x}
}
= 
\begin{bmatrix}
{^I_G\bar{q}^{\top}} &
\mathbf{b}^{\top}_{g} &
{^G\mathbf{V}^{\top}_{I}} &
\mathbf{b}^{\top}_{a} &
{^G\mathbf{p}^{\top}_{I}} &
{^G\boldsymbol{\Pi}^{\top}}
\end{bmatrix}^{\top}
\end{equation}
Based on the plane measurements \eqref{plane-meas-model}, we can compute the Jacobians as [see~\eqref{eq:plane-jac0}, \eqref{eq:plane-jac1} and \eqref{eq:plane-jac2}]: 
%
\begin{align}
&\scalemath{1}{
	\mathbf{H}^{(\pi)}_{I_k}
	=
	{^{I_k}_G\hat{\mathbf{R}}}
	\begin{bmatrix}
	\mathbf{H}_{\pi 1} &
	\mathbf{0}_{3} & \mathbf{0}_{3}  & \mathbf{0}_{3} & 
	-    
	{^G\hat{\mathbf{n}}_{\pi}} {^G\hat{\mathbf{n}}^{\top}_{\pi}} &
	\mathbf{H}_{\pi 2}
	\end{bmatrix}
}
\end{align}
where:
\begin{align}
\mathbf{H}_{\pi 1} &=
{\left(
	{^G\hat{d}}
	- {^G\hat{\mathbf{n}}^{\top}_{\pi}} 
	{^G\hat{\mathbf{P}}_{I_k}}
	\right)
	\lfloor {^G\hat{\mathbf{n}}_{\pi}} \times   \rfloor{^{I_k}_G\hat{\mathbf{R}}^{\top}} } 
\\
\mathbf{H}_{\pi 2}	&=
\scalemath{1}
{
	\frac{
		\left(
		{^G\hat{d}} - {^G\hat{\mathbf{n}}^{\top}_{\pi}}{^G\hat{\mathbf{P}}_{I_k}}
		\right)\mathbf{I}_{3}
		- {^G\hat{\mathbf{n}}_{\pi}}{^G\hat{\mathbf{P}}^{\top}_{I_k}}
		+
		2 {^G\hat{\mathbf{n}}^{\top}_{\pi}}{^G\hat{\mathbf{P}}_{I_k}}
		{^G\hat{\mathbf{n}}_{\pi}}{^G\hat{\mathbf{n}}^{\top}_{\pi}}
		}{^G\hat{d}} 
		}
\end{align}
The block row of the observability matrix is computed by:
\begin{align}
&
\scalemath{1}{
\mathbf{H}^{(\pi)}_{I_k}\boldsymbol{\Phi}_{(k,1)}
}
=
{^{I_k}_G\hat{\mathbf{R}}} 
\begin{bmatrix}
\boldsymbol{\Gamma}_{\pi 1} & \boldsymbol{\Gamma}_{\pi 2}  & 
-{^G\hat{\mathbf{n}}_{\pi}} {^G\hat{\mathbf{n}}^{\top}_{\pi}}\delta t_k &
\boldsymbol{\Gamma}_{\pi 3} & 
-   
{^G\hat{\mathbf{n}}_{\pi}} {^G\hat{\mathbf{n}}^{\top}_{\pi}} &
\boldsymbol{\Gamma}_{\pi 4}
\end{bmatrix}
\end{align}
where 
\begin{align}
\boldsymbol{\Gamma}_{\pi 1}
&=
\scalemath{1}{
\big[\left(
{^Gd} - {^G\hat{\mathbf{n}}^{\top}_{\pi}}{^G\hat{\mathbf{P}}_{I_k}}
\right)
\lfloor {^G\hat{\mathbf{n}}_{\pi}}\times\rfloor
- {^G\hat{\mathbf{n}}_{\pi}}{^G\hat{\mathbf{n}}^{\top}_{\pi}}
\lfloor {^G\hat{\mathbf{P}}_{I_1}} 
+ {^G\hat{\mathbf{V}}_{I_1}}{\delta t_k}}
- \frac{1}{2}{^G\mathbf{g}}{\delta t^2_k}
- {^G\hat{\mathbf{P}}_{I_k}} \times\rfloor \big]
{^G_{I_1}\hat{\mathbf{R}}} 
\\
\scalemath{1}{
\boldsymbol{\Gamma}_{\pi 2} 
}
&=
\scalemath{1}
{
\left(
{^G\hat{d}} - {^G\hat{\mathbf{n}}^{\top}_{\pi}}{^G\hat{\mathbf{P}}_{I_k}}
\right)
\lfloor  {^G\hat{\mathbf{n}}_{\pi}} \times\rfloor
{^{I_k}_G\hat{\mathbf{R}}^{\top}}
\boldsymbol{\Phi}_{12}
- 
{^G\hat{\mathbf{n}}_{\pi}} {^G\hat{\mathbf{n}}^{\top}_{\pi}}
\boldsymbol{\Phi}_{52}
}
\\ 
\scalemath{1}{
\boldsymbol{\Gamma}_{\pi 3} 
}
&=
\scalemath{1}{
- {^G\hat{\mathbf{n}}_{\pi}} {^G\hat{\mathbf{n}}^{\top}_{\pi}}
\boldsymbol{\Phi}_{54} 
}
\\
\boldsymbol{\Gamma}_{\pi 4} 
&= \mathbf{H}_{\pi2}
\end{align}
With that, we have the following result:
\begin{lem} \label{lem_plane_obs}
	The aided INS with a single plane feature has at least 7 unobservable directions:
	\begin{align} \label{eq_plane_obs}
	\mathbf{N}_{\boldsymbol{\pi}}
	&=
	\begin{bmatrix}
	\mathbf{N}_g   & \mathbf{0}_{12\times 3}    & \mathbf{N}_{123} \\
	-\lfloor {^G\hat{\mathbf{P}}_{I_1}\times} \rfloor {^G\mathbf{g}} & {^G_{\Pi}\hat{\mathbf{R}}}   &\mathbf{0}_{3} &   \\
	-\lfloor {{^G\hat{\boldsymbol{\Pi}}}}\times \rfloor {^G\mathbf{g}} &  {^G\hat{\mathbf{n}}_{\pi}}{\mathbf{e}^{\top}_3} &
	\mathbf{0}_{3}  \\
	\end{bmatrix} 
	=:
	\begin{bmatrix}
	\mathbf{N}_{\boldsymbol{\pi} 1}  &  \mathbf{N}_{\boldsymbol{\pi} 2:4}  & \mathbf{N}_{\boldsymbol{\pi} 5:7} 
	\end{bmatrix}
	\end{align}
In above expression,   
given $^G\hat{\mathbf{n}}^{\perp}_1$ and $^G\hat{\mathbf{n}}^{\perp}_2$ that are the unit vector orthonormal to each other and perpendicular to $^G\hat{\mathbf{n}}_{\pi}$,
we have defined $\mathbf{N}_{123}$ and the plane orientation $^G_{\Pi}\mathbf{R}$ as follows:
\begin{align}
\mathbf{N}_{123} 
&=
\begin{bmatrix}
\mathbf{0}_{3\times 1} & \mathbf{0}_{3\times 1} &   {^{I_1}_G\hat{\mathbf{R}}}{^G\hat{\mathbf{n}}_{\pi}} \\
\mathbf{0}_{3\times 1} & \mathbf{0}_{3\times 1} & \mathbf{0}_{3\times 1} \\
{^G\hat{\mathbf{n}}^{\perp}_1} &  {^G\hat{\mathbf{n}}^{\perp}_2} & \mathbf{0}_{3\times 1} \\
\mathbf{0}_{3\times 1} & \mathbf{0}_{3\times 1} & \mathbf{0}_{3\times 1}
\end{bmatrix}
\\
{^G_{\Pi}\hat{\mathbf{R}}} 
&=
\begin{bmatrix}
{^G\hat{\mathbf{n}}^{\perp}_{ 1}} & {^G\hat{\mathbf{n}}^{\perp}_{ 2}} &{^G\hat{\mathbf{n}}_{\pi}}
\end{bmatrix}
\end{align}

\end{lem}
\begin{proof}
	See Appendix \ref{apd_plane_single_obs}. 
\end{proof}

Note that as compared to \cite{Guo2013iros} where it was shown that the VINS with  bearing measurements to planes has 12 unobservable directions, 
we analytically show here that, given the direct plane measurements~\eqref{plane-meas-model}, the aided INS with with a single plane feature  has at least 7 unobservable directions:
(i) $\mathbf{N}_{\pi 1}$ that relates to the rotation around the gravity, 
(ii) $\mathbf{N}_{\pi 2:4}$ that associate with the position of the sensor platform, 
(iii) $\mathbf{N}_{\pi5:6}$ that correspond to the motions parallel to the plane, 
and (iv) $\mathbf{N}_{\pi 7}$  that corresponds to the rotation around the plane normal direction. This analysis is directly verified by numerical simulation results (see Fig. \ref{fig_all_rank}).

\subsubsection{Multiple Planes}

Assuming that there are $s>1$ plane features in the state vector, 
we first note that the rotation between plane $i$ and plane $j$ ($i,j\in \{1,\ldots s\}$) is	given by:
${^{\Pi_i}_{\Pi_j}\hat{\mathbf{R}}}	=	{^G_{\Pi_i}\hat{\mathbf{R}}}^{\top}{^G_{\Pi_j}\hat{\mathbf{R}}}$.
We then can prove the following result:
\begin{lem}\label{lem_multiple_plane}
	The aided INS system with $s>1$ plane features in the state vector has the following unobservable directions:
\begin{align}
\mathbf{N}_{\Pi}
&=
\begin{bmatrix}
\mathbf{N}_g   & \mathbf{0}_{12\times 1}   & \mathbf{N}_{i\times j}    \\
-\lfloor {^G\hat{\mathbf{P}}_{I_1}} \times \rfloor {^G\mathbf{g}} & {^G_{\Pi_i}\hat{\mathbf{R}}} &    \mathbf{0}_{3\times 1} \\
-\lfloor {^G\hat{\boldsymbol{\Pi}}_1} \times\rfloor {^G\mathbf{g}} &  {^G\hat{\mathbf{n}}_{\pi_1}}{\mathbf{e}^{\top}_3}{^{\Pi_1}_{\Pi_i}\hat{\mathbf{R}}}
&  \mathbf{0}_{3\times 1}  \\
\vdots  &   \vdots &  \vdots \\
-\lfloor {^G\hat{\boldsymbol{\Pi}}_{s}} \times\rfloor {^G\mathbf{g}} &  {^G\hat{\mathbf{n}}_{\pi_{s}}}{\mathbf{e}^{\top}_3}{^{\Pi_{s}}_{\Pi_{i}}\hat{\mathbf{R}}}
&  \mathbf{0}_{3\times 1} 
\end{bmatrix} 
=:
\begin{bmatrix}
\mathbf{N}_{\Pi 1} &  \mathbf{N}_{\Pi {2:4}} &  \mathbf{N}_{\Pi {5}}
\end{bmatrix}
\end{align}
where $^G\mathbf{n}_{\pi_i}$ is the normal direction vector for plane $i$ ($i, j\in \left\{1\ldots s\right\}$) and $\mathbf{N}_{i\times j}$ is defined by:
%
\begin{equation}
\mathbf{N}_{i\times j}=
	\begin{bmatrix}
	\mathbf{0}_{1\times 6} & 
	\left(\lfloor {^G\hat{\mathbf{n}}^{\perp}_{\pi_i}} \times \rfloor {^G\hat{\mathbf{n}}^{\perp}_{\pi_j}}\right)^{\top} & 
	\mathbf{0}_{1\times 3}
	\end{bmatrix}^{\top}
\end{equation}
Depending on the number of planes and their properties, we have the following observatioins:
	\begin{itemize}
		\item If $s=2$ and the planes are not parallel, the system will have at least 5 unobservable directions given by $\mathbf{N}_{\Pi 1:5}$.
		\item If $s\geq 3$ and these planes' intersections are not parallel, the system will have at least 4 unobservable directions given by $\mathbf{N}_{\Pi 1:4}$. 
	\end{itemize}

\end{lem}
\begin{proof}
	See Appendix \ref{apd_plane}.
\end{proof}
%
%

Up to this point, 
we have  shown from the system observability perspective that the minimal CP representation is an appealing parameterization 
in part because  it preserves the observability properties for aided INS with plane features. 
Note that if all the planes are parallel, the linearized system will have three more unobservable directions corresponding to the motion perpendicular to the plane normal direction and rotation around the normal direction. 
If all the planes' intersections are parallel, then the motion along the plane intersection direction is unobservable. 
For completeness and comparison, in Appendix \ref{apd_hesse_plane} we also present the observability analysis for the aided INS with plane feature using the Hesse form (i.e., Model 2 in Table \ref{tab_plane}). 

%% file: sections/pt_line_plane.tex
\section{Observability Analysis for  Aided INS  with Heterogeneous Features} 
\label{sec:heterogeneous-features}

In this section, we  study  the observability properties for the aided INS with different combinations of geometrical features including points, lines and planes. 
To keep presentation concise, in the following we first consider that case of one feature of each type included in the state vector, and then extend to the general case of multiple heterogeneous features.

\subsection{Point and Line Measurements}

Consider a point feature $^G\mathbf{P}_{\mathbf{f}}$  and a line feature $^G\mathbf{L}$  in the state vector, yielding
$ ^G\mathbf{x}_{\mathbf{f}} = \begin{bmatrix}
{^G\mathbf{P}^{\top}_{\mathbf{f}}}  &  {^G\mathbf{L}^{\top}}
\end{bmatrix}^{\top}$ [see~\eqref{eq:state000}].
If the sensor measures both the point and line features,  we have the following measurement model [see~\eqref{eq_x_meas} and \eqref{line-meas-im}]:
\begin{eqnarray} \label{eq_point_line}
	\mathbf{z}_{pl}=
	\begin{bmatrix}
	\mathbf{z}_{p} \\
	\mathbf{z}_{l}
	\end{bmatrix}
\end{eqnarray}
For  observability analysis, we  compute the measurement Jacobians w.r.t.  the point and line feature as [see~\eqref{eq:H_I-point} and \eqref{eq:H_I-line}]:
%
\begin{align}
\scalemath{1}{\mathbf{H}^{(pl)}_{I_k}}
=&
\scalemath{1}{
	\begin{bmatrix}
	\mathbf{H}_{proj,k} 
	&  \mathbf{0}_{2\times 3}  \\
	\mathbf{0}_{2\times 3} & 
	\mathbf{H}_{l,k}
	\mathbf{K}
	\\
	\end{bmatrix} 
	\begin{bmatrix}
	{^{I_k}_G\hat{\mathbf{R}}}   & \mathbf{0}_3 \\	\mathbf{0}_3  & {^{I_k}_G\hat{\mathbf{R}}}
	\end{bmatrix}
	\times}
\begin{bmatrix}
\mathbf{H}_{p1}  &  \mathbf{0}_{3\times 9}   &  -\mathbf{I}_{3}  &  \mathbf{I}_{3} & \mathbf{0}_3 &  \mathbf{0}_{3\times 1}  \\
\mathbf{H}_{l1} & \mathbf{0}_{3\times 9} & \lfloor {^G\hat{\mathbf{v}}_{L}}\times \rfloor & \mathbf{0}_3 & \mathbf{H}_{l2}  & \mathbf{H}_{l3} 
\end{bmatrix}
\end{align}\noindent
The $k$-th block row of the observability matrix $\mathbf{M}(\mathbf{x})$ becomes:
\begin{align}
&\scalemath{1}{
\mathbf{H}^{(pl)}_{I_k}\boldsymbol{\Phi}_{(k,1)}
}
=
\scalemath{1}{
	\begin{bmatrix}
	\mathbf{H}_{proj,k} 
	&  \mathbf{0}_{2\times 3}  \\
	\mathbf{0}_{2\times 3} & 
	\mathbf{H}_{l,k}
	\mathbf{K}
	\end{bmatrix}}
	\begin{bmatrix}
	{^{I_k}_G\hat{\mathbf{R}}}   & \mathbf{0}_3 \\	\mathbf{0}_3  & {^{I_k}_G\hat{\mathbf{R}}}
	\end{bmatrix}
	\scalemath{1}{
\begin{bmatrix}
\boldsymbol{\Gamma}_1  &   \boldsymbol{\Gamma}_2  &  \boldsymbol{\Gamma}_3  &  \boldsymbol{\Gamma}_4  & -\mathbf{I}_3  & \mathbf{I}_3 & \mathbf{0}_3 & \mathbf{0}_{3\times 1} \\
\boldsymbol{\Gamma}_{l1}  &   \boldsymbol{\Gamma}_{l2} &  \lfloor {^G\hat{\mathbf{v}}_L}\times\rfloor\delta t_k   &\boldsymbol{\Gamma}_{l3} & 
\lfloor {^G\hat{\mathbf{v}}_L}\times\rfloor  & \mathbf{0}_{3}  &\boldsymbol{\Gamma}_{l4} & \boldsymbol{\Gamma}_{l5} 
\end{bmatrix}}
\end{align}
It is not difficult to find the unobservable directions as:
\begin{align}
\mathbf{N}_{pl}
&=
\begin{bmatrix}
\mathbf{N}_g	&		\mathbf{0}_{12\times 3} \\
-\lfloor{^G\hat{\mathbf{P}}_{I_1}}\times \rfloor {^G\mathbf{g}} 	& 		\mathbf{I}_3 \\ 
-\lfloor{^G\hat{\mathbf{P}}_{\mathbf{f}}}\times \rfloor {^G\mathbf{g}} 	& 		\mathbf{I}_3 \\ 
-{^G\mathbf{g}}																&		\frac{w_{2}}{w_{1}}{^G\hat{\mathbf{v}}_{\mathbf{e}}}\mathbf{e}^{\top}_{1} {^L_{G}\hat{\mathbf{R}}}		\\	
0																			&	\eta^2{w^2_{2}}\mathbf{e}^{\top}_{3} {^L_{G}\hat{\mathbf{R}}}  
\end{bmatrix} 
=:
\begin{bmatrix}
\mathbf{N}_{pl 1} & \mathbf{N}_{pl2:4}
\end{bmatrix}
\end{align}
Clearly, for the aided INS with combination features of point and line, there are also at least 4 unobservable directions: 
one is $\mathbf{N}_{pl1}$ which relates the rotation around the gravity direction, 
and the other three are $\mathbf{N}_{pl2:4}$ which correspond to the global position of the sensor platform.

Moreover, we can readily extend to the case of multiple points and lines.
Assuming $m$ points and $l$ lines in the state vector, 
we can show that there are also at least 4 unobservable directions: 
\begin{align}
	\mathbf{N}_{PL}
	&
=
	\begin{bmatrix}
	\mathbf{N}_g		&		\mathbf{0}_{12\times 3} \\
	-\lfloor{^G\hat{\mathbf{P}}_{I_1}}\times \rfloor {^G\mathbf{g}} 	& 		   \mathbf{I}_3 \\ 
	-\lfloor{^G\hat{\mathbf{P}}_{\mathbf{f}_1}}\times \rfloor {^G\mathbf{g}} 	& 	\mathbf{I}_3 \\ 
	\vdots 																		&			\vdots 					\\
	-\lfloor{^G\hat{\mathbf{P}}_{\mathbf{f}_m}}\times \rfloor {^G\mathbf{g}} 	& 	\mathbf{I}_3 \\ 
	-{^G\mathbf{g}}																&		\frac{w_{1,2}}{w_{1,1}}{^G\hat{\mathbf{v}}_{\mathbf{e}1}}\mathbf{e}^{\top}_{1}{^{L_1}_{G}\hat{\mathbf{R}}}		\\	
	0																			&	\eta^2_{1}{w^2_{1,2}}\mathbf{e}^{\top}_{3}{^{L_1}_{G}\hat{\mathbf{R}}}  \\
	\vdots 																		&	\vdots																\\													\\
	-{^G\mathbf{g}}																&		\frac{w_{l,2}}{w_{l,1}}{^G\hat{\mathbf{v}}_{\mathbf{e}l}}\mathbf{e}^{\top}_{1}{^{L_l}_{G}\hat{\mathbf{R}}^{\top}}		\\	
	0																			&	\eta^2_{l}{w^2_{l,2}}\mathbf{e}^{\top}_{3} {^{L_l}_{G}\hat{\mathbf{R}}^{\top}} 	
	\end{bmatrix} 
	=:
	\begin{bmatrix}
	\mathbf{N}_{PL 1} & \mathbf{N}_{PL 2:4}
	\end{bmatrix}
\end{align}
%

\subsection{Point and Plane Measurements} 

Consider the case where we have a point and a plane in the state vector~\eqref{eq:state000},
and thus the feature state  
$\mathbf{x}_{\mathbf{f}}=\begin{bmatrix}
^G\mathbf{P}^{\top}_{\mathbf{f}} & {^G\boldsymbol{\Pi}}
\end{bmatrix}^{\top}$. 
Therefore, in this case, the measurement model consists of the point measurement and plane measurement [see~\eqref{eq_x_meas} and \eqref{plane-meas-model}]:
\begin{equation}
	\mathbf{z}_{p \pi}=
	\begin{bmatrix}
	 \mathbf{z}_{p} \\ \mathbf{z}_{\pi}
	\end{bmatrix}
\end{equation}
The measurement Jacobian is computed as [see~\eqref{eq:H_I-point} and \eqref{eq:H_I-plane}]:
\begin{align}
\scalemath{1}{\mathbf{H}^{(p\pi)}_{I_k}}
=&
\scalemath{1}{
	\begin{bmatrix}
	\mathbf{H}_{proj,k} 
	 & \mathbf{0}_{2\times 3} \\
	\mathbf{0}_{3}  & \mathbf{I}_3
	\end{bmatrix}
	\begin{bmatrix}
	{^{I_k}_G\hat{\mathbf{R}}}   & \mathbf{0}_3 \\	\mathbf{0}_3  & {^{I_k}_G\hat{\mathbf{R}}}
	\end{bmatrix}\times 
}
\scalemath{1}{
\begin{bmatrix}
\mathbf{H}_{p1}  &  \mathbf{0}_{3\times 9}   &  -\mathbf{I}_{3}  &  \mathbf{I}_{3}  & \mathbf{0}_3 \\
\mathbf{H}_{\pi 1} & \mathbf{0}_{3\times 9} & -{^G\hat{\mathbf{n}}_{\pi}}{^G\hat{\mathbf{n}}^{\top}_{\pi}} & \mathbf{0}_3  & \mathbf{H}_{\pi 2}
\end{bmatrix}
}
\end{align}\noindent
The $k$-th block row of the observability matrix $\mathbf{M}(\mathbf{x})$ is:
\begin{align}
&\scalemath{1}{
\mathbf{H}^{(p\pi)}_{I_k}\boldsymbol{\Phi}_{(k,1)}}
=
\scalemath{1}{
	\begin{bmatrix}
	\mathbf{H}_{proj,k} 
	& \mathbf{0}_{2\times 3} \\
	\mathbf{0}_{3}  & \mathbf{I}_3
	\end{bmatrix}
	\begin{bmatrix}
	{^{I_k}_G\hat{\mathbf{R}}} &  \mathbf{0}_3 \\ \mathbf{0}_3 & {^{I_k}_G\hat{\mathbf{R}}}
	\end{bmatrix}
} 
\scalemath{1}{	
\begin{bmatrix}
\boldsymbol{\Gamma}_1  &   \boldsymbol{\Gamma}_2  &  \boldsymbol{\Gamma}_3  &  \boldsymbol{\Gamma}_4  & -\mathbf{I}_3  & \mathbf{I}_3  & \mathbf{0}_3 \\
\boldsymbol{\Gamma}_{\pi 1} & \boldsymbol{\Gamma}_{\pi 2}  & 
-{^G\hat{\mathbf{n}}_{\pi}} {^G\hat{\mathbf{n}}^{\top}_{\pi}}\delta t_k &
\boldsymbol{\Gamma}_{\pi 3} & 
-   
{^G\hat{\mathbf{n}}_{\pi}} {^G\hat{\mathbf{n}}^{\top}_{\pi}} &
\mathbf{0}_3  &  
\boldsymbol{\Gamma}_{\pi 4}
\end{bmatrix}
}
\end{align}
where $\boldsymbol{\Gamma}_i, \boldsymbol{\Gamma}_{\pi i}$, $i\in 1\ldots 4$, are the same as in the previous sections. 
The unobservable directions can hence be found as:
\begin{align}
\mathbf{N}_{p\pi}
&=
\begin{bmatrix}
\mathbf{N}_g											&		\mathbf{0}_{12\times 3} \\
-\lfloor{^G\hat{\mathbf{P}}_{I_1}}\times \rfloor {^G\mathbf{g}} 	& 		\mathbf{I}_3 \\ 
-\lfloor{^G\hat{\mathbf{P}}_{\mathbf{f}}}\times \rfloor {^G\mathbf{g}} 	& 	\mathbf{I}_3 \\ 
-\lfloor {^G\hat{\boldsymbol{\Pi}}\times} \rfloor {^G\mathbf{g}} &  {^G\hat{\mathbf{n}}_{\Pi}}{\mathbf{e}^{\top}_3} {^{\Pi}_G\hat{\mathbf{R}}}
\end{bmatrix} 
=:
\begin{bmatrix}
\mathbf{N}_{p\pi 1} & \mathbf{N}_{p\pi 2:4} 
\end{bmatrix}
\end{align}
Clearly, in this case, 
the unobservable directions of the aided INS 
consist of the rotation around the gravity direction $\mathbf{N}_{p\pi 1}$, and the global position of the sensor platform $\mathbf{N}_{p\pi 2:4}$.

Similarly, we can extend this analysis to multiple features. 
Given $m$ points and $s$ planes in the state vector, the null space of the observability matrix can be obtained as follows: 
\begin{align}
\mathbf{N}_{P\Pi}
&=
\begin{bmatrix}
\mathbf{N}_g									&		\mathbf{0}_{12\times 3} \\
-\lfloor{^G\hat{\mathbf{P}}_{I_1}}\times \rfloor {^G\mathbf{g}} 	& 		\mathbf{I}_3 \\ 
-\lfloor{^G\hat{\mathbf{P}}_{\mathbf{f}_1}}\times \rfloor {^G\mathbf{g}} 	& 		\mathbf{I}_3 \\ 
\vdots 																		&			\vdots 					\\
-\lfloor{^G\hat{\mathbf{P}}_{\mathbf{f}_m}}\times \rfloor {^G\mathbf{g}} 	& 		\mathbf{I}_3 \\ 
-\lfloor {^G\boldsymbol{\Pi}_1} \rfloor {^G\mathbf{g}} &  {^G\hat{\mathbf{n}}_{\pi_1}}{\mathbf{e}^{\top}_3}{^{\Pi_1}_{G}\hat{\mathbf{R}}}  \\
\vdots 																		&	\vdots																\\															\\
-\lfloor {^G\boldsymbol{\Pi}_s} \rfloor {^G\mathbf{g}} &  {^G\hat{\mathbf{n}}_{\pi_s}}{\mathbf{e}^{\top}_3}{^{\Pi_s}_{G}\hat{\mathbf{R}}}  	
\end{bmatrix} 
=:
\begin{bmatrix}
\mathbf{N}_{P\Pi 1} & \mathbf{N}_{P\Pi 2:4} 
\end{bmatrix}
\end{align}
%

\subsection{Line and Plane Measurements}

Now consider the case where a line and a plane is in the state vector, 
i.e., 
$\mathbf{x}_{\mathbf{f}}=\begin{bmatrix}
^G\mathbf{L}^{\top} & {^G\boldsymbol{\Pi}^{\top}}
\end{bmatrix}^{\top}$. 
The measurement model is given by [see~\eqref{line-meas-im} and \eqref{plane-meas-model}]:
%
\begin{equation}
	\mathbf{z}_{l\pi} = 
	\begin{bmatrix}
	\mathbf{z}_{l} \\
	\mathbf{z}_{\pi}
	\end{bmatrix}
\end{equation}
%
The measurement Jacobian becomes [see~\eqref{eq:H_I-line} and \eqref{eq:H_I-plane}]:
\begin{align}
\scalemath{0.9}{\mathbf{H}^{(l\pi)}_{I_k}}
=&
\scalemath{1}{
	\begin{bmatrix}
	\mathbf{H}_{l,k}
	\mathbf{K}
	& \mathbf{0}_{2\times 3}  \\
	 \mathbf{0}_{3} & \mathbf{I}_3
	\end{bmatrix} 
	\begin{bmatrix}
	{^{I_k}_G\hat{\mathbf{R}}}  &  \mathbf{0}_3  \\ \mathbf{0}_3 & {^{I_k}_G\hat{\mathbf{R}}} 
	\end{bmatrix}
}
\scalemath{1}{
\begin{bmatrix}
\mathbf{H}_{l1} & \mathbf{0}_{3\times 9} & \lfloor {^G\hat{\mathbf{v}}_{L}}\times \rfloor &  \mathbf{H}_{l2}  & \mathbf{H}_{l3} & \mathbf{0}_3 \\
\mathbf{H}_{\pi 1} & \mathbf{0}_{3\times 9} & -{^G\hat{\mathbf{n}}_{\pi}}{^G\hat{\mathbf{n}}^{\top}_{\pi}} &  \mathbf{0}_3 & \mathbf{0}_{3\times 1} & \mathbf{H}_{\pi 2}
\end{bmatrix}
}
\end{align}\noindent
%
The $k$-th block row of the observability matrix $\mathbf{M}(\mathbf{x})$ is:
\begin{align}\label{eq_m_line_plane}
&\scalemath{1}{
\mathbf{H}^{(l\pi)}_{I_k}\boldsymbol{\Phi}_{(k,1)}
}
= 
\scalemath{1}{
	\begin{bmatrix}
	\mathbf{H}_{l,k} 
	\mathbf{K}
	& \mathbf{0}_{2\times 3}  \\
 \mathbf{0}_{ 3} & \mathbf{I}_3
	\end{bmatrix}
	\begin{bmatrix}
	{^{I_k}_G\hat{\mathbf{R}}}  &  \mathbf{0}_3  \\ \mathbf{0}_3 & {^{I_k}_G\hat{\mathbf{R}}} 
	\end{bmatrix}
}
\scalemath{0.9}{
\begin{bmatrix}
\boldsymbol{\Gamma}_{l1}  &   \boldsymbol{\Gamma}_{l2} &  \lfloor {^G\hat{\mathbf{v}}_L}\times\rfloor\delta t_k   &\boldsymbol{\Gamma}_{l3} & 
\lfloor {^G\hat{\mathbf{v}}_L}\times\rfloor   &\boldsymbol{\Gamma}_{l4} & \boldsymbol{\Gamma}_{l5} & \mathbf{0}_3\\
\boldsymbol{\Gamma}_{\pi 1} & \boldsymbol{\Gamma}_{\pi 2}  & 
-{^G\hat{\mathbf{n}}_{\pi}} {^G\hat{\mathbf{n}}^{\top}_{\pi}}\delta t_k &
\boldsymbol{\Gamma}_{\pi 3} & 
-   
{^G\hat{\mathbf{n}}_{\pi}} {^G\hat{\mathbf{n}}^{\top}_{\pi}} &
  \mathbf{0}_3 & \mathbf{0}_{3\times 1} &
\boldsymbol{\Gamma}_{\pi 4}
\end{bmatrix}
}
\end{align}
Based on that, we have the following result:
\begin{lem}\label{lem_line_plane}
	The aided INS with a single line and plane feature generally has the following unobservable directions:
\begin{align}
\mathbf{N}_{l\pi}
&=
\begin{bmatrix}
\mathbf{N}_g			&		\mathbf{0}_{12\times 3}   &      \mathbf{N}_{\mathbf{v}}   \\
-\lfloor{^G\hat{\mathbf{P}}_{I_1}}\times \rfloor {^G\mathbf{g}} 	& 		\mathbf{I}_3 & \mathbf{0}_{3\times 1}\\ 
-{^G\mathbf{g}}																&		\frac{w_{2}}{w_{1}}{^G\hat{\mathbf{v}}_{\mathbf{e}}}\mathbf{e}^{\top}_{1}		  &      \mathbf{0}_{3\times 1}   \\	
0																			&	\eta^2{w^2_{2}}\mathbf{e}^{\top}_{3} {^L_{G}\hat{\mathbf{R}}}   &      {0}    \\
-\lfloor {^G\hat{\boldsymbol{\Pi}}\times} \rfloor {^G\mathbf{g}} &  {^G\hat{\mathbf{n}}_{\pi}}{\mathbf{e}^{\top}_3}{^{\Pi}_{G}\hat{\mathbf{R}}}  &      \mathbf{0}_{3\times 1}   
\end{bmatrix} 
=:
\begin{bmatrix}
\mathbf{N}_{l\pi 1} & \mathbf{N}_{l\pi 2:4} & \mathbf{N}_{l\pi 5}
\end{bmatrix}
\end{align}
In particular,  depending on the feature properties, we have:
\begin{itemize}
		\item If the line is parallel to the plane, the linearized system will have at least 5 unobservable directions given by $\mathbf{N}_{l\pi 1:5}$. 
		\item If the line is not parallel to the plane, the linearized system will have at least 4 unobservable directions given by $\mathbf{N}_{l\pi 1:4}$.
	\end{itemize}
\end{lem}
\begin{proof}
See Appendix \ref{apd_proof_lem1}. 	
\end{proof}

Similarly, we can extend the  analysis to the case of multiple lines and planes. 
Given  $l$ lines and $s$ planes, the unobservable directions fo the aided INS with general, unparallel lines and planes are given by:
\begin{align}
\mathbf{N}_{L\Pi}
&=
\begin{bmatrix}
\mathbf{N}_g		&		\mathbf{0}_{12\times 3} \\
-\lfloor{^G\hat{\mathbf{P}}_{I_1}}\times \rfloor {^G\mathbf{g}} 	& 	 \mathbf{I}_3 \\ 
-{^G\mathbf{g}}																&		\frac{w_{1,2}}{w_{1,1}}{^G\hat{\mathbf{v}}_{\mathbf{e}1}}\mathbf{e}^{\top}_{1}{^{L_1}_{G}\hat{\mathbf{R}}}		\\	
0																			&	\eta^2_{1}{w^2_{1,2}}\mathbf{e}^{\top}_{3}{^{L_1}_{G}\hat{\mathbf{R}}}  \\
\vdots 																		&	\vdots																\\														\\
-{^G\mathbf{g}}																&		\frac{w_{l,2}}{w_{l,1}}{^G\hat{\mathbf{v}}_{\mathbf{e}l}}\mathbf{e}^{\top}_{1}{^{L_l}_{G}\hat{\mathbf{R}}}		\\	
0																			&	
\eta^2_{l}{w^2_{l,2}}\mathbf{e}^{\top}_{3} {^{L_l}_{G}\hat{\mathbf{R}}}  \\
-\lfloor {^G\boldsymbol{\Pi}_1} \times\rfloor {^G\mathbf{g}} &  {^G\hat{\mathbf{n}}_{\pi_1}}{\mathbf{e}^{\top}_3}{^{\Pi_1}_{G}\hat{\mathbf{R}}}  \\
\vdots 																		&	\vdots																\\															\\
-\lfloor {^G\boldsymbol{\Pi}_s} \times\rfloor {^G\mathbf{g}} &  {^G\hat{\mathbf{n}}_{\pi_s}}{\mathbf{e}^{\top}_3}{^{\Pi_s}_{G}\hat{\mathbf{R}}} 	
\end{bmatrix} 
=:
\begin{bmatrix}
\mathbf{N}_{L\Pi 1} & \mathbf{N}_{L\Pi 2:4} 
\end{bmatrix}
\end{align}
%

\subsection{Point, Line and Plane Measurements}

Lastly, let us consider the case where  all three types of features (a single point, line, and plane) are in  in the state vector, i.e., 
$\mathbf{x}_{\mathbf{f}}= \begin{bmatrix}
^G\mathbf{P}^{\top}_{\mathbf{f}} &  {^G\mathbf{L}^{\top}} &  {^G\boldsymbol{\Pi}^{\top}}   
\end{bmatrix}^{\top}$.  
The measurement model becomes [see~\eqref{eq_x_meas}, \eqref{line-meas-im} and \eqref{plane-meas-model}]:
\begin{equation}\label{eq_point_line_plane}
	\mathbf{z}_{pl\pi}=
	\begin{bmatrix}
	\mathbf{z}_{p} \\ \mathbf{z}_{l}  \\ \mathbf{z}_{\pi}
	\end{bmatrix}
\end{equation}
The measurement Jacobian can be computed as [see~\eqref{eq:H_I-point}, \eqref{eq:H_I-line} and  \eqref{eq:H_I-plane}]:
%
\begin{align}
	&\scalemath{1}{\mathbf{H}^{(pl\pi)}_{I_k}}
=
\scalemath{1}{
	\begin{bmatrix}
	\mathbf{H}_{proj,k} {^{I_k}_G\hat{\mathbf{R}}}  &  \mathbf{0}_{2\times 3} & \mathbf{0}_{2\times 3} \\
	\mathbf{0}_{2\times 3} & 
	\mathbf{H}_{l,k}
	\mathbf{K}
	{^{I_k}_G\hat{\mathbf{R}}} & \mathbf{0}_{2\times 3}  \\
	\mathbf{0}_{3} & \mathbf{0}_{3} & 
	{^{I_k}_G\hat{\mathbf{R}}} 
	\end{bmatrix} 
}
\scalemath{1}{
\begin{bmatrix}
\mathbf{H}_{p1}  &  \mathbf{0}_{3\times 9}   &  -\mathbf{I}_{3}  &  \mathbf{I}_{3} & \mathbf{0}_3 &  \mathbf{0}_{3\times 1} & \mathbf{0}_3 \\
\mathbf{H}_{l1} & \mathbf{0}_{3\times 9} & \lfloor {^G\hat{\mathbf{v}}_{L}}\times \rfloor & \mathbf{0}_3 & \mathbf{H}_{l2}  & \mathbf{H}_{l3} & \mathbf{0}_3 \\
\mathbf{H}_{\pi 1} & \mathbf{0}_{3\times 9} & -{^G\hat{\mathbf{n}}_{\pi}}{^G\hat{\mathbf{n}}^{\top}_{\pi}} & \mathbf{0}_3 & \mathbf{0}_3 & \mathbf{0}_{3\times 1} & \mathbf{H}_{\pi 2}
\end{bmatrix}
}
\end{align}
The $k$-th block row of the observability matrix $\mathbf{M}(\mathbf{x})$ is:
\begin{align}\label{eq_m_pt_line_plane}
&
\scalemath{1}{	\mathbf{H}^{(pl\pi)}_{I_k}\boldsymbol{\Phi}_{(k,1)}}
	=
\scalemath{1}{
	\begin{bmatrix}
	\mathbf{H}_{proj,k} {^{I_k}_G\hat{\mathbf{R}}}  &  \mathbf{0}_{2\times 3} & \mathbf{0}_{2\times 3} \\
	\mathbf{0}_{2\times 3} & 
	\mathbf{H}_{l,k}
	\mathbf{K}
	{^{I_k}_G\hat{\mathbf{R}}} & \mathbf{0}_{2\times 3}  \\
	\mathbf{0}_{3} & \mathbf{0}_{3} & 
	{^{I_k}_G\hat{\mathbf{R}}} 
	\end{bmatrix}
} 
\times
\\ 
&
\scalemath{1}{
\begin{bmatrix}
\boldsymbol{\Gamma}_1  &   \boldsymbol{\Gamma}_2  &  \boldsymbol{\Gamma}_3  &  \boldsymbol{\Gamma}_4  & -\mathbf{I}_3  & \mathbf{I}_3 & \mathbf{0}_3 & \mathbf{0}_{3\times 1} & \mathbf{0}_3 \\
\boldsymbol{\Gamma}_{l1}  &   \boldsymbol{\Gamma}_{l2} &  \lfloor {^G\hat{\mathbf{v}}_L}\times\rfloor\delta t_k   &\boldsymbol{\Gamma}_{l3} & 
\lfloor {^G\hat{\mathbf{v}}_L}\times\rfloor  & \mathbf{0}_{3}  &\boldsymbol{\Gamma}_{l4} & \boldsymbol{\Gamma}_{l5} & \mathbf{0}_3\\
\boldsymbol{\Gamma}_{\pi 1} & \boldsymbol{\Gamma}_{\pi 2}  & 
-{^G\hat{\mathbf{n}}_{\pi}} {^G\hat{\mathbf{n}}^{\top}_{\pi}}\delta t_k &
\boldsymbol{\Gamma}_{\pi 3} & 
- {^G\hat{\mathbf{n}}_{\pi}} {^G\hat{\mathbf{n}}^{\top}_{\pi}} &
\mathbf{0}_3  &  \mathbf{0}_3 & \mathbf{0}_{3\times 1} &
\boldsymbol{\Gamma}_{\pi 4}
\end{bmatrix}}
\nonumber
\end{align}
The observability properties of this aided INS are given by:
\begin{lem}\label{lem_pt_line_plane}
The aided INS with one point, one line and one plane feature in the state vector, has  at least 4 unobservable directions:
\begin{align} \label{eq:N-plpi-1}
\mathbf{N}_{pl\pi}
&=
\begin{bmatrix}
\mathbf{N}_g											&		\mathbf{0}_{12\times 3} \\
-\lfloor{^G\hat{\mathbf{P}}_{I_1}}\times \rfloor {^G\mathbf{g}} 	& 		\mathbf{I}_3 \\ 
-\lfloor{^G\hat{\mathbf{P}}_{\mathbf{f}}}\times \rfloor {^G\mathbf{g}} 	& 	\mathbf{I}_3 \\ 
-{^G\mathbf{g}}																&		\frac{w_{2}}{w_{1}}{^G\hat{\mathbf{v}}_{\mathbf{e}}}\mathbf{e}^{\top}_{1}{^L_{G}\hat{\mathbf{R}}}		\\	
0		&	\eta^2{w^2_{2}}\mathbf{e}^{\top}_{3} {^L_{G}\hat{\mathbf{R}}}  \\
-\lfloor {^G\hat{\boldsymbol{\Pi}}\times} \rfloor {^G\mathbf{g}} &  {^G\hat{\mathbf{n}}_{\pi}}{\mathbf{e}^{\top}_3}{^{\Pi}_{G}\hat{\mathbf{R}}}
\end{bmatrix} 
=:
\begin{bmatrix}
\mathbf{N}_{pl\pi 1} & \mathbf{N}_{pl\pi 2:4}
\end{bmatrix}
\end{align}
\end{lem}

\begin{proof}
	See Appendix \ref{apd_proof_lem2}. 
\end{proof}
%
%

We also extend this analysis to multiple points, lines and planes. 
Given  $m$ points, $l$ lines and $s$ planes for the estimation, the unobservable directions can be found  as follows: 
\begin{align}
\scalemath{1}{\mathbf{N}_{PL\Pi}}
&=
\begin{bmatrix}
\mathbf{N}_g					&		\mathbf{0}_{12\times 3} \\
-\lfloor{^G\hat{\mathbf{P}}_{I_1}}\times \rfloor {^G\mathbf{g}} 	& 		\mathbf{I}_3  \\ 
-\lfloor{^G\hat{\mathbf{P}}_{\mathbf{f}_1}}\times \rfloor {^G\mathbf{g}} 	& 	\mathbf{I}_3  \\ 
\vdots 																		&			\vdots 					\\
-\lfloor{^G\hat{\mathbf{P}}_{\mathbf{f}_m}}\times \rfloor {^G\mathbf{g}} 	& 	\mathbf{I}_3  \\ 
-{^G\mathbf{g}}																&		\frac{w_{1,2}}{w_{1,1}}{^G\hat{\mathbf{v}}_{\mathbf{e}1}}\mathbf{e}^{\top}_{1}{^{L_1}_{G}\hat{\mathbf{R}}}		\\	
0																			&	\eta^2_{1}{w^2_{1,2}}\mathbf{e}^{\top}_{3}{^{L_1}_{G}\hat{\mathbf{R}}}  \\
\vdots 																		&	\vdots																\\													\\
-{^G\mathbf{g}}																&		\frac{w_{l,2}}{w_{l,1}}{^G\hat{\mathbf{v}}_{\mathbf{e}l}}\mathbf{e}^{\top}_{1}{^{L_l}_{G}\hat{\mathbf{R}}}		\\	
0	&	\eta^2_{l}{w^2_{l,2}}\mathbf{e}^{\top}_{3} {^{L_l}_{G}\hat{\mathbf{R}}}  \\
-\lfloor {^G\boldsymbol{\Pi}_1} \times\rfloor {^G\mathbf{g}} &  {^G\hat{\mathbf{n}}_{\pi_1}}{\mathbf{e}^{\top}_3}{^{\Pi_1}_{G}\hat{\mathbf{R}}}  \\
\vdots 																		&	\vdots																\\														\\
-\lfloor {^G\boldsymbol{\Pi}_s} \times\rfloor {^G\mathbf{g}} &  {^G\hat{\mathbf{n}}_{\pi_s}}{\mathbf{e}^{\top}_3}{^{\Pi_s}_{G}\hat{\mathbf{R}}}
\end{bmatrix}  
=:
\begin{bmatrix}
	\mathbf{N}_{PL\Pi 1} & \mathbf{N}_{PL\Pi 2:4} 
	\end{bmatrix}
\end{align}
%

%% file: sections/global.tex
\section{Observability Analysis of Aided INS with Global Measurements}
\label{sec:global-meas}

Aided INS may also have access to (partially) global measurements provided by, for example, GPS receivers, sun/star sensors, 
barometers and compasses.
Intuitively, such measurements would alter the system observability properties, even if only partial (not full 6 DOF pose) information is available.
In this section, we systematically examine the impacts of such measurements on the system observability.

\subsection{Global Position Measurements}

We consider the case where besides the point, line and plane measurements,  
global position measurements are also available from, for example, a GPS receiver or a barometer.
In the following, we use such a global measurement individually along $x$, $y$ and $z$-axis.

\subsubsection{Global x Measurement}
If sensor's translation along $x$ direction is known, then the additional global $x$-axis measurement is given by $z^{(x)} = \mathbf{e}^{\top}_1{^G\mathbf{P}_{I}}$.
The measurement Jacobian and the block row of observability matrix   can be computed as [see \eqref{eq_m_pt_line_plane}]:
\begin{align}
&\scalemath{1}{
	\mathbf{H}_{I_k}\boldsymbol{\Phi}_{(k,1)}
	=
	\begin{bmatrix}
	\mathbf{H}^{(pl\pi)}_{I_k}\boldsymbol{\Phi}_{(k,1)} \\
	\mathbf{H}^{(x)}_{I_k}\boldsymbol{\Phi}_{(k,1)} 
	\end{bmatrix}
}
\end{align}
where  $\mathbf{H}^{(x)}_{I_k}$ is the global $x$ measurement Jacobian, yielding: 
\begin{equation}
	\mathbf{H}^{(x)}_{I_k}  \boldsymbol{\Phi}_{(k,1)} 
	=
	\begin{bmatrix}
	\mathbf{0}_{1\times 12} & 
	\mathbf{e}^{\top}_1 & 
	\mathbf{0}_{1\times 3} &
	\mathbf{0}_{1\times 4} &
	\mathbf{0}_{1\times 3}
	\end{bmatrix}	
\end{equation}
%
%
%
We can show that the unobservable subspace becomes:
\begin{align}
\mathbf{N}_{x}
&=
\begin{bmatrix}
\mathbf{0}_{12\times 2}    \\
\mathbf{A}_x \\
\mathbf{A}_x \\
\frac{w_{2}}{w_{1}}{^G\hat{\mathbf{v}}_{\mathbf{e}}}\mathbf{e}^{\top}_{1}{^L_{G}\hat{\mathbf{R}}}
\mathbf{A}_x		\\	
\eta^2{w^2_{2}}\mathbf{e}^{\top}_{3} {^L_{G}\hat{\mathbf{R}}}   
\mathbf{A}_x  \\
{^G\hat{\mathbf{n}}_{\pi}}{\mathbf{e}^{\top}_3}{^{\Pi}_{G}\hat{\mathbf{R}}}
\mathbf{A}_x
\end{bmatrix}
\end{align}
where $\mathbf{A}_x = \begin{bmatrix} \mathbf{0}_{2\times 1} & \mathbf{I}_2 \end{bmatrix}^{\top}$. 
As compared to $\mathbf{N}$ in \eqref{eq:N-plpi-1} without global $x$ measurements, 
both the global translation in $x$ direction and the rotation around the gravity direction become observable.
\subsubsection{Global y Measurement}
If sensor's translation along $y$ direction is known, then the additional global $y$-axis measurement is given by $z^{(y)} = \mathbf{e}^{\top}_2{^G\mathbf{P}_{I}}$.
The measurement Jacobian and the block row of observability matrix  can be computed as [see \eqref{eq_m_pt_line_plane}]:
\begin{align}
&\scalemath{1}{
	\mathbf{H}_{I_k}\boldsymbol{\Phi}_{(k,1)}
	=
	\begin{bmatrix}
	\mathbf{H}^{(pl\pi)}_{I_k}\boldsymbol{\Phi}_{(k,1)} \\
	\mathbf{H}^{(y)}_{I_k}\boldsymbol{\Phi}_{(k,1)} 
	\end{bmatrix}
}
\end{align}
where  $\mathbf{H}^{(y)}_{I_k}$ is the global $y$ measurement Jacobian, yielding: 
\begin{align}
&\mathbf{H}^{(y)}_{I_k}  \boldsymbol{\Phi}_{(k,1)} 
=
\begin{bmatrix}
\mathbf{0}_{1\times 12} & 
\mathbf{e}^{\top}_1 & 
\mathbf{0}_{1\times 3} &
\mathbf{0}_{1\times 4} &
\mathbf{0}_{1\times 3}
\end{bmatrix}	
\end{align}
%
%
%
We can show that the unobservable subspace becomes:
\begin{align}
\mathbf{N}_{y}
&=
\begin{bmatrix}
\mathbf{0}_{12\times 2}    \\
\mathbf{A}_y \\
\mathbf{A}_y \\
\frac{w_{2}}{w_{1}}{^G\hat{\mathbf{v}}_{\mathbf{e}}}\mathbf{e}^{\top}_{1}{^L_{G}\hat{\mathbf{R}}}
\mathbf{A}_y		\\	
\eta^2{w^2_{2}}\mathbf{e}^{\top}_{3} {^L_{G}\hat{\mathbf{R}}}   
\mathbf{A}_y  \\
{^G\hat{\mathbf{n}}_{\pi}}{\mathbf{e}^{\top}_3}{^{\Pi}_{G}\hat{\mathbf{R}}}
\mathbf{A}_y
\end{bmatrix}
\end{align}
where $\mathbf{A}_y = \begin{bmatrix}  1 & 0 & 0 \\ 0 & 0 & 1 \end{bmatrix}^{\top}$. 
As compared to $\mathbf{N}$ in \eqref{eq:N-plpi-1} without global $y$ measurements, 
both the global translation in $y$ direction and the rotation around the gravity direction become observable.
%

%

%
\subsubsection{Global z Measurement}

Proceeding similarly, 
if the global translation in $z$ direction is directly measured, e.g., by a barometer,  
we have an additional global $z$-axis measurement  $z^{(z)} = \mathbf{e}^{\top}_3{^G\mathbf{P}_{I}}$. 
In this case, the block row of the observability matrix  becomes:
\begin{align}
&\scalemath{1}{
	\mathbf{H}_{I_k}\boldsymbol{\Phi}_{(k,1)}
	=
	\begin{bmatrix}
	\mathbf{H}^{(pl\pi)}_{I_k}\boldsymbol{\Phi}_{(k,1)} \\
	\mathbf{H}^{(z)}_{I_k}\boldsymbol{\Phi}_{(k,1)} 
	\end{bmatrix}
}
\end{align}
Since $\mathbf{e}_3$ is parallel to $^G\mathbf{g}$, we have $ \mathbf{e}^{\top}_3 \lfloor {^G\mathbf{P}_{I_1}} \times \rfloor {^G\mathbf{g}} = 0 $. 
Therefore, the system's unobservable directions become:
\begin{align}
\mathbf{N}_{z}
&=
\begin{bmatrix}
\mathbf{N}_g				&		\mathbf{0}_{12\times 2} \\
-\lfloor{^G\hat{\mathbf{P}}_{I_1}}\times \rfloor {^G\mathbf{g}} 	& 		
\mathbf{A}_z  \\
-\lfloor{^G\hat{\mathbf{P}}_{\mathbf{f}}}\times \rfloor {^G\mathbf{g}} 	& 		
\mathbf{A}_z  \\
-{^G\mathbf{g}}																&		\frac{w_{2}}{w_{1}}{^G\hat{\mathbf{v}}_{\mathbf{e}}}\mathbf{e}^{\top}_{1}{^L_{G}\hat{\mathbf{R}}}
\mathbf{A}_z 		\\	
0																			&	
\eta^2{w^2_{2}}\mathbf{e}^{\top}_{3} {^L_{G}\hat{\mathbf{R}}}   
\mathbf{A}_z  \\
-\lfloor {^G\hat{\boldsymbol{\Pi}}\times} \rfloor {^G\mathbf{g}} &  {^G\hat{\mathbf{n}}_{\pi}}{\mathbf{e}^{\top}_3}{^{\Pi}_{G}\hat{\mathbf{R}}}
\mathbf{A}_z 
\end{bmatrix}
\end{align}
where $\mathbf{A}_z = \begin{bmatrix}  \mathbf{I}_2   & \mathbf{0}_{2\times 1} \end{bmatrix}^{\top}$. 
Clearly, only translation in $z$ becomes observable, 
while, different from the previous case of the global $x$ or $y$ measurements, 
the rotation around the gravity direction remains unobservable.

\subsection{Global Orientation Measurements}

We here consider the case where 
the aided INS has access to global orientation measurements,
for example, provided by a sun sensor, 
or a magnetic compass, 
or by detecting a plane with known orientation~\cite{Guo2013iros,Panahandeh2013IROS}:
$\mathbf z^{(n)} = {^I\mathbf{N}_{n}} = {^{I}_G\mathbf{R}}{^G\mathbf{N}_{n}}$.
In this case, 
the Jacobian and the block row of the observability matrix can be computed as:
\begin{align}
&\scalemath{1}{
	\mathbf{H}_{I_k}\boldsymbol{\Phi}_{(k,1)}
	=
	\begin{bmatrix}
	\mathbf{H}^{(pl\pi)}_{I_k}\boldsymbol{\Phi}_{(k,1)} \\
	\mathbf{H}^{(n)}_{I_k}\boldsymbol{\Phi}_{(k,1)} 
	\end{bmatrix}
}
\end{align}
where $\mathbf{H}^{({n})}_{I_k}$ is the orientation measurement Jacobian, yielding:
\begin{equation}
\mathbf{H}^{({n})}_{I_k} \boldsymbol{\Phi}_{(k,1)}
=
		{^{I_k}_G\hat{\mathbf{R}}} 
		\begin{bmatrix}
		\lfloor {^G\mathbf{N}_{n}}\times \rfloor{^G_{I_1}}\hat{\mathbf{R}} & 
		\boldsymbol{\Gamma}_5 &
		\mathbf{0}_{3\times 19} & 
		\end{bmatrix}	
\end{equation}
where $\scalemath{1}{\boldsymbol{\Gamma}_5=\lfloor {^G\mathbf{N}_{n}}\times\rfloor{^{I_k}_G\hat{\mathbf{R}}^{\top}}\boldsymbol{\Phi}_{12}}$. 
If $^G\mathbf{N}_{n}$ is not parallel to $^G\mathbf{g}$, i.e., $\scalemath{0.85}{\lfloor{^G\mathbf{N}}\times \rfloor {^G\mathbf{g}} \neq 0}$,
 the rotation around the gravity direction becomes observable, and the unobservable directions are:
\begin{align}
\mathbf{N}_{n}
&=
\begin{bmatrix}
  \mathbf{0}_{12\times 3}    \\
\mathbf{I}_{3} \\		
\mathbf{I}_{3} \\		
\frac{w_{2}}{w_{1}}{^G\hat{\mathbf{v}}_{\mathbf{e}}}\mathbf{e}^{\top}_{1}{^L_{G}\hat{\mathbf{R}}}		\\		
\eta^2{w^2_{2}}\mathbf{e}^{\top}_{3} {^L_{G}\hat{\mathbf{R}}}    \\ {^G\hat{\mathbf{n}}_{\pi}}{\mathbf{e}^{\top}_3}{^{\Pi}_{G}\hat{\mathbf{R}}}
\end{bmatrix}
\end{align}
In summary, as expected,  the global measurements will make the aided INS more observable.
In particular,  if a global full position measurements by GPS  or  a prior map are available, the system will become fully observable,
while global orientation measurements can make the rotation around gravitational direction observable, 
as long as this orientation is not parallel to the direction of gravity.
%

%
%

%% file: sections/range_bearing.tex
\section{Analysis of Degenerate Motion}
\label{sec:degenerate}


While Wu et al.~\cite{Wu2017ICRA} have recently reported that pure translation and constant acceleration are degenerate for monocular VINS with point features,
in this section,
we here perform a comprehensive study of degeneration motion for the aided INS with heterogeneous features of points, lines and planes,
which is important to identify  in order to keep  estimators healthy.

In particular, to ease our analysis, 
we use the range and bearing parameterization (i.e., spherical coordinates) of a point feature, 
instead of its conventional 3D position:
\begin{equation}\label{eq:point-spherical}
\mathbf {x_f} := \begin{bmatrix}
r_{\mathbf f} \\ \theta \\ \phi
\end{bmatrix} 
\Rightarrow
{r_{\mathbf{f}}}{\mathbf{b}_{\mathbf{f}}}
=
{r_{\mathbf{f}}}
\begin{bmatrix}
\cos \theta \cos \phi \\
\sin \theta \cos \phi \\
\sin \phi
\end{bmatrix}
\!\!=\!\!
\begin{bmatrix}
x_{\mathbf{f}} \\
y_{\mathbf{f}} \\
z_{\mathbf{f}} 
\end{bmatrix} =: \mathbf{P}_{\mathbf{f}}
\end{equation}
where $r_{\mathbf{f}}$ is the range, $\theta$ and $\phi$ are the horizontal and elevation angle of the point.
In this case (point features), the block row of the observability matrix can be computed as  [see~\eqref{eq:block-row-single-point}]:
\begin{align}
\scalemath{1}{}
	&\mathbf{H}^{(p)}_{I_k}
	\boldsymbol{\Phi}_{(k,1)}
	=
	\mathbf{H}_{proj,k}
	{^{I_k}_G\hat{\mathbf{R}}}
	\begin{bmatrix}
	\boldsymbol{\Gamma_1} &
	\boldsymbol{\Gamma_2} & 
	\boldsymbol{\Gamma_3} & 
	\boldsymbol{\Gamma_4} & 
	-\mathbf{I}_3 & 
	\hat{\mathbf{b}}_{\mathbf{f}} &
	{^G\hat{r}_{\mathbf{f}}}\cos \hat{\phi} \hat{\mathbf{b}}^{\perp}_1 & 
	{^G\hat{r}_{\mathbf{f}}}\hat{\mathbf{b}}^{\perp}_2 
	\end{bmatrix}
\end{align}
where 
	\begin{align}
	\hat{\mathbf{b}}^{\perp}_1 
	&
	\scalemath{1}{
		=
		\begin{bmatrix}
		-\sin \hat{\theta}  & \cos \hat{\theta}  & 0 
		\end{bmatrix}^{\top}
	}
	\\
	\hat{\mathbf{b}}^{\perp}_2
	&
	\scalemath{1}{
		=
		\begin{bmatrix}
		-\cos \hat{\theta} \sin \hat{\phi} & -\sin \hat{\theta} \sin \hat{\phi} & \cos \hat{\phi}  
		\end{bmatrix}^{\top}
	}
	\end{align}
By inspection, the unobservable directions can be found as:
	\begin{align}
	&\mathbf{N}_{rb}
	=
	\scalemath{0.9}{
		\begin{bmatrix}
		\mathbf{N}_{g}   &  \mathbf{0}_{12\times 3}   \\
		-\lfloor {^G\hat{\mathbf{P}}_{I_1}}\times \rfloor {^G\mathbf{g}}   &   \mathbf{I}_3     \\
		0  &  \hat{\mathbf{b}}^{\top}_{\mathbf{f}}  \\
		-g  &  \frac{\left(\hat{\mathbf{b}}^{\perp}_{1}\right)^{\top} }{^G\hat{r}_{\mathbf{f}}\cos \hat{\phi}}     \\
		0  &  \frac{\left(\hat{\mathbf{b}}^{\perp}_{2}\right)^{\top}}{^G\hat{r}_{\mathbf{f}}}  
		\end{bmatrix}
	}
	\scalemath{1}{
		=:
		\begin{bmatrix}
		\mathbf{N}_{rb,r} &  \mathbf{N}_{rb,p} 	
		\end{bmatrix}
	}	
	\end{align}
where $\mathbf{N}_{rb,p}$ and $\mathbf{N}_{rb,r}$ are the unobservable directions associated with  the global translation and the global rotation around the gravity direction, which, as expected, agrees with the preceding analysis~\eqref{eq_point_obs}.
 

\subsection{Pure Translation}

Based the above analysis of point measurements, we show that 
given point, line and plane measurements \eqref{eq_point_line_plane}, 
if the sensor undergoes pure translation, the system gains the following additional unobservable directions 
(by noting that the state vector~\eqref{eq:state000} includes the IMU state, one point in spherical coordinates~\eqref{eq:point-spherical}, one line and one plane):
	\begin{equation}
	\scalemath{1}{
		\mathbf{N}_{R}
		=
		\begin{bmatrix}
		{^{I_1}_G\hat{\mathbf{R}}} \\
		 \mathbf{0}_3  \\  
		- \lfloor {^G\hat{\mathbf{V}}_{I_1}} \times \rfloor \\ 
		-{^{I_1}_G\hat{\mathbf{R}}}\lfloor {^G\mathbf{g}} \times \rfloor \\ 
		- \lfloor {^G\hat{\mathbf{P}}_{I_1}} \times \rfloor \\ 
		-\boldsymbol{\Theta} \\
		-\mathbf{I}_{3} \\
		0_{1\times 3}  \\
		-\lfloor{^G\boldsymbol{\Pi}}\times\rfloor
		\end{bmatrix}}
	\end{equation}
where $
\scalemath{0.9}{
	\boldsymbol{\Theta}
	=
	\begin{bmatrix}
	0      &  0          &   0 \\
	{^G\hat{r}_{\mathbf{f}}}\cos \hat{\theta} \tan \hat{\phi}     &     \sin \hat{\theta} \tan \hat{\phi}    &   -1  \\
	-{^G\hat{r}_{\mathbf{f}}}\sin \hat{\theta}     &   \cos \hat{\theta}   &  0
	\end{bmatrix}
	}\noindent$.
%
%
%
Similar to \cite{Wu2017ICRA},  this unobservable direction can be easily verified [see~\eqref{eq_m_pt_line_plane}]:
\begin{equation}\label{eq:degen-translation}
\scalemath{1}{
	\mathbf{H}^{(pl\pi)}_{I_k}\boldsymbol{\Phi}_{(k,1)}\mathbf{N}_{R}
	=
	\begin{bmatrix}
	\mathbf{H}^{(p)}_{I_k}\boldsymbol{\Phi}_{(k,1)}\mathbf{N}_{R} \\
	\mathbf{H}^{(l)}_{I_k}\boldsymbol{\Phi}_{(k,1)}\mathbf{N}_{R} \\
	\mathbf{H}^{(\pi)}_{I_k}\boldsymbol{\Phi}_{(k,1)}\mathbf{N}_{R} 
	\end{bmatrix}} = \mathbf 0
\end{equation}
Specifically, 
using \eqref{eq:PHI54} and \eqref{eq:GAMMA4}, we have this useful identity: 
$\boldsymbol{\Gamma_4}{^{I_1}_G\hat{\mathbf{R}}} - \frac{1}{2}\delta t^2_k \mathbf{I}_3 = - \bm\Phi_{54}{^{I_1}_G\hat{\mathbf{R}}} - \frac{1}{2}\delta t^2_k \mathbf{I}_3 =\mathbf 0$.
With this, we can easily verify each block row of \eqref{eq:degen-translation} as follows:
\begin{align}
	&\scalemath{0.9}{
	\mathbf{H}^{(p)}_{I_k}\boldsymbol{\Phi}_{(k,1)}\mathbf{N}_{R}
	=
	-\mathbf{H}_{proj,k}{^{I_k}_G\hat{\mathbf{R}}}
	\left(\boldsymbol{\Gamma_4}{^{I_1}_G\hat{\mathbf{R}}} - \frac{1}{2}\delta t^2_k \mathbf{I}_3\right)
	\lfloor {^G\mathbf{g}}\times \rfloor
	=
	\mathbf{0}} \nonumber
	\\
	&\scalemath{0.85}{
	\mathbf{H}^{(l)}_{I_k}\boldsymbol{\Phi}_{(k,1)}\mathbf{N}_{R}
	=
	\mathbf{H}_{l,k}{^{I_k}_G\hat{\mathbf{R}}}
	\Big[
	\lfloor {^G\hat{\mathbf{v}}_{L}} \times \rfloor
	\left(-\boldsymbol{\Phi}_{54}{^{I_1}_G\hat{\mathbf{R}}} - \frac{1}{2}\delta t^2_k \mathbf{I}_3\right)
	\lfloor {^G\mathbf{g}}\times \rfloor
	- } 
	\scalemath{0.85}
	{
	\lfloor {^G\hat{\mathbf{v}}_{L}} \times \rfloor \lfloor {^G\hat{\mathbf{P}}_{I_k}} \times \rfloor  
	- \lfloor \lfloor {^G\hat{\mathbf{P}}_{I_k}} \times \rfloor {^G\hat{\mathbf{v}}_{L}} \times \rfloor
	+ \lfloor {^G\hat{\mathbf{P}}_{I_k}} \times \rfloor\lfloor {^G\hat{\mathbf{v}}_{L}} \times \rfloor
	\Big]
	= \mathbf{0} \nonumber
	}
	\\
	&\scalemath{0.9}{
	\mathbf{H}^{(\pi)}_{I_k}\boldsymbol{\Phi}_{(k,1)}\mathbf{N}_{R}
	=
	{^{I_k}_G\hat{\mathbf{R}}}
	{^G\hat{\mathbf{n}}_{\pi}}{^G\hat{\mathbf{n}}^{\top}_{\pi}}
	\left(\boldsymbol{\Phi}_{54}{^{I_1}_G\hat{\mathbf{R}}} + \frac{1}{2}\delta t^2_k \mathbf{I}_3\right)
	\lfloor {^G\mathbf{g}}\times \rfloor
	=
	\mathbf{0}} \nonumber	
\end{align}
%
%
where  we have also employed the identities: $\lfloor (\mathbf a \times \mathbf b) \times \rfloor = \mathbf b \mathbf a^\top - \mathbf a \mathbf b^\top$ and $\lfloor \mathbf a\times \rfloor \lfloor \mathbf b \times \rfloor = \mathbf b\mathbf a^\top -\mathbf a^\top \mathbf b\mathbf I$.

We see from  $\boldsymbol{\Theta}$  that its first row corresponding to the range of the point feature [see~\eqref{eq:point-spherical}] are all zeros 
and thus this unobservable direction~\eqref{eq:degen-translation}  relates to the bearing of the feature.
Note also that the global rotation of the sensor becomes unobservable, rather than only the global yaw is unobservable for  general motions [see \eqref{eq:degen-translation} and \eqref{eq:N-plpi-1}]. 
It is important to notice that  {\em no} assumption is made about the type of sensors used, and thus, 
 the aided INS  with generic sensors (not including global sensors) with pure translation will all gain additional unobservable directions of $\mathbf{N}_{R}$.  

\subsection{Constant Acceleration}

As it is not straightforward to have direct plane measurements~\eqref{plane-meas-model} for INS aided by a monocular camera, 
to ease our analysis of VINS, from now on we focus on the point and line measurements \eqref{eq_point_line}. 
In particular, if the camera moves with constant local acceleration,
%
%
i.e., $^{I}\mathbf{a}$ is constant, then the system will have one more unobservable direction given by:
	\begin{equation} \label{eq:Na}
	\scalemath{.95}{
	\mathbf{N}_{a}=
	\begin{bmatrix}
	\mathbf{0}_{1\times 6} & 
	{^G\hat{\mathbf{V}}^{\top}_{I_1}} &
	-{^{I}\hat{\mathbf{a}}^{\top}} & 
	{^G\hat{\mathbf{P}}^{\top}_{I_1}} & 
	{^G\hat{r}_{\mathbf{f}}}\mathbf{e}^{\top}_1 & 
	\mathbf{0}_{1\times 3} &
	-\frac{w_2}{w_1}
	\end{bmatrix}^{\top}
}
	\end{equation}
Since a monocular amera provides only bearing measurements, 
$\mathbf{H}_{proj,k}=\mathbf{H}_{b,k} = \begin{bmatrix}
	^I\hat{\mathbf{b}}^{\top}_{\perp1,k} \\
	^I\hat{\mathbf{b}}^{\top}_{\perp2,k} 
	\end{bmatrix}$, 
	where $^I\hat{\mathbf b}_{\perp i,k}$ ($i=1,2$) are  orthogonal  to $^I\hat{\mathbf{b}}_{\mathbf f}$
(see Appendix~\ref{apd_meas}). 
In this case, we have:
\begin{equation}
	\mathbf{H}_{I_k}\boldsymbol{\Phi}_{(k,1)}\mathbf{N}_a
	=
	\begin{bmatrix}
	\mathbf{H}^{(p)}_{I_k}\boldsymbol{\Phi}_{(k,1)}\mathbf{N}_a \\
	\mathbf{H}^{(l)}_{I_k}\boldsymbol{\Phi}_{(k,1)}\mathbf{N}_a
	\end{bmatrix} = \mathbf 0
\end{equation}
which can be verified by using the identity shown in~\cite{Wu2017ICRA}:
$\boldsymbol{\Gamma_4}{^I\mathbf{a}}={^G\hat{\mathbf{P}}_{I_k}} - {^G\hat{\mathbf{P}}_{I_1}} -{^G\hat{\mathbf{V}}_{I_1}}\delta t_k$:
\begin{align}
	&\mathbf{H}^{(p)}_{I_k}\boldsymbol{\Phi}_{(k,1)}\mathbf{N}_{a} 
	=
	\mathbf{H}_{b,k}{^{I_k}_G\hat{\mathbf{R}}} 
	\left(
	-{^G\hat{\mathbf{V}}_{I_1}}\delta t_k - \boldsymbol{\Gamma_4}{^I\hat{\mathbf{a}}} - {^G\hat{\mathbf{P}}_{I_1}} + {^Gr_{\mathbf{f}}}{^G\hat{\mathbf{b}}_{\mathbf{f}}}
	\right)
	\nonumber
	\\
	&
	=
	\mathbf{H}_{b,k}{^{I_k}_G\hat{\mathbf{R}}}
	\left(
	{^G\hat{\mathbf{P}}_{\mathbf{f}}} - {^G\hat{\mathbf{P}}_{I_k}} 
	\right)
	=
	\mathbf{H}_{b,k}
	{^{I_k}\hat{\mathbf{P}}_{\mathbf{f}}}
	\overset{\eqref{eq_mono_meas}}{=}\mathbf{0}
	\label{eq:degenerate-temp1}
	\\
	&\mathbf{H}^{(l)}_{I_k}\boldsymbol{\Phi}_{(k,1)}\mathbf{N}_{a} 
	=
	\scalemath{1}{}
	\mathbf{H}_{l,k}\mathbf{K}{^{I_k}_G\hat{\mathbf{R}}} 
	\left(\lfloor {^G\hat{\mathbf{v}}_{L}} \times \rfloor {^G\hat{\mathbf{P}}_{I_k}}
	+ \frac{w^2_2}{w^2_1} {^G\hat{\mathbf{n}}_{L}} + \lfloor {^G\hat{\mathbf{P}}_{I_k}} \times \rfloor {^G\hat{\mathbf{v}}_{L}}
	\right)
	\nonumber
	\\
	&=
	\frac{w^2_2}{w^2_1}
	\mathbf{H}_{l,k}\mathbf{K}{^{I_k}_G\hat{\mathbf{R}}}
	{^G\hat{\mathbf{n}}_{L}} \overset{\eqref{eq:def-H_l},\eqref{eq:def-l-prime}}{=} \mathbf{0}
	\label{eq:degenerate-temp2}
\end{align}
%
%
%
Note that this unobservable subspace~\eqref{eq:Na}  relates only to the scale as its nonzeros all appear on the scale-sensitive states.
If using a sensor  that can provide the scale (such as stereo and RGBD cameras), 
this unobservable direction will vanish.

\subsection{Pure Rotation}

If the sensor has only rotational motion, then $^G\mathbf{P}_{I_k}=\mathbf{0}_{3\times 1}$. 
For monocular-camera based point and line measurements~\eqref{eq_point_line}, 
the system will  gain the following extra unobservable directions corresponding to the feature scale:
\begin{align}
\mathbf{N}_s
&=
\begin{bmatrix}
\mathbf{0}_{1\times 15}  & \mathbf{e}^{\top}_1 & \mathbf{0}_{1\times 3} & 0 \\
\mathbf{0}_{1\times 15}  & \mathbf{0}_{1\times 3} & \mathbf{0}_{1\times 3} & 1 
\end{bmatrix}^{\top} \\
\Rightarrow~
	\mathbf{H}_{I_k}\boldsymbol{\Phi}_{(k,1)}\mathbf{N}_s
	&=
	\begin{bmatrix}
	\mathbf{H}^{(p)}_{I_k}\boldsymbol{\Phi}_{(k,1)}\mathbf{N}_s \\
	\mathbf{H}^{(l)}_{I_k}\boldsymbol{\Phi}_{(k,1)}\mathbf{N}_s
	\end{bmatrix}	= \mathbf 0
\end{align}
which can be seen as follows [see~\eqref{eq:degenerate-temp1} and \eqref{eq:degenerate-temp2}]:
\begin{align}
&\mathbf{H}^{(p)}_{I_k}\boldsymbol{\Phi}_{(k,1)}\mathbf{N}_s	
=
\begin{bmatrix}
	\mathbf{H}_{b,k}
	{^{I}_G\hat{\mathbf{R}}}
	{^G\hat{\mathbf{b}}_{\mathbf{f}}} \\
		\mathbf{0}
\end{bmatrix}
=	
\begin{bmatrix}
\mathbf{0}  \\
\mathbf{0}
\end{bmatrix}
\\
&\mathbf{H}^{(l)}_{I_k}\boldsymbol{\Phi}_{(k,1)}\mathbf{N}_s	
=
\begin{bmatrix}
\mathbf{0} \\
-\frac{w_2}{w_1}
\mathbf{H}_{l,k}\mathbf{K}{^{I_k}_G\hat{\mathbf{R}}}
{^G\hat{\mathbf{n}}_L}
\end{bmatrix}
=	
\begin{bmatrix}
\mathbf{0}  \\
\mathbf{0}
\end{bmatrix}
\end{align}
Note that the first row of $\mathbf{N}_s$ relates to the scale of the point feature (range $r_{\mathbf f}$), and the second row  to the scale of the line (the shortest distance from the origin to the line),
which implies that we have more unobservable directions related to the feature scales. 

\subsection{Moving Towards Point Feature}

With the point and line measurements~\eqref{eq_point_line}, if the camera moving towards the point feature, 
 the system will  gain one more unobservable direction related to the point scale (range):
\begin{equation}
\scalemath{1}{
\mathbf{N}_1
=
\begin{bmatrix}
\mathbf{0}_{1\times 15}   & \mathbf{e}^{\top}_1 & \mathbf{0}_{1\times 3} & 0 \end{bmatrix}^{\top}}
\end{equation}\noindent
This degenerate motion indicates that the sensor is moving along the direction of the point feature's bearing direction, that is: $^G\mathbf{P}_{I_k}=\alpha {^G\mathbf{b}_{\mathbf{f}}}$, where $\alpha$ denotes the scale of the sensor's motion. 
Then, we can arrive at:
\begin{align}
^{I_k}\mathbf{P}_{\mathbf{f}} 
=
{^{I_k}r}_{\mathbf{f}}{^{I_k}\mathbf{b}_{\mathbf{f}}}
=
{^{I_k}_G}\mathbf{R}\left({^Gr_{\mathbf{f}}}-\alpha\right){^G\mathbf{b}_{\mathbf{f}}}
\end{align}
Similar to the case of pure rotation, 
we can verify the additional unobservable direction $\mathbf{N}_1$ as follows [see~\eqref{eq:degenerate-temp1}]:
%
\begin{align}
\mathbf{H}_{I_k}\boldsymbol{\Phi}_{(k,1)}\mathbf{N}_1
&=
\begin{bmatrix}
\mathbf{H}^{(p)}_{I_k}\boldsymbol{\Phi}_{(k,1)}\mathbf{N}_1 \\
\mathbf{H}^{(l)}_{I_k}\boldsymbol{\Phi}_{(k,1)}\mathbf{N}_1
\end{bmatrix}
\!\!=\!\!
\begin{bmatrix}
\frac{{^{I_k}\hat{r}_{\mathbf{f}}}}{{^G\hat{r}_{\mathbf{f}}}-\alpha}
\mathbf{H}_{b,k}
{^{I_k}\hat{\mathbf{b}}_{\mathbf{f}}} \\
\mathbf{0}
\end{bmatrix}
\!\!=\!\!
\begin{bmatrix}
\mathbf{0} \\
\mathbf{0}
\end{bmatrix}	
\end{align}
%
%
%

%
%
%

\subsection{Moving in Parallel to Line Feature}

Similarly, 
if the camera is moving in parallel to the line feature,  the system will also gain one more unobservable direction related to this line feature's scale (distance):
\begin{align}
\scalemath{1}{
	\mathbf{N}_1
	=
	\begin{bmatrix}
	\mathbf{0}_{1\times 15}   & \mathbf{0}_{1\times 3} & \mathbf{0}_{1\times 3} & 1 \end{bmatrix}^{\top}}
\end{align}
This degenerate motion indicates that the sensor is moving parallel to  the line direction, that is, $\lfloor ^G\mathbf{P}_{I_k} \times \rfloor {^G\mathbf{v}_{L}} =\mathbf{0}$. 
Then, we have the following verification [see~\eqref{eq:degenerate-temp2}]:
\begin{align}
&\mathbf{H}_{I_k}\boldsymbol{\Phi}_{(k,1)}\mathbf{N}_1
= 
\begin{bmatrix}
\mathbf{H}^{(p)}_{I_k}\boldsymbol{\Phi}_{(k,1)}\mathbf{N}_1 \\
\mathbf{H}^{(l)}_{I_k}\boldsymbol{\Phi}_{(k,1)}\mathbf{N}_1
\end{bmatrix}
=
\begin{bmatrix}
\mathbf{0} \\
-\frac{w_2}{w_1}
\mathbf{H}_{l,k}\mathbf{K}{^{I_k}_G\hat{\mathbf{R}}}
{^G\hat{\mathbf{n}}_L}
\end{bmatrix}
=	
\begin{bmatrix}
\mathbf{0}  \\
\mathbf{0}
\end{bmatrix}	
\end{align}
%

%% file: sections/simulation1.tex
\section{Monte Carlo Simulations}
\label{sec:sim}

To validate our observability analysis of aided INS using heterogeneous geometric features, 
we perform extensive Monte Carlo simulations of vision-aided INS:\footnote{Note that similar results as presented in this section would be expected if other aiding sensors are used, 
for example, acoustic-aided INS was developed in our recent work~\cite{Yang2017icra}.} 
(i) visual-inertial SLAM (VI-SLAM),
and (ii) visual-inertial odometry (VIO),
which are among the most popular localization technologies in part due to their ubiquitous availability and  complementary sensing modality.
To this end, we have adapted both the EKF-based VI-SLAM and MSCKF-based VIO algorithms to fuse measurements of points, lines, planes and their different combinations.
To the best of our knowledge, algorithmically, we, for the first time, introduce and evaluate the EKF-based VI-SLAM/VIO approaches with heterogeneous features (which are common in structured environments).
In particular, we have compared two different EKFs in both VI-SLAM and VIO: 
(i) the {\em ideal} EKF that uses true states as the linearization points in computing filter Jacobians and has been shown to have correct observability properties and expected to be consistent, 
thus being used as the benchmark in simulations as in the literature (e.g., \cite{Huang2008ICRA,Huang2010IJRR,Huang2014ICRA,Hesch2013TRO,Hesch2014IJRR});
and 
(ii) the {\em standard} EKF that uses current state estimates as the linearization points in computing filter Jacobians and has been shown to be overconfident (inconsistent)~\cite{Huang2014ICRA,Hesch2013TRO,Hesch2014IJRR}.
The metrics used to evaluate estimation performance are the root mean squared error (RMSE)
and the average normalized (state) estimation error squared (NEES)~\cite{Bar-Shalom2001}.
The RMSE provides a measure of accuracy, while the NEES
is a standard criterion for evaluating estimator consistency, which (implicitly) indicates the correctness of the EKF system observability.
%
%

The simulated trajectories and different geometric features are shown in the left of Figs.~\ref{fig_all_slam} and~\ref{fig_all_vio},
where we simulate a camera/IMU sensor suite is moving on the sinusoidal trajectories to collect measurements to different features.
For the results in Fig.~\ref{fig_all_slam}, we developed the EKF-based VI-SLAM algorithm, 
which simultaneously  preforms visual-inertial localization and mapping by keeping different features in the state vector.
In contrast,
for the results in Fig.~\ref{fig_all_vio}, we adapated the MSCKF-based VIO~\cite{Mourikis2007ICRA}, 
which estimates only the sensor poses while marginalizing out different (not only points) features with null space operation.
%
%
%
It is clear from these results of both VI-SLAM and VIO in Figs. \ref{fig_all_slam} and \ref{fig_all_vio} 
that the standard EKF/MSCKF performs worse than the benchmark ideal filter, 
which agrees with the literature (with point features only)~\cite{Huang2014ICRA,Hesch2013TRO,Hesch2014IJRR}.
This  again reflects the importance of system observability for consistent state estimation.

Moreover, in order to directly validate the unobservable subspace of the aided INS found in our analysis,
using the same simulation setup as above but with a single feature,
we have  constructed the observability matrix of the ideal EKF-based VI-SLAM with a single point (or line or plane) 
and numerically computed the dimension of its null space, which is shown in Fig. \ref{fig_all_rank}. 
Clearly, the dimension of the unobservable subspace for the (ideal) VI-SLAM with a single point (line or plane) is 4 (5 or 7), 
which agrees with our analysis.

\begin{figure*}
	\centering
	\includegraphics[scale = 0.38]	{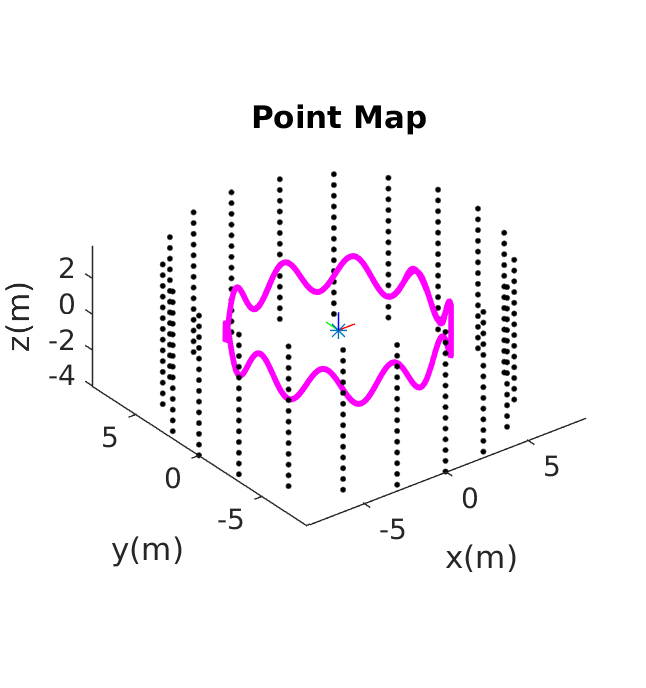}
	\includegraphics[scale = 0.35]	{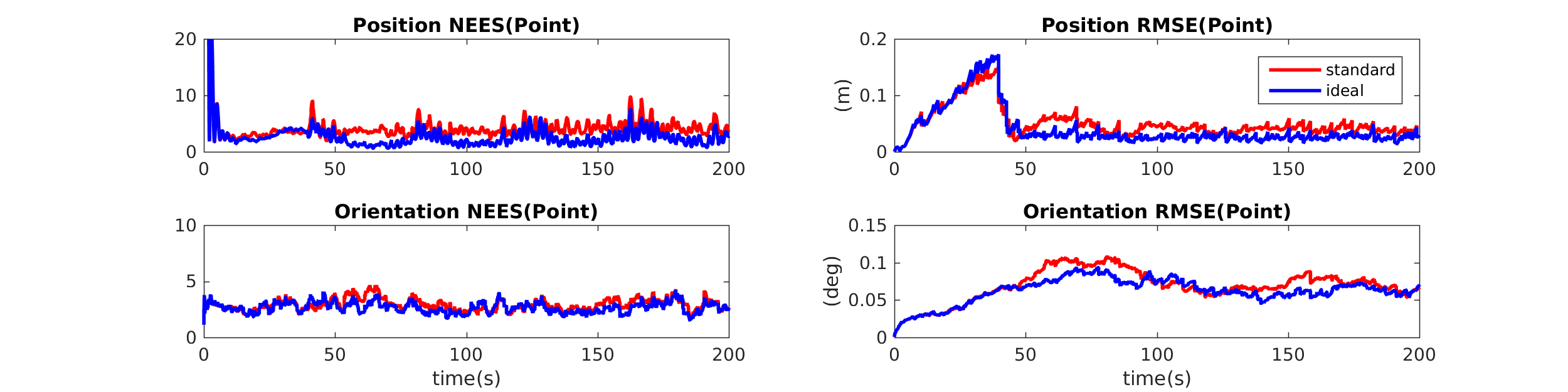}
	\includegraphics[scale = 0.38]	{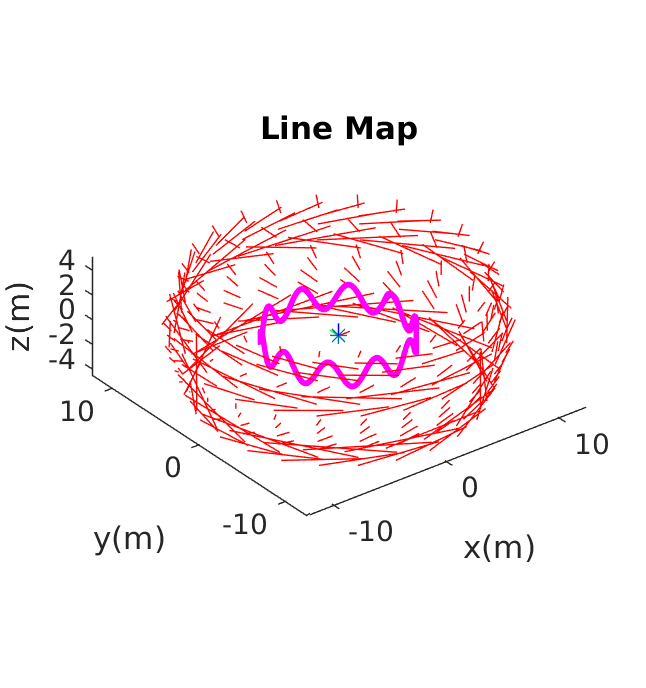}
	\includegraphics[scale = 0.35]	{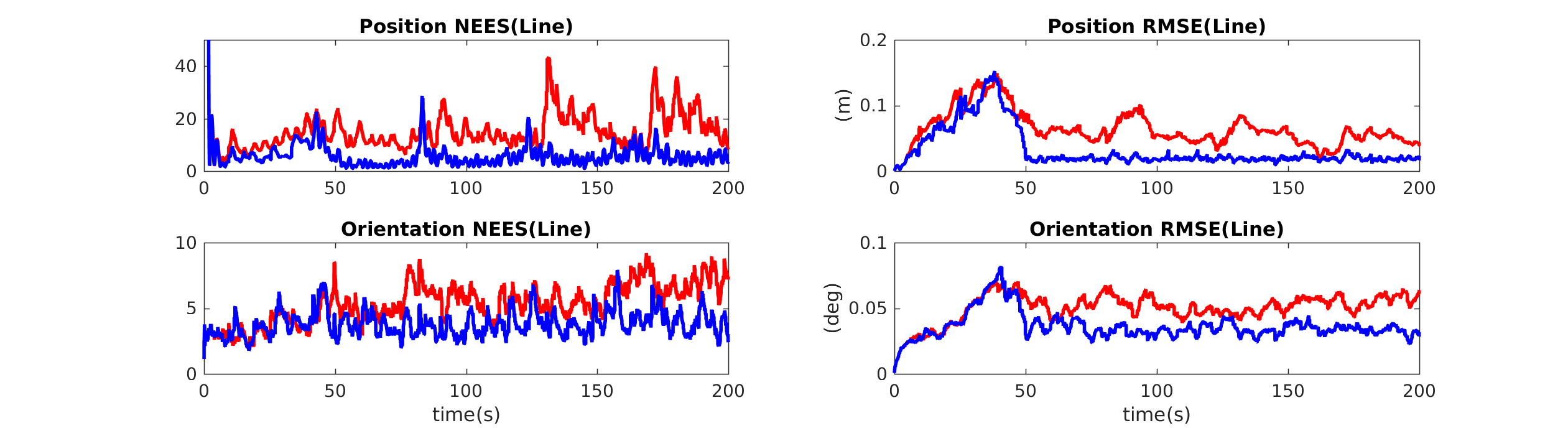}
	\includegraphics[scale = 0.38]	{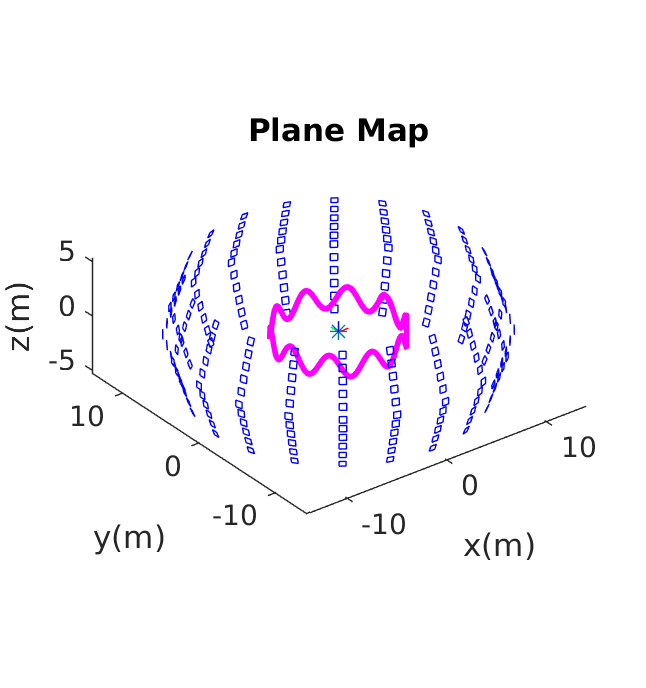}
	\includegraphics[scale = 0.35]	{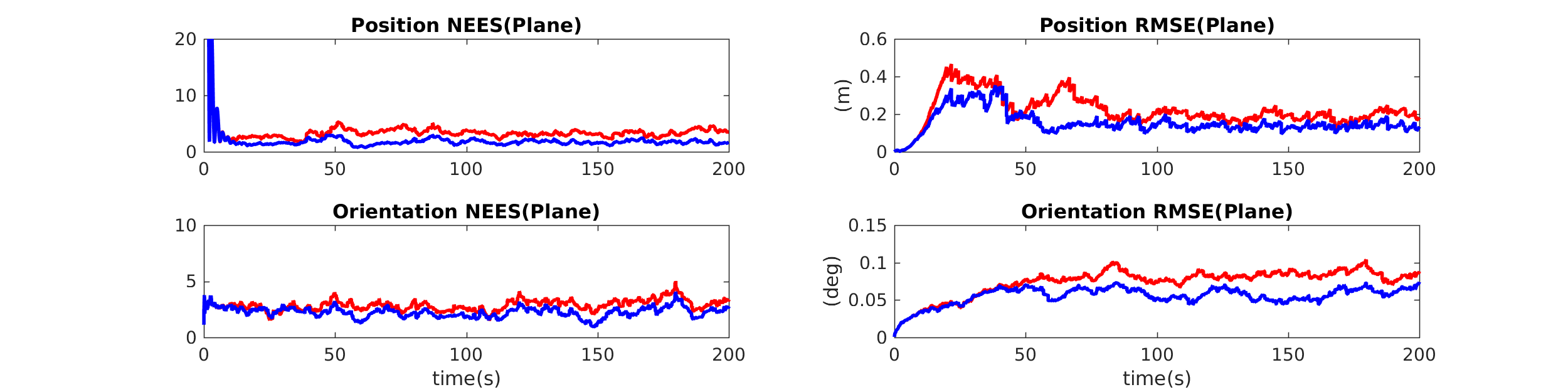}
	\includegraphics[scale = 0.38]	{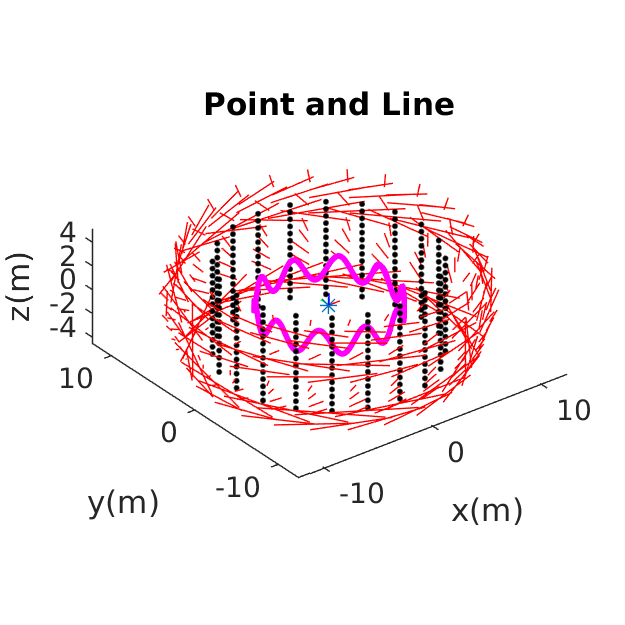}
	\includegraphics[scale = 0.35]	{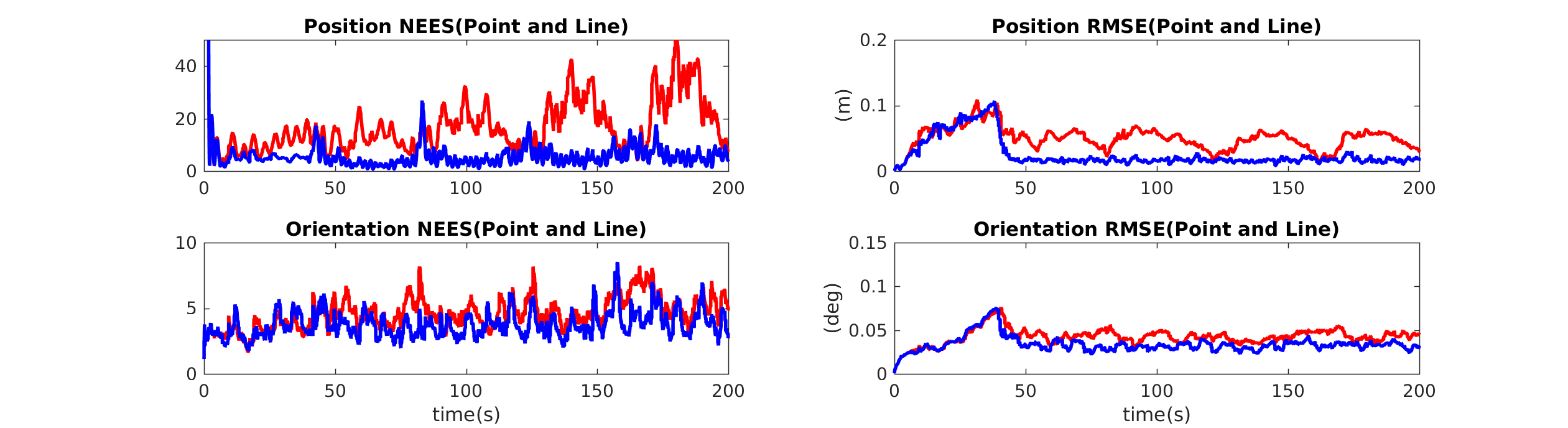}
	\includegraphics[scale = 0.38]	{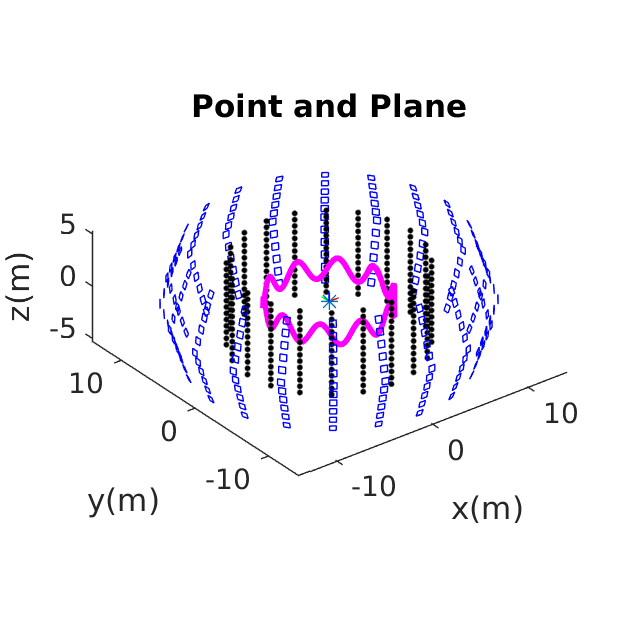}
	\includegraphics[scale = 0.35]	{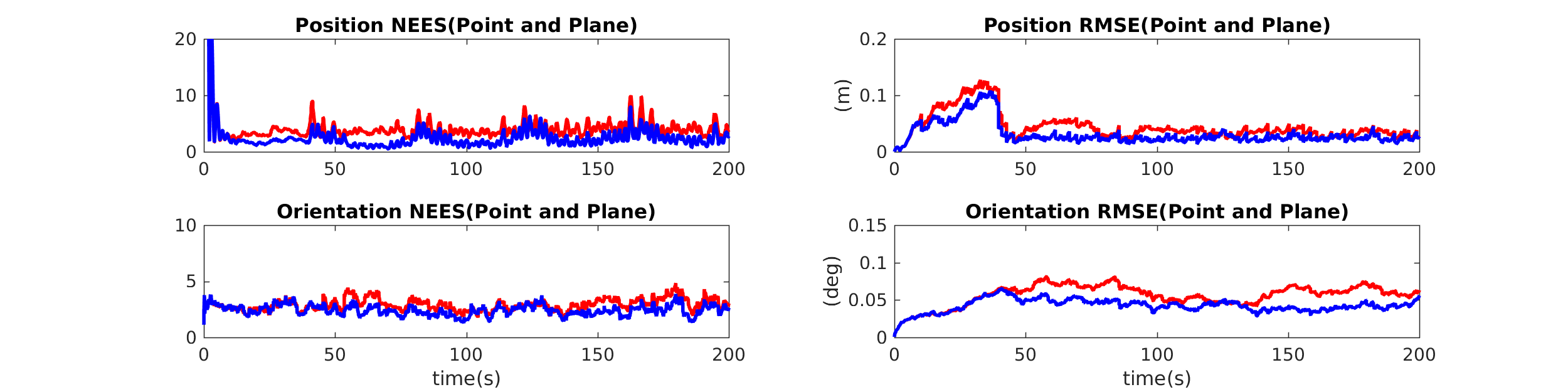}
	\includegraphics[scale = 0.38]	{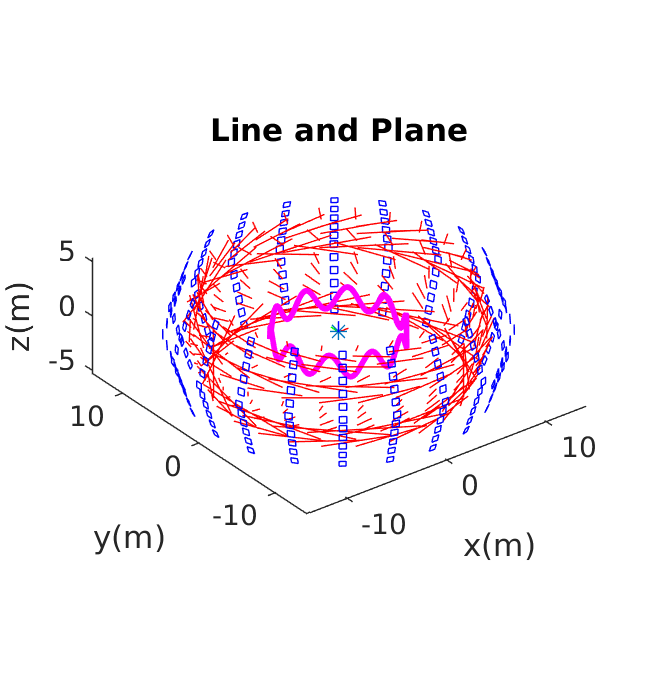}
	\includegraphics[scale = 0.35]	{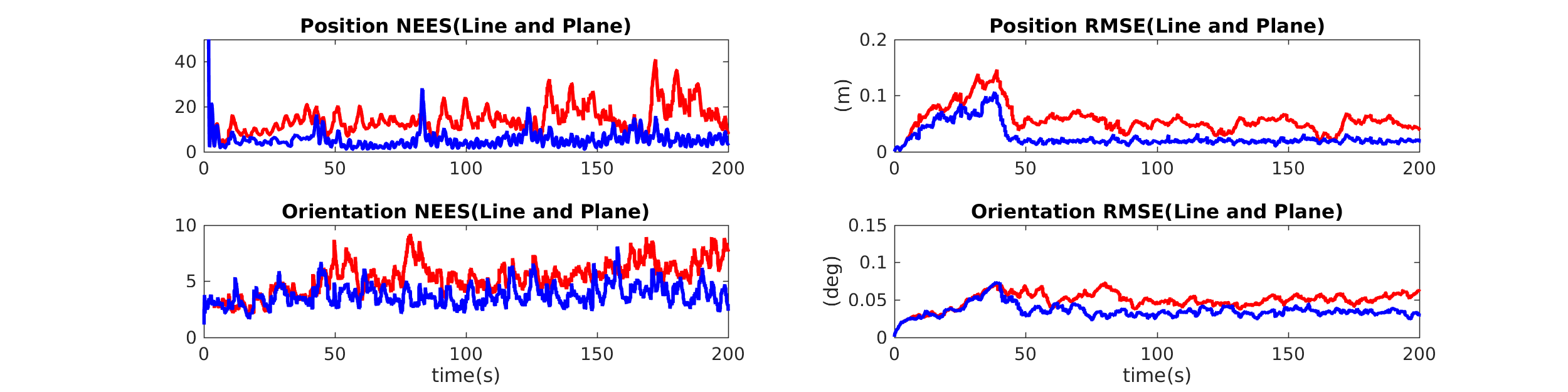}
	\includegraphics[scale = 0.38]	{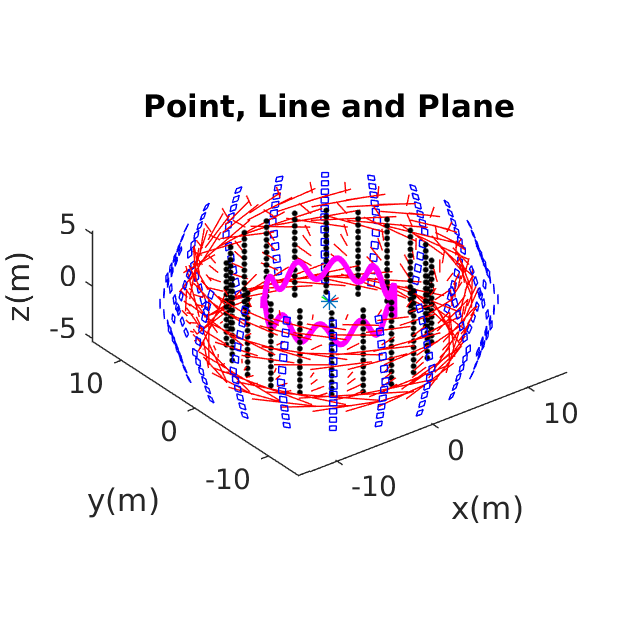}
	\includegraphics[scale = 0.35]	{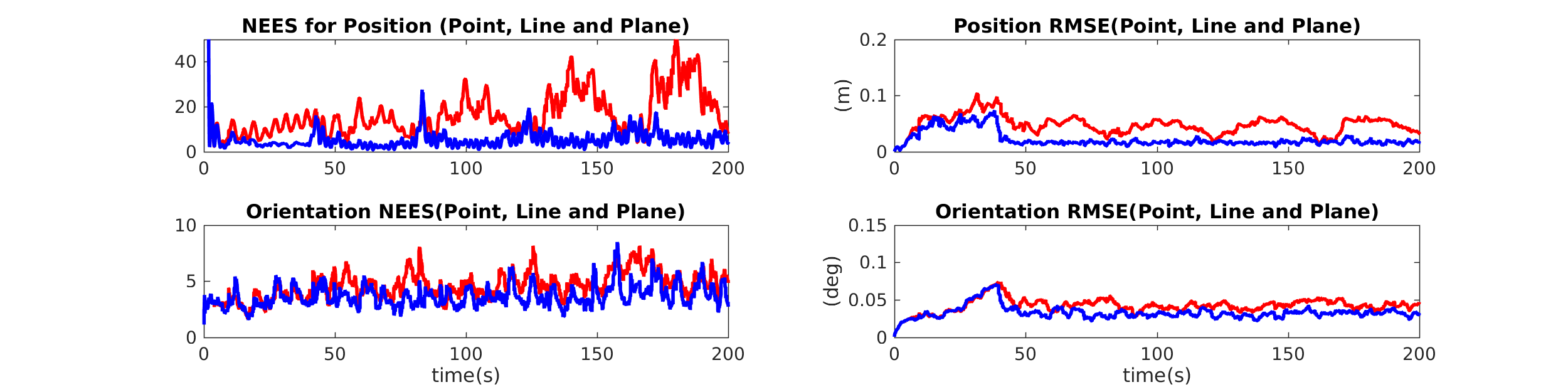}	
	\caption{Monte-Carlo results of EKF-based VI-SLAM using different geometric features.
		}
	\label{fig_all_slam}
\end{figure*}

%% file: sections/simulation2.tex
\begin{figure*}
	\centering
	\includegraphics[scale = 0.35]	{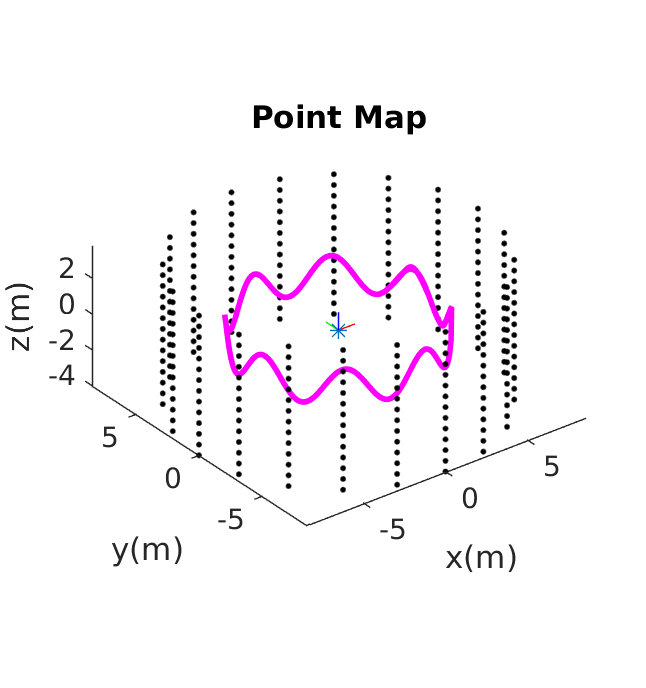}
	\includegraphics[scale = 0.35]	{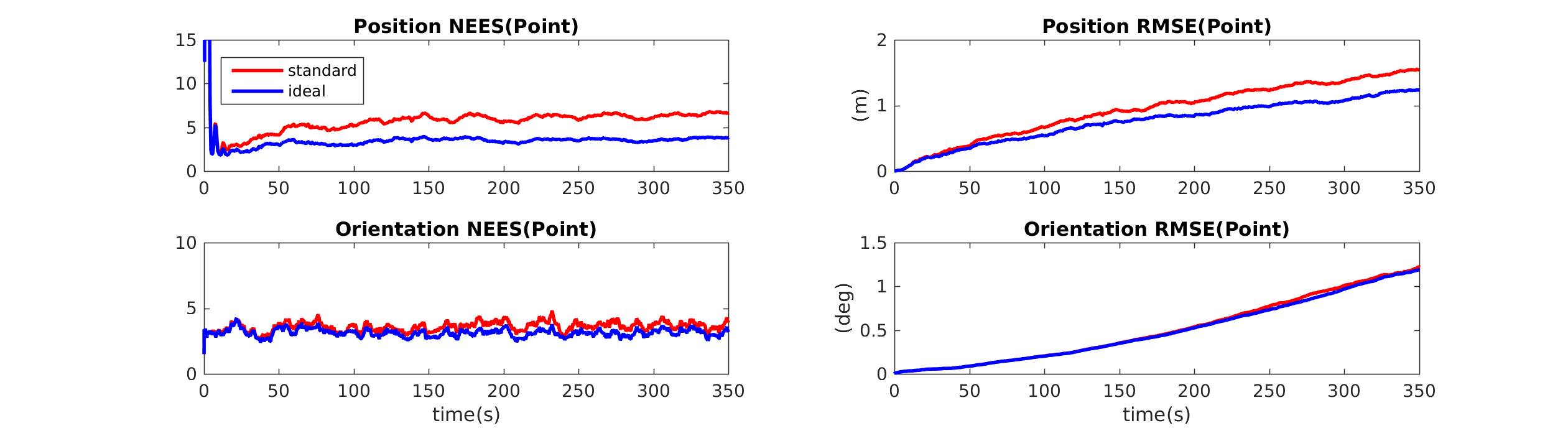}
	\includegraphics[scale = 0.35]	{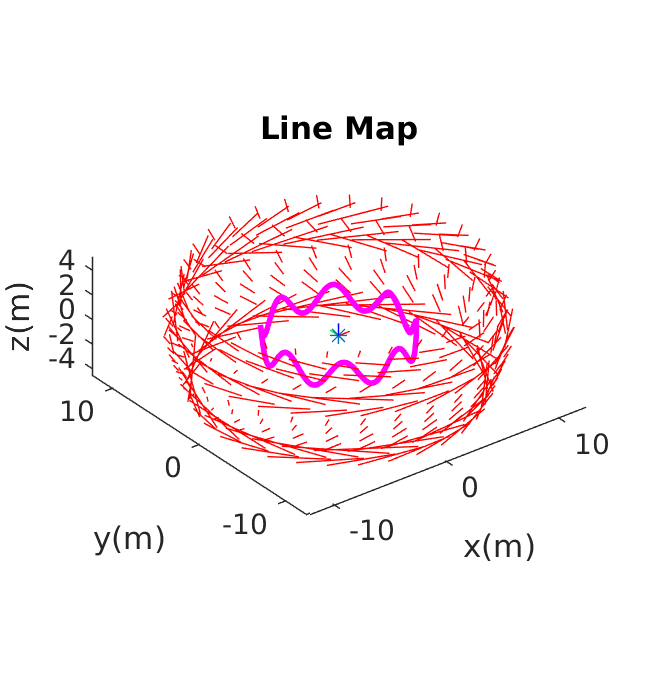}
	\includegraphics[scale = 0.35]	{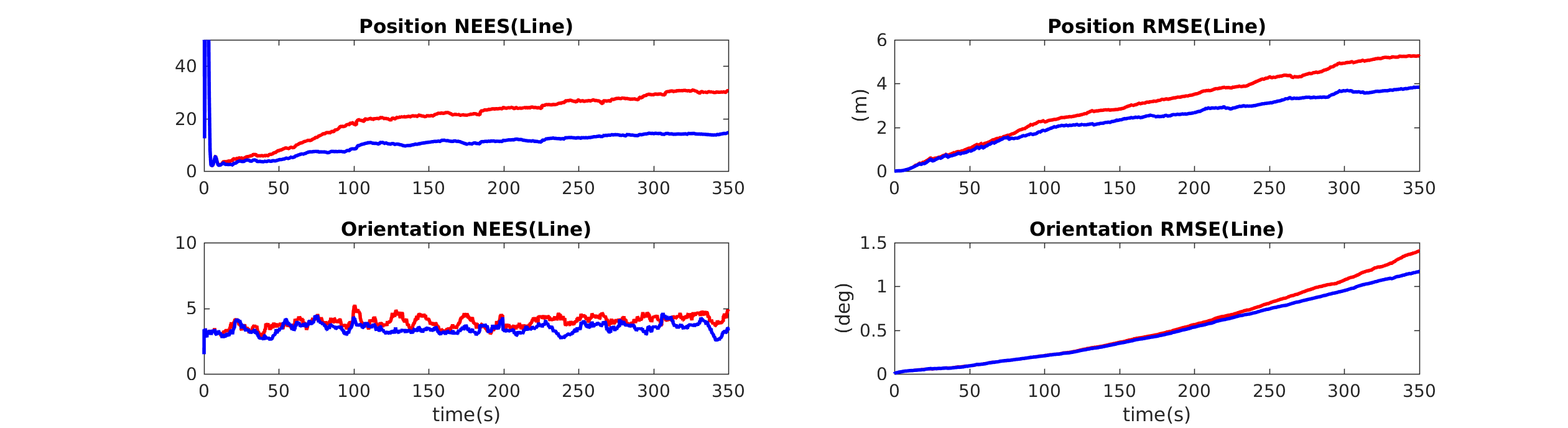}
	\includegraphics[scale = 0.35]	{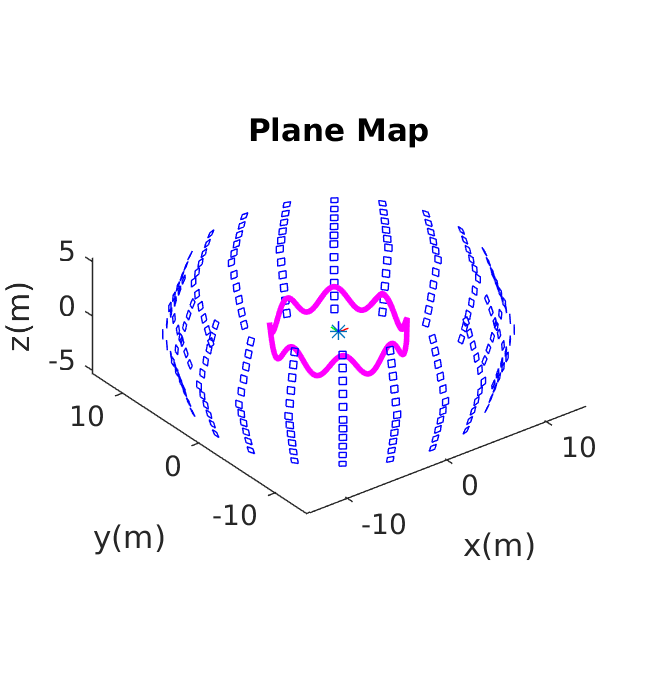}
	\includegraphics[scale = 0.35]	{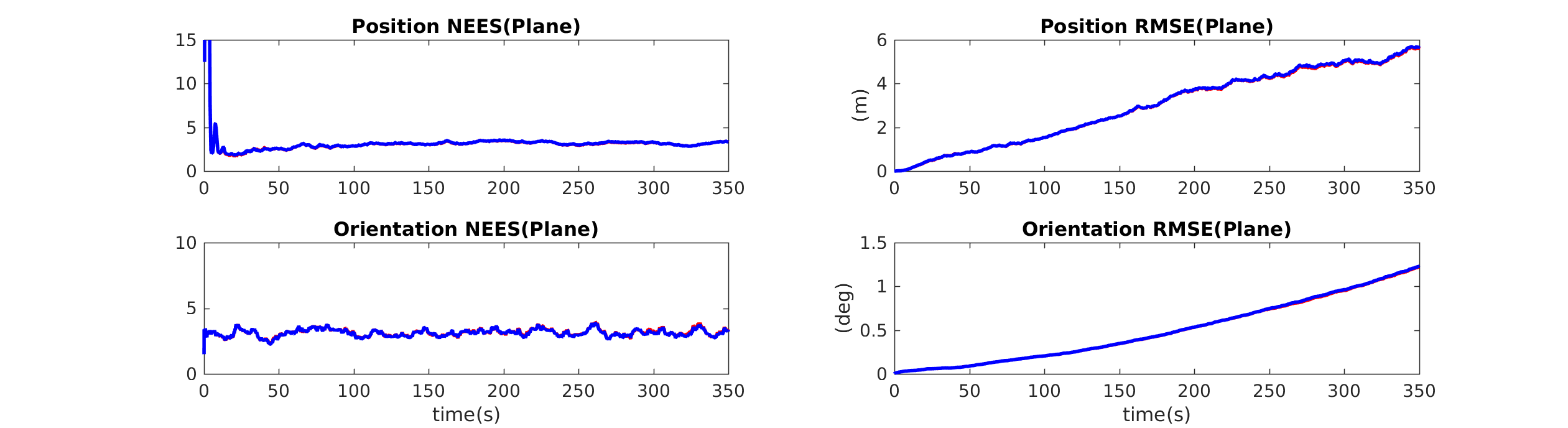}
	\includegraphics[scale = 0.35]	{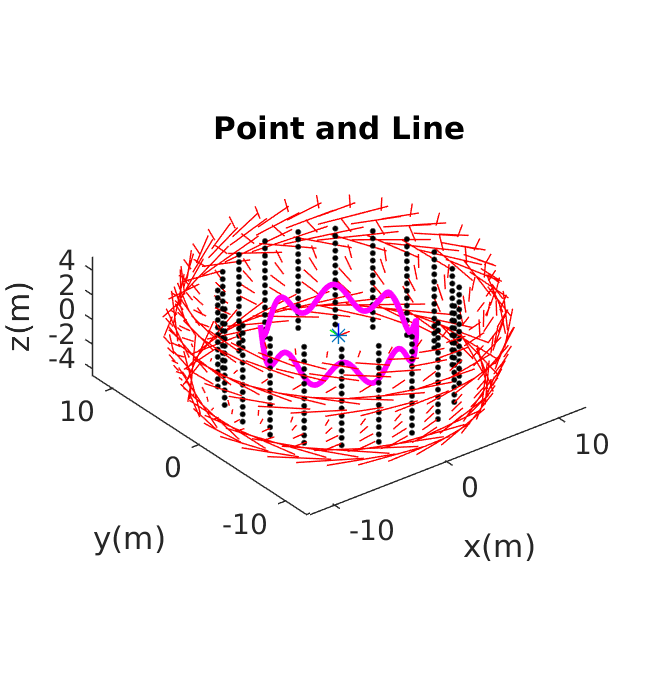}
	\includegraphics[scale = 0.35]	{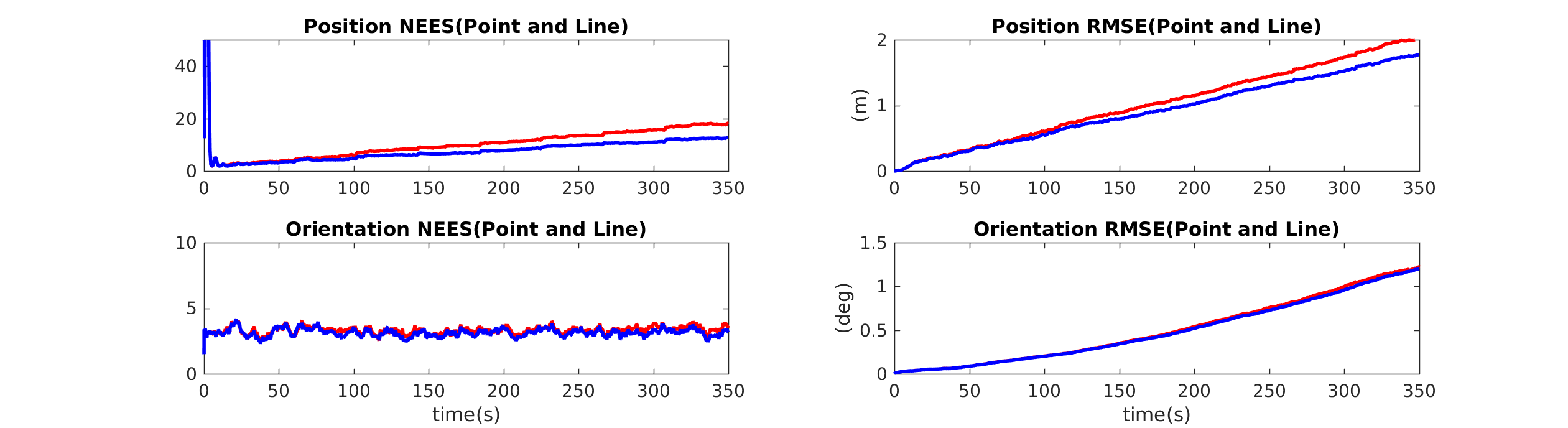}
	\includegraphics[scale = 0.35]	{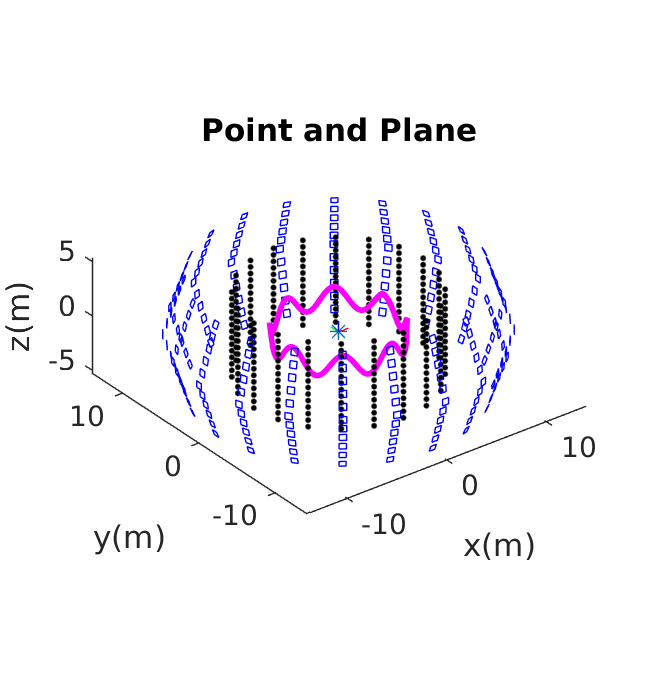}
	\includegraphics[scale = 0.35]	{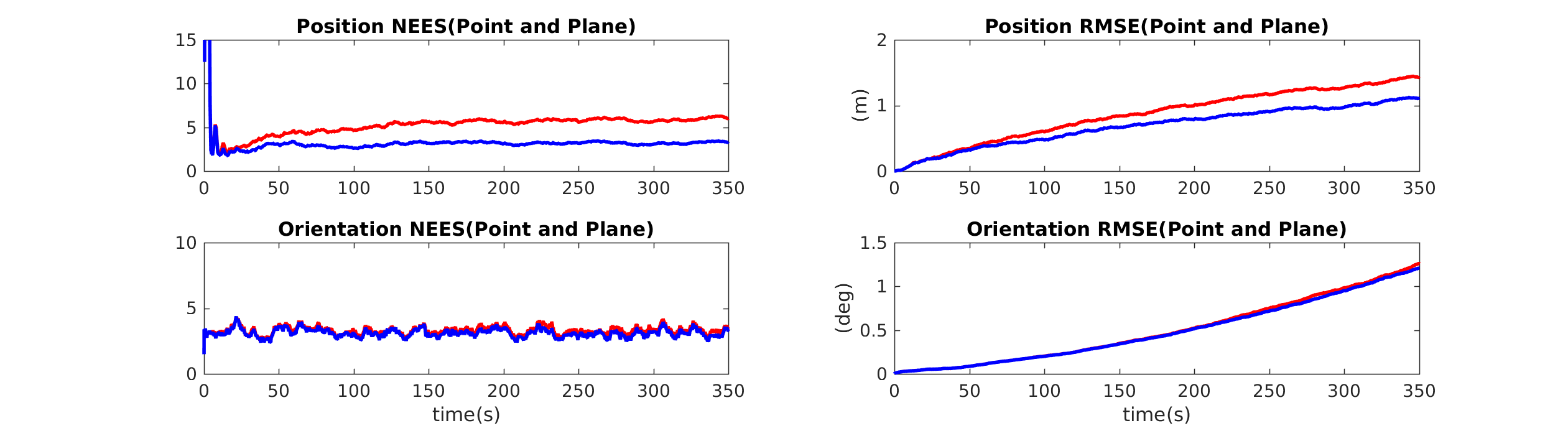}
	\includegraphics[scale = 0.35]	{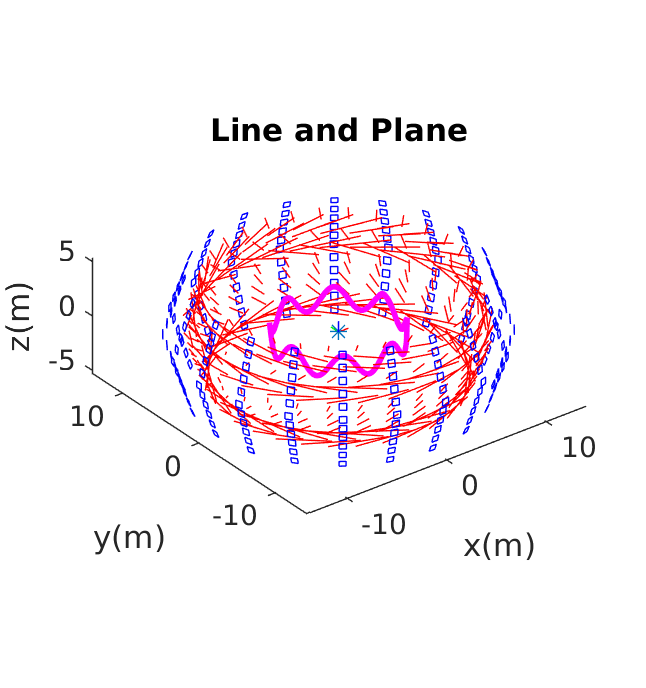}
	\includegraphics[scale = 0.35]	{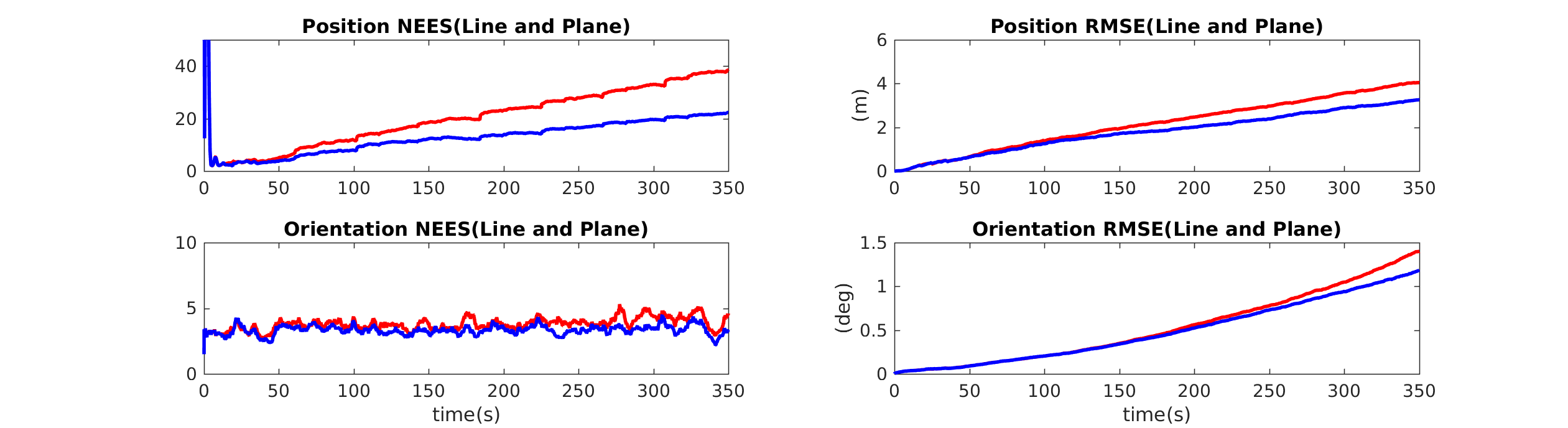}
	\includegraphics[scale = 0.35]	{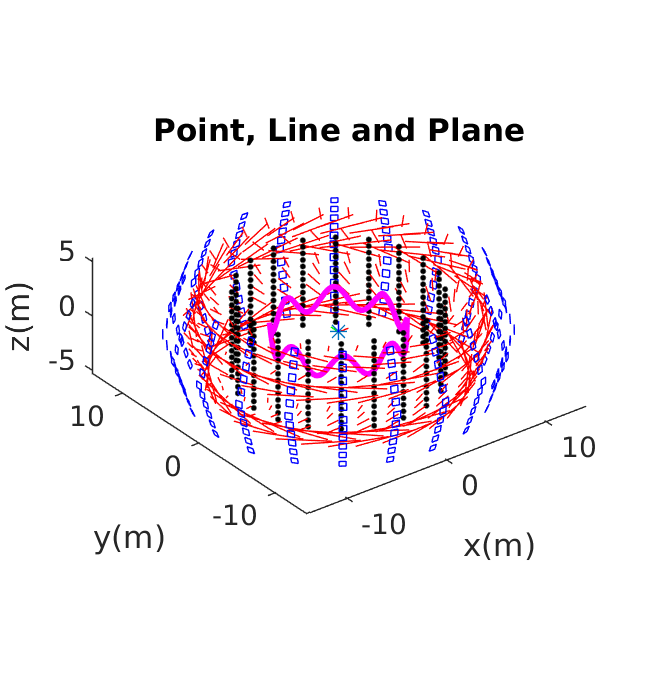}
	\includegraphics[scale = 0.35]	{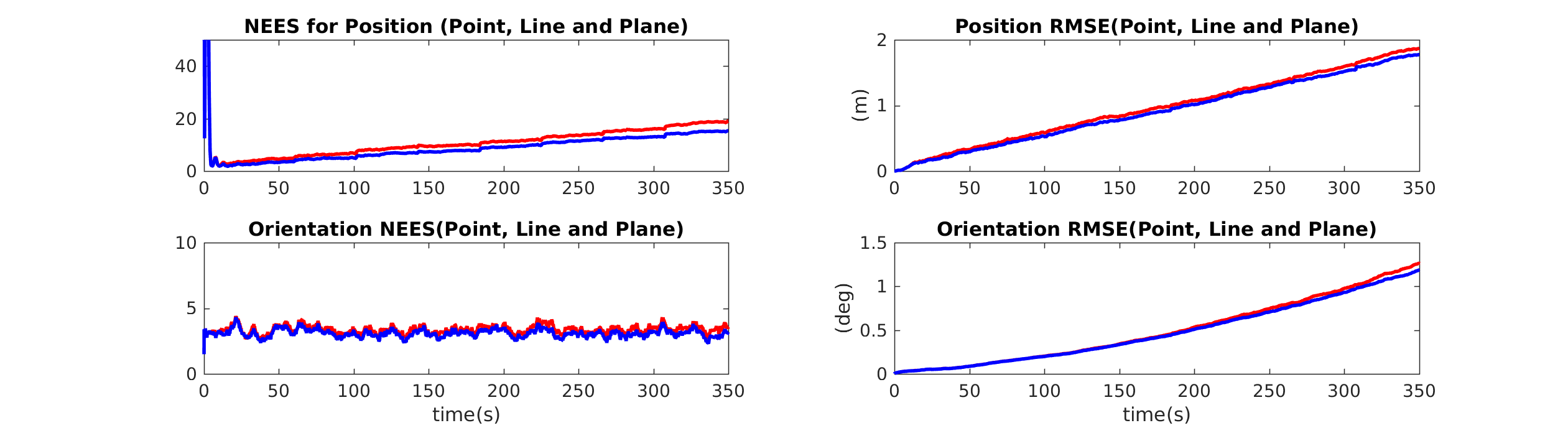}	
	\caption{Monte-Carlo results of MSCKF-based VIO using different geometric features. 
		}
	\label{fig_all_vio}
\end{figure*}

\begin{figure}
	\centering
	\includegraphics[scale = 0.5]{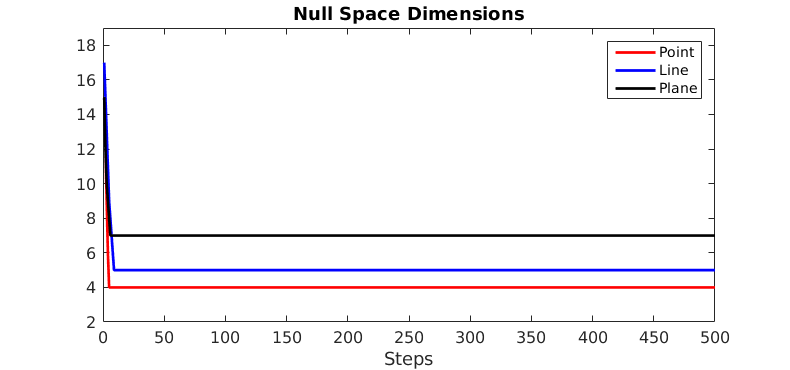}
	\caption{The dimension of the null space of the observability matrix (i.e., unobservable subspace) that is numerically computed during the simulation of the ideal-EKF based VI-SLAM with a single point, line or plane.}
	\label{fig_all_rank}
\end{figure}

%% file: sections/conclusion.tex
\section{Conclusions and Future Work} \label{sec:concl}

In this paper, we have performed observability analysis for aided INS with different geometric features including points, lines and planes, which are detected from generic range and/or bearing measurements. 
encompassing vision-aided INS (VINS) as a special case. 
As in practice, most aided-INS estimators are built based on the linearized systems, whose observability properties directly impact the estimation performance, 
this work has primarily focused on observability analysis of the linearized aided INS with points, lines, planes and their combinations.
In particular, 
in the case of point features, we analytically show that the aided INS (both linearized and nonlinear)  using generic range and/or bearing measurements has 4 unobservable directions.
In the case of lines (planes), we prove that there exist at least 5 (7) unobservable directions with a single line (plane) feature, and for the first time, analytically derived the unobservable directions for multiple lines and planes. 
We have generalized this observability analysis for linearized aided INS with different combinations of point, line and plane features and summarized important results in Table~\ref{tab_summary}. 
Based on this analysis, we have also systematically investigated the effects of global measurements on the observability of aided INS, 
and found, as expected, that global information improves the system observability.  
Moreover, we have identified comprehensively 5 types of degenerate motions that negatively impact the system observability and should be avoided if possible (otherwise, extra sensors may be needed).
Additionally, it should be noted that, during our analysis,
we have employed 2 different point representation (Euclidean and spherical coordinates), 2 different plane representation (CP and Hesse form) and 2 different line measurement models (projective  and direct), to  analytically show that the systems have the same observability properties regardless these different representations.
To validate our analysis, 
we have developed EKF-based VI-SLAM and MSCKF-based VIO using heterogeneous geometric features of points, lines and planes,
and evaluated their performance extensively in Monte-Carlo simulations.

In the future, we will leverage the insights gained from this observability analysis to design consistent estimators for aided INS with different geometric features by enforcing proper observability constraints as in our prior work~\cite{Huang2012thesis}. 
We will also investigate the (stochastic) observability of aided INS under adversarial attacks \cite{Yang2017ISRR} or unknown inputs \cite{Martinelli2018TAC} 
in order to design secure estimators for robot navigation.

\begin{table} 
	\centering
	\caption{Summary of Observability Analysis of Aided INS}
	\label{tab_summary}
	\renewcommand{\arraystretch} {1.25}
	\begin{tabular}{|l|l|}
		\hline
		Features                        & Unobservable Directions \\ \hline
		Single or multiple points                 & 4 $(\mathbf{N})$                                    \\ 
		Unparallel lines                         & 4 $(\mathbf{N}_{L})$                                  \\ 
		Planes with unparallel intersections     & 4 $(\mathbf{N}_{\Pi 1:4})$                           \\ 
		Point and line                            & 4 $(\mathbf{N}_{PL})$                                \\ 
		Point and plane                           & 4 $(\mathbf{N}_{P\Pi})$                               \\ 
		Single line unparallel to planes 		& 4 $(\mathbf{N}_{l\pi 1:4})$                            \\ 
		Plane Intersections unparallel to lines  & 4 $(\mathbf{N}_{L\Pi})$                                \\ 
		Point, line and plane                     & 4 $(\mathbf{N}_{PL\Pi})$                                \\ \hline
		Single line                               & 5 $(\mathbf{N}_{l})$                                 \\ 
		Single line parallel to single plane   & 5 $(\mathbf{N}_{l\pi 1:5})$                            \\ 
		Two un-parallel planes                    & 5 $(\mathbf{N}_{\Pi})$                               \\ \hline
		Single plane                              & 7 $(\mathbf{N}_{\pi})$                              \\ \hline
	\end{tabular}
\end{table}

%% file: appendix/state_transition.tex
\section{State Transition Matrix} \label{apd_state_transition}
Follow the reference \cite{Hesch2013TRO}, the state transition matrix $\boldsymbol{\Phi}_{(k,1)}$ is defined as:
\begin{equation}
	\boldsymbol{\Phi}_{(k,1)}
	=
	\begin{bmatrix}
	\boldsymbol{\Phi}_{11} & 
	\boldsymbol{\Phi}_{12} & 
	\mathbf{0}_3 & 
	\mathbf{0}_3 & 
	\mathbf{0}_3 & 
	\mathbf{0}_3
	\\
	\mathbf{0}_3 & 
	\mathbf{I}_3 &  
	\mathbf{0}_3 & 
	\mathbf{0}_3 & 
	\mathbf{0}_3 & 
	\mathbf{0}_3
	\\
	\boldsymbol{\Phi}_{31} & 
	\boldsymbol{\Phi}_{32} & 
	\mathbf{I}_3 & 
	\boldsymbol{\Phi}_{34} & 
	\mathbf{0}_3 & 
	\mathbf{0}_3
	\\
	\mathbf{0}_3 & 
	\mathbf{0}_3 &  
	\mathbf{0}_3 & 
	\mathbf{I}_3 & 
	\mathbf{0}_3 & 
	\mathbf{0}_3
	\\
	\boldsymbol{\Phi}_{51} & 
	\boldsymbol{\Phi}_{52} & 
	\boldsymbol{\Phi}_{53} &
	\boldsymbol{\Phi}_{54} & 
	\mathbf{I}_3 & 
	\mathbf{0}_3
	\\
	\mathbf{0}_3 & 
	\mathbf{0}_3 &  
	\mathbf{0}_3 & 
	\mathbf{0}_3 & 
	\mathbf{0}_3 & 
	\mathbf{I}_3					
	\end{bmatrix}	
\end{equation}
where we have:
\begin{eqnarray}
	\boldsymbol{\Phi}_{11}
	&=&
	^{I_k}_{I_1}\mathbf{R}\\
	\boldsymbol{\Phi}_{31}
	&=&
	-\lfloor\left({^G\mathbf{V}_{I_k}} - {^G\mathbf{V}_{I_1}}\right)+ {^G\mathbf{g}\delta t_k}\times\rfloor
	{^G_{I_1}}\mathbf{R}\\
	\boldsymbol{\Phi}_{51}
	&=&
	\lfloor {^G\mathbf{P}_{I_1}} + {^G\mathbf{V}_{I_1}}\delta t_k -\frac{1}{2}{^G\mathbf{g}}\delta t^2_k
	-{^G\mathbf{P}_{I_k}}\times\rfloor {^G_{I_1}}\mathbf{R}
	\\
	\boldsymbol{\Phi}_{12}
	&=&
	-
	\int_{t_1}^{t_k} {^{I_{\tau}}_{I_k}}\mathbf{R}^{\top}\mathbf{d}\tau
	\\
	\boldsymbol{\Phi}_{32}
	&=&
	\int_{t_1}^{t_k}
	{^{I_s}_G}\mathbf{R}^{\top}
	\lfloor {^{I_s}\mathbf{a}}\times\rfloor
	\int_{t_1}^{s}
	{^{I_{\tau}}_{I_s}}\mathbf{R}^{\top}
	\mathbf{d}\tau
	\mathbf{d}s
	\\
	\boldsymbol{\Phi}_{52}
	&=&
	\int_{t_1}^{t_k}
	\int_{t_1}^{\theta}
	{^{I_s}_G}\mathbf{R}^{\top}
	\lfloor {^{I_s}\mathbf{a}}\times\rfloor
	\int_{t_1}^{t_s}
	{^{I_{\tau}}_{I_s}}\mathbf{R}^{\top}
	\mathbf{d}\tau
	\mathbf{d}s	
	\mathbf{d}\theta
	\\
	\boldsymbol{\Phi}_{53}
	&=&
	\mathbf{I}_{3}\delta t_k
	\\
	\boldsymbol{\Phi}_{34}
	&=&
	-
	\int_{t_1}^{t_k}
	{^{I_{\tau}}_G\mathbf{R}^{\top}}
	\mathbf{d}\tau
	\\
	\boldsymbol{\Phi}_{54}
	&=&
	-
	\int_{t_1}^{t_k}
	\int_{t_1}^{t_s}
	{^{I_{\tau}}_G\mathbf{R}^{\top}}
	\mathbf{d}\tau	
	\mathbf{d}s
\end{eqnarray}

%% file: appendix/measurements_point.tex
\section{Sensor Measurements for Point Features} \label{apd_meas}
In this section, we will analyze the measurement model for lase sensors, camera sensors and 2D imaging sonars. In this section, we refer to frame $\{X\}$ as the sensor frame. 

\subsection{1D Range Finder}
1D range Finder can only get the range measurement of the point feature, and the measurement model can be described as:
\begin{equation}\label{eq_1D_range}
z^{(r)}   
= 
\sqrt{^x\mathbf{P}_{\mathbf{f}}^{\top}{^x\mathbf{P}_{\mathbf{f}}}} + \mathbf{n}^{(r)} 
= 
\sqrt{(^xx_{\mathbf{f}})^2+(^xy_{\mathbf{f}})^2+(^xz_{\mathbf{f}})^2} + \mathbf{n}^{(r)} 
\end{equation}
where $^xr_{\mathbf{f}}=\sqrt{^x\mathbf{p}_f^{\top}{^x\mathbf{p}_f}}$ represents the range for the point feature in the frame $\{X\}$. And we can linearizing the measurement model at $^x\hat{\mathbf{P}}_{\mathbf{f}}$ as:
\begin{equation}\label{eq_1D_range_linear}
	\tilde{z}^{(r)}
	\simeq
	\mathbf{H}_{r}{^x\tilde{\mathbf{P}}_{\mathbf{f}}}+\mathbf{n}^{(r)}
	=
	\frac{^x\hat{\mathbf{P}}^{\top}_{\mathbf{f}}}{^x\hat{r}_{\mathbf{f}}}{^x\tilde{\mathbf{P}}_{\mathbf{f}}}+\mathbf{n}^{(r)}
\end{equation}

\subsection{Mono-camera}
Mono-camera can only get the bearing measurements of the point feature, and the measurement model can be represented as:
\begin{equation}\label{eq_mono_cam}
\mathbf{z}^{(b)}  
= 
\left[
\begin{array}{c}
\frac{\mathbf{e}_1^{\top}{^x\mathbf{P}_{\mathbf{f}}}}{\mathbf{e}_3^{\top}{^x\mathbf{P}_{\mathbf{f}}}} \\
\frac{\mathbf{e}_2^{\top}{^x\mathbf{P}_{\mathbf{f}}}}{\mathbf{e}_3^{\top}{^x\mathbf{P}_{\mathbf{f}}}} 
\end{array}
\right] + \mathbf{n}^{(b)}
= 
\left[
\begin{array}{c}
\frac{^xx_{\mathbf{f}}}{^xz_{\mathbf{f}}}  \\
\frac{^xy_{\mathbf{f}}}{^xz_{\mathbf{f}}}
\end{array}
\right]  + \mathbf{n}^{(b)}
\end{equation}
where $\mathbf{e}_i\in \mathbb{R}^{3\times1}$ for $i=1,2,3$ and $\mathbf{e}_1=[1\quad0\quad0]^{\top}$, $\mathbf{e}_2=[0\quad1\quad0]^{\top}$ and $\mathbf{e}_3=[0\quad0\quad1]^{\top}$. Inspired by \cite{Yang2017SSRR}, we use a more universal measurement model for point feature with Mono-camera as:
\begin{equation}\label{eq_mono_meas}
	\mathbf{z}^{(b)} =
	\mathbf{h}_{b}\left({^x\mathbf{P}_{\mathbf{f}}}, {\mathbf{n}^{(b)}}\right)
	=
	\begin{bmatrix}
	^x\mathbf{b}^{\top}_{\perp1} \\
	^x\mathbf{b}^{\top}_{\perp2} 
	\end{bmatrix}
	{^x\mathbf{P}_{\mathbf{f}}}
	+
	{^xz_{\mathbf{f}}}
	\begin{bmatrix}
	^x\mathbf{b}^{\top}_{\perp1} \\
	^x\mathbf{b}^{\top}_{\perp2} 
	\end{bmatrix}
	\begin{bmatrix}
	\mathbf{I}_2 \\
	\mathbf{0}_{1\times 2}
	\end{bmatrix}
	\mathbf{n}^{(b)}
\end{equation}
where $\mathbf{b}_{\perp i},i\in\{1,2\}$ are two perpendicular vectors to the bearing $^x\mathbf{b}_{\mathbf
	f}$, and they can be constructed from \cite{Yang2017SSRR}. The advantage of this model is that it is suitable for both fish eye and normal projective camera model. And the linearized model with $^x\hat{\mathbf{P}}_{\mathbf{f}}$ is:
\begin{equation}\label{eq_mono_meas_linear}
	\tilde{\mathbf{z}}^{(b)}
	\simeq
	\mathbf{H}_{b}{^x\tilde{\mathbf{P}}_{\mathbf{f}}} +
	\mathbf{H}_{n}\mathbf{n}^{(b)}
	=
	\begin{bmatrix}
	^x\mathbf{b}^{\top}_{\perp1} \\
	^x\mathbf{b}^{\top}_{\perp2} 
	\end{bmatrix}
	{^x\tilde{\mathbf{P}}_{\mathbf{f}}}
	+
	{^x\hat{z}_{\mathbf{f}}}
	\begin{bmatrix}
	^x\hat{\mathbf{b}}^{\top}_{\perp1} \\
	^x\hat{\mathbf{b}}^{\top}_{\perp2} 
	\end{bmatrix}
	\begin{bmatrix}
	\mathbf{I}_2 \\
	\mathbf{0}_{1\times 2}
	\end{bmatrix}
	\mathbf{n}^{(b)}
\end{equation}

\subsection{2D Imaging Sonar}
2D imaging sonar's measurement contains the range and horizontal bearing measurement of a point \cite{Yang2017icra}, and the  model can be represented as:
\begin{equation}\label{eq_son}
\mathbf{z} = 
\begin{bmatrix}
z^{(r)}  \\
z^{(b)}
\end{bmatrix}
+
\begin{bmatrix}
\mathbf{n}^{(r)} \\
\mathbf{n}^{(b)}
\end{bmatrix}
=
\left[
\begin{array}{c}
\sqrt{{^x\mathbf{P}_{\mathbf{f}}^{\top}}{^x\mathbf{P}_{\mathbf{f}}}}  \\
^x\theta _{\mathbf{f}}
\end{array}
\right] + 
\begin{bmatrix}
\mathbf{n}^{(r)} \\
\mathbf{n}^{(b)}
\end{bmatrix}
=
\left[
\begin{array}{c}
\sqrt{(^xx_{\mathbf{f}})^2+(^xy_{\mathbf{f}})^2+(^xz_{\mathbf{f}})^2}       \\
^x\theta _{\mathbf{f}}
\end{array}
\right] + 
\begin{bmatrix}
\mathbf{n}^{(r)} \\
\mathbf{n}^{(b)}
\end{bmatrix}
\end{equation}
and similar to the case of the Mono-camera, we rewrite the bearing measurement as:
\begin{equation}
	z^{(b)}
	=
	\mathbf{h}_{b}\left({^x\mathbf{P}_{\mathbf{f}}}, \mathbf{n}^{(b)}\right)
	=
	{^xr_{\mathbf{f}}}
	\begin{bmatrix}
	\cos \left({^x\theta_{\mathbf{f}}+ n^{(b)}}\right) \cos \phi  &
	\sin \left({^x\theta_{\mathbf{f}}+ n^{(b)}}\right) \cos \phi  & 
	\sin \phi
	\end{bmatrix}
	{^x\mathbf{b}_{\perp}}
\end{equation}
where $^x\mathbf{b}_{\perp} = \begin{bmatrix}-\sin \theta & \cos \theta &  0\end{bmatrix}^{\top}$. Therefore, the linearized sonar measurement model with $^x\hat{\mathbf{P}}_{\mathbf{f}}$ as:
\begin{equation}
	\tilde{\mathbf{z}}
	=
	\begin{bmatrix}
	\tilde{z}^{(r)} \\
	\tilde{z}^{(b)}
	\end{bmatrix}
	\begin{bmatrix}
	\mathbf{H}_{r}{^x\tilde{\mathbf{P}}_{\mathbf{f}}}+                  \mathbf{n}^{(r)}  \\
	\mathbf{H}_{b}{^x\tilde{\mathbf{P}}_{\mathbf{f}}} +
	\mathbf{H}_{n}\mathbf{n}^{(b)}
	\end{bmatrix}
\end{equation}
\begin{eqnarray}
	\mathbf{H}_{b}  &=& {^x\mathbf{b}^{\top}_{\perp}}  \\
	\mathbf{H}_{n}  &=& {^x\mathbf{b}^{\top}_{\perp}}{^x\hat{\mathbf{P}}_{\mathbf{f}}}
	\left(-sin \left(^x\hat{\theta}_{\mathbf{f}}\right)\cos \left(^x\hat{\phi}_{\mathbf{f}}\right) + \cos \left(^x\hat{\theta}_{\mathbf{f}}\right)\cos \left(^x\hat{\phi}_{\mathbf{f}}\right)\right)	
\end{eqnarray}
where $^x\hat{\theta}_{\mathbf{f}}=\arctan\frac{^x\hat{y}_{\mathbf{f}}}{^x\hat{x}_{\mathbf{f}}}$, $^x\hat{\phi}_{\mathbf{f}}=\arctan\frac{^x\hat{z}_{\mathbf{f}}}{\sqrt{^x\hat{x}^2_{\mathbf{f}}+^x\hat{y}^2_{\mathbf{f}}}}$.

\subsection{2D LiDAR}
2D lidar measurement is quite similar to sonar measurement, except for an extra constraint that $^xz_f=\mathbf{e}_3^{\top}{^x\mathbf{p}_f}=0$ (or we can see as $\phi=0$). In order to distinguish with Eq. \eqref{eq_son}, we add this constrain to the model, and hence:
\begin{equation}
\mathbf{z} =
\begin{bmatrix}
z^{(r)}  \\
\mathbf{z}^{(b)}
\end{bmatrix}
=
\left[
\begin{array}{c}
\sqrt{{^x\mathbf{P}_{\mathbf{f}}^{\top}}{^x\mathbf{P}_{\mathbf{f}}}}  \\
^x\theta_{\mathbf{f}} \\
\mathbf{e}_3^{\top}{^x\mathbf{P}_{\mathbf{f}}}		
\end{array}
\right] + 
\begin{bmatrix}
n^{(r)}  \\
\mathbf{n}^{(b)}_{2\times 1}
\end{bmatrix}
=
\left[
\begin{array}{c}
\sqrt{(^xx_{\mathbf{f}})^2+(^xy_{\mathbf{f}})^2+(^xz_{\mathbf{f}})^2}       	\\
^x\theta_{\mathbf{f}}						\\
^xz_{\mathbf{f}}
\end{array}
\right] +
\begin{bmatrix}
n^{(r)}  \\
\mathbf{n}^{(b)}_{2\times 1}
\end{bmatrix}
\end{equation}
where $\mathbf{n}^{(b)}=\begin{bmatrix} n^{(b)}_1 & n^{(b)}_2\end{bmatrix}^{\top}$. 
Similarly, we can rewrite the bearing measurement as:
\begin{equation}
	\mathbf{z}^{(b)}
	=
	\mathbf{h}_b\left(^x\mathbf{P}_{\mathbf{f}},\mathbf{n}^{(b)}\right)
	=
	\begin{bmatrix}
	{^x\mathbf{b}^{\top}_{\perp1}}
	{^xr_{\mathbf{f}}}
	\begin{bmatrix}
	\cos \left({^x\theta_{\mathbf{f}}+ n^{(b)}_1}\right) \cos \phi  &
	\sin \left({^x\theta_{\mathbf{f}}+ n^{(b)}_1}\right) \cos \phi  & 
	\sin \phi
	\end{bmatrix}^{\top} \\
	^x\mathbf{b}_{\perp2}^{\top}{^x\mathbf{P}_{\mathbf{f}}} + n^{(b)}_2
	\end{bmatrix}
\end{equation}
where $\mathbf{b}_{\perp1}=\begin{bmatrix}-\sin \left(^x\theta_{\mathbf{f}}\right) & \cos \left(^x\theta_{\mathbf{f}}\right) &  0\end{bmatrix}^{\top}$, $\mathbf{b}_{\perp2}=\begin{bmatrix}
0 & 0 & 1\end{bmatrix}^{\top}$. Therefore, the linearized system with $^x\hat{\mathbf{P}}_{\mathbf{f}}$ can be described as:
\begin{equation}
	\tilde{\mathbf{z}}
	=
	\begin{bmatrix}
	\tilde{z}^{(r)} \\
	\tilde{\mathbf{z}}^{(b)}
	\end{bmatrix}
	=
	\begin{bmatrix}
	\mathbf{H}_{r}{^x\tilde{\mathbf{P}}_{\mathbf{f}}} + n^{(r)}  \\
	\mathbf{H}_{b}{^x\tilde{\mathbf{P}}_{\mathbf{f}}} + \mathbf{H}_{n}\mathbf{n}^{(b)}
	\end{bmatrix}
\end{equation}
where:
\begin{eqnarray}
	\mathbf{H}_{b}
	&=&
	\begin{bmatrix}
	\mathbf{b}^{\top}_{\perp1} \\
	\mathbf{b}^{\top}_{\perp2} 	
	\end{bmatrix}  \\
	\mathbf{H}_{n}
	&=&
	\begin{bmatrix}
	{^x\mathbf{b}^{\top}_{\perp1}}{^x\hat{\mathbf{P}}_{\mathbf{f}}}
	\left(-sin \left(^x\hat{\theta}_{\mathbf{f}}\right)\cos \left(^x\hat{\phi}_{\mathbf{f}}\right) + \cos \left(^x\hat{\theta}_{\mathbf{f}}\right)\cos \left(^x\hat{\phi}_{\mathbf{f}}\right)\right) \\
	^x\mathbf{b}^{\top}_{\perp2}
	\end{bmatrix}
\end{eqnarray}
where $^x\hat{\theta}_{\mathbf{f}}=\arctan\frac{^x\hat{y}_{\mathbf{f}}}{^x\hat{x}_{\mathbf{f}}}$, $^x\hat{\phi}_{\mathbf{f}}=\arctan\frac{^x\hat{z}_{\mathbf{f}}}{\sqrt{^x\hat{x}^2_{\mathbf{f}}+^x\hat{y}^2_{\mathbf{f}}}}$.

\subsection{3D LiDAR}
3D LiDAR can directly get the range and bearing information of the feature, therefore the measurement model can be denoted as: 
\begin{equation}\label{eq_3d_lidar}
\mathbf{z} = 
\begin{bmatrix}
z^{(r)}  \\
\mathbf{z}^{(b)}
\end{bmatrix}
=
\begin{bmatrix}
\sqrt{{^x\mathbf{p}_{\mathbf{f}}^{\top}}{^x\mathbf{p}_{\mathbf{f}}}}  \\
^x\theta_{\mathbf{f}} \\
^x\phi_{\mathbf{f}}
\end{bmatrix}
+
\begin{bmatrix}
{n}^{(r)}  \\
\mathbf{n}^{(b)}
\end{bmatrix}
=
\left[
\begin{array}{c}
\sqrt{(^xx_{\mathbf{f}})^2+(^xy_{\mathbf{f}})^2+(^xz_{\mathbf{f}})^2}       \\
^x\theta_{\mathbf{f}} \\
^x\phi_{\mathbf{f}}
\end{array}
\right] + 
\begin{bmatrix}
{n}^{(r)}  \\
\mathbf{n}^{(b)}
\end{bmatrix}
\end{equation}
where $\mathbf{n}^{(b)}=\begin{bmatrix} n^{(b)}_1 & n^{(b)}_2\end{bmatrix}^{\top}$. 
Similarly, we can rewrite the bearing measurement as:
\begin{equation}
\scalemath{0.9}{
\mathbf{z}^{(b)}
=
\mathbf{h}_b\left(^x\mathbf{P}_{\mathbf{f}},\mathbf{n}^{(b)}\right)
=
\begin{bmatrix}
{^x\mathbf{b}^{\top}_{\perp1}}
{^xr_{\mathbf{f}}}
\begin{bmatrix}
\cos \left({^x\theta_{\mathbf{f}}+ n^{(b)}_1}\right) \cos \phi  &
\sin \left({^x\theta_{\mathbf{f}}+ n^{(b)}_1}\right) \cos \phi  & 
\sin \phi
\end{bmatrix}^{\top} \\
{^x\mathbf{b}^{\top}_{\perp2}}
{^xr_{\mathbf{f}}}
\begin{bmatrix}
\cos \left({^x\theta_{\mathbf{f}}}\right) \cos \left({^x\phi_{\mathbf{f}}+ n^{(b)}_2}\right)  &
\sin \left({^x\theta_{\mathbf{f}}}\right) \cos \left({^x\phi_{\mathbf{f}}+ n^{(b)}_2}\right)  & 
\sin \left({^x\phi_{\mathbf{f}}+ n^{(b)}_2}\right)
\end{bmatrix}^{\top} 
\end{bmatrix}
}
\end{equation}
where $\mathbf{b}_{\perp i},i\in\{1,2\}$ are two perpendicular vectors to the bearing $^x\mathbf{b}_{\mathbf
	f}$, and they can be constructed from \cite{Yang2017SSRR}. Therefore, the linearized system with $^x\hat{\mathbf{P}}_{\mathbf{f}}$ can be described as:
\begin{equation}
\tilde{\mathbf{z}}
=
\begin{bmatrix}
\tilde{z}^{(r)} \\
\tilde{\mathbf{z}}^{(b)}
\end{bmatrix}
=
\begin{bmatrix}
\mathbf{H}_{r}{^x\tilde{\mathbf{P}}_{\mathbf{f}}} + n^{(r)}  \\
\mathbf{H}_{b}{^x\tilde{\mathbf{P}}_{\mathbf{f}}} + \mathbf{H}_{n}\mathbf{n}^{(b)}
\end{bmatrix}
\end{equation}
where:
\begin{eqnarray}
\mathbf{H}_{b}
&=&
\begin{bmatrix}
\mathbf{b}^{\top}_{\perp1} \\
\mathbf{b}^{\top}_{\perp2} 	
\end{bmatrix}  \\
\mathbf{H}_{n}
&=&
\begin{bmatrix}
{^x\mathbf{b}^{\top}_{\perp1}}{^x\hat{\mathbf{P}}_{\mathbf{f}}}
\left(-\sin \left(^x\hat{\theta}_{\mathbf{f}}\right)\cos \left(^x\hat{\phi}_{\mathbf{f}}\right) + \cos \left(^x\hat{\theta}_{\mathbf{f}}\right)\cos \left(^x\hat{\phi}_{\mathbf{f}}\right)\right) \\
{^x\mathbf{b}^{\top}_{\perp2}}{^x\hat{\mathbf{P}}_{\mathbf{f}}}
\left(-\cos \left(^x\hat{\theta}_{\mathbf{f}}\right)\sin \left(^x\hat{\phi}_{\mathbf{f}}\right) -
 \sin \left(^x\hat{\theta}_{\mathbf{f}}\right)\sin \left(^x\hat{\phi}_{\mathbf{f}}\right) +
 \cos \left(^x\hat{\phi}_{\mathbf{f}}\right)\right)
\end{bmatrix}
\end{eqnarray}
where $^x\hat{\theta}_{\mathbf{f}}=\arctan\frac{^x\hat{y}_{\mathbf{f}}}{^x\hat{x}_{\mathbf{f}}}$, $^x\hat{\phi}_{\mathbf{f}}=\arctan\frac{^x\hat{z}_{\mathbf{f}}}{\sqrt{^x\hat{x}^2_{\mathbf{f}}+^x\hat{y}^2_{\mathbf{f}}}}$.

\subsection{RGBD Camera}
RGBD camera can also get the range and bearing information of the feature, therefore:
\begin{equation}\label{eq_rgbd}
\mathbf{z} = 
\begin{bmatrix}
z^{(r)}  \\
\mathbf{z}^{(b)}
\end{bmatrix}
=
\left[
\begin{array}{c}
\sqrt{{^x\mathbf{P}_{\mathbf{f}}^{\top}}{^x\mathbf{Pp}_{\mathbf{f}}}}  \\
\frac{\mathbf{e}_1^{\top}{^x\mathbf{P}_{\mathbf{f}}}}{\mathbf{e}_3^{\top}{^x\mathbf{P}_{\mathbf{f}}}} \\
\frac{\mathbf{e}_2^{\top}{^x\mathbf{P}_{\mathbf{f}}}}{\mathbf{e}_3^{\top}{^x\mathbf{P}_{\mathbf{f}}}}
\end{array}
\right] + 
\begin{bmatrix}
n^{(r)} \\
\mathbf{n}^{(b)}
\end{bmatrix}
=
\left[
\begin{array}{c}
\sqrt{(^xx_{\mathbf{f}})^2+(^xy_{\mathbf{f}})^2+(^xz_{\mathbf{f}})^2}       \\
\frac{^xx_{\mathbf{f}}}{^xz_{\mathbf{f}}}							\\
\frac{^xy_{\mathbf{f}}}{^xz_{\mathbf{f}}}
\end{array}
\right] + 
\begin{bmatrix}
n^{(r)} \\
\mathbf{n}^{(b)}
\end{bmatrix}
\end{equation}
Therefore, we can rewrite the measurement model as:
\begin{equation}
    \mathbf{z}
    =
    \begin{bmatrix}
    z^{(r)} \\
    \mathbf{z}^{(b)}
    \end{bmatrix}
	=
	\begin{bmatrix}
	\sqrt{^x\mathbf{P}_{\mathbf{f}}^{\top}{^x\mathbf{P}_{\mathbf{f}}}} + \mathbf{n}^{(r)}  \\
	\mathbf{h}_{b}\left({^x\mathbf{P}_{\mathbf{f}}, \mathbf{n}^{(b)}}\right)
	\end{bmatrix}
	=
	\begin{bmatrix}
	\sqrt{^x\mathbf{P}_{\mathbf{f}}^{\top}{^x\mathbf{P}_{\mathbf{f}}}} + \mathbf{n}^{(r)}  \\
	\begin{bmatrix}
	^x\mathbf{b}^{\top}_{\perp1} \\
	^x\mathbf{b}^{\top}_{\perp2} 
	\end{bmatrix}
	{^x\mathbf{P}_{\mathbf{f}}}
	+
	{^xz_{\mathbf{f}}}
	\begin{bmatrix}
	^x\mathbf{b}^{\top}_{\perp1} \\
	^x\mathbf{b}^{\top}_{\perp2} 
	\end{bmatrix}
	\begin{bmatrix}
	\mathbf{I}_2 \\
	\mathbf{0}_{1\times 2}
	\end{bmatrix}
	\mathbf{n}^{(b)}
	\end{bmatrix}	
\end{equation}
And we can linearize the system with $^x\hat{\mathbf{P}}_{\mathbf{f}}$ as:
\begin{equation}
	\tilde{\mathbf{z}}
	=
	\begin{bmatrix}
	\tilde{z}^{(r)}  \\
	\tilde{\mathbf{z}}^{(b)}
	\end{bmatrix}
	\simeq
	\begin{bmatrix}
	\mathbf{H}_{r}{^x\tilde{\mathbf{P}}_{\mathbf{f}}} + n^{(r)} \\
	\mathbf{H}_{b}{^x\tilde{\mathbf{P}}_{\mathbf{f}}} + \mathbf{H}_{n}\mathbf{n}^{(b)}
	\end{bmatrix}
	=
	\begin{bmatrix}
	\frac{^x\hat{\mathbf{P}}^{\top}_{\mathbf{f}}}{^x\hat{r}_{\mathbf{f}}}{^x\tilde{\mathbf{P}}_{\mathbf{f}}}+\mathbf{n}^{(r)} \\
	\begin{bmatrix}
	^x\mathbf{b}^{\top}_{\perp1} \\
	^x\mathbf{b}^{\top}_{\perp2} 
	\end{bmatrix}
	{^x\tilde{\mathbf{P}}_{\mathbf{f}}}
	+
	{^x\hat{z}_{\mathbf{f}}}
	\begin{bmatrix}
	^x\hat{\mathbf{b}}^{\top}_{\perp1} \\
	^x\hat{\mathbf{b}}^{\top}_{\perp2} 
	\end{bmatrix}
	\begin{bmatrix}
	\mathbf{I}_2 \\
	\mathbf{0}_{1\times 2}
	\end{bmatrix}
	\mathbf{n}^{(b)}
	\end{bmatrix}
\end{equation}

\subsection{Stereo Camera}
Stereo-camera are two mono-cameras with known extrinsic transformations. Without lost of generalities, we assume input images have already been rectified, thus the measurement model can be described as:
\begin{equation}\label{eq_stereo}
\mathbf{z}  = 
\left[
\begin{array}{c}
\frac{\mathbf{e}_1^{\top}{^x\mathbf{P}_{\mathbf{f}}}}{\mathbf{e}_3^{\top}{^x\mathbf{P}_{\mathbf{f}}}} \\
\frac{\mathbf{e}_1^{\top}{^x\mathbf{P}_{\mathbf{f}}}-b_s}{\mathbf{e}_3^{\top}{^x\mathbf{P}_{\mathbf{f}}}} \\
\frac{\mathbf{e}_2^{\top}{^x\mathbf{P}_{\mathbf{f}}}}{\mathbf{e}_3^{\top}{^x\mathbf{P}_{\mathbf{f}}}} 
\end{array}
\right] + 
\begin{bmatrix}
n^{(b)}_1 \\
n^{(b)}_1 \\
n^{(b)}_2 \\
\end{bmatrix}
= 
\left[
\begin{array}{c}
\frac{^xx_{\mathbf{f}}}{^xz_{\mathbf{f}}}  \\
\frac{^xx_{\mathbf{f}}-b_s}{^xz_{\mathbf{f}}}  \\
\frac{^xy_{\mathbf{f}}}{^xz_{\mathbf{f}}}
\end{array}
\right]  + 
\begin{bmatrix}
n^{(b)}_1 \\
n^{(b)}_1 \\
n^{(b)}_2 \\
\end{bmatrix}
\end{equation}
where $b$ is the baseline for the stereo-camera, which is a known scalar. Similar to the case of Mono-camera, we can rewrite the Stereo camera measurement as:
\begin{equation}
	\mathbf{z}
	=
	\begin{bmatrix}
	\mathbf{z}^{(b)}_{L} \\
	\mathbf{z}^{(b)}_{R}
	\end{bmatrix}
	=
	\begin{bmatrix}
	\mathbf{h}_{b_L}\left(^x\mathbf{P}_{\mathbf{f}},\mathbf{n}^{(b)}\right) \\
	\mathbf{h}_{b_R}\left(^x\mathbf{P}_{\mathbf{f}},\mathbf{n}^{(b)}\right) \\
	\end{bmatrix}
	=
	\begin{bmatrix}
	\begin{bmatrix}
	^x\mathbf{b}^{\top}_{\perp1L} \\
	^x\mathbf{b}^{\top}_{\perp2L} 
	\end{bmatrix}
	{^x\mathbf{P}_{\mathbf{f}}}
	+
	{^xz_{\mathbf{f}}}
	\begin{bmatrix}
	^x\mathbf{b}^{\top}_{\perp1L} \\
	^x\mathbf{b}^{\top}_{\perp2L} 
	\end{bmatrix}
	\begin{bmatrix}
	\mathbf{I}_2 \\
	\mathbf{0}_{1\times 2}
	\end{bmatrix}
	\mathbf{n}^{(b)} \\
	\begin{bmatrix}
	^x\mathbf{b}^{\top}_{\perp1R} \\
	^x\mathbf{b}^{\top}_{\perp2R} 
	\end{bmatrix}
	{^x\mathbf{P}'_{\mathbf{f}}}
	+
	{^xz_{\mathbf{f}}}
	\begin{bmatrix}
	^x\mathbf{b}^{\top}_{\perp1R} \\
	^x\mathbf{b}^{\top}_{\perp2R} 
	\end{bmatrix}
	\begin{bmatrix}
	\mathbf{I}_2 \\
	\mathbf{0}_{1\times 2}
	\end{bmatrix}
	\mathbf{n}^{(b)}
	\end{bmatrix}
\end{equation}
where $\mathbf{z}^{(b)}_{L}$ and $\mathbf{z}^{(b)}_{R}$ represents the bearing measurement from the left-and right camera, respectively, $^x\mathbf{P}'_{\mathbf{f}}=\begin{bmatrix}{^xx_{\mathbf{f}}-b_s} & ^xy_{\mathbf{f}} & {^xz_{\mathbf{f}}}\end{bmatrix}^{\top}$. With $^x\hat{\mathbf{P}}_{\mathbf{f}}$, we can linearize the system as:
\begin{equation}
	\tilde{\mathbf{z}}
	=
	\begin{bmatrix}
	\tilde{\mathbf{z}}^{(b)}_{L}  \\
	\tilde{\mathbf{z}}^{(b)}_{R}
	\end{bmatrix}
	=
	\begin{bmatrix}
	\mathbf{H}_{b_L}{^x\tilde{\mathbf{P}}_{\mathbf{f}}}+\mathbf{H}_{n}\mathbf{n^{(b)}} \\
	\mathbf{H}_{b_R}{^x\tilde{\mathbf{P}}_{\mathbf{f}}}+\mathbf{H}_{n}\mathbf{n^{(b)}} \\
	\end{bmatrix}
	=
		\begin{bmatrix}
		\begin{bmatrix}
		^x\hat{\mathbf{b}}^{\top}_{\perp1_L} \\
		^x\hat{\mathbf{b}}^{\top}_{\perp2_L} 
		\end{bmatrix}
		{^x\tilde{\mathbf{P}}_{\mathbf{f}}}
		+
		{^x\hat{z}_{\mathbf{f}}}
		\begin{bmatrix}
		^x\hat{\mathbf{b}}^{\top}_{\perp1_L} \\
		^x\hat{\mathbf{b}}^{\top}_{\perp2_L} 
		\end{bmatrix}
		\begin{bmatrix}
		\mathbf{I}_2 \\
		\mathbf{0}_{1\times 2}
		\end{bmatrix}
		\mathbf{n}^{(b)} \\
		\begin{bmatrix}
		^x\hat{\mathbf{b}}^{\top}_{\perp1_R} \\
		^x\hat{\mathbf{b}}^{\top}_{\perp2_R} 
		\end{bmatrix}
		{^x\tilde{\mathbf{P}}_{\mathbf{f}}}
		+
		{^x\hat{z}_{\mathbf{f}}}
		\begin{bmatrix}
		^x\hat{\mathbf{b}}^{\top}_{\perp1_R} \\
		^x\hat{\mathbf{b}}^{\top}_{\perp2_R} 
		\end{bmatrix}
		\begin{bmatrix}
		\mathbf{I}_2 \\
		\mathbf{0}_{1\times 2}
		\end{bmatrix}
		\mathbf{n}^{(b)}
		\end{bmatrix}
\end{equation}

To sum up, the measurement model and its linearized model for aided INS can be generalized as \eqref{eq_x_meas}. 

\begin{table}[h]
	\caption{Measurement Model for Different Sensors}
	\label{tab_point_meas}
	\centering
	\begin{tabular}{cccc}
		\hline \hline
		Sensor    			&   Range 	    	&Full Bearing   	&Partial Bearing 	\\
		\hline
		1D range finder  	&  $\checkmark$  	&              		&   				\\
		mono-camera      	&           		&    $\checkmark$   &     				\\
		sonar			 	&  $\checkmark$  	&    				&  	$\checkmark$	\\
		2D lidar         	&  $\checkmark$	 	&    $\checkmark$	&       			\\
		3D lidar 			&  $\checkmark$  	&    $\checkmark$   &   				\\
		RGBD-camera    	 	&  $\checkmark$  	&    $\checkmark$   &   				\\
		stereo-camera    	&  $\checkmark$  	&    $\checkmark$   &   				\\
		\hline
	\end{tabular}
\end{table}

All the sensor measurements to point feature can be seen as the combination of range measurements and bearing measurements (as shown in Table \ref{tab_point_meas}). (Be noted from the table that camera sensors can get the full bearing measurements, which in some sense is equivalent that we get the information of $\theta$ and $\phi$. The sonar can only get partial bearing information ($\theta$), so we label it as partial bearing measurement in the table.) 
Therefore, in order to fully analyze the observability property of aided INS, we will analyze the range only  measurement model and bearing only measurement model in the next section respectively.

%% file: appendix/proof_projection.tex
\section{Projection Matrix for Line}\label{apd_line_proj}
Assume we have two points $^I\mathbf{P}_{1}= [p_{x1}, p_{y1}, p_{z1}]^{\top}$ and $^I\mathbf{P}_{2}=[p_{x2}, p_{y2}, p_{z2}]^{\top}$ in the line and their related projection points as $\mathbf{x}_1$ and $\mathbf{x}_2$. 
The first step is to compute the line parameters in the image. 
From camera projection model, we have: 
\begin{eqnarray}
	\mathbf{x}_1 = 
	\begin{bmatrix}
	u_1 \\ v_1  \\ 1
	\end{bmatrix}
	=
	\begin{bmatrix}
	f_1   &   0  &  c_1 \\
	0    &   f_2  &  c_2 \\
	0   &  0  &  1
	\end{bmatrix}
	\begin{bmatrix}
	\frac{p_{x1}}{p_{z1}} \\
	\frac{p_{y1}}{p_{z1}} \\
	1
	\end{bmatrix}
\end{eqnarray}
Therefore, we have:
\begin{eqnarray}
	u_1 &=& \frac{p_{x1}}{p_{z1}} f_1 + c_1 \\
	v_1 &=& \frac{p_{y1}}{p_{z1}} f_2 + c_2 
\end{eqnarray}

Assume a 2D line can be parameterized as: $u=av+c$. And $\mathbf{x}_1$ and $\mathbf{x}_2$ are in the line, therefore, we have:
\begin{equation}
	a = \frac{u_1 - u_2}{v_1 -v_2}
\end{equation}
Plug in the $\mathbf{x}_1=[u_1, v_1, 1]^{\top}$, we have:
\begin{equation}
u_1 = 	\frac{u_1 - u_2}{v_1 -v_2}v_1 + c \Rightarrow c= \frac{u_2v_1 - u_1 v_2}{v_1 - v_2}
\end{equation}
Therefore, we can write the projected 2D line as:
\begin{equation}
	(u_1 -u_2)v - (v_1 - v_2)u + (u_2v_1 - u_1v_2) = 0
\end{equation}

Then, we can get:
\begin{small}
\begin{equation}
	\left( \frac{p_{x1}}{p_{z1}}f_1 - \frac{p_{x2}}{p_{z2}}f_1\right)v 
	- \left( \frac{p_{y1}}{p_{z1}}f_2 - \frac{p_{y2}}{p_{z2}}f_2\right)u
	+ \left(
	\left( \frac{p_{x2}}{p_{z2}}f_1+ c_1\right)\left( \frac{p_{y1}}{p_{z1}}f_2+ c_2\right)
	-
	\left( \frac{p_{x1}}{p_{z1}}f_1+ c_1\right)\left( \frac{p_{y2}}{p_{z2}}f_2+ c_2\right)
	\right)=0
\end{equation}
\end{small}\noindent
Finally we can get:
%
\begin{align}
\begin{bmatrix}
u & v & 1
\end{bmatrix}
\begin{bmatrix}
p_{x2}p_{z1} - p_{x1}p_{z2}   \\
- \left(p_{y2}p_{z1} - p_{y1}p_{z2}\right) \\
\left( p_{x1}p_{y2} - p_{x2}p_{y1}\right)f_1f_2
+ \left( p_{x1}p_{z2} - p_{x2}p_{z1}\right)f_1c_2
+ \left( p_{y2}p_{z1} - p_{y1}p_{z2}\right)f_2c_1
\end{bmatrix}
=0
\end{align}

Then we can get the $\mathbf{l}'$ as:
\begin{eqnarray}
	\mathbf{l}' &=& 
	\begin{bmatrix}
	l_1 \\ l_2 \\ l_3
	\end{bmatrix}
	=
	\begin{bmatrix}
	 \left(p_{y1}p_{z2} - p_{y2}p_{z1}\right)f_2\\
	\left(p_{x2}p_{z1} - p_{x1}p_{z2}\right)f_1  \\
	\left( p_{x1}p_{y2} - p_{x2}p_{y1}\right)f_1f_2
	+ \left( p_{x1}p_{z2} - p_{x2}p_{z1}\right)f_1c_2
	+ \left( p_{y2}p_{z1} - p_{y1}p_{z2}\right)f_2c_1
	\end{bmatrix} \\
	&=&
	\begin{bmatrix}
	f_2 & 0  &  0 \\
	0   &  f_1  &  0\\
	-f_2c_1  &  - f_1c_2 &  f_1f_2
	\end{bmatrix}
	\begin{bmatrix}
	\left(p_{y1}p_{z2} - p_{y2}p_{z1}\right) \\
	\left(p_{x2}p_{z1} - p_{x1}p_{z2}\right)  \\
	\left( p_{x1}p_{y2} - p_{x2}p_{y1}\right)
	\end{bmatrix} \\
	&=& 
	\begin{bmatrix}
	f_2 & 0  &  0 \\
	0   &  f_1  &  0\\
	-f_2c_1  &  - f_1c_2 &  f_1f_2
	\end{bmatrix}
	\lfloor {^I\mathbf{P}_{1}} \times \rfloor {^I\mathbf{P}_{2}}
\end{eqnarray}
%

%% file: appendix/line_jacobian.tex
\section{Line Measurement Jacobians} \label{apd_line_jacob}

We compute the measurement Jacobians   by perturbation.
Specifically, based on \eqref{eq:def-L} and \eqref{eq:def-RL} and with a small perturbation $\delta \boldsymbol{\theta}_L$, we can find $\frac{\partial {^G{\mathbf{L}}}}{\partial \delta \boldsymbol{\theta}_L}$  as follows:
\begin{align}
	{^G\mathbf{L}} 
	&\simeq
	\scalemath{0.9}{
	\left(\mathbf{I} - \lfloor \delta \boldsymbol{\theta}_L\rfloor\right)
	\hat{\mathbf{R}}_L
	\begin{bmatrix}
	\norm{^G\hat{\mathbf{n}}_L} &  0 \\
	0  & \norm{^G\hat{\mathbf{v}}_L} \\
	0 & 0
	\end{bmatrix} }
     =
     \scalemath{1}{
	 {^G\hat{\mathbf{L}}} 
	 + 
	 \left[
	 \lfloor {^G\hat{\mathbf{n}}_L}\rfloor  \  
	 \lfloor {^G\hat{\mathbf{v}}_L}\rfloor
	 \right]
	 \delta \boldsymbol{\theta}_L
	}
\end{align}
Similarly, 
we perturb \eqref{eq:def-WL} with $\delta \boldsymbol{\phi}_L$ and obtain:
\begin{align}
	\mathbf{W}_L & 
    \simeq
	\begin{bmatrix}
	\cos \hat{\phi}_L - \delta{\phi}_L  \sin \hat{\phi}_L  &  
	-\sin \hat{\phi}_L - \delta{\phi}_L  \cos \hat{\phi}_L\\
	\sin \hat{\phi}_L + \delta{\phi}_L  \cos \hat{\phi}_L &  
	\cos \hat{\phi}_L  - \delta{\phi}_L  \sin \hat{\phi}_L
	\end{bmatrix}
	\nonumber
\end{align}
With that, we have:
\begin{displaymath}
\scalemath{0.9}{\cos \phi_L = \frac{\norm{^G\mathbf{n}_L}}{\sqrt{\norm{^G\mathbf{n}_L}^2+\norm{^G\mathbf{v}_L}^2}}},\  
\scalemath{0.9}{\sin \phi_L = \frac{\norm{^G\mathbf{v}_L}}{\sqrt{\norm{^G\mathbf{n}_L}^2+\norm{^G\mathbf{v}_L}^2}}}
\end{displaymath}
And then we get:
\begin{align}
	\norm{^G\mathbf{n}_L} 
	&\simeq
	\scalemath{0.9}
	{ {\sqrt{\norm{^G\hat{\mathbf{n}}_L}^2+\norm{^G\hat{\mathbf{v}}_L}^2}}
		\left(\cos \hat{\phi}_L - \delta{\phi}_L  \sin \hat{\phi}_L\right) }
	= \norm{^G\hat{\mathbf{n}}_L} - \norm{^G\hat{\mathbf{v}}_L} \delta \phi_L 
	\\
	\norm{^G\mathbf{v}_L} 
	&\simeq \scalemath{0.9}
	{{\sqrt{\norm{^G\hat{\mathbf{n}}_L}^2+\norm{^G\hat{\mathbf{v}}_L}^2}}
		\left(\sin \hat{\phi}_L + \delta{\phi}_L  \cos \hat{\phi}_L\right) }
	= \norm{^G\hat{\mathbf{v}}_L} + \norm{^G\hat{\mathbf{n}}_L} \delta \phi_L 
\end{align}
Finally, we can read out the Jacobian $\frac{\partial {^G{\mathbf{L}}}}{\partial \delta \boldsymbol{\phi}_L}$ from:
\begin{align}
	{^G\mathbf{L}} 
	&\simeq  
	\begin{bmatrix}
	{^G\hat{\mathbf{n}}_L} -  \frac{w_2}{w_1} {^G\hat{\mathbf{n}}_L} \delta \phi_L  &
	{^G\hat{\mathbf{v}}_L} +  \frac{w_1}{w_2}{^G\hat{\mathbf{v}}_L}  \delta \phi_L 
	\end{bmatrix}
\end{align}
%

%% file: appendix/single_line.tex
\section{Proof of Lemma \ref{lem_single_line}} \label{apd_single_line}

For $\mathbf{N}_{l1}$, we have:
\begin{align}
\mathbf{H}^{(l)}_{I_k}\boldsymbol{\Phi}_{(k,1)}\mathbf{N}_{l1}
&=
\mathbf{H}_{l,k}
\mathbf{K}{^{I_k}_G\hat{\mathbf{R}}}
\big(
\lfloor {^G}\mathbf{g}\times\rfloor \lfloor {^G\hat{\mathbf{P}}_{I_k}}\times\rfloor {^G\hat{\mathbf{v}}_{L}} 
+
\lfloor{^G\hat{\mathbf{v}}_{L}} \rfloor \lfloor{^G}\mathbf{g}\times \rfloor {^G\hat{\mathbf{P}}_{I_k}}
 +
\lfloor {^G\hat{\mathbf{P}}_{I_k}}\times\rfloor \lfloor{^G\hat{\mathbf{v}}_{L}\times} \rfloor {^G\hat{\mathbf{g}}} 
\big)
=
\mathbf{0}
\end{align}
For $\mathbf{N}_{l2}$, we have:
\begin{align} \label{eq_nl_2}
\mathbf{H}^{(l)}_{I_k}\boldsymbol{\Phi}_{(k,1)}\mathbf{N}_{l2}
&=
\mathbf{H}_{l,k}
\mathbf{K}{^{I_k}_G\hat{\mathbf{R}}}
\left(
\lfloor {^G\hat{\mathbf{v}}_{L}} \times \rfloor {^G\hat{\mathbf{n}}_{\mathbf{e}}}
+
\lfloor {^G\hat{\mathbf{n}}_{\mathbf{e}}} \times \rfloor {^G\hat{\mathbf{v}}_{L}}
\right) 
=
\mathbf{0}
\end{align}
For $\mathbf{N}_{l3}$, we have:
\begin{align} \label{eq_nl_3}
\mathbf{H}^{(l)}_{I_k}\boldsymbol{\Phi}_{(k,1)}\mathbf{N}_{l3}
&=
\mathbf{H}_{l,k}
\mathbf{K}{^{I_k}_G\hat{\mathbf{R}}}
\left(\lfloor {^G\hat{\mathbf{v}}_{L}} \times \rfloor {^G\hat{\mathbf{v}}_{\mathbf{e}}}\right)
=
\mathbf{0}
\end{align}
For $\mathbf{N}_{l4}$, we have:
\begin{align}
&\mathbf{H}^{(l)}_{I_k}\boldsymbol{\Phi}_{(k,1)}\mathbf{N}_{l4}=
-\frac{w_1 w^2_2}{w^2_1+w^2_2}
\mathbf{H}_{k,l}
\mathbf{K}
{^{I_k}\mathbf{n}'_{L}}\\
\label{eq_line_norm_vec}
&{^{I_k}\mathbf{n}'_{L}} 
=
{^{I_k}_G\hat{\mathbf{R}}}
\lfloor \left(\frac{w_1}{w_2}\lfloor {^G\hat{\mathbf{n}}_e}\times\rfloor {^G\mathbf{\mathbf{v}}_{e}}  + {^G\mathbf{P}_{I_k}} \right)
\times\rfloor
{^G\hat{\mathbf{v}}_e} 
\end{align}
%
%
By geometry, ${^{I_k}\mathbf{n}'_{L}}$ is parallel to ${^{I_k}\mathbf{n}_{L}}$. 
Since $\mathbf{l}'=\mathbf{K}{^{I_k}\mathbf{n}_{L}}$ is the null space of 
$\mathbf{H}_{l,k}$
, and thus, 
$\mathbf{K}{^{I_k}\mathbf{n}'_{L}}$ is also the null space of 
$\mathbf{H}_{l,k}$. 
As a result, we have $\mathbf{H}^{(l)}_{I_k}\boldsymbol{\Phi}_{(k,1)}\mathbf{N}_{l4}=\mathbf{0}$.
For $\mathbf{N}_{l5}$, we have:
\begin{align} \label{eq_nl_5}
&\mathbf{H}^{(l)}_{I_k}\boldsymbol{\Phi}_{(k,1)}\mathbf{N}_{l5}=
\mathbf{H}_{k,l}
\mathbf{K}
{^{I_k}_G\hat{\mathbf{R}}}
\lfloor {^G\hat{\mathbf{v}}_L}\times\rfloor {^G\hat{\mathbf{v}}_{\mathbf{e}}} \delta t_k
=\mathbf{0}
\end{align}
It should be noted  that the structure of $\mathbf{N}_l$ reveals that $\mathbf{N}_{l1:5}$ are linearly independent. 

%% file: appendix/direct_line.tex
\section{Observability with Direct Line Measurements} \label{apd_direct_line}

Given point clouds, we may directly extract a 3D line feature parameterized by its direction $^I\mathbf{v}_L$ and distance to the sensor $^Id=\frac{\norm{^I\mathbf{n}_{L}}}{\norm{^I\mathbf{v}_{L}}}$. 
We describe such measurements as:
%
\begin{equation}
	\mathbf{z} = 
	\begin{bmatrix}
	\lfloor ^I\mathbf{v}_m \times \rfloor {^I\mathbf{v}_{L}} 
	\\
	^Id
	\end{bmatrix}
\end{equation}
where $^I\mathbf{v}_{m}$ and $^Id$ are the line direction measurement and the distance measurement in the sensor's frame, respectively. 
The measurement Jacobian is computed by:
\begin{align}
	\frac{\partial \tilde{\mathbf{z}}}{\partial \tilde{\mathbf{x}}}
	&=
	\frac{\partial \tilde{\mathbf{z}}}{\partial {^I\mathbf{L}}}
	\frac{\partial {^I\mathbf{L}}}{\partial \tilde{\mathbf{x}}} \\
	\frac{\partial \tilde{\mathbf{z}}}{\partial {^I\mathbf{L}}}
	&=
	\begin{bmatrix}
	\mathbf{0}_3   &  \lfloor {^I\mathbf{v}_m}\times\rfloor \\
	\frac{^I\mathbf{n}^{\top}_L}{\norm{^I\mathbf{n}_L}\norm{^I\mathbf{v}_L}}  &
	-\frac{\norm{^I\mathbf{n}_L}}{\norm{^I\mathbf{v}_L}^3}{^I\mathbf{v}^{\top}_L}
	\end{bmatrix}
\\
	\frac{\partial {^I\mathbf{L}}}{\partial \tilde{\mathbf{x}}}
	&=
	\scalemath{1}{
	\begin{bmatrix}
	{^I_G\hat{\mathbf{R}}}  &  \mathbf{0}_3   \\
	\mathbf{0}_3    &   {^I_G\hat{\mathbf{R}}}
	\end{bmatrix}}
	\scalemath{0.9}{
	\begin{bmatrix}
	\mathbf{H}_{l1}& 
	\mathbf{0}_{3\times 9} &
	\lfloor {^G\hat{\mathbf{v}}_L}\times\rfloor &
	\mathbf{H}_{l2} &
	\mathbf{H}_{l3}  \\
	\lfloor {^G\hat{\mathbf{v}}_L}\times\rfloor {^I_G\hat{\mathbf{R}}}^{\top} & 
	\mathbf{0}_{3\times 9} &
	\mathbf{0}_{3}  & 
	\lfloor {^G\hat{\mathbf{v}}_L}\times\rfloor & 
	\frac{w_1}{w_2}{^G\hat{\mathbf{v}}_L}
	\end{bmatrix}
	}
\end{align}
Then, the block row of the observability matrix is given by:
\begin{align}
	&\scalemath{1}{
	\mathbf{H}^{(l)}_{I_k}\boldsymbol{\Phi}_{(k,1)} }
	=
	\scalemath{1}{
	\frac{\partial \tilde{\mathbf{z}}}{\partial {^I\mathbf{L}}}
	\begin{bmatrix}
	{^I_G\hat{\mathbf{R}}}  &  \mathbf{0}_3   \\
	\mathbf{0}_3    &   {^I_G\hat{\mathbf{R}}}
	\end{bmatrix} \times }
	\\
	&	
	\scalemath{1}{
	\begin{bmatrix}
	\boldsymbol{\Gamma}_{l1} &
	\boldsymbol{\Gamma}_{l2} &
	\lfloor {^G\hat{\mathbf{v}}_L\times}\rfloor\delta t_k &
	\boldsymbol{\Gamma}_{l3} &
	\lfloor {^G\hat{\mathbf{v}}_L\times}\rfloor &
	\boldsymbol{\Gamma}_{l4} &
	\boldsymbol{\Gamma}_{l5} 
	\\
	\lfloor {^G\hat{\mathbf{v}}_L}\times\rfloor {^G_{I_1}\hat{\mathbf{R}}}  & 
	\lfloor {^G\hat{\mathbf{v}}_L}\times\rfloor {^G_{I_k}\hat{\mathbf{R}}} \boldsymbol{\Phi}_{12} &
	\mathbf{0}_3  & 
	\mathbf{0}_3  & 
	\mathbf{0}_3 & 
	\lfloor {^G\hat{\mathbf{v}}_L}\times\rfloor & 
	\frac{w_1}{w_2}{^G\hat{\mathbf{v}}_L}
	\end{bmatrix}}
\nonumber
\end{align}
In this case, we show that the linearized aided INS system with a line feature still has at least  5 unobservable directions 
$\mathbf{N}_l$~\eqref{eq_line_obs}.
To see this, we can first verify $\mathbf{N}_{l2:3}$ and $\mathbf{N}_{l5}$ in a similarly way as \eqref{eq_nl_2}, \eqref{eq_nl_3}and \eqref{eq_nl_5}. 
Then, we need to verify $\mathbf{N}_{l1}$ and $\mathbf{N}_{l4}$. 
For $\mathbf{N}_{l1}$, we can have $\mathbf{H}^{(l)}_{I_k}\boldsymbol{\Phi}_{k,1}\mathbf{N}_{l1}=\mathbf{0}$ with:
\begin{align}
	&
	\scalemath{1}{
	\begin{bmatrix}
	\lfloor {^G}\mathbf{g}\times\rfloor \lfloor {^G\hat{\mathbf{P}}_{I_k}}\times\rfloor {^G\hat{\mathbf{v}}_{L}} +
	\lfloor{^G\hat{\mathbf{v}}_{L}} \rfloor \lfloor{^G}\mathbf{g}\times \rfloor {^G\hat{\mathbf{P}}_{I_k}} +
	\lfloor {^G\hat{\mathbf{P}}_{I_k}}\times\rfloor \lfloor{^G\hat{\mathbf{v}}_{L}\times} \rfloor {^G\hat{\mathbf{g}}}  \\
	\lfloor{^G\hat{\mathbf{v}}_{L}\times} \rfloor {^G\hat{\mathbf{g}}} - \lfloor{^G\hat{\mathbf{v}}_{L}\times} \rfloor {^G\hat{\mathbf{g}}}
	\end{bmatrix}} =\mathbf{0}
	\nonumber
\end{align} 
For $\mathbf{N}_{l4}$, we have:
\begin{align} 
	&\frac{\partial \tilde{\mathbf{z}}}{\partial {^I\mathbf{L}}}
	\scalemath{0.9}{
	\begin{bmatrix}
	 -{^{I_k}\mathbf{n}'_{L}} \\
	{^{I_k}_G\hat{\mathbf{R}}}{^G\hat{\mathbf{v}}_e}
	\end{bmatrix}}
	\!=\!
	\scalemath{0.9}{
	\begin{bmatrix}
	\lfloor {^{I_k}\mathbf{v}_m}\rfloor {^{I_k}\mathbf{v}_e}  \\
	\frac{\norm{^{I_k}\mathbf{n}_L}}{\norm{^{I_k}\mathbf{v}_L}^2}{^{I_k}\mathbf{n}^{\top}_e}{^{I_k}\mathbf{n}_e} -
	\frac{\norm{^{I_k}\mathbf{n}_L}}{\norm{^{I_k}\mathbf{v}_L}^2}{^{I_k}\mathbf{v}^{\top}_e}{^{I_k}\mathbf{v}_e}
	\end{bmatrix}}
	\!=\!
	\begin{bmatrix}
	\mathbf{0} \\
	\mathbf{0}
	\end{bmatrix}
%
\end{align}
where $^{I_k}\mathbf{n}_e$ is the unit direction vector of $^{I_k}\mathbf{n}_L$, $^{I_k}d$ is the distance of line to sensor in frame $\{I_k\}$. ${^{I_k}\mathbf{n}'_{L}}$ is given by \eqref{eq_line_norm_vec} and $^{I_k}\mathbf{n}'_L = - {^{I_k}d}{^{I_k}\mathbf{n}_e}$.
With that we see that  $\mathbf{N}_{l4}$ lies in the null space:
\begin{align}
	\mathbf{H}^{(l)}_{I_k}\boldsymbol{\Phi}_{k,1}\mathbf{N}_{l4}
	&=
	\frac{w_1 w^2_2}{w^2_1+w^2_2}
	\frac{\partial \tilde{\mathbf{z}}}{\partial {^I\mathbf{L}}}
	\scalemath{0.9}{
	\begin{bmatrix}
		-{^{I_k}\mathbf{n}'_{L}} \\
		{^{I_k}_G\hat{\mathbf{R}}}{^G\hat{\mathbf{v}}_e}
	\end{bmatrix}}
	=
	\begin{bmatrix}
	\mathbf{0} \\
	\mathbf{0}
	\end{bmatrix}
	%
\end{align}
%
%
We thus conclude that the aided INS with direct line measurements has the same unobservable directions $\mathbf{N}_{l}$ with  projective line measurements.

%% file: appendix/multiple_lines.tex

\section{Proof of Lemma \ref{lem_multiple_lines}} \label{apd_line}

With $l$ line features  in the state vector, the block row of the observability matrix $\mathbf{M}(\mathbf{x})$ can be constructed by:
\begin{equation}
\scalemath{1}{
	\mathbf{H}^{(l)}_{I_k}\boldsymbol{\Phi}_{(k,1)}
	=
	\begin{bmatrix}
	\mathbf{H}^{L_1}_{I_k}\boldsymbol{\Phi}_{(k,1)} \\
	\mathbf{H}^{L_2}_{I_k}\boldsymbol{\Phi}_{(k,1)} \\
	\vdots  \\
	\mathbf{H}^{L_l}_{I_k}\boldsymbol{\Phi}_{(k,1)}
	\end{bmatrix}
	=
	\begin{bmatrix}
	\mathbf{M}^{L_1}_{k} \\
	\mathbf{M}^{L_2}_{k} \\
	\vdots \\
	\mathbf{M}^{L_l}_{k} \\
	\end{bmatrix}
}
\end{equation}\noindent
%
%
where $\mathbf{H}^{L_i}_{I_k}$ and $\mathbf{M}^{L_i}_{I_k}$ are the measurement Jacobians and the $i$-th block row of observability matrix for line feature $i$, respectively. 
It is straightforward to verify $\mathbf{N}_{L1}$, which relates to the rotation around the gravitational direction. 
Since we have multiple unparallel lines, the unobservable direction along the line direction diminishes. 
Therefore, the main task to prove that $\mathbf{N}_{L2:4}$ are the null space of $\mathbf{H}_{I_k}\boldsymbol{\Phi}_{(k,1)}$. 
$\mathbf{N}_{L2:4}$ are related to the sensor position and, from the analysis \eqref{eq_line_obs}, for a single line feature, we can easily find vectors $\boldsymbol{\alpha}_j, \boldsymbol{\beta}_j, \boldsymbol{\gamma}_j$ which are the null space for $\mathbf{M}^{L_j}_{k}$ for feature $j$, and $\boldsymbol{\alpha}_i, \boldsymbol{\beta}_i, \boldsymbol{\gamma}_i$ which are the null space for $\mathbf{M}^{L_i}_{k}$ for feature $i$. 
%
\begin{align}
\scalemath{1}{
	\mathbf{M}^{L_i}_{k}
	\begin{bmatrix}
	\alpha_i & \beta_i & \gamma_i
	\end{bmatrix}
}
	&=
	\mathbf{0}_{3\times 1} \\
\scalemath{1}{	
	\mathbf{M}^{L_j}_{k}
	\begin{bmatrix}
	\alpha_j & \beta_j & \gamma_j
	\end{bmatrix}
}
	&=
	\mathbf{0}_{3\times 1}
\end{align}
If we can show that $\{\alpha_j, \beta_j, \gamma_j\}$ can be linearly dependent on $\{\alpha_i, \beta_i, \gamma_i\}$, i.e., 
\begin{equation} \label{eq_linear_transformation}
\scalemath{1}{
	\begin{bmatrix}
	\alpha_j & \beta_j & \gamma_j
	\end{bmatrix}
	=
	\begin{bmatrix}
	\alpha_i & \beta_i & \gamma_i
	\end{bmatrix}
	\boldsymbol{\Lambda}
}
\end{equation}
%
then both $\alpha_i, \beta_i, \gamma_i$ and $\alpha_j, \beta_j, \gamma_j$ share the same bases if $\boldsymbol{\Lambda}$ is invertible, and they are the null space for both $\mathbf{M}^{L_j}_{k}$ and $\mathbf{M}^{L_i}_{k}$, that is:
\begin{equation}\label{eq_null_line}
\scalemath{1}{
	\begin{bmatrix}
	\mathbf{M}^{L_i}_{k} \\
	\mathbf{M}^{L_j}_{k}
	\end{bmatrix}
	\begin{bmatrix}
	\alpha_i & \beta_i & \gamma_i 
	\end{bmatrix}
	=
	\begin{bmatrix}
	\mathbf{M}^{L_i}_{k} \\
	\mathbf{M}^{L_j}_{k}
	\end{bmatrix}
	\begin{bmatrix}
	\alpha_j & \beta_j & \gamma_j 
	\end{bmatrix}
	=	
	\mathbf{0}
}
\end{equation}
Based on the definition of $\mathbf{N}_{L2:4}$, $\mathbf{N}_{2:4}$ is the null space for line $i$. Then for the null space $\mathbf{N}^{(j)}_{L2:4}$ for line $j$, we can have:
\begin{equation}
\scalemath{1}{
	\mathbf{N}^{(j)}_{L2:4}=
	\begin{bmatrix}
	\alpha_j & \beta_j & \gamma_j
	\end{bmatrix}
	=
	\begin{bmatrix}
	\alpha_i & \beta_i & \gamma_i
	\end{bmatrix}
	{^{L_i}_{L_j}\mathbf{R}}
	=
	\mathbf{N}_{L2:4}
	{^{L_i}_{L_j}\mathbf{R}}
}
\end{equation}
As ${^{L_i}_{L_j}\mathbf{R}}$ is a rotation matrix, it is invertible. 
With \eqref{eq_linear_transformation} and \eqref{eq_null_line}, we have that both $\mathbf{N}_{L2:4}$ and $\mathbf{N}^{(j)}_{L2:4}$ are the null space of $\mathbf{M}^{L_j}_{k}$ and $\mathbf{M}^{L_i}_{k}$. 
We have no assumption for the choice of $i$ and $j$, and thus, $\mathbf{N}_{L2:4}$ is in the null space of $\mathbf{H}^{(l)}_{I_k}\boldsymbol{\Phi}_{(k,1)}$.

%% file: appendix/plane_obs.tex
\section{Proof of Lemma \ref{lem_plane_obs} } \label{apd_plane_single_obs}

First of all, it is straightforward to verify $\mathbf{N}_{\pi 1}$ as:
%
%
\begin{align}
\mathbf{H}^{(\pi)}_{I_k}\boldsymbol{\Phi}_{(k,1)}\mathbf{N}_{\pi 1}
&=
{^{I_k}_G\hat{\mathbf{R}}}
\left(
{^G\hat{{d}}_{\pi}}
\lfloor {^G\hat{\mathbf{n}}_{\pi}}\times \rfloor
-
\lfloor {^G\hat{\boldsymbol{\Pi}}_{\pi}}\times \rfloor
\right)
=
\mathbf{0}
\end{align}

We here show $\mathbf{N}_{\pi 2:4}$ is in the null space:
\begin{align}
	\mathbf{H}^{(\pi)}_{I_k}\boldsymbol{\Phi}_{(k,1)}\mathbf{N}_{\pi 2:4}
	&=
	{^{I_k}_G\hat{\mathbf{R}}}
	\left(
	-{^G\hat{\mathbf{n}}_{\pi}}{^G\hat{\mathbf{n}}^{\top}_{\pi}}{^G_{\Pi}\hat{\mathbf{R}}}
	+
	\boldsymbol{\Gamma}_{\pi 4} {^G\hat{\mathbf{n}}_{\pi}} {\mathbf{e}^{\top}_3}
	\right)
	\nonumber
	\\
	&=
	{^{I_k}_G\hat{\mathbf{R}}}
	\left(
	-{^G\hat{\mathbf{n}}_{\pi}}{\mathbf{e}^{\top}_3}
	+
	 {^G\hat{\mathbf{n}}_{\pi}} {\mathbf{e}^{\top}_3}
	\right) = \mathbf{0}
\end{align}
We then verify $\mathbf{N}_{\pi 5:6}$ are also in the null space:
\begin{align}
\mathbf{H}^{(\pi)}_{I_k}\boldsymbol{\Phi}_{(k,1)}\mathbf{N}_{\pi 5}
&=
-
{^{I_k}_G\hat{\mathbf{R}}}
\left({^G\hat{\mathbf{n}}_{\pi}}{^G\hat{\mathbf{n}}^{\top}_{\pi}} {^G\hat{\mathbf{n}}^{\perp}_{1}} \delta t_k \right)
\!\! =\!\! 
\mathbf{0}
\\
\mathbf{H}^{(\pi)}_{I_k}\boldsymbol{\Phi}_{(k,1)}\mathbf{N}_{\pi 6}
&=
-
{^{I_k}_G\hat{\mathbf{R}}}
\left({^G\hat{\mathbf{n}}_{\pi}}{^G\hat{\mathbf{n}}^{\top}_{\pi}} {^G\hat{\mathbf{n}}^{\perp}_{2}} \delta t_k \right)
\!\! =\!\! 
\mathbf{0}
\end{align}
To verify $\mathbf{N}_{\pi 7}$, we have:
\begin{align}
&\mathbf{H}^{(\pi)}_{I_k}\boldsymbol{\Phi}_{(k,1)}\mathbf{N}_{\pi 7}
=
{^{I_k}_G\hat{\mathbf{R}}}
\scalemath{0.9}{
	\big[\left(
	{^Gd} - {^G\hat{\mathbf{n}}^{\top}_{\pi}}{^G\hat{\mathbf{P}}_{I_k}}
	\right)
	\lfloor {^G\hat{\mathbf{n}}_{\pi}}\times\rfloor}
	- 
{^G\hat{\mathbf{n}}_{\pi}}{^G\hat{\mathbf{n}}^{\top}_{\pi}}
\lfloor {^G\hat{\mathbf{P}}_{I_1}} 
+ {^G\hat{\mathbf{V}}_{I_1}}{\delta t_k}
- \frac{1}{2}{^G\mathbf{g}}{\delta t^2_k}
- 
{^G\hat{\mathbf{P}}_{I_k}} \times\rfloor \big]
{^G\hat{\mathbf{n}}_{\pi}}
=\mathbf{0}
\end{align}
Finally, from the structure of $\mathbf{N}_{\pi}$, it can be seen that $\mathbf{N}_{\pi 1:7}$ are linearly independent. 

%% file: appendix/multiple_planes.tex
\section{Proof of Lemma \ref{lem_multiple_plane}}\label{apd_plane}

Given $s$ plane features in the state vector,  the block row of the observability matrix becomes:
\begin{equation}
\scalemath{1}{
	\mathbf{H}^{(\pi)}_{I_k}\boldsymbol{\Phi}_{(k,1)}
	=
	\begin{bmatrix}
	\mathbf{H}^{\Pi_1}_{I_k}\boldsymbol{\Phi}_{(k,1)} \\
	\mathbf{H}^{\Pi_2}_{I_k}\boldsymbol{\Phi}_{(k,1)} \\
	\vdots  \\
	\mathbf{H}^{\Pi_s}_{I_k}\boldsymbol{\Phi}_{(k,1)}
	\end{bmatrix}
	=:
	\begin{bmatrix}
	\mathbf{M}^{\Pi_1}_{k} \\
	\mathbf{M}^{\Pi_2}_{k} \\
	\vdots \\
	\mathbf{M}^{\Pi_s}_{k} \\
	\end{bmatrix}
}
\end{equation}
First of all, 
it is straightforward to verify $\mathbf{N}_{\Pi 1}$ that is related to the rotation around the gravitational direction. 
If $s=2$, the state vector has two plane features, whose intersection $\lfloor {^G\mathbf{n}_{\pi_1}}\times\rfloor {^G\mathbf{n}_{\pi_2}}$ is perpendicular to both plane normal directions. 
Therefore, the sensor motion along this intersection line $\mathbf{N}_{\Pi 5}$ will become unobservable. 

Similar to Appendix~\ref{apd_line}, the main task is to prove $\mathbf{N}_{\Pi 2:4}$ are in the null space of $\mathbf{H}^{(\pi)}_{I_k}\boldsymbol{\Phi}_{(k,1)}$. 
To this end, 
we can easily find vectors $\boldsymbol{\alpha}_j, \boldsymbol{\beta}_j, \boldsymbol{\gamma}_j$ which are the null space of $\mathbf{M}^{\Pi_j}_{k}$ for feature $j$, 
and $\boldsymbol{\alpha}_i, \boldsymbol{\beta}_i, \boldsymbol{\gamma}_i$ which are the null space of $\mathbf{M}^{\Pi_i}_{k}$ for feature $i$. 
Therefore, we have:
\begin{align}
\scalemath{1}{
\mathbf{M}^{\Pi_i}_{k}
\begin{bmatrix}
\alpha_i & \beta_i & \gamma_i
\end{bmatrix}
}
&=
\mathbf{0}_{3\times 1} \\
\scalemath{1}{
\mathbf{M}^{\Pi_j}_{k}
\begin{bmatrix}
\alpha_j & \beta_j & \gamma_j
\end{bmatrix}
}
&=
\mathbf{0}_{3\times 1}
\end{align}
If we can show that $\{\alpha_j, \beta_j, \gamma_j\}$ can be linearly represented by $\{\alpha_i, \beta_i, \gamma_i\}$ as \eqref{eq_linear_transformation},   
both $\alpha_i, \beta_i, \gamma_i$ and $\alpha_j, \beta_j, \gamma_j$ share the same bases, 
and thus, they are the null space of both $\mathbf{M}^{\Pi_j}_{k}$ and $\mathbf{M}^{\Pi_i}_{k}$, that is:
\begin{equation}\label{eq_null_plane}
\scalemath{1}{
\begin{bmatrix}
\mathbf{M}^{\Pi_i}_{k} \\
\mathbf{M}^{\Pi_j}_{k}
\end{bmatrix}
\begin{bmatrix}
\alpha_i & \beta_i & \gamma_i
\end{bmatrix}
=
\begin{bmatrix}
\mathbf{M}^{\Pi_i}_{k} \\
\mathbf{M}^{\Pi_j}_{k}
\end{bmatrix}
\begin{bmatrix}
\alpha_j & \beta_j & \gamma_j
\end{bmatrix}
=
\mathbf{0}
}
\end{equation}
Based on the definition, $\mathbf{N}_{\Pi2:4}$ is the null space for feature $i$, and thus, the null space $\mathbf{N}^{(j)}_{\Pi2:4}$ for feature $j$ can be written as:
\begin{equation}
\scalemath{0.9}{
\mathbf{N}^{(j)}_{\Pi2:4}=
\begin{bmatrix}
\alpha_j & \beta_j & \gamma_j
\end{bmatrix}
=
\begin{bmatrix}
\alpha_i & \beta_i & \gamma_i
\end{bmatrix}
{^{\Pi_i}_{\Pi_j}\mathbf{R}}
=
\mathbf{N}_{\Pi2:4}
{^{\Pi_i}_{\Pi_j}\mathbf{R}}
}
\end{equation}
Since ${^{\Pi_i}_{\Pi_j}\mathbf{R}}$ is a rotation matrix, it is invertible. 
Based on~\eqref{eq_linear_transformation} and \eqref{eq_null_plane}, 
both $\mathbf{N}_{\Pi2:4}$ and $\mathbf{N}^{(j)}_{\Pi2:4}$ are the null space of $\mathbf{M}^{\Pi_j}_{k}$ and $\mathbf{M}^{L_i}_{k}$. 
As no assumption is made for the choice of $i$ and $j$, $\mathbf{N}_{\Pi2:4}$ is also in the null space of $\mathbf{H}^{(\pi)}_{I_k}\boldsymbol{\Phi}_{(k,1)}$.

%% file: appendix/plane_feature.tex
\section{Observability for Plane Features of Hesse Form} \label{apd_hesse_plane}

Recall that a plane feature in Hesse form (Mode 2) is given by: 
%
%
$\boldsymbol{\Pi}
=
\begin{bmatrix}
\mathbf{n}^{\top}_{\pi} & d

\end{bmatrix}^{\top}$,
%
and we have defined: $\mathbf{n}^{\top}_{\pi}\mathbf{P}_{\mathbf{f}}-d = 0$.
To find a minimal representation for the error state,
we use the horizontal angle $\theta$ and the elevation angle $\phi$ represent the normal direction $\mathbf{n}_{\pi}$, i.e., 
\begin{equation}
\scalemath{0.9}{
	\mathbf{n}_{\pi}
	=
	\begin{bmatrix}
	n_1 \\ n_2 \\ n_3
	\end{bmatrix}
	=
	\begin{bmatrix}
	\cos \theta \cos \phi \\
	\sin \theta \cos \phi \\
	\sin \phi
	\end{bmatrix}
}
\end{equation}
Then, the error state of the plane feature is denoted by:
$\scalemath{0.85}{
\tilde{\boldsymbol{\Pi}}
=
\begin{bmatrix}
\tilde{\theta} &
\tilde{\phi} &
\tilde{d}
\end{bmatrix}^{\top}}$.
%
%
We assume a direct plane measurement (e.g., provided by plane detection from point clouds):
\begin{equation}
\scalemath{0.9}{
	\mathbf{z}
	=
	\begin{bmatrix}
	^{I}\theta &
	^{I}\phi &
	^{I}d
	\end{bmatrix}^{\top}
}
\end{equation}\noindent
The measurement Jacobian w.r.t. $^I\mathbf{n}_{\pi}$ and $^Id$ is computed by:
\begin{equation}
\scalemath{1}{
\mathbf{H}_{\Pi}
=
\begin{bmatrix}
-\frac{\hat{n}_2}{\hat{n}^2_1 + \hat{n}^2_2} &  \frac{\hat{n}_1}{\hat{n}^2_1 + \hat{n}^2_2}  &  0   &     0  \\
-\frac{\hat{n}_1\hat{n}_3}{\sqrt{\hat{n}^2_1+\hat{n}^2_2}}  & -\frac{\hat{n}_2\hat{n}_3}{\sqrt{\hat{n}^2_1+\hat{n}^2_2}}  & \sqrt{\hat{n}^2_1+\hat{n}^2_2}   & 0 \\
0 & 0 & 0 & 1
\end{bmatrix}
}
\end{equation}
%
Then, the measurement Jacobian w.r.t. the state and the block row of the observability matrix can be obtained as follows:
\begin{align}
&
\mathbf{H}_{I_k} := \frac{\partial \tilde{\mathbf{z}}}{\partial \tilde{\mathbf{x}}}	=  
\scalemath{0.8}{
	\mathbf{H}_{\Pi}
	\begin{bmatrix}
	\lfloor {^I_G\hat{\mathbf{R}}}{^G\hat{\mathbf{n}}_{\pi}} \times \rfloor & 
	\mathbf{0}_{3\times 9} & 
	\mathbf{0}_3 & 
	{^I_G\hat{\mathbf{R}}}{^G\hat{\mathbf{n}}^{\perp}_1} \cos \hat{\phi}& 
	{^I_G\hat{\mathbf{R}}}{^G\hat{\mathbf{n}}^{\perp}_2} & 
	\mathbf{0}_{3\times 1}
	\\
	\mathbf{0}_{1\times 9} & 
	\mathbf{0}_{1\times 3} & 
	-{^G\hat{\mathbf{n}}^{\top}_{\pi}} & 
	-{^G\hat{\mathbf{P}}^{\top}_I}{^G\hat{\mathbf{n}}^{\perp}_1} \cos \hat{\phi} &
	-{^G\hat{\mathbf{P}}^{\top}_I}{^G\hat{\mathbf{n}}^{\perp}_2}  &
	1
	\end{bmatrix}
} 
\\
&
\mathbf{H}_{I_k}\boldsymbol{\Phi}_{(k,1)}	=  
\scalemath{0.8}{
		\mathbf{H}_{\Pi}
	\begin{bmatrix}
	\boldsymbol{\Gamma}_{\Pi 1}  & 
	\boldsymbol{\Gamma}_{\Pi 2}  & 
	\begin{bmatrix}
	\mathbf{0}_3 \\
	-{^G\hat{\mathbf{n}}^{\top}_{\pi}}\delta t_k
	\end{bmatrix} & 
	\boldsymbol{\Gamma}_{\Pi 3}  & 
	\begin{bmatrix}
	\mathbf{0}_3 \\
	-{^G\hat{\mathbf{n}}^{\top}_{\pi}}
	\end{bmatrix} &
	\boldsymbol{\Gamma}_{\Pi 4} 
	\end{bmatrix}
}
\end{align}
where
\begin{align}
\boldsymbol{\Gamma}_{\Pi 1}
&=
\scalemath{0.85}{
\begin{bmatrix}
{^{I_k}_G\hat{\mathbf{R}}}\lfloor{^G\hat{\mathbf{n}}_{\pi}}\times \rfloor \\
-{^G\mathbf{n}^{\top}_{\pi}}
\lfloor \left({^G\hat{\mathbf{P}}_{I_1}} + {^G\hat{\mathbf{V}}_{I_1}}\delta t_k 
- \frac{1}{2}{^G\mathbf{g}}\delta t^2_k - {^G\hat{\mathbf{P}}_{I_k}}\right)
\times \rfloor
\end{bmatrix}{^G_{I_1}\hat{\mathbf{R}}}
}
\\
\boldsymbol{\Gamma}_{\Pi 2}
&=
\scalemath{0.9}{
\begin{bmatrix}
{^{I_k}_G\hat{\mathbf{R}}}\lfloor{^G\hat{\mathbf{n}}_{\pi}}\times \rfloor 
{^{I_k}_G\hat{\mathbf{R}}}^{\top}\boldsymbol{\Phi}_{12}\\
-{^G\hat{\mathbf{n}}^{\top}_{\pi}}
\boldsymbol{\Phi}_{52}
\end{bmatrix}
}
\\
\boldsymbol{\Gamma}_{\Pi 3}
&=
\scalemath{0.9}{
\begin{bmatrix}
\mathbf{0}_3 \\
-{^G\hat{\mathbf{n}}^{\top}_{\pi}}
\boldsymbol{\Phi}_{54}
\end{bmatrix}
}
\\
\boldsymbol{\Gamma}_{\Pi 4}
&=
\scalemath{0.9}{
\begin{bmatrix}
{^{I_k}_G\hat{\mathbf{R}}}{^G\hat{\mathbf{n}}^{\perp}_1}\cos \hat{\phi} & 
{^{I_k}_G\hat{\mathbf{R}}}{^G\hat{\mathbf{n}}^{\perp}_2} & 
\mathbf{0}_{3\times 1} \\
-{^G\hat{\mathbf{P}}^{\top}_{I_k}}{^G\hat{\mathbf{n}}^{\perp}_1} \cos \hat{\phi}  &
-{^G\hat{\mathbf{P}}^{\top}_{I_k}}{^G\hat{\mathbf{n}}^{\perp}_2}  &
1
\end{bmatrix}
}
\\
{^G\mathbf{n}^{\perp}_1}
&= \begin{bmatrix}
-\sin \hat{\theta} &  \cos\hat{ \theta} &  0
\end{bmatrix}^{\top}
\\
{^G\mathbf{n}^{\perp}_2}
&= \begin{bmatrix}
-\cos \hat{\theta} \sin \hat{\phi} &  -\sin \hat{\theta} \sin \hat{\phi} &  \cos \hat{\phi}
\end{bmatrix}^{\top}
\end{align}
It is not difficult to see that the aided INS with a single plane feature will have at least 7 unobservable directions:
\begin{align}
\scalemath{1}{
	\mathbf{N}_{\pi}}
&=
\scalemath{1}{
	\begin{bmatrix}
	\mathbf{N}_g  & \mathbf{0}_{12\times 3}  & \mathbf{N}_{123}  &   \\
	-\lfloor{^G\hat{\mathbf{P}}_{I_1}}\times \rfloor {^G\mathbf{g}}   & {^G_{\Pi}\hat{\mathbf{R}}}   & \mathbf{0}_{3}\\
	-g \mathbf{e}_1 & \mathbf{e}_3 \mathbf{e}^{\top}_3 & \mathbf{0}_{3}  \\	
	\end{bmatrix}}
	=:
	\begin{bmatrix}
	\mathbf{N}_{\pi 1} & \mathbf{N}_{\pi 2:4} & \mathbf{N}_{\pi 5:7}
	\end{bmatrix}
\end{align}
Note that $\mathbf{N}_{\pi 1}$ relates to the rotation around the gravitational direction, $\mathbf{N}_{\pi 2:4}$ relates to the sensor's global translation, $\mathbf{N}_{\pi 5:6}$ relate to the sensor motion perpendicular to the plane's normal direction, while $\mathbf{N}_{\pi 7}$ relates to the rotation around the normal direction of the plane.

Proceeding similarly, 
we can reach the same conclusion as in Lemma~\ref{lem_multiple_plane} in the case of  $s>1$ plane features in the state vector -- that is, 
(i) 5 unobservable directions ($\mathbf{N}_{\Pi 1:5}$) for $s=2$ unparallel planes, 
and (ii) 4 unobservable directions ($\mathbf{N}_{\Pi 1:4}$) for $s>2$ planes with unparallel intersections.  
	\begin{align}
		\mathbf{N}_{\Pi}
		&=
		\begin{bmatrix}
		\mathbf{N}_g & \mathbf{0}_{12\times 3} & \mathbf{N}_{i\times j}\\
		-\lfloor{^G\hat{\mathbf{P}}_{I_1}}\times \rfloor {^G\mathbf{g}}   & {^G_{\Pi i}\hat{\mathbf{R}}} & \mathbf{0}_{3\times 1}\\
		-g\mathbf{e}_1  & \mathbf{e}_{3}\mathbf{e}^{\top}_3{^{\Pi 1}_{\Pi i}\hat{\mathbf{R}}} & \mathbf{0}_{3\times 1}\\
		\vdots &  \vdots & \vdots \\
		-g\mathbf{e}_1  & \mathbf{e}_{3}\mathbf{e}^{\top}_3{^{\Pi s}_{\Pi i}\hat{\mathbf{R}}} & \mathbf{0}_{3\times 1}\\
		\end{bmatrix} 
		=: 
	\begin{bmatrix}
		\mathbf{N}_{\Pi 1}  &  \mathbf{N}_{\Pi 2:4} & \mathbf{N}_{\Pi 5}
		\end{bmatrix}	
	\end{align}
for $i,j\in \{1\ldots s \}$.	

%% file: appendix/lem1.tex
\section{Proof of Lemma \ref{lem_line_plane}}\label{apd_proof_lem1}

First of all, it is straightforward to verify the null space $\mathbf{N}_{l\pi 1}$ that relates to the rotation around the gravity direction. 
If the line is parallel to the plane, we have the line direction vector $^G\mathbf{v}_{\mathbf{e}}$ is perpendicular to the plane normal direction $^G\mathbf{n}_{\Pi}$, 
i.e., $^G\mathbf{v}^{\top}_{\mathbf{e}}{^G\mathbf{n}_{\Pi}}=0$, 
then we have:
\begin{equation}
	\mathbf{H}^{(l\pi)}_{I_k}\boldsymbol{\Phi}_{(k,1)}\mathbf{N}_{l\pi 5}
	=
	\scalemath{0.85}{
	\begin{bmatrix}
	\mathbf{H}_{l,k} 
	\mathbf{K}
	{^{I_k}_G\hat{\mathbf{R}}} & \mathbf{0}  \\
	\mathbf{0} & 
	{^{I_k}_G\hat{\mathbf{R}}} 
	\end{bmatrix}
	\begin{bmatrix}
	\lfloor {^G\hat{\mathbf{v}}_{L}}\rfloor{^G\hat{\mathbf{v}}_{\mathbf{e}}} \\
	-{^G\hat{\mathbf{n}}_{\pi}}{^G\hat{\mathbf{n}}^{\top}_{\pi}}{^G\hat{\mathbf{v}}^{\top}_{\mathbf{e}}}
	\end{bmatrix}
}
	\delta t_k
	=
	\mathbf{0}
\end{equation}
Clearly, if the line is parallel to the plane, the system will have one more unobservable direction $\mathbf{N}_{l\pi 5}$. 
To verify  $\mathbf{N}_{l\pi 2:4}$ is in the null space, for simplicity, we write~\eqref{eq_m_line_plane} as:
\begin{equation}
	\mathbf{H}^{(l\pi)}_{I_k}\boldsymbol{\Phi}_{(k,1)}=
	\begin{bmatrix}
	\mathbf{H}^{(l)}_{I_k}\boldsymbol{\Phi}_{(k,1)} \\
	\mathbf{H}^{(\pi)}_{I_k}\boldsymbol{\Phi}_{(k,1)} 
	\end{bmatrix}
\end{equation}
where $\mathbf{H}^{l}_{I_k}$ and $\mathbf{H}^{\pi}_{I_k}$ are the Jacobians w.r.t. the IMU state and the line and plane features. 
It is not hard to verify that:
\begin{align}
	\label{eq_pi_lem1}
	&\mathbf{H}^{(l)}_{I_k}\boldsymbol{\Phi}_{(k,1)}\mathbf{N}_{l\pi 2:4} {^G_L\hat{\mathbf{R}}} = \mathbf{0} \\
	\label{eq_pi_lem2}
	&\mathbf{H}^{(\pi)}_{I_k}\boldsymbol{\Phi}_{(k,1)}\mathbf{N}_{l\pi 2:4}{^{G}_{\Pi}\hat{\mathbf{R}}} =\mathbf{0}
\end{align}
Since ${{^{G}_{\Pi}\hat{\mathbf{R}}}}$ and ${{^G_{L}\hat{\mathbf{R}}}}$ are rotation matrices (of full rank), 
by multiplying ${{^G_{L}\hat{\mathbf{R}}^{\top}}}$ and ${{^{G}_{\Pi}\hat{\mathbf{R}}^{\top}}}$ from the left-hand side of \eqref{eq_pi_lem1} and \eqref{eq_pi_lem2}, respectively, 
we have:
\begin{align}
&\mathbf{H}^{(l)}_{I_k}\boldsymbol{\Phi}_{(k,1)}\mathbf{N}_{l\pi 2:4}=\mathbf{0} \\
&\mathbf{H}^{(\pi)}_{I_k}\boldsymbol{\Phi}_{(k,1)}\mathbf{N}_{l\pi 2:4}=\mathbf{0}
\end{align}
which shows: $\mathbf{H}^{(l\pi)}_{I_k}\boldsymbol{\Phi}_{(k,1)}\mathbf{N}_{l\pi 2:4}=\mathbf{0}$.

%% file: appendix/lem2.tex
\section{Proof of Lemma \ref{lem_pt_line_plane}} \label{apd_proof_lem2}

First of all, it is not difficult to verify the null space $\mathbf{N}_{pl\pi 1}$ corresponding to the rotation around the gravity direction. 
What we need to verify is that $\mathbf{N}_{pl\pi 2:4}$ is in the unobservable subspace. 
To this end, for simplicity, we write $k$-th block row of the observability matrix as:
\begin{equation}
\mathbf{H}^{(pl\pi)}_{I_k}\boldsymbol{\Phi}_{(k,1)}=
\begin{bmatrix}
\mathbf{H}^{(p)}_{I_k}\boldsymbol{\Phi}_{(k,1)} \\
\mathbf{H}^{(l)}_{I_k}\boldsymbol{\Phi}_{(k,1)} \\ 
\mathbf{H}^{(\pi)}_{I_k}\boldsymbol{\Phi}_{(k,1)}
\end{bmatrix}
\end{equation}
where $\mathbf{H}^{(p)}_{I_k}$,  $\mathbf{H}^{(l)}_{I_k}$ and $\mathbf{H}^{(\pi)}_{I_k}$ are the Jacobians w.r.t. the IMU state  and the point, line and plane features, respectively. 
We can easily verify the following:
\begin{align}
&\mathbf{H}^{(p)}_{I_k}\boldsymbol{\Phi}_{(k,1)}\mathbf{N}_{pl\pi 2:4} = \mathbf{0} \\
\label{eq_pt_line_plane_lem1}
&\mathbf{H}^{(l)}_{I_k}\boldsymbol{\Phi}_{(k,1)}\mathbf{N}_{pl\pi 2:4}{^G_{L}\hat{\mathbf{R}}} = \mathbf{0} \\
\label{eq_pt_line_plane_lem2}
&\mathbf{H}^{(\pi)}_{I_k}\boldsymbol{\Phi}_{(k,1)}\mathbf{N}_{pl\pi 2:4} {^G_{\Pi}\hat{\mathbf{R}}}=\mathbf{0}
\end{align}
Similarly, since ${{^{G}_{\Pi}\hat{\mathbf{R}}}}$ and ${^G_{L}\hat{\mathbf{R}}}$ are rotation matrices of full rank, by multiplying both sides of \eqref{eq_pt_line_plane_lem1} and \eqref{eq_pt_line_plane_lem2} from left with ${^{G}_{\Pi}\hat{\mathbf{R}}^{\top}}$ and ${^G_{L}\hat{\mathbf{R}}}^{\top}$, respectively, we have:
\begin{align}
&\mathbf{H}^{(l)}_{I_k}\boldsymbol{\Phi}_{(k,1)}\mathbf{N}_{pl\pi 2:4}=\mathbf{0}\\
&\mathbf{H}^{(\pi)}_{I_k}\boldsymbol{\Phi}_{(k,1)}\mathbf{N}_{pl\pi 2:4}=\mathbf{0}
\end{align}
Thus, we reach: $\mathbf{H}^{(pl\pi)}_{I_k}\boldsymbol{\Phi}_{(k,1)}\mathbf{N}_{pl\pi 2:4}=\mathbf{0}$.

%% file: sections/lie_derivative.tex
\section{Unobservable Directions for Point Features}\label{apd_proof}

\input{appendix/nonlinear_observability}

\input{appendix/propagation}
\input{appendix/observability_point_feature}
\input{appendix/appendix}

%% file: appendix/nonlinear_observability.tex
\subsection{Nonlinear Observability Analysis}
we first provide an overview of the nonlinear observability rank condition test \cite{Hermann1977TAC} and summarize the method in \cite{Huang2010IJRR}\cite{Guo2013ICRA}\cite{Panahandeh2013ros}\cite{Hesch2014IJRR} for finding the unobservable modes of nonlinear system. 
Consider a nonlinear system:
\begin{equation} \label{eq_general_system}
	\left \{
		\begin{array}{l}
		\dot{\mathbf{x}} = \mathbf{f}_{0}(\mathbf{x}) + \sum_{i=1}^{\ell}\mathbf{f}_{i}(\mathbf{x})u_{i} \\
		\mathbf{z} = \mathbf{h}(\mathbf{x})
		\end{array} 
	\right.
\end{equation}
where $\mathbf{x}\in \mathbb{R}^m$ is the state vector, $\mathbf{u} = [ u_1 \quad \cdots \quad u_{\ell} ]\in \mathbb{R}^{\ell}$ is the system input, $\mathbf{z}\in \mathbb{R}^k$ is the system output, and $\mathbf{f}_i$ for $i\in \{0,\ldots,\ell\}$ is the process function. 

The zeroth order Lie derivative of a measurement function $\mathbf{h}$ is the function itself, i.e., $\mathcal{L}^0\mathbf{h} = \mathbf{h}(\mathbf{x})$. For any $n$-th order Lie derivative, $\mathcal{L}^n\mathbf{h}$, the $n+1$-th order Lie derivative $\mathcal{L}^{n+1}_{\mathbf{f}_i}\mathbf{h}$ with respect to a process function $\mathbf{f}_i$ can be computed as:
\begin{equation}
	\mathcal{L}^{n+1}_{\mathbf{f}_i}\mathbf{h}=\nabla\mathcal{L}^n\mathbf{h}\cdot\mathbf{f}_i
\end{equation}
where $\nabla$ denotes the gradient operator with respect to $\mathbf{x}$ and $"\cdot"$ represents the vector inner product. Similarly, mixed higher order Lie derivatives can be defined as:
\begin{equation}
	\mathcal{L}^n_{\mathbf{f}_i\mathbf{f}_{j\ldots}\mathbf{f}_k}\mathbf{h} = \mathcal{L}_{\mathbf{f}_i}(\mathcal{L}^{n-1}_{\mathbf{f}_{j\ldots}\mathbf{f}_k}\mathbf{h}) =
	\nabla \mathcal{L}^{n-1}_{\mathbf{f}_{j\ldots}\mathbf{f}_k} \mathbf{h}\cdot\mathbf{f}_i 
\end{equation}
where $i,j,k\in \{0,\ldots,\ell\}$.

The observability of a nonlinear system is determined by calculating the dimension of the space spanned by the gradients of Lie derivative of its output functions\cite{Hermann1977TAC}. Hence, the observability matrix $\mathbf{O}$ of system \eqref{eq_general_system} is defined as:
\begin{equation}
	\mathbf{O} \triangleq 
	\left[
		\begin{array}{c}
			\nabla\mathcal{L}^0\mathbf{h}         \\
			\nabla\mathcal{L}^1_{\mathbf{f}_i}\mathbf{h}   \\
			\vdots										\\
			\nabla\mathcal{L}^n_{\mathbf{f}_i\mathbf{f}_{j\ldots}\mathbf{f}_k}\mathbf{h} \\
			\vdots	
		\end{array}
	\right]
\end{equation}
To prove that a system is observable, it suffices to show that $\mathbf{O}$ is of full column rank. However, to prove that a system is unobservable, we have to find the null space of matrix $\mathbf{O}$, which may have infinitely many rows. This can be very challenging especially for high-dimensional systems, such as aided INS. To address this issue, we adopt the method proposed by \cite{Guo2013ICRA} for analyzing observability of nonlinear systems in the form of Eq. \eqref{eq_general_system}.

\begin{thm}\label{thm_observability}
	Assume that there exists a nonlinear transformation $\boldsymbol{\beta}(\mathbf{x}) = [\boldsymbol{\beta}_1({\mathbf{x}})^{\top}\ldots\boldsymbol{\beta}_n(\mathbf{x})^{\top}]^{\top}$(i.e., a set of basis functions) of the variable $\mathbf{x}$, such that:
	 \begin{enumerate}
	 	\item The system measurement equation can be written as a function of $\boldsymbol{\beta}$, i.e., $\mathbf{z}=\mathbf{h}(\mathbf{x})=\overline{\mathbf{h}}(\boldsymbol{\beta})$
	 	\item $\frac{\partial \boldsymbol{\beta}}{\partial \mathbf{x}}\mathbf{f}_j$, for $j=\{0,\ldots,\ell\}$, is a function of $\boldsymbol{\beta}$
	 \end{enumerate}
	 Then the observability matrix of system \eqref{eq_general_system} can be factorized as: $\mathbf{O}=\Xi\Omega$ where $\Xi$ is the observability matrix of the system:
	 \begin{equation} \label{eq_transformed_system}
		 \left \{
			 \begin{array}{l}
				 \dot{\boldsymbol{\beta}} = \mathbf{g}_{0}(\boldsymbol{\beta}) + \sum_{i=1}^{\ell}\mathbf{g}_{i}(\boldsymbol{\beta})u_{i} \\
				 \mathbf{z} = \overline{\mathbf{h}}(\boldsymbol{\beta})
			 \end{array} 
		 \right.
	 \end{equation}
	 and $\Omega$ can be represented as:
	\begin{equation}
		\Omega = \frac{\partial \boldsymbol{\beta}}{\partial \mathbf{x}}
	\end{equation}
\end{thm}
\begin{proof}
	See \cite{Guo2013ICRA}.
\end{proof}
Note that system \eqref{eq_transformed_system} results by pre-multiplying the process function to system \eqref{eq_general_system} with $\frac{\partial \boldsymbol{\beta}}{\partial \mathbf{x}}$:
\begin{displaymath}
	\left \{
		\begin{array}{l}
			\frac{\partial \boldsymbol{\beta}}{\partial \mathbf{x}}\frac{\partial \mathbf{x}}{\partial t} = 
			\frac{\partial \boldsymbol{\beta}}{\partial \mathbf{x}}\mathbf{f}_0(\mathbf{x})+ 
			\frac{\partial \boldsymbol{\beta}}{\partial \mathbf{x}}\sum_{i=1}^{\ell}\mathbf{f}_i(\mathbf{x})u_i \\
			\mathbf{z} = \mathbf{h}(\mathbf{x})
		\end{array}  
	\right.
	\Rightarrow
	\left \{
		\begin{array}{l}
			\dot{\boldsymbol{\beta}} = \mathbf{g}_{0}(\boldsymbol{\beta}) + \sum_{i=1}^{\ell}\mathbf{g}_{i}(\boldsymbol{\beta})u_{i} \\
			\mathbf{z} = \overline{\mathbf{h}}(\boldsymbol{\beta})
		\end{array} 
	\right.
\end{displaymath}
where $\mathbf{g}_i(\boldsymbol{\beta})\triangleq\frac{\partial \boldsymbol{\beta}}{\partial \mathbf{x}}$ and $\overline{\mathbf{h}}(\boldsymbol{\beta})\triangleq \mathbf{h}(\mathbf{x})$.
\begin{cor}\label{cor_observability}
	If $\Xi$ is of full column rank, i.e., system \eqref{eq_transformed_system} is observable, then the unobservable directions of system \eqref{eq_transformed_system} will be spanned by the null vectors of $\Omega$.
\end{cor}
\begin{proof}
	From $\mathbf{O}=\Xi\Omega$, we have $null(\mathbf{O})=null(\Omega)\cup(null(\Xi)\cap range(\Omega))$. Therefore, if $\Xi$ is of full column rank, i.e., system \eqref{eq_transformed_system} is observable, then $null(\mathbf{O})=null(\Omega)$.
\end{proof}
Base on Theorem \ref{thm_observability} and Corollary \ref{cor_observability}, to find the unobservable directions of a system, we first need to define the basis functions, $\boldsymbol{\beta}$, which fulfill the first and second conditions of Theorem \ref{thm_observability}. Then, we should prove that the infinite-dimensional matrix $\Xi$ has full column rank, which satisfies the conditions of Corollary \ref{cor_observability}. 


%% file: appendix/propagation.tex
\subsection{System Propagation Model}
For the aided INS, the IMU measurements are used for state propagation while the measurements from exteroceptive sensor are used for state update.
%
%
The INS state $\mathbf{x}_I$ can be defined as:
\begin{equation}
	\mathbf{x}_{I} = [{^I_{G}\mathbf{s}^{\top}} \quad {\mathbf{b}_g^{\top}} \quad {^G\mathbf{v}_I^{\top}} \quad {\mathbf{b}_a^{\top}} \quad {^G\mathbf{p}_I^{\top}}]^{\top}
\end{equation}
where $^I_G\mathbf{s}$ is the Cayley-Gibbs-Rodriguez parameterization \cite{Shuster1993JAS} representing the orientation of the global frame $\{G\}$ in the IMU frame of reference $\{I\}$. The time-continuous system evolution model:
\begin{eqnarray}
	^I_G\dot{\mathbf{s}}(t)  & =  & \mathbf{D}(^I\boldsymbol{\omega}(t)-\mathbf{b}_g(t))   \nonumber \\
	\dot{\mathbf{b}}_g(t)    & =  & \mathbf{n}_g													  \nonumber \\
	^G\dot{\mathbf{V}}_I(t)  & =  & ^G\mathbf{a}(t)={^G\mathbf{g}}+\mathbf{R}(^I_G\mathbf{s}(t))^{\top}(^I\mathbf{a}(t)-\mathbf{b}_a(t)) \\
	\dot{\mathbf{b}}_a(t)    & =  & \mathbf{n}_a													  \nonumber \\
	^G\dot{\mathbf{P}}_I(t)  & =  & ^G\mathbf{V}_I(t)													  \nonumber
\end{eqnarray} 
where $\mathbf{D}\triangleq \frac{\partial \mathbf{s}}{\partial \boldsymbol{\theta}}=\frac{1}{2}(\mathbf{I}+\lfloor\mathbf{s}\times\rfloor+{\mathbf{s}\mathbf{s}^{\top}})$, $\boldsymbol{\theta} = \alpha \hat{\mathbf{k}}$ represents a rotation by an angle $\alpha$ around the axis $\hat{\mathbf{k}}$, $^I\boldsymbol{\omega}(t) = [\omega_1 \quad \omega_2 \quad \omega_3]^{\top}$ and $^I\mathbf{a}(t)=[a_1 \quad a_2 \quad a_3]^{\top}$ are the rotational velocity and linear acceleration respectively, measured by the IMU and represented in $\{I\}$. $^G\mathbf{g}$ is the gravitational acceleration, $\mathbf{R}(\mathbf{s})$ is the rotation matrix corresponding to $\mathbf{s}$, and $\mathbf{n}_g$ and $\mathbf{n}_a$ are the gyroscope and accelerometer biases driving white Gaussian noises. 

%% file: appendix/observability_point_feature.tex
\subsection{Observability Analysis for Point Feature}

Based on the above analysis, the key is to prove $\Xi$ is of full rank and then to find the unobservable direction from the $\Omega$. $\Omega$ is determined by the basis functions $\boldsymbol{\beta}$. That means if we can find the same basis functions set $\boldsymbol{\beta}$ for  aided INS, we can prove that these systems have the same unobservable directions. Therefore, the only job left unfinished is to check the rank of different $\Xi$s for these systems.  

\subsubsection{Basis Functions For Point Measurement}
With the generalized point measurement model and state propagation model, we can define the state vector as:
\begin{equation}
\mathbf{x} = [{^I_{G}\mathbf{s}^{\top}} \quad {\mathbf{b}_g^{\top}} \quad {^G\mathbf{V}_I^{\top}} \quad {\mathbf{b}_a^{\top}} \quad {^G\mathbf{P}_I^{\top}} \quad {^G\mathbf{P}_{\mathbf{f}}^{\top}}]^{\top}
\end{equation}
For simplicity, we retain only a few of the subscripts and superscripts in the state elements and denote the system state vector as:
\begin{equation}
\mathbf{x} = [{\mathbf{s}^{\top}} \quad {\mathbf{b}_g^{\top}} \quad {\mathbf{V}^{\top}} \quad {\mathbf{b}_a^{\top}} \quad {\mathbf{P}^{\top}} \quad {\mathbf{P}_{\mathbf{f}}^{\top}}]^{\top}
\end{equation}
Then the system state equation can be rewritten as:
\begin{equation}
\left[
\begin{array}{c}
\dot{\mathbf{s}}        \\
\dot{\mathbf{b}}_g		\\
\dot{\mathbf{V}}		\\
\dot{\mathbf{b}}_a		\\
\dot{\mathbf{P}}		\\
\dot{\mathbf{P}}_{\mathbf{f}}			
\end{array}
\right]
=
\underbrace{	
	\left[
	\begin{array}{c}
	-\mathbf{D}\mathbf{b}_g   \\
	\mathbf{0}				  \\
	\mathbf{g}-\mathbf{R}^{\top}\mathbf{b}_a	\\
	\mathbf{0}				\\
	\mathbf{v}				\\
	\mathbf{0}
	\end{array}
	\right]}
_{\mathbf{f}_0}
+
\underbrace{	
	\left[
	\begin{array}{c}
	\mathbf{D}		\\
	\mathbf{0}		\\
	\mathbf{0}		\\
	\mathbf{0}		\\
	\mathbf{0}		\\
	\mathbf{0}
	\end{array}
	\right]}_{\mathbf{F}_1}
\boldsymbol{\omega}
+
\underbrace{	
	\left[
	\begin{array}{c}
	\mathbf{0}		\\
	\mathbf{0}		\\
	\mathbf{R}^{\top} \\
	\mathbf{0}		\\
	\mathbf{0}		\\
	\mathbf{0}
	\end{array}
	\right]}_{\mathbf{F}_2}
\mathbf{a}
\end{equation}
where $\mathbf{R}\triangleq\mathbf{\mathbf{R}(\mathbf{s})}$. Note that $\mathbf{f}_0$ is a $18\times1$ vector, while $\mathbf{F}_1$ and $\mathbf{F_2}$ are both $24\times3$ matrices which is a compact form for representing process functions as:
\begin{eqnarray}
\mathbf{F}_1\boldsymbol{\omega} & =  &  \mathbf{f}_{11}\omega_1 + \mathbf{f}_{12}\omega_2 + \mathbf{f}_{13}\omega_3  \\ 
\mathbf{F}_2\mathbf{a} 			& =  &  \mathbf{f}_{21}a_1 + \mathbf{f}_{22}a_2 + \mathbf{f}_{23}a_3  
\end{eqnarray}

Since all the terms in the preceding projections are defined based on the existing basis functions, we have found a complete basis set (see \ref{apd_point_basis}):
\begin{equation}
	\boldsymbol{\beta} =
	\left[
		\begin{array}{c}
			\boldsymbol{\beta}_1	\\
			\boldsymbol{\beta}_2	\\
			\boldsymbol{\beta}_3	\\
			\boldsymbol{\beta}_4	\\
			\boldsymbol{\beta}_5	
		\end{array}
	\right]
	=
	\left[
		\begin{array}{c}
			\mathbf{R}(\mathbf{p}_{\mathbf{f}}-\mathbf{p})	\\
			\mathbf{b}_g						\\
			\mathbf{R}\mathbf{v}				\\
			\mathbf{b}_a						\\
			\mathbf{R}\mathbf{g}
		\end{array}
	\right]
\end{equation}

Therefore, the new system with $\boldsymbol{\beta}$ basis:
\begin{equation}
			\left[
				\begin{array}{c}
					\dot{\boldsymbol{\beta}_1}	\\
					\dot{\boldsymbol{\beta}_2}	\\
					\dot{\boldsymbol{\beta}_3}	\\
					\dot{\boldsymbol{\beta}_4}	\\
					\dot{\boldsymbol{\beta}_5}	\\
				\end{array}
			\right]
		=
\underbrace{\left[
				\begin{array}{c}
					-\lfloor\boldsymbol{\beta}_1\times\rfloor\boldsymbol{\beta}_2-\boldsymbol{\beta}_3	\\
					\mathbf{0}		\\
					-\lfloor\boldsymbol{\beta}_3\times\rfloor\boldsymbol{\beta}_2+\boldsymbol{\beta}_5-\boldsymbol{\beta}_4	\\
					\mathbf{0}			\\
					-\lfloor\boldsymbol{\beta}_5\times\rfloor\boldsymbol{\beta}_2		\\
				\end{array}
			\right]}_{\mathbf{g}_{0}}
		+
\underbrace{\left[
				\begin{array}{c}
					\lfloor\boldsymbol{\beta}_1\times\rfloor	\\
					\mathbf{0}			\\
					\lfloor\boldsymbol{\beta}_3\times\rfloor	\\
					\mathbf{0}			\\
					\lfloor\boldsymbol{\beta}_5\times\rfloor	\\
				\end{array}
			\right]}_{\mathbf{G}_1}\boldsymbol{\omega}
		+
\underbrace{\left[
				\begin{array}{c}
					\mathbf{0}	\\
					\mathbf{0}	\\
					\mathbf{I}_3	\\
					\mathbf{0}	\\
					\mathbf{0}	\\
				\end{array}
			\right]}_{\mathbf{G}_2}\mathbf{a}			
\end{equation}
where $\mathbf{g}_0$ is a $18\times1$ vector, while $\mathbf{G}_1$ and $\mathbf{G}_2$ are both $24\times3$ matrices which is compact form for representing process functions as:
\begin{eqnarray}
	\mathbf{G}_1\boldsymbol{\omega} & = & \mathbf{g}_{11}\omega_1 + \mathbf{g}_{12}\omega_2 + \mathbf{g}_{13}\omega_3  \\
	\mathbf{G}_2\mathbf{a} & = & \mathbf{g}_{21}a_1 + \mathbf{g}_{22}a_2 + \mathbf{g}_{23}a_3  
\end{eqnarray}
Base on the Theorem \ref{thm_observability}, the observability matrix $\mathbf{O}$ of the aided INS is the product of observability matrix $\Xi$ with the derivatives of the basis functions $\Omega$. In what follows, we will first prove that matrix $\Xi$ is of full column rank. Then, the null space of matrix $\Omega$ corresponds to the unobservable directions of the aided INS. 

From the generalized measurement model Eq. \eqref{eq_x_meas}, the $\Xi$ contains two parts:
\begin{equation}
	\Xi=
	\left[
		\begin{array}{c}
		\Xi^{(r)}		\\
		\Xi^{(b)}
		\end{array}
	\right]
\end{equation}
where $\Xi^{(r)}$ and $\Xi^{(b)}$ represents observability matrix from the range measurement and bearing measurement respectively. Therefore, in order to prove that matrix $\Xi$ is of full column rank, we will inspect the column rank of $\Xi^{(r)}$ and $\Xi^{(b)}$ respectively. In Appendix \ref{apd_rank_test_r} and \ref{apd_rank_test_b} we showed that for range measurement and fulling bearing measurement $\Xi^{(r)}$ and $\Xi^{(b)}$ will have full column rank

\subsubsection{Unobservable Direction}
According to the basis set of $\boldsymbol{\beta}$, we have:
\begin{equation}
	\Omega = \frac{\partial \boldsymbol{\beta}}{\partial \mathbf{x}}=
	\left[
		\begin{array}{cccccc}
			\lfloor\mathbf{R}(\mathbf{p}_{\mathbf{f}}-\mathbf{p})\times\rfloor\frac{\partial \boldsymbol{\theta}}{\partial \mathbf{s}} &
			\mathbf{0}  & \mathbf{0}  &  \mathbf{0} & -\mathbf{R} &  \mathbf{R}    \\
			\mathbf{0}  & \mathbf{I}_3  & \mathbf{0}   & \mathbf{0} & \mathbf{0}  & \mathbf{0}   \\
			\lfloor\mathbf{R}\mathbf{v}\times\rfloor\frac{\partial \boldsymbol{\theta}}{\partial \mathbf{s}} & 
			\mathbf{0}  & \mathbf{R} & \mathbf{0} & \mathbf{0} & \mathbf{0}    \\
			\mathbf{0}  & \mathbf{0}  & \mathbf{0}   & \mathbf{I}_3 & \mathbf{0}  & \mathbf{0}   \\
			\lfloor\mathbf{R}\mathbf{g}\times\rfloor\frac{\partial \boldsymbol{\theta}}{\partial \mathbf{s}} & 
			\mathbf{0}  & \mathbf{0} & \mathbf{0} & \mathbf{0} & \mathbf{0}    \\
		\end{array}
	\right]
\end{equation}
Assuming A is the null space of $\Omega$, and A should have the following form:
\begin{equation}
	\mathbf{A} = \left[\mathbf{A}_1^{\top}\quad \mathbf{A}_2^{\top}\quad \mathbf{A}_3^{\top}\quad \mathbf{A}_4^{\top}\quad \mathbf{A}_5^{\top}\quad \mathbf{A}_6^{\top}\right]^{\top}\neq\mathbf{0}
\end{equation}
such that:
\begin{equation}
	\Omega\mathbf{A}=\mathbf{0}
\end{equation}
Hence, the system's unobservable directions can be described as:
\begin{equation}
	\mathbf{A}
	=
	\left[
		\begin{array}{cc}
			\frac{\partial \mathbf{s}}{\partial \mathbf{\theta}}\mathbf{R}\mathbf{g}  &  \mathbf{0}    \\
			\mathbf{0}   &  \mathbf{0}				\\
			-\lfloor\mathbf{v}\times\rfloor\mathbf{g}    &   \mathbf{0}			\\
			\mathbf{0}		&  \mathbf{0}			\\
			-\lfloor\mathbf{p}\times\rfloor\mathbf{g}  &   \mathbf{I}_3		\\
			-\lfloor\mathbf{p}_{\mathbf{f}}\times\rfloor\mathbf{g} &  \mathbf{I}_3	
		\end{array}
	\right]
\end{equation}
Therefore, the unobservable directions are the global position of exteroceptive sensor and the the point landmark, and the rotation about the gravity vector. 

%% file: appendix/appendix.tex
\section{Basis Function and Rank Test for Point Measurement}

\subsection{Basis Functions for Point Measurement} \label{apd_point_basis}

According to the two conditions of Theorem \ref{thm_observability}, we define the system's first basis function according to Eq. \eqref{eq_feat}:
\begin{equation}
\boldsymbol{\beta_1}\triangleq \mathbf{R}(\mathbf{p}_{\mathbf{f}}-\mathbf{p})
\end{equation}
According to the second condition of Theorem \ref{thm_observability}, we will compute:
\begin{eqnarray}
\frac{\partial \boldsymbol{\beta}_1}{\partial \mathbf{x}}
&  =  & 
\left[\frac{\partial \boldsymbol{\beta}_1}{\partial \mathbf{s}} \quad 
\frac{\partial \boldsymbol{\beta}_1}{\partial \mathbf{b}_g} \quad 
\frac{\partial \boldsymbol{\beta}_1}{\partial \mathbf{v}} \quad 
\frac{\partial \boldsymbol{\beta}_1}{\partial \mathbf{b}_a} \quad 
\frac{\partial \boldsymbol{\beta}_1}{\partial \mathbf{p}} \quad 
\frac{\partial \boldsymbol{\beta}_1}{\partial \mathbf{p}_{\mathbf{f}}}\right]		\\
& = &
\left[\lfloor\mathbf{R}(\mathbf{p}_{\mathbf{f}}-\mathbf{p})\times\rfloor\frac{\partial \boldsymbol{\theta}}{\partial \mathbf{s}} \quad
\mathbf{0}  \quad
\mathbf{0}  \quad
\mathbf{0}  \quad
-\mathbf{R} \quad
\mathbf{R}  \right]		\\
\frac{\partial \boldsymbol{\beta}_1}{\partial \mathbf{x}}\mathbf{f}_0  & =  & -\lfloor\mathbf{R}(\mathbf{p}_{\mathbf{f}}-\mathbf{p})\times\rfloor\mathbf{b}_g-\mathbf{R}\mathbf{v} \triangleq 
-\lfloor \boldsymbol{\beta}_1\times\rfloor\boldsymbol{\beta}_2 - \boldsymbol{\beta}_3    \\
\frac{\partial \boldsymbol{\beta}_1}{\partial \mathbf{x}}\mathbf{f}_{1i}  &  =  &
\lfloor \mathbf{R}(\mathbf{p}_{\mathbf{f}}-\mathbf{p})\times \rfloor\mathbf{e}_{i}	\triangleq	\lfloor\boldsymbol{\beta}_1\times\rfloor\mathbf{e}_{i}					\\
\frac{\partial \boldsymbol{\beta}_1}{\partial \mathbf{x}}\mathbf{f}_{2i}  &  =   &  \mathbf{0}
\end{eqnarray}
where $\frac{\partial \boldsymbol{\theta}}{\partial \mathbf{s}}\mathbf{D}=\frac{\partial \boldsymbol{\theta}}{\partial \mathbf{s}}\frac{\partial \mathbf{s}}{\partial \boldsymbol{\theta}}=\mathbf{I}_3$, $i\in\{1,2,3\}$ and we have defined two new basis elements: $\boldsymbol{\beta}_2\triangleq\mathbf{b}_a$, $\boldsymbol{\beta}_3\triangleq\mathbf{R}\mathbf{v}$.

Similarly, for the span of $\boldsymbol{\beta}_2$, we have:
\begin{eqnarray}
\frac{\partial \boldsymbol{\beta}_2}{\partial \mathbf{x}}
&  =  & 
\left[\frac{\partial \boldsymbol{\beta}_2}{\partial \mathbf{s}} \quad 
\frac{\partial \boldsymbol{\beta}_2}{\partial \mathbf{b}_g} \quad 
\frac{\partial \boldsymbol{\beta}_2}{\partial \mathbf{v}} \quad 
\frac{\partial \boldsymbol{\beta}_2}{\partial \mathbf{b}_a} \quad 
\frac{\partial \boldsymbol{\beta}_2}{\partial \mathbf{p}} \quad 
\frac{\partial \boldsymbol{\beta}_2}{\partial \mathbf{p}_{\mathbf{f}}}\right]		\\
& = &
\left[\mathbf{0} \quad
\mathbf{I}_3  \quad
\mathbf{0}  \quad
\mathbf{0}  \quad
\mathbf{0} \quad
\mathbf{0}  \right]	\\
\frac{\partial \boldsymbol{\beta}_2}{\partial \mathbf{x}}\mathbf{f}_0 & = & \mathbf{0}    \\ 
\frac{\partial \boldsymbol{\beta}_2}{\partial \mathbf{x}}\mathbf{f}_{1i}& = &\mathbf{0}   \\
\frac{\partial \boldsymbol{\beta}_2}{\partial \mathbf{x}}\mathbf{f}_{2i}& = &\mathbf{0}
\end{eqnarray}
where $i\in\{1,2,3\}$.

Then, for the span of $\boldsymbol{\beta}_3$, we have:
\begin{eqnarray}
\frac{\partial \boldsymbol{\beta}_3}{\partial \mathbf{x}}
&  =  & 
\left[\frac{\partial \boldsymbol{\beta}_3}{\partial \mathbf{s}} \quad 
\frac{\partial \boldsymbol{\beta}_3}{\partial \mathbf{b}_g} \quad 
\frac{\partial \boldsymbol{\beta}_3}{\partial \mathbf{v}} \quad 
\frac{\partial \boldsymbol{\beta}_3}{\partial \mathbf{b}_a} \quad 
\frac{\partial \boldsymbol{\beta}_3}{\partial \mathbf{p}} \quad 
\frac{\partial \boldsymbol{\beta}_3}{\partial \mathbf{p}_{\mathbf{f}}}\right]		\\
& = &
\left[\lfloor\mathbf{R}\mathbf{v}\times\rfloor\frac{\partial \boldsymbol{\theta}}{\partial \mathbf{s}} \quad
\mathbf{0}  \quad
\mathbf{R}  \quad
\mathbf{0}  \quad
\mathbf{0} \quad
\mathbf{0}  \right]	\\
\frac{\partial \boldsymbol{\beta}_3}{\partial \mathbf{x}}\mathbf{f}_0 & = & 
-\lfloor\mathbf{R}\mathbf{v}\times\rfloor\mathbf{b}_g+\mathbf{R}\mathbf{g}-\mathbf{b}_a \triangleq
-\lfloor\boldsymbol{\beta}_3\times\rfloor\boldsymbol{\beta}_2 + \boldsymbol{\beta}_5 - \boldsymbol{\beta}_4   \\ 
\frac{\partial \boldsymbol{\beta}_3}{\partial \mathbf{x}}\mathbf{f}_{1i}& = &
\lfloor\mathbf{R}\mathbf{v}\times\rfloor\mathbf{e}_i \triangleq
\lfloor\boldsymbol{\beta}\times\rfloor \mathbf{e}_i  \\
\frac{\partial \boldsymbol{\beta}_3}{\partial \mathbf{x}}\mathbf{f}_{2i}& = & \mathbf{I}_3\mathbf{e}_i
\end{eqnarray}
where $i\in \{1,2,3\}$, and we have defined $\boldsymbol{\beta}_4 \triangleq \mathbf{b}_a$ and $\boldsymbol{\beta}_5 \triangleq \mathbf{R}\mathbf{g}$. 

Then, for the span of $\boldsymbol{\beta}_4$ and $\boldsymbol{\beta}_5$ we have:
\begin{eqnarray}
\frac{\partial \boldsymbol{\beta}_4}{\partial \mathbf{x}}
&  =  & 
\left[\frac{\partial \boldsymbol{\beta}_4}{\partial \mathbf{s}} \quad 
\frac{\partial \boldsymbol{\beta}_4}{\partial \mathbf{b}_g} \quad 
\frac{\partial \boldsymbol{\beta}_4}{\partial \mathbf{v}} \quad 
\frac{\partial \boldsymbol{\beta}_4}{\partial \mathbf{b}_a} \quad 
\frac{\partial \boldsymbol{\beta}_4}{\partial \mathbf{p}} \quad 
\frac{\partial \boldsymbol{\beta}_4}{\partial \mathbf{p}_{\mathbf{f}}}\right]		\\
& = &
\left[
\mathbf{0} \quad
\mathbf{0}  \quad
\mathbf{0}  \quad
\mathbf{I}_3  \quad
\mathbf{0} \quad
\mathbf{0}  
\right]	\\
\frac{\partial \boldsymbol{\beta}_4}{\partial \mathbf{x}}\mathbf{f}_0 & = & \mathbf{0}    \\ 
\frac{\partial \boldsymbol{\beta}_4}{\partial \mathbf{x}}\mathbf{f}_{1i}& = &\mathbf{0}   \\
\frac{\partial \boldsymbol{\beta}_4}{\partial \mathbf{x}}\mathbf{f}_{2i}& = &\mathbf{0}
\end{eqnarray}
where $i\in\{1,2,3\}$.
\begin{eqnarray}
\frac{\partial \boldsymbol{\beta}_5}{\partial \mathbf{x}}
&  =  & 
\left[\frac{\partial \boldsymbol{\beta}_5}{\partial \mathbf{s}} \quad 
\frac{\partial \boldsymbol{\beta}_5}{\partial \mathbf{b}_g} \quad 
\frac{\partial \boldsymbol{\beta}_5}{\partial \mathbf{v}} \quad 
\frac{\partial \boldsymbol{\beta}_5}{\partial \mathbf{b}_a} \quad 
\frac{\partial \boldsymbol{\beta}_5}{\partial \mathbf{p}} \quad 
\frac{\partial \boldsymbol{\beta}_5}{\partial \mathbf{p}_{\mathbf{f}}}\right]		\\
& = &
\left[
\lfloor\mathbf{R}\mathbf{g}\times\rfloor\frac{\partial \boldsymbol{\theta}}{\partial \mathbf{s}} \quad
\mathbf{0}  \quad
\mathbf{0}  \quad
\mathbf{0}  \quad
\mathbf{0} \quad
\mathbf{0}  
\right]	\\
\frac{\partial \boldsymbol{\beta}_5}{\partial \mathbf{x}}\mathbf{f}_0 & = & -\lfloor\mathbf{R}\mathbf{g}\times\rfloor\mathbf{b}_g 
\triangleq -\lfloor\boldsymbol{\beta}_5\times\rfloor\boldsymbol{\beta}_2 \\ 
\frac{\partial \boldsymbol{\beta}_5}{\partial \mathbf{x}}\mathbf{f}_{1i}& = &\lfloor\mathbf{R}\mathbf{g}\times\rfloor\mathbf{e}_i 
\triangleq \lfloor\boldsymbol{\beta}_5 \times\rfloor\mathbf{e}_i \\
\frac{\partial \boldsymbol{\beta}_5}{\partial \mathbf{x}}\mathbf{f}_{2i}& = &\mathbf{0}
\end{eqnarray}
where $i\in\{1,2,3\}$.

\subsection{Rank test for $\Xi^{(r)}$} \label{apd_rank_test_r}
Since the for the range measurement:$r=\sqrt{^x\mathbf{P}_{\mathbf{f}}^{\top}{^x\mathbf{P}_{\mathbf{f}}}}$ and $r\geq0$, we take $r^2={^x\mathbf{P}_{\mathbf{f}}^{\top}}{^x\mathbf{P}_{\mathbf{f}}}$ as the equivalent measurement to simplify the mathematical analysis. Hence, the range measurement model can be expressed in terms of basis functions as:
\begin{equation}
\overline{\mathbf{h}}^{(r)}=\boldsymbol{\beta}_1^{\top}\boldsymbol{\beta}_1
\end{equation}

Then we will perform the nonlinear observability rank condition test according to \cite{Hermann1977TAC}. 
\begin{itemize}
	\item The zeroth-order Lie derivatives of the measurement function is:
	\begin{equation}
	\mathcal{L}^0\overline{\mathbf{h}}^{(r)}=\boldsymbol{\beta}_1^{\top}\boldsymbol{\beta}_1
	\end{equation}
	Then, the gradient of the zeroth order Lie derivative is:
	\begin{equation}
	\nabla\mathcal{L}^0\overline{\mathbf{h}}^{(r)}
	=
	\frac{\partial \overline{\mathbf{h}}^{(r)}}{\partial \boldsymbol{\beta}}
	=
	\left[2\boldsymbol{\beta}_1^{\top} \quad \mathbf{0} \quad \mathbf{0} \quad \mathbf{0} \quad \mathbf{0} \right]
	\end{equation}
	\item The first-order Lie derivative of $\overline{\mathbf{h}}^{(r)}$ with respect to $\mathbf{g}_0$, $\mathbf{g}_{1i}$ and $\mathbf{g}_{2i}$ are computed respectively, as:
	\begin{eqnarray}
	\mathcal{L}^1_{\mathbf{g}_0}\overline{\mathbf{h}}^{(r)}
	& = &
	\nabla\mathcal{L}^0\overline{\mathbf{h}}^{(r)}\cdot\mathbf{g}_0
	=
	-2\boldsymbol{\beta}_1^{\top}\boldsymbol{\beta}_3  		\\
	\mathcal{L}^1_{\mathbf{g}_{1i}}\overline{\mathbf{h}}^{(r)}
	& = &
	\nabla\mathcal{L}^0\overline{\mathbf{h}}^{(r)}\cdot\mathbf{g}_{1i}
	=
	\mathbf{0}  								\\
	\mathcal{L}^1_{\mathbf{g}_{2i}}\overline{\mathbf{h}}^{(r)}
	& = &
	\nabla\mathcal{L}^0\overline{\mathbf{h}}^{(r)}\cdot\mathbf{g}_{2i}
	=
	\mathbf{0} 				
	\end{eqnarray}
	while the corresponding gradients are given by:
	\begin{eqnarray}
	\nabla\mathcal{L}^1_{\mathbf{g}_0}\overline{\mathbf{h}}^{(r)}
	& = &
	\frac{\partial \mathcal{L}^1_{\mathbf{g}_0}\overline{\mathbf{h}}^{(r)}}{\partial \boldsymbol{\beta}}
	=
	\left[-2\boldsymbol{\beta}_3^{\top} \quad \mathbf{0} \quad -2\boldsymbol{\beta}_1^{\top} \quad \mathbf{0} \quad \mathbf{0}\right]					\\
	\nabla\mathcal{L}^1_{\mathbf{g}_{1i}}\overline{\mathbf{h}}^{(r)}
	& = &
	\frac{\partial \mathcal{L}^1_{\mathbf{g}_{1i}}\overline{\mathbf{h}}^{(r)}}{\partial \boldsymbol{\beta}}
	= \left[\mathbf{0} \quad \mathbf{0} \quad \mathbf{0} \quad \mathbf{0} \quad \mathbf{0}\right]						\\
	\nabla\mathcal{L}^1_{\mathbf{g}_{2i}}\overline{\mathbf{h}}^{(r)}
	& = &
	\frac{\partial \mathcal{L}^1_{\mathbf{g}_{2i}}\overline{\mathbf{h}}^{(r)}}{\partial \boldsymbol{\beta}}
	= \left[\mathbf{0} \quad \mathbf{0} \quad \mathbf{0} \quad \mathbf{0} \quad \mathbf{0}\right]				
	\end{eqnarray}
	\item The second-order Lie derivatives are as following:
	\begin{eqnarray}
	\mathcal{L}^2_{\mathbf{g}_0\mathbf{g}_0}\overline{\mathbf{h}}^{(r)}
	& = &
	\nabla\mathcal{L}^1_{\mathbf{g}_0}\overline{\mathbf{h}}^{(r)}\cdot\mathbf{g}_0
	=
	2\boldsymbol{\beta}_3^{\top}\boldsymbol{\beta}_3 
	- 2\boldsymbol{\beta}_1^{\top}\boldsymbol{\beta}_5 
	+ 2\boldsymbol{\beta}_1^{\top}\boldsymbol{\beta}_4		
	\\
	\mathcal{L}^2_{\mathbf{g}_0\mathbf{g}_{1i}}\overline{\mathbf{h}}^{(r)}
	& = &
	\nabla\mathcal{L}^2_{\mathbf{g}_0}\overline{\mathbf{h}}^{(r)}\cdot\mathbf{g}_{1i}
	=
	\mathbf{0}  								
	\\
	\mathcal{L}^2_{\mathbf{g}_0\mathbf{g}_{2i}}\overline{\mathbf{h}}^{(r)}
	& = &
	\nabla\mathcal{L}^1_{\mathbf{g}_0}\overline{\mathbf{h}}^{(r)}\cdot\mathbf{g}_{2i}
	=
	-2\boldsymbol{\beta}_1^{\top}\mathbf{e}_{i} 				
	\end{eqnarray}
	while the corresponding gradients are:
	\begin{eqnarray}
	\nabla\mathcal{L}^2_{\mathbf{g}_0\mathbf{g}_0}\overline{\mathbf{h}}^{(r)}
	& = &
	\frac{\partial \mathcal{L}^2_{\mathbf{g}_0\mathbf{g}_0}\overline{\mathbf{h}}^{(r)}}{\partial \boldsymbol{\beta}}
	=
	\left[-2(\boldsymbol{\beta}_5^{\top}-\boldsymbol{\beta}_4^{\top}) \quad \mathbf{0} \quad 4\boldsymbol{\beta}_3^{\top} \quad 2\boldsymbol{\beta}_1^{\top} \quad -2\boldsymbol{\beta}_1^{\top}\right]					\\
	\nabla\mathcal{L}^2_{\mathbf{g}_0\mathbf{g}_{1i}}\overline{\mathbf{h}}^{(r)}
	& = &
	\frac{\partial \mathcal{L}^2_{\mathbf{g}_0\mathbf{g}_{1i}}\overline{\mathbf{h}}^{(r)}}{\partial \boldsymbol{\beta}}
	= \left[\mathbf{0} \quad \mathbf{0} \quad \mathbf{0} \quad \mathbf{0} \quad \mathbf{0}\right]						\\
	\nabla\mathcal{L}^2_{\mathbf{g}_0\mathbf{g}_{2i}}\overline{\mathbf{h}}^{(r)}
	& = &
	\frac{\partial \mathcal{L}^2_{\mathbf{g}_0\mathbf{g}_{2i}}\overline{\mathbf{h}}^{(r)}}{\partial \boldsymbol{\beta}}
	= \left[-2\mathbf{e}_1^{\top} \quad \mathbf{0} \quad \mathbf{0} \quad \mathbf{0} \quad \mathbf{0}\right]				
	\end{eqnarray}	
	\item The third-order Lie derivatives are as following:
	\begin{eqnarray}
	\mathcal{L}^3_{\mathbf{g}_0\mathbf{g}_0\mathbf{g}_0}\overline{\mathbf{h}}^{(r)}
	&=&
	\nabla\mathcal{L}^2_{\mathbf{g}_0\mathbf{g}_0}\overline{\mathbf{h}}^{(r)}\cdot\mathbf{g}_0
	=
	6\boldsymbol{\beta}_3^{\top}\boldsymbol{\beta}_5 - 6\boldsymbol{\beta}_3^{\top}\boldsymbol{\beta}_4 - 2\boldsymbol{\beta}_4^{\top}\lfloor\boldsymbol{\beta}_1\times\rfloor\boldsymbol{\beta}_2  
	\\
	\mathcal{L}^3_{\mathbf{g}_0\mathbf{g}_0\mathbf{g}_{1i}}\overline{\mathbf{h}}^{(r)}
	&=&
	\nabla\mathcal{L}^2_{\mathbf{g}_0\mathbf{g}_0}\overline{\mathbf{h}}^{(r)}\cdot\mathbf{g}_{1i}
	=
	2\boldsymbol{\beta}_4^{\top}\lfloor\boldsymbol{\beta}_1\times\rfloor\mathbf{e}_i
	\\
	\mathcal{L}^3_{\mathbf{g}_0\mathbf{g}_0\mathbf{g}_{2i}}\overline{\mathbf{h}}^{(r)}
	&=&
	\nabla\mathcal{L}^2_{\mathbf{g}_0\mathbf{g}_0}\overline{\mathbf{h}}^{(r)}\cdot\mathbf{g}_{2i}
	=
	4\boldsymbol{\beta}_3^{\top}\mathbf{e}_i
	\\
	\mathcal{L}^3_{\mathbf{g}_0\mathbf{g}_{2i}\mathbf{g}_0}\overline{\mathbf{h}}^{(r)}
	&=&
	\nabla\mathcal{L}^2_{\mathbf{g}_0\mathbf{g}_{2i}}\overline{\mathbf{h}}^{(r)}\cdot\mathbf{g}_{0}
	=
	2\mathbf{e}_i^{\top}\lfloor\boldsymbol{\beta}_1\times\rfloor\boldsymbol{\beta}_2+2\mathbf{e}_i^{\top}\boldsymbol{\beta}_3
	\\
	\mathcal{L}^3_{\mathbf{g}_0\mathbf{g}_{2i}\mathbf{g}_{1j}}\overline{\mathbf{h}}^{(r)}
	&=&
	\nabla\mathcal{L}^2_{\mathbf{g}_0\mathbf{g}_{2i}}\overline{\mathbf{h}}^{(r)}\cdot\mathbf{g}_{1j}
	=
	-2\mathbf{e}_i^{\top}\lfloor\boldsymbol{\beta}_1\times\rfloor\mathbf{e}_j						
	\end{eqnarray}
	while the corresponding gradients are:
	\begin{eqnarray}
	\nabla\mathcal{L}^3_{\mathbf{g}_0\mathbf{g}_0\mathbf{g}_0}\overline{\mathbf{h}}^{(r)}
	&=&
	\frac{\partial \mathcal{L}^3_{\mathbf{g}_0\mathbf{g}_0\mathbf{g}_0}\overline{\mathbf{h}}^{(r)}}{\partial \boldsymbol{\beta}}					
	\\
	&=&
	\scalemath{0.9}{
	\left[
	2\boldsymbol{\beta}_4^{\top}\lfloor\boldsymbol{\beta}_2\times\rfloor \quad -2\boldsymbol{\beta}_4\lfloor\boldsymbol{\beta}_1\times\rfloor \quad
	6(\boldsymbol{\beta}_5^{\top}-\boldsymbol{\beta}_4^{\top}) \quad
	-6\boldsymbol{\beta}_3^{\top} + 2\boldsymbol{\beta}_2^{\top}\lfloor\boldsymbol{\beta}_1\times\rfloor \quad
	6\boldsymbol{\beta}_3^{\top} 
	\right]} \nonumber
	\\
	\nabla\mathcal{L}^3_{\mathbf{g}_0\mathbf{g}_0\mathbf{g}_{1i}}\overline{\mathbf{h}}^{(r)}
	&=&
	\frac{\partial \mathcal{L}^3_{\mathbf{g}_0\mathbf{g}_0\mathbf{g}_{1i}}\overline{\mathbf{h}}^{(r)}}{\partial \boldsymbol{\beta}}	
	=
	\left[
	-2\boldsymbol{\beta}_4^{\top}\lfloor\mathbf{e}_i\times\rfloor \quad 
	\mathbf{0} \quad
	\mathbf{0} \quad
	-2\mathbf{e}_i^{\top}\lfloor\boldsymbol{\beta}_1\times\rfloor \quad
	\mathbf{0} 
	\right]
	\\
	\nabla\mathcal{L}^3_{\mathbf{g}_0\mathbf{g}_0\mathbf{g}_{2i}}\overline{\mathbf{h}}^{(r)}
	&=&
	\frac{\partial \mathcal{L}^3_{\mathbf{g}_0\mathbf{g}_0\mathbf{g}_{2i}}\overline{\mathbf{h}}^{(r)}}{\partial \boldsymbol{\beta}}	
	=
	\left[
	\mathbf{0} \quad 
	\mathbf{0} \quad
	4\mathbf{e}_i^{\top} \quad
	\mathbf{0} \quad
	\mathbf{0} 
	\right]	
	\\
	\nabla\mathcal{L}^3_{\mathbf{g}_0\mathbf{g}_{2i}\mathbf{g}_0}\overline{\mathbf{h}}^{(r)}
	&=&
	\frac{\partial \mathcal{L}^3_{\mathbf{g}_0\mathbf{g}_{2i}\mathbf{g}_0}\overline{\mathbf{h}}^{(r)}}{\partial \boldsymbol{\beta}}	
	=
	\left[
	-2\mathbf{e}_i^{\top}\lfloor\boldsymbol{\beta}_2\times\rfloor \quad 
	2\mathbf{e}_i^{\top}\lfloor\boldsymbol{\beta}_1\times\rfloor \quad
	2\mathbf{e}_i^{\top} \quad
	\mathbf{0} \quad
	\mathbf{0} 
	\right]
	\\
	\nabla\mathcal{L}^3_{\mathbf{g}_0\mathbf{g}_{2i}\mathbf{g}_{1j}}\overline{\mathbf{h}}^{(r)}
	&=&
	\frac{\partial \mathcal{L}^3_{\mathbf{g}_0\mathbf{g}_{2i}\mathbf{g}_{1j}}\overline{\mathbf{h}}^{(r)}}{\partial \boldsymbol{\beta}}	
	=
	\left[
	2\mathbf{e}_i^{\top}\lfloor\mathbf{e}_j\times\rfloor \quad 
	\mathbf{0} \quad
	\mathbf{0} \quad
	\mathbf{0} \quad
	\mathbf{0} 
	\right]										
	\end{eqnarray}
	\item The fourth-order Lie derivatives are as following:
	\begin{eqnarray}
	\mathcal{L}^4_{\mathbf{g}_0\mathbf{g}_0\mathbf{g}_0\mathbf{g}_0}\overline{\mathbf{h}}^{(r)}
	&=&
	\nabla\mathcal{L}^3_{\mathbf{g}_0\mathbf{g}_0\mathbf{g}_0}\overline{\mathbf{h}}^{(r)}\cdot\mathbf{g}_0
	\nonumber
	\\
	&=&
	-2\boldsymbol{\beta}_4^{\top}\lfloor\boldsymbol{\beta}_2\times\rfloor\lfloor\boldsymbol{\beta}_1\times\rfloor\boldsymbol{\beta}_2
	-8\boldsymbol{\beta}_4^{\top}\lfloor\boldsymbol{\beta}_2\times\rfloor\boldsymbol{\beta}_3
	+6(\boldsymbol{\beta}_4^{\top}-\boldsymbol{\beta}_5^{\top})(\boldsymbol{\beta}_4-\boldsymbol{\beta}_5)
	\\
	\mathcal{L}^4_{\mathbf{g}_0\mathbf{g}_0\mathbf{g}_{2i}\mathbf{g}_0}\overline{\mathbf{h}}^{(r)}
	&=&
	\nabla\mathcal{L}^4_{\mathbf{g}_0\mathbf{g}_0\mathbf{g}_{2i}}\overline{\mathbf{h}}^{(r)}\cdot\mathbf{g}_{0}
	=
	-4\mathbf{e}_i^{\top}\lfloor\boldsymbol{\beta}_3\times\rfloor\boldsymbol{\beta}_2
	+4\mathbf{e}_i^{\top}\boldsymbol{\beta}_5-4\mathbf{e}_i^{\top}\boldsymbol{\beta}_4			
	\end{eqnarray}	
	while the corresponding gradients are:
	\begin{small}
		\begin{equation}\label{eq_4th_lie}
		\scalemath{0.9}{
		\nabla\mathcal{L}^4_{\mathbf{g}_0\mathbf{g}_{0}\mathbf{g}_0\mathbf{g}_0}\overline{\mathbf{h}}^{(r)}
		=
		\left[
		\frac{\partial \mathcal{L}^4_{\mathbf{g}_0\mathbf{g}_{0}\mathbf{g}_0\mathbf{g}_0}\overline{\mathbf{h}}^{(r)}}{\partial \boldsymbol{\beta}_1}	 \  
		\frac{\partial \mathcal{L}^4_{\mathbf{g}_0\mathbf{g}_{0}\mathbf{g}_0\mathbf{g}_0}\overline{\mathbf{h}}^{(r)}}{\partial \boldsymbol{\beta}_2}	 \  
		\frac{\partial \mathcal{L}^4_{\mathbf{g}_0\mathbf{g}_{0}\mathbf{g}_0\mathbf{g}_0}\overline{\mathbf{h}}^{(r)}}{\partial \boldsymbol{\beta}_3}	 \  
		\frac{\partial \mathcal{L}^4_{\mathbf{g}_0\mathbf{g}_{0}\mathbf{g}_0\mathbf{g}_0}\overline{\mathbf{h}}^{(r)}}{\partial \boldsymbol{\beta}_4}	 \  
		\frac{\partial \mathcal{L}^4_{\mathbf{g}_0\mathbf{g}_{0}\mathbf{g}_0\mathbf{g}_0}\overline{\mathbf{h}}^{(r)}}{\partial \boldsymbol{\beta}_5}	 \  
		\right]	}
		\end{equation}
	\end{small}
	\begin{equation}
	\scalemath{0.9}{
	\nabla\mathcal{L}^4_{\mathbf{g}_0\mathbf{g}_0\mathbf{g}_{2i}\mathbf{g}_{0}}\overline{\mathbf{h}}^{(r)}
	=
	\frac{\partial \mathcal{L}^4_{\mathbf{g}_0\mathbf{g}_0\mathbf{g}_{2i}\mathbf{g}_{0}}\overline{\mathbf{h}}^{(r)}}{\partial \boldsymbol{\beta}}	
	=
	\left[
	\mathbf{0} \quad 
	-4\mathbf{e}_i^{\top}\lfloor\boldsymbol{\beta}_3\times\rfloor \quad
	4\mathbf{e}_i^{\top}\lfloor\boldsymbol{\beta}_2\times\rfloor \quad
	-4\mathbf{e}_i^{\top} \quad
	4\mathbf{e}_i^{\top} 
	\right]	}		
	\end{equation}
	where the terms in Eq. \eqref{eq_4th_lie} are:
	\begin{eqnarray}
	\frac{\partial \mathcal{L}^4_{\mathbf{g}_0\mathbf{g}_{0}\mathbf{g}_0\mathbf{g}_0}\overline{\mathbf{h}}^{(r)}}{\partial \boldsymbol{\beta}_1}
	&=&
	2\boldsymbol{\beta}_4^{\top}\lfloor\boldsymbol{\beta}_2\times\rfloor^2		
	\\
	\frac{\partial \mathcal{L}^4_{\mathbf{g}_0\mathbf{g}_{0}\mathbf{g}_0\mathbf{g}_0}\overline{\mathbf{h}}^{(r)}}{\partial \boldsymbol{\beta}_2}
	&=&
	-2\boldsymbol{\beta}_4^{\top}\lfloor\boldsymbol{\beta}_2\times\rfloor\lfloor\boldsymbol{\beta}_1\times\rfloor
	+2\boldsymbol{\beta}_2^{\top}\lfloor\boldsymbol{\beta}_1\times\rfloor\lfloor\boldsymbol{\beta}_4\times\rfloor
	+8\boldsymbol{\beta}_4^{\top}\lfloor\boldsymbol{\beta}_3\times\rfloor
	\\
	\frac{\partial \mathcal{L}^4_{\mathbf{g}_0\mathbf{g}_{0}\mathbf{g}_0\mathbf{g}_0}\overline{\mathbf{h}}^{(r)}}{\partial \boldsymbol{\beta}_3}
	&=&
	-8\boldsymbol{\beta}_4^{\top}\lfloor\boldsymbol{\beta}_2\times\rfloor
	\\
	\frac{\partial \mathcal{L}^4_{\mathbf{g}_0\mathbf{g}_{0}\mathbf{g}_0\mathbf{g}_0}\overline{\mathbf{h}}^{(r)}}{\partial \boldsymbol{\beta}_4}
	&=&
	-2\boldsymbol{\beta}_2^{\top}\lfloor\boldsymbol{\beta}_1\times\rfloor\lfloor\boldsymbol{\beta}_2\times\rfloor
	+8\boldsymbol{\beta}_3^{\top}\lfloor\boldsymbol{\beta}_2\times\rfloor
	+12(\boldsymbol{\beta}_4^{\top}-\boldsymbol{\beta}_5^{\top})		
	\\
	\frac{\partial \mathcal{L}^4_{\mathbf{g}_0\mathbf{g}_{0}\mathbf{g}_0\mathbf{g}_0}\overline{\mathbf{h}}^{(r)}}{\partial \boldsymbol{\beta}_5}
	&=&
	12(\boldsymbol{\beta}_5^{\top}-\boldsymbol{\beta}_4^{\top})					
	\end{eqnarray}
	\item The fifth-order Lie derivatives are as following:
	\begin{eqnarray}
	\mathcal{L}^5_{\mathbf{g}_0\mathbf{g}_0\mathbf{g}_{2i}\mathbf{g}_0\mathbf{g}_{1j}}\overline{\mathbf{h}}^{(r)}
	&=&
	\nabla\mathcal{L}^4_{\mathbf{g}_0\mathbf{g}_0\mathbf{g}_{2i}\mathbf{g}_0}\overline{\mathbf{h}}^{(r)}\cdot\mathbf{g}_{1j}
	=
	4\mathbf{e}_i^{\top} \lfloor\boldsymbol{\beta}_2\times\rfloor \lfloor\boldsymbol{\beta}_3\times\rfloor \mathbf{e}_j
	+4\mathbf{e}_i^{\top} \lfloor\boldsymbol{\beta}_5\times\rfloor	 \mathbf{e}_j	
	\end{eqnarray}	
	while the corresponding gradients are:
	\begin{eqnarray}
	\nabla\mathcal{L}^5_{\mathbf{g}_0\mathbf{g}_0\mathbf{g}_{2i}\mathbf{g}_{0}\mathbf{g}_{1j}}\overline{\mathbf{h}}^{(r)}
	&=&
	\frac{\partial \mathcal{L}^5_{\mathbf{g}_0\mathbf{g}_0\mathbf{g}_{2i}\mathbf{g}_{0}\mathbf{g}_{1j}}\overline{\mathbf{h}}^{(r)}}{\partial \boldsymbol{\beta}}	
	\\
	&=&
	\left[
	\mathbf{0} \quad 
	-4\mathbf{e}_j^{\top}\lfloor\boldsymbol{\beta}_3\times\rfloor  \lfloor\mathbf{e}_i\times\rfloor  \quad
	-4\mathbf{e}_i^{\top}\lfloor\boldsymbol{\beta}_2\times\rfloor  \lfloor\mathbf{e}_j\times\rfloor \quad
	\mathbf{0} \quad
	-4\mathbf{e}_i^{\top} \lfloor\mathbf{e}_j\times\rfloor
	\right]	
	\nonumber		
	\end{eqnarray}
\end{itemize}

Therefore, we can construct the $\Xi^{(r)}$ matrix \eqref{eq_Xi_r} and we can find out that $\Xi^{(r)}$ is of full rank. 
\begin{equation} \label{eq_Xi_r}
\resizebox{0.9\hsize}{!}
{$
	\Xi^{(r)}
	=
	\begin{bmatrix}
	\nabla \mathcal{L}^2_{\mathbf{g}_0\mathbf{g}_{21}}\overline{\mathbf{h}}^{(r)} \\
	\nabla \mathcal{L}^2_{\mathbf{g}_0\mathbf{g}_{22}}\overline{\mathbf{h}}^{(r)} \\
	\nabla \mathcal{L}^2_{\mathbf{g}_0\mathbf{g}_{23}}\overline{\mathbf{h}}^{(r)} \\
	\nabla\mathcal{L}^5_{\mathbf{g}_0\mathbf{g}_0\mathbf{g}_{21}\mathbf{g}_{0}\mathbf{g}_{12}}\overline{\mathbf{h}}^{(r)} \\
	\nabla\mathcal{L}^5_{\mathbf{g}_0\mathbf{g}_0\mathbf{g}_{22}\mathbf{g}_{0}\mathbf{g}_{13}}\overline{\mathbf{h}}^{(r)} \\
	\nabla\mathcal{L}^5_{\mathbf{g}_0\mathbf{g}_0\mathbf{g}_{23}\mathbf{g}_{0}\mathbf{g}_{11}}\overline{\mathbf{h}}^{(r)} \\			
	\nabla \mathcal{L}^3_{\mathbf{g}_0\mathbf{g}_0\mathbf{g}_{21}}\overline{\mathbf{h}}^{(r)} \\
	\nabla \mathcal{L}^3_{\mathbf{g}_0\mathbf{g}_0\mathbf{g}_{22}}\overline{\mathbf{h}}^{(r)} \\
	\nabla \mathcal{L}^3_{\mathbf{g}_0\mathbf{g}_0\mathbf{g}_{23}}\overline{\mathbf{h}}^{(r)} \\
	\nabla \mathcal{L}^4_{\mathbf{g}_0\mathbf{g}_0\mathbf{g}_{21}\mathbf{g}_0}\overline{\mathbf{h}}^{(r)}  \\
	\nabla \mathcal{L}^4_{\mathbf{g}_0\mathbf{g}_0\mathbf{g}_{22}\mathbf{g}_0}\overline{\mathbf{h}}^{(r)}  \\
	\nabla \mathcal{L}^4_{\mathbf{g}_0\mathbf{g}_0\mathbf{g}_{23}\mathbf{g}_0}\overline{\mathbf{h}}^{(r)}  \\	
	\nabla\mathcal{L}^5_{\mathbf{g}_0\mathbf{g}_0\mathbf{g}_{21}\mathbf{g}_{0}\mathbf{g}_{13}}\overline{\mathbf{h}}^{(r)} \\
	\nabla\mathcal{L}^5_{\mathbf{g}_0\mathbf{g}_0\mathbf{g}_{22}\mathbf{g}_{0}\mathbf{g}_{11}}\overline{\mathbf{h}}^{(r)} \\
	\nabla\mathcal{L}^5_{\mathbf{g}_0\mathbf{g}_0\mathbf{g}_{23}\mathbf{g}_{0}\mathbf{g}_{12}}\overline{\mathbf{h}}^{(r)}						
	\end{bmatrix}
	=
	\begin{bmatrix}
	-2\mathbf{e}^{\top}_1  &   \mathbf{0}  &  \mathbf{0}  &  \mathbf{0}  &  \mathbf{0}  \\
	-2\mathbf{e}^{\top}_2  &   \mathbf{0}  &  \mathbf{0}  &  \mathbf{0}  &  \mathbf{0}  \\
	-2\mathbf{e}^{\top}_3  &   \mathbf{0}  &  \mathbf{0}  &  \mathbf{0}  &  \mathbf{0}  \\
	\mathbf{0} & 
	-4\mathbf{e}_2^{\top}\lfloor\boldsymbol{\beta}_3\times\rfloor  \lfloor\mathbf{e}_1\times\rfloor  &
	-4\mathbf{e}_1^{\top}\lfloor\boldsymbol{\beta}_2\times\rfloor  \lfloor\mathbf{e}_2\times\rfloor &
	\mathbf{0} &
	-4\mathbf{e}_1^{\top} \lfloor\mathbf{e}_2\times\rfloor	 \\
	\mathbf{0} & 
	-4\mathbf{e}_3^{\top}\lfloor\boldsymbol{\beta}_3\times\rfloor  \lfloor\mathbf{e}_2\times\rfloor  &
	-4\mathbf{e}_2^{\top}\lfloor\boldsymbol{\beta}_2\times\rfloor  \lfloor\mathbf{e}_3\times\rfloor &
	\mathbf{0} &
	-4\mathbf{e}_2^{\top} \lfloor\mathbf{e}_3\times\rfloor	 \\
	\mathbf{0} & 
	-4\mathbf{e}_1^{\top}\lfloor\boldsymbol{\beta}_3\times\rfloor  \lfloor\mathbf{e}_3\times\rfloor  &
	-4\mathbf{e}_3^{\top}\lfloor\boldsymbol{\beta}_2\times\rfloor  \lfloor\mathbf{e}_1\times\rfloor &
	\mathbf{0} &
	-4\mathbf{e}_3^{\top} \lfloor\mathbf{e}_1\times\rfloor  \\
	\mathbf{0}  &   \mathbf{0}  &  4\mathbf{e}^{\top}_1  &  \mathbf{0}  &  \mathbf{0}  \\
	\mathbf{0}  &   \mathbf{0}  &  4\mathbf{e}^{\top}_2  &  \mathbf{0}  &  \mathbf{0}  \\
	\mathbf{0}  &   \mathbf{0}  &  4\mathbf{e}^{\top}_3  &  \mathbf{0}  &  \mathbf{0}  \\
	\mathbf{0} & -4\mathbf{e}^{\top}_1\lfloor\boldsymbol{\beta}_3\times\rfloor	& 4\mathbf{e}^{\top}_1\lfloor\boldsymbol{\beta}_2\times\rfloor	& -4\mathbf{e}^{\top}_1  &  4\mathbf{e}^{\top}_1   \\	
	\mathbf{0} & -4\mathbf{e}^{\top}_1\lfloor\boldsymbol{\beta}_3\times\rfloor	& 4\mathbf{e}^{\top}_2\lfloor\boldsymbol{\beta}_2\times\rfloor	& -4\mathbf{e}^{\top}_2  &  4\mathbf{e}^{\top}_2   \\	
	\mathbf{0} & -4\mathbf{e}^{\top}_1\lfloor\boldsymbol{\beta}_3\times\rfloor	& 4\mathbf{e}^{\top}_3\lfloor\boldsymbol{\beta}_2\times\rfloor	& -4\mathbf{e}^{\top}_3  &  4\mathbf{e}^{\top}_3   \\
	\mathbf{0} & 
	-4\mathbf{e}_3^{\top}\lfloor\boldsymbol{\beta}_3\times\rfloor  \lfloor\mathbf{e}_1\times\rfloor  &
	-4\mathbf{e}_1^{\top}\lfloor\boldsymbol{\beta}_2\times\rfloor  \lfloor\mathbf{e}_3\times\rfloor &
	\mathbf{0} &
	-4\mathbf{e}_1^{\top} \lfloor\mathbf{e}_3\times\rfloor	 \\
	\mathbf{0} & 
	-4\mathbf{e}_1^{\top}\lfloor\boldsymbol{\beta}_3\times\rfloor  \lfloor\mathbf{e}_2\times\rfloor  &
	-4\mathbf{e}_2^{\top}\lfloor\boldsymbol{\beta}_2\times\rfloor  \lfloor\mathbf{e}_1\times\rfloor &
	\mathbf{0} &
	-4\mathbf{e}_2^{\top} \lfloor\mathbf{e}_1\times\rfloor	 \\
	\mathbf{0} & 
	-4\mathbf{e}_2^{\top}\lfloor\boldsymbol{\beta}_3\times\rfloor  \lfloor\mathbf{e}_3\times\rfloor  &
	-4\mathbf{e}_3^{\top}\lfloor\boldsymbol{\beta}_2\times\rfloor  \lfloor\mathbf{e}_2\times\rfloor &
	\mathbf{0} &
	-4\mathbf{e}_3^{\top} \lfloor\mathbf{e}_2\times\rfloor								
	\end{bmatrix}
	$}
\end{equation}
Given random motion, the diagonal block of the $\Xi^{(r)}$ are all of full rank (3). Therefore, $\Xi^{(r)}$ is of full column rank. 
\subsection{Rank test for $\Xi^{(b)}$} \label{apd_rank_test_b}
%
%

%
%

For the analysis, with the generalized point measurement model\eqref{eq_x_meas}, we consider the noise free case, and define  $\boldsymbol{\gamma}=\mathbf{b}_{\perp1}$, $\boldsymbol{\gamma}=\mathbf{b}_{\perp2}$.
%
%
Then we will perform the nonlinear observability rank condition test according to \cite{Hermann1977TAC}. 
\begin{itemize}
	\item The zeroth-order Lie derivatives of the measurement function is:
	\begin{equation}
	\mathcal{L}^{0}\overline{\mathbf{h}}^{(b)}
	=
	\left[
	\begin{array}{c}
	\mathcal{L}^{0}\overline{\mathbf{h}}^{(b)}_1  \\
	\mathcal{L}^{0}\overline{\mathbf{h}}^{(b)}_2
	\end{array}
	\right]
	=
	\left[
	\begin{array}{c}
	\boldsymbol{\gamma}^{\top}_1\boldsymbol{\beta}_1  \\
	\boldsymbol{\gamma}^{\top}_2\boldsymbol{\beta}_1				
	\end{array}
	\right]
	=
	\left[
	\begin{array}{c}
	\boldsymbol{\gamma}^{\top}_1 \\
	\boldsymbol{\gamma}^{\top}_2
	\end{array}
	\right]
	\boldsymbol{\beta}_1		
	\end{equation}
	Then, the gradients of the zeroth-order Lie derivative is:
	\begin{equation}
	\nabla\mathcal{L}^{0}\overline{\mathbf{h}}^{(b)}
	=
	\left[
	\begin{array}{c}
	\nabla\mathcal{L}^{0}\overline{\mathbf{h}}^{(b)}_1  \\
	\nabla\mathcal{L}^{0}\overline{\mathbf{h}}^{(b)}_2				
	\end{array}
	\right]
	=
	\left[
	\begin{array}{c}
	\frac{\partial \overline{\mathbf{h}}^{(b)}_1}{\partial \boldsymbol{\beta}}  \\
	\frac{\partial \overline{\mathbf{h}}^{(b)}_2}{\partial \boldsymbol{\beta}}
	\end{array}
	\right]
	=
	\left[
	\begin{array}{ccccc}
	\boldsymbol{\gamma}^{\top}_1 & \mathbf{0} & \mathbf{0} & \mathbf{0} & \mathbf{0}   \\
	\boldsymbol{\gamma}^{\top}_2 & \mathbf{0} & \mathbf{0} & \mathbf{0} & \mathbf{0} 
	\end{array}
	\right]
	=
	\left[
	\begin{array}{c}
	\boldsymbol{\gamma}^{\top}_1 \\
	\boldsymbol{\gamma}^{\top}_2
	\end{array}
	\right]
	\left[
	\begin{array}{ccccc}
	\mathbf{I}_3 & \mathbf{0} & \mathbf{0} & \mathbf{0} & \mathbf{0}   
	\end{array}
	\right]	\nonumber	
	\end{equation}
	\item The first-order Lie derivative of $\overline{\mathbf{h}}^{(b)}$ with respect to $\mathbf{g}_0$, $\mathbf{g}_{1i}$ and $\mathbf{g}_{2i}$ are computed respectively, as:
	\begin{eqnarray}
	\mathcal{L}^1_{\mathbf{g}_0}\overline{\mathbf{h}}^{(b)} 
	&=&
	\nabla\mathcal{L}^0\overline{\mathbf{h}}^{(b)}\cdot \mathbf{g}_0
	=
	\left[
	\begin{array}{c}
	\nabla\mathcal{L}^0\overline{\mathbf{h}}^{(b)}_1 \cdot \mathbf{g}_0  \\
	\nabla\mathcal{L}^0\overline{\mathbf{h}}^{(b)}_2 \cdot \mathbf{g}_0
	\end{array}
	\right]\\
	&=&
	\left[
	\begin{array}{c}
	-\boldsymbol{\gamma}^{\top}_1\lfloor\boldsymbol{\beta}_1\times\rfloor\boldsymbol{\beta}_2-\boldsymbol{\gamma}^{\top}_1\boldsymbol{\beta}_3  \\
	-\boldsymbol{\gamma}^{\top}_2\lfloor\boldsymbol{\beta}_1\times\rfloor\boldsymbol{\beta}_2-\boldsymbol{\gamma}^{\top}_2\boldsymbol{\beta}_3
	\end{array}
	\right]
	=
	\left[
	\begin{array}{c}
	\boldsymbol{\gamma}^{\top}_1  \\
	\boldsymbol{\gamma}^{\top}_2
	\end{array}
	\right]
	\left[-\lfloor\boldsymbol{\beta}_1\times\rfloor\boldsymbol{\beta}_2-\mathbf{I}_3\boldsymbol{\beta}_3 \right]
	\\
	\mathcal{L}^1_{\mathbf{g}_{1i}}\overline{\mathbf{h}}^{(b)} 
	&=&
	\scalemath{0.9}{
	\nabla\mathcal{L}^0\overline{\mathbf{h}}^{(b)}\cdot \mathbf{g}_{1i}
	=
	\left[
	\begin{array}{c}
	\nabla\mathcal{L}^0\overline{\mathbf{h}}^{(b)}_1 \cdot \mathbf{g}_{1i}  \\
	\nabla\mathcal{L}^0\overline{\mathbf{h}}^{(b)}_2 \cdot \mathbf{g}_{1i}
	\end{array}
	\right]
	=
	\left[
	\begin{array}{c}
	\boldsymbol{\gamma}^{\top}_1\lfloor\boldsymbol{\beta}_1\times\rfloor\mathbf{e}_i  \\
	\boldsymbol{\gamma}^{\top}_2\lfloor\boldsymbol{\beta}_1\times\rfloor\mathbf{e}_i
	\end{array}
	\right]
	=
	\left[
	\begin{array}{c}
	\boldsymbol{\gamma}^{\top}_1 \\
	\boldsymbol{\gamma}^{\top}_2
	\end{array}
	\right]
	\left[\lfloor\boldsymbol{\beta}_1\times\rfloor\mathbf{e}_i\right]}
	\\
	\mathcal{L}^1_{\mathbf{g}_{2i}}\overline{\mathbf{h}}^{(b)} 
	&=&
	\nabla\mathcal{L}^0\overline{\mathbf{h}}^{(b)}\cdot \mathbf{g}_{2i}
	=
	\left[
	\begin{array}{c}
	\nabla\mathcal{L}^0\overline{\mathbf{h}}^{(b)}_1 \cdot \mathbf{g}_{2i}  \\
	\nabla\mathcal{L}^0\overline{\mathbf{h}}^{(b)}_2 \cdot \mathbf{g}_{2i}
	\end{array}
	\right]
	=
	\left[
	\begin{array}{c}
	\mathbf{0} \\
	\mathbf{0}
	\end{array}
	\right]				
	\end{eqnarray}
	while the corresponding gradients are given by:
	\begin{small}
		\begin{eqnarray}
		\nabla\mathcal{L}^1_{\mathbf{g}_0}\overline{\mathbf{h}}^{(b)}
		&=&
		\left[
		\begin{array}{c}
		\nabla\mathcal{L}^1_{\mathbf{g}_0}\overline{\mathbf{h}}^{(b)}_1  \\
		\nabla\mathcal{L}^1_{\mathbf{g}_0}\overline{\mathbf{h}}^{(b)}_2  
		\end{array}
		\right]
		=
		\left[
		\begin{array}{c}
		\frac{\partial \mathcal{L}^1_{\mathbf{g}_0}\overline{\mathbf{h}}^{(b)}_1}{\partial \boldsymbol{\beta}}  \\
		\frac{\partial \mathcal{L}^1_{\mathbf{g}_0}\overline{\mathbf{h}}^{(b)}_2}{\partial \boldsymbol{\beta}}
		\end{array}
		\right]
		=
		\left[
		\begin{array}{ccccc}
		\boldsymbol{\gamma}^{\top}_1\lfloor\boldsymbol{\beta}_2\times\rfloor & -\boldsymbol{\gamma}^{\top}_1\lfloor\boldsymbol{\beta}_1\times\rfloor &
		-\boldsymbol{\gamma}^{\top}_1	&  \mathbf{0} & \mathbf{0}  \\
		\boldsymbol{\gamma}^{\top}_2\lfloor\boldsymbol{\beta}_2\times\rfloor & -\boldsymbol{\gamma}^{\top}_2\lfloor\boldsymbol{\beta}_1\times\rfloor &
		-\boldsymbol{\gamma}^{\top}_2	&  \mathbf{0} & \mathbf{0}									
		\end{array}
		\right]
		\\
		&=&
		\left[
		\begin{array}{c}
		\boldsymbol{\gamma}^{\top}_1 \\
		\boldsymbol{\gamma}^{\top}_2
		\end{array}
		\right]
		\left[
		\begin{array}{ccccc}
		\lfloor\boldsymbol{\beta}_2\times\rfloor  & -\lfloor\boldsymbol{\beta}_1\times\rfloor
		& -\mathbf{I}_3  &  \mathbf{0}  & \mathbf{0}			
		\end{array}
		\right]
		\\
		\nabla\mathcal{L}^1_{\mathbf{g}_{1i}}\overline{\mathbf{h}}^{(b)}
		&=&
		\left[
		\begin{array}{c}
		\nabla\mathcal{L}^1_{\mathbf{g}_{1i}}\overline{\mathbf{h}}^{(b)}_1  \\
		\nabla\mathcal{L}^1_{\mathbf{g}_{1i}}\overline{\mathbf{h}}^{(b)}_2  
		\end{array}
		\right]
		=
		\left[
		\begin{array}{c}
		\frac{\partial \mathcal{L}^1_{\mathbf{g}_{1i}}\overline{\mathbf{h}}^{(b)}_1}{\partial \boldsymbol{\beta}}  \\
		\frac{\partial \mathcal{L}^1_{\mathbf{g}_{1i}}\overline{\mathbf{h}}^{(b)}_2}{\partial \boldsymbol{\beta}}
		\end{array}
		\right]
		=
		\left[
		\begin{array}{ccccc}
		-\boldsymbol{\gamma}^{\top}_1\lfloor\mathbf{e}_i\times\rfloor & \mathbf{0} &
		\mathbf{0}	&  \mathbf{0} & \mathbf{0}  \\
		-\boldsymbol{\gamma}^{\top}_2\lfloor\mathbf{e}_i\times\rfloor & \mathbf{0} &
		\mathbf{0}	&  \mathbf{0} & \mathbf{0}									
		\end{array}
		\right]
		\\
		&=&
		\left[
		\begin{array}{c}
		\boldsymbol{\gamma}^{\top}_1 \\
		\boldsymbol{\gamma}^{\top}_2
		\end{array}
		\right]
		\left[
		\begin{array}{ccccc}
		-\lfloor\mathbf{e}_i\times\rfloor  & \mathbf{0}
		& \mathbf{0}  &  \mathbf{0}  & \mathbf{0}			
		\end{array}
		\right]
		\\
		\nabla\mathcal{L}^1_{\mathbf{g}_{1i}}\overline{\mathbf{h}}^{(b)}
		&=&
		\left[
		\begin{array}{c}
		\nabla\mathcal{L}^1_{\mathbf{g}_{1i}}\overline{\mathbf{h}}^{(b)}_1  \\
		\nabla\mathcal{L}^1_{\mathbf{g}_{1i}}\overline{\mathbf{h}}^{(b)}_2  
		\end{array}
		\right]
		=
		\left[
		\begin{array}{c}
		\frac{\partial \mathcal{L}^1_{\mathbf{g}_{1i}}\overline{\mathbf{h}}^{(b)}_1}{\partial \boldsymbol{\beta}}  \\
		\frac{\partial \mathcal{L}^1_{\mathbf{g}_{1i}}\overline{\mathbf{h}}^{(b)}_2}{\partial \boldsymbol{\beta}}
		\end{array}
		\right]
		=
		\left[
		\begin{array}{c}
		\mathbf{0}  \\
		\mathbf{0}
		\end{array}
		\right]						
		\end{eqnarray}
	\end{small}
	\item The second-order Lie derivatives are as following:
	\begin{eqnarray}
	\mathcal{L}^2_{\mathbf{g}_0\mathbf{g}_0}\overline{\mathbf{h}}^{(b)}
	&=&
	\nabla\mathcal{L}^1_{\mathbf{g}_0}\overline{\mathbf{h}}^{(b)} \cdot \mathbf{g}_0
	=
	\left[
	\begin{array}{c}
	\nabla\mathcal{L}^1_{\mathbf{g}_0}\overline{\mathbf{h}}^{(b)}_1\cdot\mathbf{g}_0 \\ \nabla\mathcal{L}^1_{\mathbf{g}_0}\overline{\mathbf{h}}^{(b)}_2\cdot\mathbf{g}_0
	\end{array}
	\right]
	\\
	&=&
	\left[
	\begin{array}{c}
	\boldsymbol{\gamma}^{\top}_1 \\
	\boldsymbol{\gamma}^{\top}_2
	\end{array}
	\right]
	\left[
	-\lfloor\boldsymbol{\beta}_2\times\rfloor \lfloor\boldsymbol{\beta}_1\times\rfloor\boldsymbol{\beta}_2
	-\lfloor\boldsymbol{\beta}_2\times\rfloor\boldsymbol{\beta}_3
	+\lfloor\boldsymbol{\beta}_3\times\rfloor\boldsymbol{\beta}_2
	-\boldsymbol{\beta}_5
	+\boldsymbol{\beta}_4
	\right]
	\\
	\mathcal{L}^2_{\mathbf{g}_0\mathbf{g}_{1i}}\overline{\mathbf{h}}^{(b)}
	&=&
	\scalemath{0.9}{
	\nabla\mathcal{L}^1_{\mathbf{g}_0}\overline{\mathbf{h}}^{(b)} \cdot \mathbf{g}_{1i}
	=
	\left[
	\begin{array}{c}
	\nabla\mathcal{L}^1_{\mathbf{g}_0}\overline{\mathbf{h}}^{(b)}_1\cdot\mathbf{g}_{1i} \\ \nabla\mathcal{L}^1_{\mathbf{g}_0}\overline{\mathbf{h}}^{(b)}_2\cdot\mathbf{g}_{1i}
	\end{array}
	\right]
	=
	\left[
	\begin{array}{c}
	\boldsymbol{\gamma}^{\top}_1 \\
	\boldsymbol{\gamma}^{\top}_2
	\end{array}
	\right]
	\left[
	\lfloor\boldsymbol{\beta}_2\times\rfloor \lfloor\boldsymbol{\beta}_1\times\rfloor
	-\lfloor\boldsymbol{\beta}_3\times\rfloor
	\right]
	\mathbf{e}_i
    }
	\\
	\mathcal{L}^2_{\mathbf{g}_{1i}\mathbf{g}_0}\overline{\mathbf{h}}^{(b)}
	&=&
	\scalemath{0.85}{
	\nabla\mathcal{L}^1_{\mathbf{g}_{1i}}\overline{\mathbf{h}}^{(b)} \cdot \mathbf{g}_{0}
	=
	\left[
	\begin{array}{c}
	\nabla\mathcal{L}^1_{\mathbf{g}_{1i}}\overline{\mathbf{h}}^{(b)}_1\cdot\mathbf{g}_{0} \\ \nabla\mathcal{L}^1_{\mathbf{g}_{1i}}\overline{\mathbf{h}}^{(b)}_2\cdot\mathbf{g}_{0}
	\end{array}
	\right]
	=
	\left[
	\begin{array}{c}
	\boldsymbol{\gamma}^{\top}_1 \\
	\boldsymbol{\gamma}^{\top}_2
	\end{array}
	\right]
	\left[
	\lfloor\mathbf{e}_i\times\rfloor \lfloor\boldsymbol{\beta}_1\times\rfloor\boldsymbol{\beta}_2
	+\lfloor\mathbf{e}_i\times\rfloor\boldsymbol{\beta}_3
	\right]
	}
	\\
	\mathcal{L}^2_{\mathbf{g}_{1i}\mathbf{g}_{1j}}\overline{\mathbf{h}}^{(b)}
	&=&
	\nabla\mathcal{L}^1_{\mathbf{g}_{1i}}\overline{\mathbf{h}}^{(b)} \cdot \mathbf{g}_{1j}
	=
	\left[
	\begin{array}{c}
	\nabla\mathcal{L}^1_{\mathbf{g}_{1i}}\overline{\mathbf{h}}^{(b)}_1\cdot\mathbf{g}_{1j} \\ \nabla\mathcal{L}^1_{\mathbf{g}_{1i}}\overline{\mathbf{h}}^{(b)}_2\cdot\mathbf{g}_{1j}
	\end{array}
	\right]
	=
	\left[
	\begin{array}{c}
	\boldsymbol{\gamma}^{\top}_1 \\
	\boldsymbol{\gamma}^{\top}_2
	\end{array}
	\right]
	\left[
	-\lfloor\mathbf{e}_i\times\rfloor \lfloor\boldsymbol{\beta}_1\times\rfloor\mathbf{e}_j
	\right]							
	\end{eqnarray}
	with the corresponding gradients as:
	\begin{small}
		\begin{eqnarray}
		\nabla\mathcal{L}^2_{\mathbf{g}_0\mathbf{g}_{0}}\overline{\mathbf{h}}^{(b)}
		&=&
		\left[
		\begin{array}{c}
		\nabla\mathcal{L}^2_{\mathbf{g}_0\mathbf{g}_{0}}\overline{\mathbf{h}}^{(b)}_1  \\
		\nabla\mathcal{L}^2_{\mathbf{g}_0\mathbf{g}_{0}}\overline{\mathbf{h}}^{(b)}_2  
		\end{array}
		\right]
		=
		\left[
		\begin{array}{c}
		\frac{\partial \mathcal{L}^2_{\mathbf{g}_0\mathbf{g}_{0}}\overline{\mathbf{h}}^{(b)}_1}{\partial \boldsymbol{\beta}}  \\
		\frac{\partial \mathcal{L}^2_{\mathbf{g}_0\mathbf{g}_{0}}\overline{\mathbf{h}}^{(b)}_2}{\partial \boldsymbol{\beta}}
		\end{array}
		\right]
		\\
		&=&
		\scalemath{0.9}{
		\left[
		\begin{array}{c}
		\boldsymbol{\gamma}^{\top}_1 \\
		\boldsymbol{\gamma}^{\top}_2
		\end{array}
		\right]
		\left[
		\lfloor\boldsymbol{\beta}_2\times\rfloor^2 \quad 
		-\lfloor\boldsymbol{\beta}_2\times\rfloor\lfloor\boldsymbol{\beta}_1\times\rfloor
		+ \lfloor \lfloor\boldsymbol{\beta}_1\times\rfloor   \boldsymbol{\beta}_2\times\rfloor
		+ 2\lfloor\boldsymbol{\beta}_3\times\rfloor \quad 
		-2\lfloor\boldsymbol{\beta}_2\times\rfloor \quad 
		\mathbf{I}_3 \ 
		-\mathbf{I}_3
		\right]
        } 
        \nonumber
		\\
		\nabla\mathcal{L}^2_{\mathbf{g}_0\mathbf{g}_{1i}}\overline{\mathbf{h}}^{(b)}
		&=&
		\left[
		\begin{array}{c}
		\nabla\mathcal{L}^2_{\mathbf{g}_0\mathbf{g}_{1i}}\overline{\mathbf{h}}^{(b)}_1  \\
		\nabla\mathcal{L}^2_{\mathbf{g}_0\mathbf{g}_{1i}}\overline{\mathbf{h}}^{(b)}_2  
		\end{array}
		\right]
		=
		\left[
		\begin{array}{c}
		\frac{\partial \mathcal{L}^2_{\mathbf{g}_0\mathbf{g}_{1i}}\overline{\mathbf{h}}^{(b)}_1}{\partial \boldsymbol{\beta}}  \\
		\frac{\partial \mathcal{L}^2_{\mathbf{g}_0\mathbf{g}_{1i}}\overline{\mathbf{h}}^{(b)}_2}{\partial \boldsymbol{\beta}}
		\end{array}
		\right]
		\\
		&=&
		\left[
		\begin{array}{c}
		\boldsymbol{\gamma}^{\top}_1 \\
		\boldsymbol{\gamma}^{\top}_2
		\end{array}
		\right]
		\left[
		-\lfloor\boldsymbol{\beta}_2\times\rfloor\lfloor\mathbf{e}_i\times\rfloor \quad 
		-\lfloor\lfloor\boldsymbol{\beta}_1\times\rfloor\mathbf{e}_i\times\rfloor \quad 
		\lfloor\mathbf{e}_i\times\rfloor \quad 
		\mathbf{0} \quad  
		\mathbf{0}
		\right]
		\\
		\nabla\mathcal{L}^2_{\mathbf{g}_{1i}\mathbf{g}_{0}}\overline{\mathbf{h}}^{(b)}
		&=&
		\left[
		\begin{array}{c}
		\nabla\mathcal{L}^2_{\mathbf{g}_{1i}\mathbf{g}_{0}}\overline{\mathbf{h}}^{(b)}_1  \\
		\nabla\mathcal{L}^2_{\mathbf{g}_{1i}\mathbf{g}_{0}}\overline{\mathbf{h}}^{(b)}_2  
		\end{array}
		\right]
		=
		\left[
		\begin{array}{c}
		\frac{\partial \mathcal{L}^2_{\mathbf{g}_{1i}\mathbf{g}_{0}}\overline{\mathbf{h}}^{(b)}_1}{\partial \boldsymbol{\beta}}  \\
		\frac{\partial \mathcal{L}^2_{\mathbf{g}_{1i}\mathbf{g}_{0}}\overline{\mathbf{h}}^{(b)}_2}{\partial \boldsymbol{\beta}}
		\end{array}
		\right]
		\\
		&=&
		\left[
		\begin{array}{c}
		\boldsymbol{\gamma}^{\top}_1 \\
		\boldsymbol{\gamma}^{\top}_2
		\end{array}
		\right]
		\left[
		-\lfloor\mathbf{e}_i\times\rfloor\lfloor\boldsymbol{\beta}_2\times\rfloor \quad 
		\lfloor\mathbf{e}_i\times\rfloor\lfloor\boldsymbol{\beta}_1\times\rfloor \quad 
		\lfloor\mathbf{e}_i\times\rfloor \quad 
		\mathbf{0} \quad  
		\mathbf{0}
		\right]
		\\
		\nabla\mathcal{L}^2_{\mathbf{g}_{1i}\mathbf{g}_{1j}}\overline{\mathbf{h}}^{(b)}
		&=&
		\scalemath{0.9}{
		\left[
		\begin{array}{c}
		\nabla\mathcal{L}^2_{\mathbf{g}_{1i}\mathbf{g}_{1j}}\overline{\mathbf{h}}^{(b)}_1  \\
		\nabla\mathcal{L}^2_{\mathbf{g}_{1i}\mathbf{g}_{1j}}\overline{\mathbf{h}}^{(b)}_2  
		\end{array}
		\right]
		=
		\left[
		\begin{array}{c}
		\frac{\partial \mathcal{L}^2_{\mathbf{g}_{1i}\mathbf{g}_{1j}}\overline{\mathbf{h}}^{(b)}_1}{\partial \boldsymbol{\beta}}  \\
		\frac{\partial \mathcal{L}^2_{\mathbf{g}_{1i}\mathbf{g}_{1j}}\overline{\mathbf{h}}^{(b)}_2}{\partial \boldsymbol{\beta}}
		\end{array}
		\right]
		=
		\left[
		\begin{array}{c}
		\boldsymbol{\gamma}^{\top}_1 \\
		\boldsymbol{\gamma}^{\top}_2
		\end{array}
		\right]
		\left[
		\lfloor\mathbf{e}_i\times\rfloor\lfloor\mathbf{e}_j\times\rfloor \quad 
		\mathbf{0} \quad 
		\mathbf{0} \quad 
		\mathbf{0} \quad 
		\mathbf{0}
		\right]	}											
		\end{eqnarray}
	\end{small}	

	\item The third-order Lie derivatives are as following:
	\begin{eqnarray}
	\mathcal{L}^3_{\mathbf{g}_0\mathbf{g}_0\mathbf{g}_{1i}}\overline{\mathbf{h}}^{(b)}
	&=&
	\nabla\mathcal{L}^2_{\mathbf{g}_0\mathbf{g}_0}\overline{\mathbf{h}}^{(b)} \cdot \mathbf{g}_{1i}
	=
	\left[
	\begin{array}{c}
	\nabla\mathcal{L}^2_{\mathbf{g}_0\mathbf{g}_0}\overline{\mathbf{h}}^{(b)}_1\cdot\mathbf{g}_{1i} \\ \nabla\mathcal{L}^2_{\mathbf{g}_0\mathbf{g}_0}\overline{\mathbf{h}}^{(b)}_2\cdot\mathbf{g}_{1i}
	\end{array}
	\right]
	\\
	&=&
	\left[
	\begin{array}{c}
	\boldsymbol{\gamma}^{\top}_1 \\
	\boldsymbol{\gamma}^{\top}_2
	\end{array}
	\right]
	\left[
	\lfloor\boldsymbol{\beta}_2\times\rfloor^2 \lfloor\boldsymbol{\beta}_1\times\rfloor
	-2\lfloor\boldsymbol{\beta}_2\times\rfloor \lfloor\boldsymbol{\beta}_3\times\rfloor
	-\lfloor\boldsymbol{\beta}_5\times\rfloor
	\right]
	\mathbf{e}_i
	\\
	\mathcal{L}^3_{\mathbf{g}_0\mathbf{g}_0\mathbf{g}_{2i}}\overline{\mathbf{h}}^{(b)}
	&=&
	\nabla\mathcal{L}^2_{\mathbf{g}_0\mathbf{g}_0}\overline{\mathbf{h}}^{(b)} \cdot \mathbf{g}_{2i}
	=
	\left[
	\begin{array}{c}
	\nabla\mathcal{L}^2_{\mathbf{g}_0\mathbf{g}_0}\overline{\mathbf{h}}^{(b)}_1\cdot\mathbf{g}_{2i} \\ \nabla\mathcal{L}^2_{\mathbf{g}_0\mathbf{g}_0}\overline{\mathbf{h}}^{(b)}_2\cdot\mathbf{g}_{2i}
	\end{array}
	\right]
	=
	\left[
	\begin{array}{c}
	\boldsymbol{\gamma}^{\top}_1 \\
	\boldsymbol{\gamma}^{\top}_2
	\end{array}
	\right]
	\left[
	-2\lfloor\boldsymbol{\beta}_2\times\rfloor
	\right]
	\mathbf{e}_i	 \\
	\\
	\mathcal{L}^3_{\mathbf{g}_0\mathbf{g}_{1i}\mathbf{g}_{0}}\overline{\mathbf{h}}^{(b)}
	&=&
	\nabla\mathcal{L}^2_{\mathbf{g}_0\mathbf{g}_{1i}}\overline{\mathbf{h}}^{(b)} \cdot \mathbf{g}_{0}
	=
	\left[
	\begin{array}{c}
	\nabla\mathcal{L}^2_{\mathbf{g}_0\mathbf{g}_{1i}}\overline{\mathbf{h}}^{(b)}_1\cdot\mathbf{g}_{0} \\ \nabla\mathcal{L}^2_{\mathbf{g}_0\mathbf{g}_{1i}}\overline{\mathbf{h}}^{(b)}_2\cdot\mathbf{g}_{0}
	\end{array}
	\right]
	\\
	&=&
	\left[
	\begin{array}{c}
	\boldsymbol{\gamma}^{\top}_1 \\
	\boldsymbol{\gamma}^{\top}_2
	\end{array}
	\right]
	\left[
	\lfloor\boldsymbol{\beta}_2\times\rfloor \lfloor\mathbf{e}_i\times\rfloor \lfloor\boldsymbol{\beta}_1\times\rfloor\boldsymbol{\beta}_2 
	+ \lfloor\boldsymbol{\beta}_2\times\rfloor \lfloor\mathbf{e}_i\times\rfloor\boldsymbol{\beta}_3 
	- \lfloor\mathbf{e}_i\times\rfloor \lfloor\boldsymbol{\beta}_3\times\rfloor\boldsymbol{\beta}_2  
	+ \lfloor\mathbf{e}_i\times\rfloor \boldsymbol{\beta}_5
	-\lfloor\mathbf{e}_i\times\rfloor \boldsymbol{\beta}_4  
	\right]		
	\nonumber			
	\end{eqnarray}
	with the corresponding gradients as:
	\begin{small}
		\begin{eqnarray}
		&  &
		\nabla\mathcal{L}^3_{\mathbf{g}_0\mathbf{g}_{0}\mathbf{g}_{1i}}\overline{\mathbf{h}}^{(b)}
		=
		\left[
		\begin{array}{c}
		\nabla\mathcal{L}^3_{\mathbf{g}_0\mathbf{g}_{0}\mathbf{g}_{1i}}\overline{\mathbf{h}}^{(b)}_1  \\
		\nabla\mathcal{L}^3_{\mathbf{g}_0\mathbf{g}_{0}\mathbf{g}_{1i}}\overline{\mathbf{h}}^{(b)}_2  
		\end{array}
		\right]
		=
		\left[
		\begin{array}{c}
		\frac{\partial \mathcal{L}^3_{\mathbf{g}_0\mathbf{g}_{0}\mathbf{g}_{1i}}\overline{\mathbf{h}}^{(b)}_1}{\partial \boldsymbol{\beta}}  \\
		\frac{\partial \mathcal{L}^3_{\mathbf{g}_0\mathbf{g}_{0}\mathbf{g}_{1i}}\overline{\mathbf{h}}^{(b)}_2}{\partial \boldsymbol{\beta}}
		\end{array}
		\right]
		\\
		&=&
		\left[
		\begin{array}{c}
		\boldsymbol{\gamma}^{\top}_1 \\
		\boldsymbol{\gamma}^{\top}_2
		\end{array}
		\right]
		\left[
		\frac{\partial \mathcal{L}^3_{\mathbf{g}_0\mathbf{g}_{0}\mathbf{g}_{1i}}\overline{\mathbf{h}}^{(b)}}{\partial \boldsymbol{\beta}_1} \  
		\frac{\partial \mathcal{L}^3_{\mathbf{g}_0\mathbf{g}_{0}\mathbf{g}_{1i}}\overline{\mathbf{h}}^{(b)}}{\partial \boldsymbol{\beta}_2} \  
		\frac{\partial \mathcal{L}^3_{\mathbf{g}_0\mathbf{g}_{0}\mathbf{g}_{1i}}\overline{\mathbf{h}}^{(b)}}{\partial \boldsymbol{\beta}_3} \  
		\frac{\partial \mathcal{L}^3_{\mathbf{g}_0\mathbf{g}_{0}\mathbf{g}_{1i}}\overline{\mathbf{h}}^{(b)}}{\partial \boldsymbol{\beta}_4} \  
		\frac{\partial \mathcal{L}^3_{\mathbf{g}_0\mathbf{g}_{0}\mathbf{g}_{1i}}\overline{\mathbf{h}}^{(b)}}{\partial \boldsymbol{\beta}_5}
		\right]
		\\
		&  &
		\nabla\mathcal{L}^3_{\mathbf{g}_0\mathbf{g}_{1i}\mathbf{g}_{0}}\overline{\mathbf{h}}^{(b)}
		=
		\left[
		\begin{array}{c}
		\nabla\mathcal{L}^3_{\mathbf{g}_0\mathbf{g}_{1i}\mathbf{g}_{0}}\overline{\mathbf{h}}^{(b)}_1  \\
		\nabla\mathcal{L}^3_{\mathbf{g}_0\mathbf{g}_{1i}\mathbf{g}_{0}}\overline{\mathbf{h}}^{(b)}_2  
		\end{array}
		\right]
		=
		\left[
		\begin{array}{c}
		\frac{\partial \mathcal{L}^3_{\mathbf{g}_0\mathbf{g}_{1i}\mathbf{g}_{0}}\overline{\mathbf{h}}^{(b)}_1}{\partial \boldsymbol{\beta}}  \\
		\frac{\partial \mathcal{L}^3_{\mathbf{g}_0\mathbf{g}_{1i}\mathbf{g}_{0}}\overline{\mathbf{h}}^{(b)}_2}{\partial \boldsymbol{\beta}}
		\end{array}
		\right]
		\\
		&=&
		\left[
		\begin{array}{c}
		\boldsymbol{\gamma}^{\top}_1 \\
		\boldsymbol{\gamma}^{\top}_2
		\end{array}
		\right]
		\left[
		\frac{\partial \mathcal{L}^3_{\mathbf{g}_0\mathbf{g}_{1i}\mathbf{g}_{0}}\overline{\mathbf{h}}^{(b)}}{\partial \boldsymbol{\beta}_1} \  
		\frac{\partial \mathcal{L}^3_{\mathbf{g}_0\mathbf{g}_{1i}\mathbf{g}_{0}}\overline{\mathbf{h}}^{(b)}}{\partial \boldsymbol{\beta}_2} \  
		\frac{\partial \mathcal{L}^3_{\mathbf{g}_0\mathbf{g}_{1i}\mathbf{g}_{0}}\overline{\mathbf{h}}^{(b)}}{\partial \boldsymbol{\beta}_3} \  
		-\lfloor \mathbf{e}_i \rfloor \  
		\lfloor \mathbf{e}_i \rfloor
		\right]
		\\
		&  &
		\scalemath{0.9}{
		\nabla\mathcal{L}^3_{\mathbf{g}_0\mathbf{g}_{0}\mathbf{g}_{2i}}\overline{\mathbf{h}}^{(b)}
		=
		\left[
		\begin{array}{c}
		\nabla\mathcal{L}^3_{\mathbf{g}_0\mathbf{g}_{0}\mathbf{g}_{2i}}\overline{\mathbf{h}}^{(b)}_1  \\
		\nabla\mathcal{L}^3_{\mathbf{g}_0\mathbf{g}_{0}\mathbf{g}_{2i}}\overline{\mathbf{h}}^{(b)}_2  
		\end{array}
		\right]
		=
		\left[
		\begin{array}{c}
		\frac{\partial \mathcal{L}^3_{\mathbf{g}_0\mathbf{g}_{0}\mathbf{g}_{2i}}\overline{\mathbf{h}}^{(b)}_1}{\partial \boldsymbol{\beta}}  \\
		\frac{\partial \mathcal{L}^3_{\mathbf{g}_0\mathbf{g}_{0}\mathbf{g}_{2i}}\overline{\mathbf{h}}^{(b)}_2}{\partial \boldsymbol{\beta}}
		\end{array}
		\right]
		=
		\left[
		\begin{array}{c}
		\boldsymbol{\gamma}^{\top}_1 \\
		\boldsymbol{\gamma}^{\top}_2
		\end{array}
		\right]
		\left[
		\mathbf{0} \quad 
		\mathbf{0} \quad 
		2\lfloor\mathbf{e}_i\times\rfloor \quad 
		\mathbf{0} \quad 
		\mathbf{0}
		\right]	
	    }		
		\end{eqnarray}
	\end{small}
	where in the equations: 
	\begin{small}
		\begin{eqnarray}
		\frac{\partial \mathcal{L}^3_{\mathbf{g}_0\mathbf{g}_{0}\mathbf{g}_{1i}}\overline{\mathbf{h}}^{(b)}}{\partial \boldsymbol{\beta}_1}
		&=&
		\left[
		\begin{array}{c}
		\frac{\partial \mathcal{L}^3_{\mathbf{g}_0\mathbf{g}_{0}\mathbf{g}_{1i}}\overline{\mathbf{h}}^{(b)}_1}{\partial \boldsymbol{\beta}_1}  \\
		\frac{\partial \mathcal{L}^3_{\mathbf{g}_0\mathbf{g}_{0}\mathbf{g}_{1i}}\overline{\mathbf{h}}^{(b)}_2}{\partial \boldsymbol{\beta}_1}
		\end{array}
		\right]
		=
		\left[
		\begin{array}{c}
		\boldsymbol{\gamma}^{\top}_1 \\
		\boldsymbol{\gamma}^{\top}_2
		\end{array}
		\right]
		\left[
		-\lfloor\boldsymbol{\beta}_2\times\rfloor^2\lfloor\mathbf{e}_i\times\rfloor
		\right]
		\\
		\frac{\partial \mathcal{L}^3_{\mathbf{g}_0\mathbf{g}_{0}\mathbf{g}_{1i}}\overline{\mathbf{h}}^{(b)}}{\partial \boldsymbol{\beta}_2}
		&=&
		\left[
		\begin{array}{c}
		\frac{\partial \mathcal{L}^3_{\mathbf{g}_0\mathbf{g}_{0}\mathbf{g}_{1i}}\overline{\mathbf{h}}^{(b)}_1}{\partial \boldsymbol{\beta}_2}  \\
		\frac{\partial \mathcal{L}^3_{\mathbf{g}_0\mathbf{g}_{0}\mathbf{g}_{1i}}\overline{\mathbf{h}}^{(b)}_2}{\partial \boldsymbol{\beta}_2}
		\end{array}
		\right]
		\\
		&=&
		\left[
		\begin{array}{c}
		\boldsymbol{\gamma}^{\top}_1 \\
		\boldsymbol{\gamma}^{\top}_2
		\end{array}
		\right]
		\left[
		-\lfloor\lfloor\boldsymbol{\beta}_2\times\rfloor\lfloor\boldsymbol{\beta}_1\times\rfloor\mathbf{e}_i\times\rfloor
		-\lfloor\boldsymbol{\beta}_2\times\rfloor\lfloor\lfloor\boldsymbol{\beta}_1\times\rfloor\mathbf{e}_i\times\rfloor
		+2\lfloor\lfloor\boldsymbol{\beta}_3\times\rfloor\mathbf{e}_i\times\rfloor
		\right]	
		\\
		\frac{\partial \mathcal{L}^3_{\mathbf{g}_0\mathbf{g}_{0}\mathbf{g}_{1i}}\overline{\mathbf{h}}^{(b)}}{\partial \boldsymbol{\beta}_3}
		&=&
		\left[
		\begin{array}{c}
		\frac{\partial \mathcal{L}^3_{\mathbf{g}_0\mathbf{g}_{0}\mathbf{g}_{1i}}\overline{\mathbf{h}}^{(b)}_1}{\partial \boldsymbol{\beta}_3}  \\
		\frac{\partial \mathcal{L}^3_{\mathbf{g}_0\mathbf{g}_{0}\mathbf{g}_{1i}}\overline{\mathbf{h}}^{(b)}_2}{\partial \boldsymbol{\beta}_3}
		\end{array}
		\right]
		=
		\left[
		\begin{array}{c}
		\boldsymbol{\gamma}^{\top}_1 \\
		\boldsymbol{\gamma}^{\top}_2
		\end{array}
		\right]
		\left[
		2\lfloor\boldsymbol{\beta}_2\times\rfloor\lfloor\mathbf{e}_i\times\rfloor
		\right]
		\\
		\frac{\partial \mathcal{L}^3_{\mathbf{g}_0\mathbf{g}_{0}\mathbf{g}_{1i}}\overline{\mathbf{h}}^{(b)}}{\partial \boldsymbol{\beta}_4}
		&=&
		\left[
		\begin{array}{c}
		\frac{\partial \mathcal{L}^3_{\mathbf{g}_0\mathbf{g}_{0}\mathbf{g}_{1i}}\overline{\mathbf{h}}^{(b)}_1}{\partial \boldsymbol{\beta}_4}  \\
		\frac{\partial \mathcal{L}^3_{\mathbf{g}_0\mathbf{g}_{0}\mathbf{g}_{1i}}\overline{\mathbf{h}}^{(b)}_2}{\partial \boldsymbol{\beta}_4}
		\end{array}
		\right]
		=
		\left[
		\begin{array}{c}
		\mathbf{0} \\
		\mathbf{0}
		\end{array}
		\right]
		\\
		\frac{\partial \mathcal{L}^3_{\mathbf{g}_0\mathbf{g}_{0}\mathbf{g}_{1i}}\overline{\mathbf{h}}^{(b)}}{\partial \boldsymbol{\beta}_5}
		&=&
		\left[
		\begin{array}{c}
		\frac{\partial \mathcal{L}^3_{\mathbf{g}_0\mathbf{g}_{0}\mathbf{g}_{1i}}\overline{\mathbf{h}}^{(b)}_1}{\partial \boldsymbol{\beta}_5}  \\
		\frac{\partial \mathcal{L}^3_{\mathbf{g}_0\mathbf{g}_{0}\mathbf{g}_{1i}}\overline{\mathbf{h}}^{(b)}_2}{\partial \boldsymbol{\beta}_5}
		\end{array}
		\right]
		=
		\left[
		\begin{array}{c}
		\boldsymbol{\gamma}^{\top}_1 \\
		\boldsymbol{\gamma}^{\top}_2
		\end{array}
		\right]
		\left[
		\lfloor\mathbf{e}_i\times\rfloor
		\right]	
		\\
		\frac{\partial \mathcal{L}^3_{\mathbf{g}_0\mathbf{g}_{1i}\mathbf{g}_{0}}\overline{\mathbf{h}}^{(b)}}{\partial \boldsymbol{\beta}_1}
		&=&
		\left[
		\begin{array}{c}
		\frac{\partial \mathcal{L}^3_{\mathbf{g}_0\mathbf{g}_{1i}\mathbf{g}_{0}}\overline{\mathbf{h}}^{(b)}_1}{\partial \boldsymbol{\beta}_1}  \\
		\frac{\partial \mathcal{L}^3_{\mathbf{g}_0\mathbf{g}_{1i}\mathbf{g}_{0}}\overline{\mathbf{h}}^{(b)}_2}{\partial \boldsymbol{\beta}_1}
		\end{array}
		\right]
		=
		\left[
		\begin{array}{c}
		\boldsymbol{\gamma}^{\top}_1 \\
		\boldsymbol{\gamma}^{\top}_2
		\end{array}
		\right]
		\left[
		-\lfloor\boldsymbol{\beta}_2\times\rfloor \lfloor\mathbf{e}_i\times\rfloor \lfloor\boldsymbol{\beta}_2\times\rfloor
		\right]	
		\\
		\frac{\partial \mathcal{L}^3_{\mathbf{g}_0\mathbf{g}_{1i}\mathbf{g}_{0}}\overline{\mathbf{h}}^{(b)}}{\partial \boldsymbol{\beta}_2}
		&=&
		\left[
		\begin{array}{c}
		\frac{\partial \mathcal{L}^3_{\mathbf{g}_0\mathbf{g}_{1i}\mathbf{g}_{0}}\overline{\mathbf{h}}^{(b)}_1}{\partial \boldsymbol{\beta}_1}  \\
		\frac{\partial \mathcal{L}^3_{\mathbf{g}_0\mathbf{g}_{1i}\mathbf{g}_{0}}\overline{\mathbf{h}}^{(b)}_2}{\partial \boldsymbol{\beta}_1}
		\end{array}
		\right]
		=
		\scalemath{0.9}{
		\left[
		\begin{array}{c}
		\boldsymbol{\gamma}^{\top}_1 \\
		\boldsymbol{\gamma}^{\top}_2
		\end{array}
		\right]
		\left[
			\lfloor\boldsymbol{\beta}_2\times\rfloor \lfloor\mathbf{e}_i\times\rfloor \lfloor\boldsymbol{\beta}_1\times\rfloor
			-\lfloor\lfloor\mathbf{e}_i\times\rfloor \lfloor\boldsymbol{\beta}_1\times\rfloor\boldsymbol{\beta}_2\times\rfloor 
			-\lfloor\mathbf{e}_i\times\rfloor \lfloor\boldsymbol{\beta}_3\times\rfloor 
			-\lfloor\lfloor\mathbf{e}_i\times\rfloor \boldsymbol{\beta}_3\times\rfloor
		\right]	
		}
		\nonumber
		\\
		\frac{\partial \mathcal{L}^3_{\mathbf{g}_0\mathbf{g}_{1i}\mathbf{g}_{0}}\overline{\mathbf{h}}^{(b)}}{\partial \boldsymbol{\beta}_3}
		&=&
		\left[
		\begin{array}{c}
		\frac{\partial \mathcal{L}^3_{\mathbf{g}_0\mathbf{g}_{1i}\mathbf{g}_{0}}\overline{\mathbf{h}}^{(b)}_1}{\partial \boldsymbol{\beta}_1}  \\
		\frac{\partial \mathcal{L}^3_{\mathbf{g}_0\mathbf{g}_{1i}\mathbf{g}_{0}}\overline{\mathbf{h}}^{(b)}_2}{\partial \boldsymbol{\beta}_1}
		\end{array}
		\right]
		=
		\left[
		\begin{array}{c}
		\boldsymbol{\gamma}^{\top}_1 \\
		\boldsymbol{\gamma}^{\top}_2
		\end{array}
		\right]
		\left[
		\lfloor\boldsymbol{\beta}_2\times\rfloor  \lfloor\mathbf{e}_i\times\rfloor 
		+ \lfloor\mathbf{e}_i\times\rfloor \lfloor\boldsymbol{\beta}_2\times\rfloor 
		\right]																		
		\end{eqnarray}
	\end{small}	\noindent
\end{itemize}

Therefore, we can construct the $\Xi^{(b)}$ matrix as \eqref{eq_Xi_b}. Note that the diagonal block are all of full column rank and, hence the matrix $\Xi^{(b)}$ is of full column rank.
\begin{equation} \label{eq_Xi_b}
\resizebox{0.9\hsize}{!}
{$
	\Xi^{(b)}
	=
	\begin{bmatrix}
	\nabla \mathcal{L}^1_{\mathbf{g}_{11}}\overline{\mathbf{h}}^{(b)} \\
	\nabla \mathcal{L}^1_{\mathbf{g}_{12}}\overline{\mathbf{h}}^{(b)} \\
	\nabla \mathcal{L}^1_{\mathbf{g}_{13}}\overline{\mathbf{h}}^{(b)} \\
	\nabla \mathcal{L}^3_{\mathbf{g}_0\mathbf{g}_0\mathbf{g}_{21}}\overline{\mathbf{h}}^{(b)} \\
	\nabla \mathcal{L}^3_{\mathbf{g}_0\mathbf{g}_0\mathbf{g}_{22}}\overline{\mathbf{h}}^{(b)} \\
	\nabla \mathcal{L}^3_{\mathbf{g}_0\mathbf{g}_0\mathbf{g}_{23}}\overline{\mathbf{h}}^{(b)} \\	
	\nabla \mathcal{L}^2_{\mathbf{g}_{11}\mathbf{g}_0}\overline{\mathbf{h}}^{(b)} \\
	\nabla \mathcal{L}^2_{\mathbf{g}_{12}\mathbf{g}_0}\overline{\mathbf{h}}^{(b)} \\
	\nabla \mathcal{L}^2_{\mathbf{g}_{13}\mathbf{g}_0}\overline{\mathbf{h}}^{(b)} \\	
	\nabla \mathcal{L}^3_{\mathbf{g}_0\mathbf{g}_{11}\mathbf{g}_0}\overline{\mathbf{h}}^{(b)} \\
	\nabla \mathcal{L}^3_{\mathbf{g}_0\mathbf{g}_{11}\mathbf{g}_0}\overline{\mathbf{h}}^{(b)} \\
	\nabla \mathcal{L}^3_{\mathbf{g}_0\mathbf{g}_{11}\mathbf{g}_0}\overline{\mathbf{h}}^{(b)} \\			
	\nabla \mathcal{L}^3_{\mathbf{g}_0\mathbf{g}_0\mathbf{g}_{11}}\overline{\mathbf{h}}^{(b)} \\
	\nabla \mathcal{L}^3_{\mathbf{g}_0\mathbf{g}_0\mathbf{g}_{12}}\overline{\mathbf{h}}^{(b)} \\
	\nabla \mathcal{L}^3_{\mathbf{g}_0\mathbf{g}_0\mathbf{g}_{13}}\overline{\mathbf{h}}^{(b)} 
	\end{bmatrix}
	=
	\begin{bmatrix}
	-\boldsymbol{\gamma}^{\top}_1 \lfloor\mathbf{e}_1\times\rfloor & \mathbf{0} & \mathbf{0} & \mathbf{0} & \mathbf{0} \\
	-\boldsymbol{\gamma}^{\top}_2 \lfloor\mathbf{e}_1\times\rfloor & \mathbf{0} & \mathbf{0} & \mathbf{0} & \mathbf{0} \\
	-\boldsymbol{\gamma}^{\top}_1 \lfloor\mathbf{e}_2\times\rfloor & \mathbf{0} & \mathbf{0} & \mathbf{0} & \mathbf{0} \\
	-\boldsymbol{\gamma}^{\top}_2 \lfloor\mathbf{e}_2\times\rfloor & \mathbf{0} & \mathbf{0} & \mathbf{0} & \mathbf{0} \\
	-\boldsymbol{\gamma}^{\top}_1 \lfloor\mathbf{e}_3\times\rfloor & \mathbf{0} & \mathbf{0} & \mathbf{0} & \mathbf{0} \\
	-\boldsymbol{\gamma}^{\top}_2 \lfloor\mathbf{e}_3\times\rfloor & \mathbf{0} & \mathbf{0} & \mathbf{0} & \mathbf{0} \\
	\mathbf{0} &  2\boldsymbol{\gamma}^{\top}_1 \lfloor\mathbf{e}_1\times\rfloor & \mathbf{0} & \mathbf{0} & \mathbf{0} \\
	\mathbf{0} &  2\boldsymbol{\gamma}^{\top}_2 \lfloor\mathbf{e}_1\times\rfloor &  \mathbf{0} & \mathbf{0} & \mathbf{0} \\
	\mathbf{0} &  2\boldsymbol{\gamma}^{\top}_1 \lfloor\mathbf{e}_2\times\rfloor &  \mathbf{0} & \mathbf{0} & \mathbf{0} \\
	\mathbf{0} &  2\boldsymbol{\gamma}^{\top}_2 \lfloor\mathbf{e}_2\times\rfloor &  \mathbf{0} & \mathbf{0} & \mathbf{0} \\
	\mathbf{0} &  2\boldsymbol{\gamma}^{\top}_1 \lfloor\mathbf{e}_3\times\rfloor &  \mathbf{0} & \mathbf{0} & \mathbf{0} \\
	\mathbf{0} &  2\boldsymbol{\gamma}^{\top}_2 \lfloor\mathbf{e}_3\times\rfloor &  \mathbf{0} & \mathbf{0} & \mathbf{0} \\	
	-\boldsymbol{\gamma}^{\top}_1 \lfloor\mathbf{e}_1\times\rfloor \lfloor\boldsymbol{\beta}_2\times\rfloor  & 
	\boldsymbol{\gamma}^{\top}_1 \lfloor\mathbf{e}_1\times\rfloor\lfloor\boldsymbol{\beta}_1\times\rfloor &
	\boldsymbol{\gamma}^{\top}_1 \lfloor\mathbf{e}_1\times\rfloor & \mathbf{0} & \mathbf{0} \\
	-\boldsymbol{\gamma}^{\top}_2 \lfloor\mathbf{e}_1\times\rfloor \lfloor\boldsymbol{\beta}_2\times\rfloor & 
	\boldsymbol{\gamma}^{\top}_2 \lfloor\mathbf{e}_1\times\rfloor\lfloor\boldsymbol{\beta}_1\times\rfloor &
	\boldsymbol{\gamma}^{\top}_2 \lfloor\mathbf{e}_1\times\rfloor & \mathbf{0} & \mathbf{0} \\	
	-\boldsymbol{\gamma}^{\top}_1 \lfloor\mathbf{e}_2\times\rfloor \lfloor\boldsymbol{\beta}_2\times\rfloor  & 
	\boldsymbol{\gamma}^{\top}_1 \lfloor\mathbf{e}_2\times\rfloor\lfloor\boldsymbol{\beta}_1\times\rfloor &
	\boldsymbol{\gamma}^{\top}_1 \lfloor\mathbf{e}_2\times\rfloor & \mathbf{0} & \mathbf{0} \\
	-\boldsymbol{\gamma}^{\top}_2 \lfloor\mathbf{e}_2\times\rfloor \lfloor\boldsymbol{\beta}_2\times\rfloor & 
	\boldsymbol{\gamma}^{\top}_2 \lfloor\mathbf{e}_2\times\rfloor\lfloor\boldsymbol{\beta}_1\times\rfloor &
	\boldsymbol{\gamma}^{\top}_2 \lfloor\mathbf{e}_2\times\rfloor & \mathbf{0} & \mathbf{0} \\	
	-\boldsymbol{\gamma}^{\top}_1 \lfloor\mathbf{e}_3\times\rfloor \lfloor\boldsymbol{\beta}_2\times\rfloor  & 
	\boldsymbol{\gamma}^{\top}_1 \lfloor\mathbf{e}_3\times\rfloor\lfloor\boldsymbol{\beta}_1\times\rfloor &
	\boldsymbol{\gamma}^{\top}_1 \lfloor\mathbf{e}_3\times\rfloor & \mathbf{0} & \mathbf{0} \\
	-\boldsymbol{\gamma}^{\top}_2 \lfloor\mathbf{e}_3\times\rfloor \lfloor\boldsymbol{\beta}_2\times\rfloor & 
	\boldsymbol{\gamma}^{\top}_2 \lfloor\mathbf{e}_3\times\rfloor\lfloor\boldsymbol{\beta}_1\times\rfloor &
	\boldsymbol{\gamma}^{\top}_2 \lfloor\mathbf{e}_3\times\rfloor & \mathbf{0} & \mathbf{0} \\
	\frac{\partial \mathcal{L}^3_{\mathbf{g}_0\mathbf{g}_{11}\mathbf{g}_{0}}\overline{\mathbf{h}}^{(b)}_1}{\partial \boldsymbol{\beta}_1} &
	\frac{\partial \mathcal{L}^3_{\mathbf{g}_0\mathbf{g}_{11}\mathbf{g}_{0}}\overline{\mathbf{h}}^{(b)}_1}{\partial \boldsymbol{\beta}_2} &
	\frac{\partial \mathcal{L}^3_{\mathbf{g}_0\mathbf{g}_{11}\mathbf{g}_{0}}\overline{\mathbf{h}}^{(b)}_1}{\partial \boldsymbol{\beta}_3} &
	-\boldsymbol{\gamma}^{\top}_1 \lfloor\mathbf{e}_1\times\rfloor &
	\boldsymbol{\gamma}^{\top}_1 \lfloor\mathbf{e}_1\times\rfloor  \\
	\frac{\partial \mathcal{L}^3_{\mathbf{g}_0\mathbf{g}_{11}\mathbf{g}_{0}}\overline{\mathbf{h}}^{(b)}_2}{\partial \boldsymbol{\beta}_1} &
	\frac{\partial \mathcal{L}^3_{\mathbf{g}_0\mathbf{g}_{11}\mathbf{g}_{0}}\overline{\mathbf{h}}^{(b)}_2}{\partial \boldsymbol{\beta}_2} &
	\frac{\partial \mathcal{L}^3_{\mathbf{g}_0\mathbf{g}_{11}\mathbf{g}_{0}}\overline{\mathbf{h}}^{(b)}_2}{\partial \boldsymbol{\beta}_3} &
	-\boldsymbol{\gamma}^{\top}_2 \lfloor\mathbf{e}_1\times\rfloor &
	\boldsymbol{\gamma}^{\top}_2 \lfloor\mathbf{e}_1\times\rfloor  \\
	\frac{\partial \mathcal{L}^3_{\mathbf{g}_0\mathbf{g}_{12}\mathbf{g}_{0}}\overline{\mathbf{h}}^{(b)}_1}{\partial \boldsymbol{\beta}_1} &
	\frac{\partial \mathcal{L}^3_{\mathbf{g}_0\mathbf{g}_{12}\mathbf{g}_{0}}\overline{\mathbf{h}}^{(b)}_1}{\partial \boldsymbol{\beta}_2} &
	\frac{\partial \mathcal{L}^3_{\mathbf{g}_0\mathbf{g}_{12}\mathbf{g}_{0}}\overline{\mathbf{h}}^{(b)}_1}{\partial \boldsymbol{\beta}_3} &
	-\boldsymbol{\gamma}^{\top}_1 \lfloor\mathbf{e}_2\times\rfloor &
	\boldsymbol{\gamma}^{\top}_1 \lfloor\mathbf{e}_2\times\rfloor  \\
	\frac{\partial \mathcal{L}^3_{\mathbf{g}_0\mathbf{g}_{12}\mathbf{g}_{0}}\overline{\mathbf{h}}^{(b)}_2}{\partial \boldsymbol{\beta}_1} &
	\frac{\partial \mathcal{L}^3_{\mathbf{g}_0\mathbf{g}_{12}\mathbf{g}_{0}}\overline{\mathbf{h}}^{(b)}_2}{\partial \boldsymbol{\beta}_2} &
	\frac{\partial \mathcal{L}^3_{\mathbf{g}_0\mathbf{g}_{12}\mathbf{g}_{0}}\overline{\mathbf{h}}^{(b)}_2}{\partial \boldsymbol{\beta}_3} &
	-\boldsymbol{\gamma}^{\top}_2 \lfloor\mathbf{e}_2\times\rfloor &
	\boldsymbol{\gamma}^{\top}_2 \lfloor\mathbf{e}_2\times\rfloor  \\
	\frac{\partial \mathcal{L}^3_{\mathbf{g}_0\mathbf{g}_{13}\mathbf{g}_{0}}\overline{\mathbf{h}}^{(b)}_1}{\partial \boldsymbol{\beta}_1} &
	\frac{\partial \mathcal{L}^3_{\mathbf{g}_0\mathbf{g}_{13}\mathbf{g}_{0}}\overline{\mathbf{h}}^{(b)}_1}{\partial \boldsymbol{\beta}_2} &
	\frac{\partial \mathcal{L}^3_{\mathbf{g}_0\mathbf{g}_{13}\mathbf{g}_{0}}\overline{\mathbf{h}}^{(b)}_1}{\partial \boldsymbol{\beta}_3} &
	-\boldsymbol{\gamma}^{\top}_1 \lfloor\mathbf{e}_3\times\rfloor &
	\boldsymbol{\gamma}^{\top}_1 \lfloor\mathbf{e}_3\times\rfloor  \\
	\frac{\partial \mathcal{L}^3_{\mathbf{g}_0\mathbf{g}_{13}\mathbf{g}_{0}}\overline{\mathbf{h}}^{(b)}_2}{\partial \boldsymbol{\beta}_1} &
	\frac{\partial \mathcal{L}^3_{\mathbf{g}_0\mathbf{g}_{13}\mathbf{g}_{0}}\overline{\mathbf{h}}^{(b)}_2}{\partial \boldsymbol{\beta}_2} &
	\frac{\partial \mathcal{L}^3_{\mathbf{g}_0\mathbf{g}_{13}\mathbf{g}_{0}}\overline{\mathbf{h}}^{(b)}_2}{\partial \boldsymbol{\beta}_3} &
	-\boldsymbol{\gamma}^{\top}_2 \lfloor\mathbf{e}_3\times\rfloor &
	\boldsymbol{\gamma}^{\top}_2 \lfloor\mathbf{e}_3\times\rfloor  \\										
	\frac{\partial \mathcal{L}^3_{\mathbf{g}_0\mathbf{g}_{0}\mathbf{g}_{11}}\overline{\mathbf{h}}^{(b)}_1}{\partial \boldsymbol{\beta}_1} &
	\frac{\partial \mathcal{L}^3_{\mathbf{g}_0\mathbf{g}_{0}\mathbf{g}_{11}}\overline{\mathbf{h}}^{(b)}_1}{\partial \boldsymbol{\beta}_2} &
	\frac{\partial \mathcal{L}^3_{\mathbf{g}_0\mathbf{g}_{0}\mathbf{g}_{11}}\overline{\mathbf{h}}^{(b)}_1}{\partial \boldsymbol{\beta}_3} &
	\mathbf{0} &
	\boldsymbol{\gamma}^{\top}_1 \lfloor\mathbf{e}_1\times\rfloor  \\
	\frac{\partial \mathcal{L}^3_{\mathbf{g}_0\mathbf{g}_{0}\mathbf{g}_{11}}\overline{\mathbf{h}}^{(b)}_2}{\partial \boldsymbol{\beta}_1} &
	\frac{\partial \mathcal{L}^3_{\mathbf{g}_0\mathbf{g}_{0}\mathbf{g}_{11}}\overline{\mathbf{h}}^{(b)}_2}{\partial \boldsymbol{\beta}_2} &
	\frac{\partial \mathcal{L}^3_{\mathbf{g}_0\mathbf{g}_{0}\mathbf{g}_{11}}\overline{\mathbf{h}}^{(b)}_2}{\partial \boldsymbol{\beta}_3} &
	\mathbf{0} &
	\boldsymbol{\gamma}^{\top}_2 \lfloor\mathbf{e}_1\times\rfloor  \\
	\frac{\partial \mathcal{L}^3_{\mathbf{g}_0\mathbf{g}_{0}\mathbf{g}_{12}}\overline{\mathbf{h}}^{(b)}_1}{\partial \boldsymbol{\beta}_1} &
	\frac{\partial \mathcal{L}^3_{\mathbf{g}_0\mathbf{g}_{0}\mathbf{g}_{12}}\overline{\mathbf{h}}^{(b)}_1}{\partial \boldsymbol{\beta}_2} &
	\frac{\partial \mathcal{L}^3_{\mathbf{g}_0\mathbf{g}_{0}\mathbf{g}_{12}}\overline{\mathbf{h}}^{(b)}_1}{\partial \boldsymbol{\beta}_3} &
	\mathbf{0} &
	\boldsymbol{\gamma}^{\top}_1 \lfloor\mathbf{e}_2\times\rfloor  \\
	\frac{\partial \mathcal{L}^3_{\mathbf{g}_0\mathbf{g}_{0}\mathbf{g}_{12}}\overline{\mathbf{h}}^{(b)}_2}{\partial \boldsymbol{\beta}_1} &
	\frac{\partial \mathcal{L}^3_{\mathbf{g}_0\mathbf{g}_{0}\mathbf{g}_{12}}\overline{\mathbf{h}}^{(b)}_2}{\partial \boldsymbol{\beta}_2} &
	\frac{\partial \mathcal{L}^3_{\mathbf{g}_0\mathbf{g}_{0}\mathbf{g}_{12}}\overline{\mathbf{h}}^{(b)}_2}{\partial \boldsymbol{\beta}_3} &
	\mathbf{0} &
	\boldsymbol{\gamma}^{\top}_2 \lfloor\mathbf{e}_2\times\rfloor  \\
	\frac{\partial \mathcal{L}^3_{\mathbf{g}_0\mathbf{g}_{0}\mathbf{g}_{13}}\overline{\mathbf{h}}^{(b)}_1}{\partial \boldsymbol{\beta}_1} &
	\frac{\partial \mathcal{L}^3_{\mathbf{g}_0\mathbf{g}_{0}\mathbf{g}_{13}}\overline{\mathbf{h}}^{(b)}_1}{\partial \boldsymbol{\beta}_2} &
	\frac{\partial \mathcal{L}^3_{\mathbf{g}_0\mathbf{g}_{0}\mathbf{g}_{13}}\overline{\mathbf{h}}^{(b)}_1}{\partial \boldsymbol{\beta}_3} &
	\mathbf{0} &
	\boldsymbol{\gamma}^{\top}_1 \lfloor\mathbf{e}_3\times\rfloor  \\
	\frac{\partial \mathcal{L}^3_{\mathbf{g}_0\mathbf{g}_{0}\mathbf{g}_{13}}\overline{\mathbf{h}}^{(b)}_2}{\partial \boldsymbol{\beta}_1} &
	\frac{\partial \mathcal{L}^3_{\mathbf{g}_0\mathbf{g}_{0}\mathbf{g}_{13}}\overline{\mathbf{h}}^{(b)}_2}{\partial \boldsymbol{\beta}_2} &
	\frac{\partial \mathcal{L}^3_{\mathbf{g}_0\mathbf{g}_{0}\mathbf{g}_{13}}\overline{\mathbf{h}}^{(b)}_2}{\partial \boldsymbol{\beta}_3} &
	\mathbf{0} &
	\boldsymbol{\gamma}^{\top}_2 \lfloor\mathbf{e}_3\times\rfloor  										
	\end{bmatrix}
	$}
\end{equation}
%

%% file: sections/ack.tex
\section*{Acknowledgments}

This work was partially supported by the University of Delaware College of Engineering, 
the NSF (IIS-1566129), and the DTRA (HDTRA1-16-1-0039).

%% file: main.bbl
\begin{thebibliography}{10}
\providecommand{\url}[1]{#1}
\csname url@samestyle\endcsname
\providecommand{\newblock}{\relax}
\providecommand{\bibinfo}[2]{#2}
\providecommand{\BIBentrySTDinterwordspacing}{\spaceskip=0pt\relax}
\providecommand{\BIBentryALTinterwordstretchfactor}{4}
\providecommand{\BIBentryALTinterwordspacing}{\spaceskip=\fontdimen2\font plus
\BIBentryALTinterwordstretchfactor\fontdimen3\font minus
  \fontdimen4\font\relax}
\providecommand{\BIBforeignlanguage}[2]{{%
\expandafter\ifx\csname l@#1\endcsname\relax
\typeout{** WARNING: IEEEtran.bst: No hyphenation pattern has been}%
\typeout{** loaded for the language `#1'. Using the pattern for}%
\typeout{** the default language instead.}%
\else
\language=\csname l@#1\endcsname
\fi
#2}}
\providecommand{\BIBdecl}{\relax}
\BIBdecl

\bibitem{Chatfield1997}
A.~B. Chatfield, \emph{Fundamentals of High Accuracy Inertial
  Navigation}.\hskip 1em plus 0.5em minus 0.4em\relax Reston, VA: American
  Institute of Aeronautics and Astronautics, Inc., 1997.

\bibitem{Mourikis2007ICRA}
A.~I. Mourikis and S.~I. Roumeliotis, ``A multi-state constraint {K}alman
  filter for vision-aided inertial navigation,'' in \emph{International
  Conference on Robotics and Automation}, Rome, Italy, Apr. 10--14, 2007, pp.
  3565--3572.

\bibitem{Mourikis2009TRO}
A.~Mourikis, N.~Trawny, S.~Roumeliotis, A.~Johnson, A.~Ansar, and L.~Matthies,
  ``Vision-aided inertial navigation for spacecraft entry, descent, and
  landing,'' \emph{IEEE Transactions on Robotics}, vol.~25, no.~2, pp. 264
  --280, Apr. 2009.

\bibitem{Zhang2017ICRA}
J.~Zhang and S.~Singh, ``Enabling aggressive motion estimation at low-drift and
  accurate mapping in real-time,'' in \emph{International Conference on
  Robotics and Automation}.\hskip 1em plus 0.5em minus 0.4em\relax Singapore:
  IEEE, May. 2017, pp. 5051--5058.

\bibitem{Geneva2018IROS}
P.~Geneva, K.~Eckenhoff, Y.~Yang, and G.~Huang, ``{LIPS}: Lidar-inertial 3d
  plane slam,'' in \emph{Proc. IEEE/RSJ International Conference on Intelligent
  Robots and Systems}, Madrid, Spain, Oct. 2018, (in review).

\bibitem{Yang2017icra}
Y.~Yang and G.~Huang, ``Acoustic-inertial underwater navigation,'' in
  \emph{Proc. of the IEEE International Conference on Robotics and Automation},
  Singapore, May 29--Jun.3, 2017.

\bibitem{Hesch2013TRO}
J.~Hesch, D.~Kottas, S.~Bowman, and S.~Roumeliotis, ``Consistency analysis and
  improvement of vision-aided inertial navigation,'' \emph{{IEEE} Transactions
  on Robotics}, vol.~PP, no.~99, pp. 1--19, 2013.

\bibitem{Hesch2014IJRR}
------, ``Camera-{IMU}-based localization: Observability analysis and
  consistency improvement,'' \emph{International Journal of Robotics Research},
  vol.~33, pp. 182--201, 2014.

\bibitem{Li2013IJRR}
M.~Li and A.~I. Mourikis, ``High-precision, consistent {EKF}-based
  visual-inertial odometry,'' \emph{International Journal of Robotics
  Research}, vol.~32, no.~6, pp. 690--711, 2013.

\bibitem{Leutenegger2014IJRR}
S.~Leutenegger, S.~Lynen, M.~Bosse, R.~Siegwart, and P.~Furgale,
  ``Keyframe-based visual-inertial odometry using nonlinear optimization,''
  \emph{International Journal of Robotics Research}, Dec. 2014.

\bibitem{Huang2014ICRA}
G.~Huang, M.~Kaess, and J.~Leonard, ``Towards consistent visual-inertial
  navigation,'' in \emph{IEEE International Conference on Robotics and
  Automation}, Hong Kong, China, May 31-Jun. 7 2014, pp. 4926--4933.

\bibitem{Qin2017TRO}
T.~Qin, P.~Li, and S.~Shen, ``Vins-mono: A robust and versatile monocular
  visual-inertial state estimator,'' \emph{arXiv preprint arXiv:1708.03852},
  2017.

\bibitem{Mur2017ICRA}
R.~Mur-Artal and J.~D. Tard{\'o}s, ``Visual-inertial monocular slam with map
  reuse,'' \emph{IEEE Robotics and Automation Letters}, vol.~2, no.~2, pp.
  796--803, 2017.

\bibitem{Huang2010IJRR}
G.~Huang, A.~I. Mourikis, and S.~I. Roumeliotis, ``Observability-based rules
  for designing consistent {EKF SLAM} estimators,'' \emph{International Journal
  of Robotics Research}, vol.~29, no.~5, pp. 502--528, Apr. 2010.

\bibitem{Huang2015ISRR}
G.~Huang, K.~Eckenhoff, and J.~Leonard, ``Optimal-state-constraint {EKF} for
  visual-inertial navigation,'' in \emph{Proc. of the International Symposium
  on Robotics Research}, Sestri Levante, Italy, Sep. 12-15 2015.

\bibitem{Zhang2017ICRAa}
T.~Zhang, K.~Wu, J.~Song, S.~Huang, and G.~Dissanayake, ``Convergence and
  consistency analysis for a 3-d invariant-ekf slam,'' \emph{IEEE Robotics and
  Automation Letters}, vol.~2, no.~2, pp. 733--740, April 2017.

\bibitem{Huang2012thesis}
\BIBentryALTinterwordspacing
G.~Huang, ``Improving the consistency of nonlinear estimators: Analysis,
  algorithms, and applications,'' Ph.D. dissertation, Department of Computer
  Science and Engineering, University of Minnesota, 2012. [Online]. Available:
  \url{http://people.csail.mit.edu/ghuang/paper/thesis.pdf}
\BIBentrySTDinterwordspacing

\bibitem{Martinelli2012TRO}
A.~Martinelli, ``Vision and {IMU} data fusion: Closed-form solutions for
  attitude, speed, absolute scale, and bias determination,'' \emph{{IEEE}
  Transactions on Robotics}, vol.~28, no.~1, pp. 44--60, 2012.

\bibitem{Kottas2012ISER}
D.~G. Kottas, J.~A. Hesch, S.~L. Bowman, and S.~I. Roumeliotis, ``On the
  consistency of vision-aided inertial navigation,'' in \emph{International
  Symposium on Experimental Robotics}, Quebec City, Canada, Jun. 17--20, 2012.

\bibitem{Guo2013ICRA}
C.~Guo and S.~Roumeliotis, ``{IMU-RGBD} camera {3D} pose estimation and
  extrinsic calibration: Observability analysis and consistency improvement,''
  in \emph{International Conference on Robotics and Automation}, Karlsruhe,
  Germany, May 6--10, 2013.

\bibitem{Yang2018ICRA}
Y.~Yang and G.~Huang, ``Aided inertial navigation with geometric features:
  Observability analysis,'' in \emph{Proc. of the IEEE International Conference
  on Robotics and Automation}, Brisbane, Australia, May 21--25, 2018.

\bibitem{Yang2018IROS}
------, ``Observability analysis for aided inertial navigation with
  combinations of point, line and plane features,'' in \emph{Proc. IEEE/RSJ
  International Conference on Intelligent Robots and Systems}, Madrid, Spain,
  Oct. 2018, (in review).

\bibitem{Martinelli2013IROS}
A.~Martinelli, ``Visual-inertial structure from motion: Observability and
  resolvability,'' in \emph{Proc. of the {IEEE/RSJ} International Conference on
  Intelligent Robots and Systems}, Tokyo, Japan, Nov. 2013, pp. 4235--4242.

\bibitem{Wu2017ICRA}
K.~J. Wu, C.~X. Guo, G.~Georgiou, and S.~I. Roumeliotis, ``Vins on wheels,'' in
  \emph{International Conference on Robotics and Automation}.\hskip 1em plus
  0.5em minus 0.4em\relax Singapore: IEEE, May. 2017, pp. 5155--5162.

\bibitem{Farrell2008}
J.~Farrell, \emph{Aided Navigation: GPS with High Rate Sensors}, 1st~ed.\hskip
  1em plus 0.5em minus 0.4em\relax New York, NY, USA: McGraw-Hill, Inc., 2008.

\bibitem{Li2013ICRAb}
M.~Li, B.~Kim, and A.~I. Mourikis, ``Real-time motion estimation on a cellphone
  using inertial sensing and a rolling-shutter camera,'' in \emph{Proc. of the
  {IEEE} International Conference on Robotics and Automation}, Karlsruhe,
  Germany, May 6--10, 2013, pp. 4697--4704.

\bibitem{Guo2014RSS}
C.~Guo, D.~Kottas, R.~DuToit, A.~Ahmed, R.~Li, and S.~Roumeliotis, ``Efficient
  visual-inertial navigation using a rolling-shutter camera with inaccurate
  timestamps,'' in \emph{Proc. of the Robotics: Science and Systems
  Conference}, Berkeley, CA, Jul. 13--17, 2014.

\bibitem{Shen2013ICRA}
S.~Shen, Y.~Mulgaonkar, N.~Michael, and V.~Kumar, ``Vision-based state
  estimation for autonomous rotorcraft {MAVs} in complex environments,'' in
  \emph{Proc. of the {IEEE} International Conference on Robotics and
  Automation}, Karlsruhe, Germany, May 6--10, 2013, pp. 1750--1756.

\bibitem{Kottas2013icra}
D.~G. Kottas and S.~I. Roumeliotis, ``Efficient and consistent vision-aided
  inertial navigation using line observations,'' in \emph{International
  Conference on Robotics and Automation}, Karlsruhe, Germany, May 6-10 2013,
  pp. 1540--1547.

\bibitem{Yu2015IROS}
H.~Yu and A.~I. Mourikis, ``Vision-aided inertial navigation with line features
  and a rolling-shutter camera,'' in \emph{International Conference on Robotics
  and Intelligent Systems}, Hamburg, Germany, October 2015, pp. 892--899.

\bibitem{Zheng2018ICRA}
F.~{Zheng}, G.~{Tsai}, Z.~{Zhang}, S.~{Liu}, C.-C. {Chu}, and H.~{Hu},
  ``{PI-VIO: Robust and Efficient Stereo Visual Inertial Odometry using Points
  and Lines},'' \emph{ArXiv e-prints}, Mar. 2018.

\bibitem{He2018Sensors}
Y.~He, J.~Zhao, Y.~Guo, W.~He, , and K.~Yuan, ``Pl vio: Tightly coupled
  monocular visual inertial odometry using point and line features,''
  \emph{Sensors}, 2018.

\bibitem{Guo2016ICRA}
C.~X. Guo, K.~Sartipi, R.~C. DuToit, G.~A. Georgiou, R.~Li, J.~O'Leary, E.~D.
  Nerurkar, J.~A. Hesch, and S.~I. Roumeliotis, ``Large-scale cooperative 3d
  visual-inertial mapping in a manhattan world,'' in \emph{IEEE International
  Conference on Robotics and Automation (ICRA)}.\hskip 1em plus 0.5em minus
  0.4em\relax Stockholm, Sweden: IEEE, May. 16--21 2016, pp. 1071--1078.

\bibitem{Guo2013iros}
C.~X. Guo and S.~I. Roumeliotis, ``Imu-rgbd camera navigation using point and
  plane features,'' in \emph{International Conference on Intelligent Robots and
  Systems}, Tokyo, Japan, Nov. 3--7, 2013, pp. 3164--3171.

\bibitem{Hesch2010ICRA}
J.~A. Hesch, F.~M. Mirzaei, G.~L. Mariottini, and S.~I. Roumeliotis, ``A
  laser-aided inertial navigation system (l-ins) for human localization in
  unknown indoor environments,'' in \emph{International Conference on Robotics
  and Automation}, Anchorage, Alaska, May. 3 - 8 2010, pp. 5376--5382.

\bibitem{Kaess2015icra}
M.~Kaess, ``Simultaneous localization and mapping with infinite planes,'' in
  \emph{2015 IEEE International Conference on Robotics and Automation (ICRA)},
  Seattle, Washington, May 2015, pp. 4605--4611.

\bibitem{Bartoli2005CVIU}
A.~Bartoli and P.~Sturm, ``Structure from motion using lines: Representation,
  triangulation and bundle adjustment,'' \emph{Computer Vision and Image
  Understanding}, vol. 100, no.~3, pp. 416--441, dec 2005.

\bibitem{Servant2008ICPR}
F.~Servant, E.~Marchand, P.~Houlier, and I.~Marchal, ``Visual planes-based
  simultaneous localization and model refinement for augmented reality,'' in
  \emph{International Conference on Pattern Recognition}, Tampa, Florida, Dec
  2008, pp. 1--4.

\bibitem{Huang2008ICRA}
G.~Huang, A.~I. Mourikis, and S.~I. Roumeliotis, ``Analysis and improvement of
  the consistency of extended {K}alman filter-based {SLAM},'' in \emph{Proc. of
  the IEEE International Conference on Robotics and Automation}, Pasadena, CA,
  May 19-23 2008, pp. 473--479.

\bibitem{Huang2008ISER}
------, ``A first-estimates {J}acobian {EKF} for improving {SLAM}
  consistency,'' in \emph{Proc. of the 11th International Symposium on
  Experimental Robotics}, Athens, Greece, Jul. 14--17, 2008.

\bibitem{Huang2009ICRA}
------, ``On the complexity and consistency of {UKF}-based {SLAM},'' in
  \emph{Proc. of the IEEE International Conference on Robotics and Automation},
  Kobe, Japan, May 12-17 2009, pp. 4401--4408.

\bibitem{Huang2013TRO}
------, ``A quadratic-complexity observability-constrained unscented {Kalman}
  filter for {SLAM},'' \emph{IEEE Transactions on Robotics}, vol.~29, no.~5,
  pp. 1226--1243, Oct. 2013.

\bibitem{Huang2011AURO}
G.~Huang, N.~Trawny, A.~I. Mourikis, and S.~I. Roumeliotis,
  ``Observability-based consistent {EKF} estimators for multi-robot cooperative
  localization,'' \emph{Autonomous Robots}, vol.~30, no.~1, pp. 99--122, Jan.
  2011.

\bibitem{Huang2011IROS}
G.~Huang, A.~I. Mourikis, and S.~I. Roumeliotis, ``An observability constrained
  sliding window filter for {SLAM},'' in \emph{International Conference on
  Intelligent Robots and Systems}, San Francisco, CA, Sep. 25-30 2011, pp.
  65--72.

\bibitem{Jones2011IJRR}
E.~S. Jones and S.~Soatto, ``Visual-inertial navigation, mapping and
  localization: A scalable real-time causal approach,'' \emph{International
  Journal of Robotics Research}, vol.~30, no.~4, pp. 407--430, Apr. 2011.

\bibitem{Hernandez2015ICRA}
J.~Hernandez, K.~Tsotsos, and S.~Soatto, ``Observability, identifiability and
  sensitivity of vision-aided inertial navigation,'' in \emph{2015 IEEE
  International Conference on Robotics and Automation (ICRA)}.\hskip 1em plus
  0.5em minus 0.4em\relax Seattle, Washington: IEEE, 2015, pp. 2319--2325.

\bibitem{Martinelli2011TRO}
A.~Martinelli, ``State estimation based on the concept of continuous symmetry
  and observability analysis: The case of calibration,'' \emph{IEEE
  Transactions on Robotics}, vol.~27, no.~2, pp. 239--255, 2011.

\bibitem{Martinelli2014ICRA}
------, ``Visual-inertial structure from motion: Observability vs minimum
  number of sensors,'' in \emph{2014 IEEE International Conference on Robotics
  and Automation (ICRA)}, Hong Kong, China, May 2014, pp. 1020--1027.

\bibitem{Martinelli2018CoRR}
------, ``Closed-form solution to cooperative visual-inertial structure from
  motion,'' \emph{CoRR}, vol. abs/1802.08515, 2018.

\bibitem{Martinelli2017BOOK}
------, \emph{Nonlinear Unknown Input Observability: The General Analytic
  Solution}.\hskip 1em plus 0.5em minus 0.4em\relax arXiv:1704.03252, 2017.

\bibitem{Martinelli2018TAC}
------, ``Nonlinear unknown input observability: Extension of the observability
  rank condition,'' \emph{IEEE Transactions on Automatic Control}, pp. 1--1,
  2018.

\bibitem{Hermann1977TAC}
R.~Hermann and A.~Krener, ``Nonlinear controllability and observability,''
  \emph{IEEE Transactions on Automatic Control}, vol.~22, no.~5, pp. 728 --
  740, Oct. 1977.

\bibitem{Mirzaei2007IROS}
F.~M. Mirzaei and S.~I. Roumeliotis, ``A {K}alman filter-based algorithm for
  {IMU}-camera calibration,'' in \emph{Proc. of the IEEE/RSJ International
  Conference on Intelligent Robots and Systems}, San Diego, CA, Oct. 29 - Nov.
  2 2007, pp. 2427--2434.

\bibitem{Kelly2011IJRR}
J.~Kelly and G.~S. Sukhatme, ``Visual-inertial sensor fusion: Localization,
  mapping and sensor-to-sensor self-calibration,'' \emph{International Journal
  of Robotics Research}, vol.~30, no.~1, pp. 56--79, Jan. 2011.

\bibitem{Panahandeh2013IROS}
G.~Panahandeh, C.~X. Guo, M.~Jansson, and S.~I. Roumeliotis, ``Observability
  analysis of a vision-aided inertial navigation system using planar features
  on the ground,'' in \emph{International Conference on Intelligent Robots and
  Systems}.\hskip 1em plus 0.5em minus 0.4em\relax Tokyo, Japan: IEEE, 2013,
  pp. 4187--4194.

\bibitem{Panahandeh2016JIRS}
G.~Panahandeh, S.~Hutchinson, P.~H{\"a}ndel, and M.~Jansson, ``Planar-based
  visual inertial navigation: Observability analysis and motion estimation,''
  \emph{Journal of Intelligent {\&} Robotic Systems}, vol.~82, no.~2, pp.
  277--299, May 2016.

\bibitem{Li2012ICRA}
M.~Li and A.~I. Mourikis, ``Improving the accuracy of {EKF}-based
  visual-inertial odometry,'' in \emph{Proc. of the {IEEE} International
  Conference on Robotics and Automation}, St. Paul, MN, May 14--18, 2012, pp.
  828--835.

\bibitem{Hesch2013WAFR}
J.~Hesch, D.~Kottas, S.~Bowman, and S.~Roumeliotis, ``Towards consistent
  vision-aided inertial navigation,'' in \emph{Algorithmic Foundations of
  Robotics X}, ser. Springer Tracts in Advanced Robotics, E.~Frazzoli,
  T.~Lozano-Perez, N.~Roy, and D.~Rus, Eds.\hskip 1em plus 0.5em minus
  0.4em\relax Springer Berlin Heidelberg, 2013, vol.~86, pp. 559--574.

\bibitem{Trawny2005_Q_TR}
N.~Trawny and S.~I. Roumeliotis, ``Indirect {K}alman filter for {3D} attitude
  estimation,'' University of Minnesota, Dept. of Comp. Sci. \& Eng., Tech.
  Rep., Mar. 2005.

\bibitem{Zuo2017IROS}
X.~Zuo, J.~Xie, Y.~Liu, and G.~Huang, ``Robust visual {SLAM} with point and
  line features,'' in \emph{Proc. of the IEEE/RSJ International Conference on
  Intelligent Robots and Systems}, Vancouver, Canada, Sep. 24-28, 2017.

\bibitem{Weingarten2006PHD}
J.~Weingarten, ``Feature-based 3d slam,'' Ph.D. dissertation, EPFL, Lausanne,
  2006.

\bibitem{Hartley2004}
R.~Hartley and A.~Zisserman, \emph{Multiple View Geometry in Computer
  Vision}.\hskip 1em plus 0.5em minus 0.4em\relax Cambridge University Press,
  2004.

\bibitem{Weingarten2004ICRA}
J.~W. Weingarten, G.~Gruener, and R.~Siegwart, ``Probabilistic plane fitting in
  3d and an application to robotic mapping,'' in \emph{International Conference
  on Robotics and Automation}, vol.~1, New Orleans, Louisiana, April 2004, pp.
  927--932 Vol.1.

\bibitem{Proenca2017CORR}
P.~F. Proen{\c{c}}a and Y.~Gao, ``Probabilistic {RGB-D} odometry based on
  points, lines and planes under depth uncertainty,'' \emph{CoRR}, vol.
  abs/1706.04034, 2017.

\bibitem{Kanatani1996}
K.~Kanatani, \emph{Statistical optimization for geometric computation : theory
  and practice}.\hskip 1em plus 0.5em minus 0.4em\relax Amsterdam ; New York :
  Elsevier, 1996.

\bibitem{Bar-Shalom2001}
Y.~Bar-Shalom, X.~R. Li, and T.~Kirubarajan, \emph{Estimation with Applications
  to Tracking and Navigation}.\hskip 1em plus 0.5em minus 0.4em\relax New York:
  John Willey and Sons, 2001.

\bibitem{Yang2017ISRR}
Y.~Yang and G.~Huang, ``Map-based localization under adversarial attacks,'' in
  \emph{Proc. of the International Symposium on Robotics Research}, Puerto
  Varas, Chile, Dec.11-14, 2017.

\bibitem{Yang2017SSRR}
Z.~Yang and S.~Shen, ``Monocular visual-inertial fusion with online
  initialization and camera-imu calibration,'' in \emph{IEEE International
  Symposium on Safety, Security, and Rescue Robotics (SSRR)}, West Lafayette,
  Indiana, Oct 2015, pp. 1--8.

\bibitem{Panahandeh2013ros}
G.~Panahandeh, C.~X. Guo, M.~Jansson, and S.~I. Roumeliotis, ``Observability
  analysis of a vision-aided inertial navigation system using planar features
  on the ground,'' in \emph{International Conference on Intelligent Robots and
  Systems}, Nov 2013, pp. 4187--4194.

\bibitem{Shuster1993JAS}
M.~D. Shuster, ``A survey of attitude representations,'' \emph{Journal of the
  Astronautical Sciences}, vol.~41, no.~4, pp. 439--517, Oct.-Dec. 1993.

\end{thebibliography}
